\documentclass[11pt]{amsart}

\usepackage{amsmath,amssymb,verbatim,eucal, amscd,color,mathrsfs}\usepackage{graphicx}
\usepackage{colordvi}
\usepackage[all,cmtip]{xy}
\usepackage[colorlinks=true, pdfstartview=FitV,
linkcolor=blue, citecolor=blue, urlcolor=red]{hyperref}
\usepackage{amsfonts,latexsym}
\usepackage{enumerate}

\usepackage{geometry,tikz,amsmath}

\usetikzlibrary{fit,matrix,positioning,intersections}

\input xy

\numberwithin{equation}{section}

\theoremstyle{plain}

\newtheorem{theorem}{Theorem}[section]				
\newtheorem{proposition}[theorem]{Proposition}		
\newtheorem{corollary}[theorem]{Corollary}
\newtheorem{lemma}[theorem]{Lemma}

\newtheorem{question}[theorem]{Question}
\newtheorem{notation}[theorem]{Notation}

\theoremstyle{definition}

\newtheorem{definition}[theorem]{Definition}
\newtheorem{remark}[theorem]{Remark}
\newtheorem{example}[theorem]{Example}

\newcommand\ep{\epsilon}

\newcommand\frg{\mathfrak{g}}
\newcommand\frh{\mathfrak{h}}

\newcommand{\CBbb}{\mathbb C}

\newcommand{\NBbb}{\mathbb N}

\newcommand{\PBbb}{\mathbb P}
\newcommand{\QBbb}{\mathbb Q}
\newcommand{\RBbb}{\mathbb R}

\newcommand{\ZBbb}{\mathbb Z}

\newcommand{\Acal}{\mathcal A}
\newcommand{\Bcal}{\mathcal B}
\newcommand{\Ccal}{\mathcal C}

\newcommand{\Gcal}{\mathcal G}

\newcommand{\Kcal}{\mathcal K}
\newcommand{\Lcal}{\mathcal L}
\newcommand{\Mcal}{\mathcal M}

\newcommand{\Ocal}{\mathcal O}

\newcommand{\Rcal}{\mathcal R}

\newcommand{\Tcal}{\mathcal T}

\newcommand{\Vcal}{\mathcal V}

\newcommand{\Xcal}{\mathcal X}
\newcommand{\Ycal}{\mathcal Y}

\newcommand{\Tscr}{\mathscr T}

\newcommand{\GO}{\mathring{\mathfrak{g}}}
\newcommand{\GA}{\mathfrak{g}(X_N^{(m)})}

\DeclareMathOperator{\Hom}{Hom}
\DeclareMathOperator{\Rep}{Rep}
\DeclareMathOperator{\Aut}{Aut}

\DeclareMathOperator{\id}{id}

\DeclareMathOperator{\rank}{rank}

\DeclareMathOperator{\tr}{tr}

\makeatletter
\newcommand*\bigcdot{\mathpalette\bigcdot@{.5}}
\newcommand*\bigcdot@[2]{\mathbin{\vcenter{\hbox{\scalebox{#2}{$\m@th#1\bullet$}}}}}
\makeatother

\makeatletter
\newcommand{\thickcolon}{\mathpalette\thick@colon\relax}
\newcommand{\thick@colon}[2]{%
	\mspace{1mu}%
	\vbox{%
		\hbox{$\m@th#1\bigcdot$}
		\nointerlineskip
		\kern.15ex
		\hbox{$\m@th#1\bigcdot$}
		\kern-.55ex
	}%
	\mspace{1mu}%
}
\makeatother

\newcommand\reg{\operatorname{reg}}

\newcommand\rk{\operatorname{rk}}

\newcommand{\op}{\operatorname{op}}

\usepackage{bbm}
\usepackage{bm}
\usepackage {mathtools}
\newcommand{\rar}[1]{\stackrel{#1}{\longrightarrow}}
\newcommand{\xoto}{\xrightarrow}
\newcommand\tildeC{\widetilde{C}}
\newcommand\tildeL{\widetilde{\Lcal}}
\newcommand\tildeD{\widetilde{D}}
\newcommand\tildeq{\widetilde{q}}
\newcommand\wtilde{\widetilde}
\newcommand\Qab{\mathbb{Q}^{\operatorname{ab}}}
\newcommand\KQab{K_{\Qab}}
\newcommand{\KC}{K_{\CBbb}}
\newcommand\unit{\mathbbm{1}}
\newcommand\Irr{\operatorname{Irr}}
\newcommand\linfun{\nu}
\newcommand\barM{\overline{\Mcal}}

\newcommand\tildeM{\widetilde{\Mcal}}
\newcommand\hatM{\widehat{\Mcal}}

\DeclareMathOperator{\Mod}{Mod}
\DeclareMathOperator{\coev}{coev}
\DeclareMathOperator{\gen}{gen}
\DeclareMathOperator{\Sp}{Sp}
\DeclareMathOperator*{\bigboxtimes}{\text{\Large $\boxtimes$}}
\DeclareMathOperator{\ev}{ev}
\newcommand{\GMod}{\overline{\mathcal{M}}{}^{\Gamma}_{g,n}}
\newcommand{\GModhat}{\widehat{\overline{\mathcal{M}}}{}^{\Gamma}_{g,n}}
\newcommand{\GModtilde}{\widetilde{\overline{\mathcal{M}}}{}^{\Gamma}_{g,n}}

\newcommand{\GModtildeA}{\widetilde{\overline{\mathcal{M}}}{}^{\Gamma}_{g,A}}
\newcommand{\Spec}{\operatorname{Spec}}
\newcommand{\hatbarM}{\widehat{\overline{\Mcal}}}
\newcommand{\tildebarM}{\widetilde{\overline{\Mcal}}}
\renewcommand{\subset}{\subseteq}
\newcommand{\onto}{\twoheadrightarrow}
\newcommand{\Gm}{\mathbb{G}_{\operatorname{m}}}

\newcommand{\confv}{\mathcal{V}_{\vec{\lambda}}(\mathbb{P}^1, \vec{z})}
\newcommand{\wfrg}{\widehat{\mathfrak{g}}}
\newcommand{\Kor}{\mathcal{K}_{\widetilde{C}, \widetilde{p}_i}}
\newcommand{\COV}{\mathbb{V}_{\vec{\lambda},\Gamma}(\widetilde{C},C,\widetilde{\bf{p}},{\bf{p}})}
\newcommand{\COVx}{\mathbb{V}_{\vec{\lambda},\Gamma}}
\newcommand{\VOC}{\mathbb{V}^{\dagger}_{\vec{\lambda},\Gamma}(\widetilde{C},C,\widetilde{\bf{p}},{\bf{p}})}

\newcommand{\sign}{\operatorname{sign}}
\newcommand{\dotimes}{\dot{\otimes}}

\newcommand{\Surf}{\mathfrak{Surf}}
\renewcommand{\H}{\pi_0\partial}
\renewcommand{\Vec}{\operatorname{Vec}}

\AtBeginDocument{%
	\def\MR#1{}
}
\begin{document}


\title[Verlinde formula for twisted conformal blocks]{Crossed modular categories and the Verlinde formula for twisted conformal blocks}
\author{Tanmay Deshpande}

\address{Tata Institute of Fundamental Research,
	Homi Bhabha Road, 
	Mumbai 400005, INDIA}
\email{tanmay@math.tifr.res.in}
%

\author{Swarnava Mukhopadhyay }
\address{Tata Institute of Fundamental Research,
	Homi Bhabha Road, 
	Mumbai 400005, INDIA}
\email{swarnava@math.tifr.res.in}


\begin{abstract}
In this paper we give a Verlinde formula for computing the ranks of the bundles of twisted conformal blocks associated with a simple Lie algebra equipped with an action of a finite group $\Gamma$ and a positive integral level $\ell$ under the assumption  ``$\Gamma$ preserves a Borel".  For $\Gamma=\mathbb{Z}/2$ and double covers of $\mathbb{P}^1$, this formula was conjectured by Birke-Fuchs-Schweigert \cite{BFS}. As a motivation for this Verlinde formula, we prove a categorical Verlinde formula  which computes the fusion coefficients for any $\Gamma$-crossed modular fusion category as defined by Turaev. 

We relate these two versions of the Verlinde formula, by formulating the notion of a $\Gamma$-crossed modular functor and show that it is very closely related to the notion of a $\Gamma$-crossed modular fusion category. We compute the Atiyah algebra and prove (with same assumptions) that the bundles of $\Gamma$-twisted conformal blocks associated with a twisted affine Lie algebra define a $\Gamma$-crossed modular functor. 

We also prove a useful criterion for rigidity of weakly fusion categories to deduce that the level $\ell$ $\Gamma$-twisted conformal blocks define a $\Gamma$-crossed modular fusion category. Along the way, we prove the equivalence between a $\Gamma$-crossed modular functor and its topological analogue. We then apply these results to derive the  Verlinde formula for twisted conformal blocks. We also  describe the S-matrices of the $\Gamma$-crossed modular fusion categories associated with twisted conformal blocks.

\end{abstract}

\maketitle
\renewcommand\contentsname{}
\setcounter{tocdepth}{1}
\tableofcontents




\section{Introduction}The Wess-Zumino-Witten (WZW) model is a two dimensional rational conformal theory and conformal blocks associated to these models were explicitly constructed in the phenomenal works of Tsuchiya, Ueno, and Yamada \cite{TUY:89}. To a simple Lie algebra $\frg$, a positive integer $\ell$, an $n$-tuple $\vec{\lambda}$ of dominant integral weights of level $\ell$  (see Section \ref{sec:repofrep}) of an untwisted affine Kac-Moody Lie algebra, the WZW model in \cite{TUY:89} associates a vector bundle $\mathbb{V}^{\dagger}_{\vec{\lambda}}(\frg,\ell)$ of finite rank on the Deligne-Mumford-Knudsen moduli stack $\overline{\mathcal{M}}_{g,n}$ of stable $n$-pointed curves of genus $g$.  These vector bundles $\mathbb{V}^{\dagger}_{\vec{\lambda}}(\frg,\ell)$  are known as the {\em bundles of conformal blocks} and their duals $\mathbb{V}_{\vec{\lambda}}(\frg,\ell)$  are referred to as the {\em sheaf of covacua}. 

Moreover, these bundles are endowed with a flat projective connection with logarithmic singularities along the boundary divisor of $\overline{\mathcal{M}}_{g,n}$ and satisfy \cite{TUY:89} the axioms of conformal field theory like {\em factorization} and {\em propagation of vacua} (see Section \ref{sec:deftwistconf}). The spaces of conformal blocks can be identified with the global sections of line bundles on moduli of $G$-bundles on a curve \cite{BF,BL,Faltings:94,KNR:94,LaszloSorger:97} as well as refinements of invariants of representations of a Lie algebra $\frg$ \cite{Fakhruddin:12,TUY:89}.

In 1987, E. Verlinde conjectured \cite{Verlinde} an explicit formula to compute the rank of the bundle $\mathbb{V}_{\vec{\lambda}}(\frg,\ell)$ which became well known as the Verlinde formula.  
Verlinde's conjectural rank formula was proved for $\operatorname{SL}(2)$ independently by the works of Bertram, Bertram-Szenes, Daskalopoulos-Wentworth and M. Thaddeus  \cite{Bert,Bert-Sz,DW,Thaddeus:92}. 
The Verlinde formula for classical groups and $\operatorname{G}_2$  was proved  by  G. Faltings \cite{Faltings:94} and for all groups by  C. Teleman  \cite{Teleman}.  Later,  it was also proved by Alekseev-Meinrenken-Woodward, Jeffrey-Kirwan \cite{AWM,JK}. 
The  Verlinde conjecture in full generality has been proved by Huang \cite{Huang:08b,Huang:08a} in a general set-up of vertex algebras. 
Verlinde's formula is a central and a distinguishing feature in this subject and we
refer to \cite{Abe:08,Belkale:08,Mukhopadhyay:12,MukhopadhyayWentworth:16,NaculichSchnitzer:90} for some applications to representation theory and algebraic geometry. 
We now discuss the twisted set-up.

Twisted WZW models associated to order two diagram automorphisms of a simple Lie algebra were constructed and studied by Shen-Wang \cite{WangShen} and also by Birke-Fuchs-Schweigert \cite{BFS}. Starting with a vertex algebra along with an action of a finite group $\Gamma$, Frenkel-Szczesny \cite{FS} constructed orbifold conformal blocks. More precisely, given a pointed  $\Gamma$-cover $(\widetilde{C},C,\widetilde{\bf{p}}, {\bf{p}})$ of a smooth curve $C$ with $n$ marked points ${\bf p}=(p_1, \dots, p_n)$ along with a lift ${\widetilde{\bf p}}$, they attached a module for a twisted vertex algebra. This construction can be  (Szczesny \cite{szcz}) be carried out in families of stable curves to get a quasi-coherent sheaf $\mathbb{V}_{\Gamma}$ on the stack $\GMod$ of pointed $\Gamma$-cover introduced in Jarvis-Kimura-Kaufmann \cite{JKK} along with a flat projective connection on the open part $\mathcal{M}{}^{\Gamma}_{g,n}$. Recently, extensions to the boundary of $\GMod$ has been considered by several authors.   

In her thesis, C. Damiolini \cite{Dam17} considered twisted conformal blocks associated to $\Gamma$-covers of curves where the marked points \'etale. She proved factorization, propagation of vacua, existence of projective connections and local freeness under the assumption that $|\Gamma|$ is of prime order.


In \cite{KH}, Hong-Kumar studied  $\Gamma$-twisted WZW-models associated to an arbitrary $\Gamma$-twisted affine Kac-Moody Lie algebra under the assumption that ``$\Gamma$ preserves a Borel subalgebra of $\frg$". They construct a locally free sheaf over $\GMod$ that satisfies ``factorization" and  ``propagation of vacua". They further show that on $\mathcal{M}^{\Gamma}_{g,n}$ twisted conformal blocks admit a flat projective connection. Similar results for diagram automorphisms of $\frg$ were also obtained by the second coauthor independently.

We now briefly recall the notion of twisted conformal blocks  \cite{Dam17,KH} and refer the reader to Section \ref{sec:deftwistconf} for a coordinate free construction. Let $\frg$ be a simple Lie algebra equipped with a $\Gamma$-action $\Gamma \rightarrow \operatorname{Aut}(\frg)$.  
For any $\gamma \in \Gamma$, we consider the twisted affine Kac-Moody Lie algebra (see Section \ref{sec:affineLiealgtwisted}) $\widehat{L}(\frg,\gamma)$.  The set of irreducible, integrable, highest weight representations of $\widehat{L}(\frg,\gamma)$ of level $\ell\in \ZBbb_{\geq 1}$ is denoted by $P^{\ell}(\frg,\gamma)$ (see Section \ref{sec:affineLiealgtwisted}). If $\gamma$ is trivial, then $P^{\ell}(\frg, \id)$ will often be denoted by $P_{\ell}(\frg)$. 

Now consider a stable nodal curve $\widetilde{C}$ with $n$ marked points $\widetilde{\bf{p}}$ with a $\Gamma$ action. Assume that $\widetilde{C}\setminus \Gamma\cdot\widetilde{\bf{p}}$ is affine on which $\Gamma$ acts freely and  each component has at least one marked point. Consider the corresponding pointed $\Gamma$-cover $(\pi:\wtilde{C}\to C,\widetilde{\bf{p}},{\bf{p}})$. Note that for each point $\wtilde{c}\in \wtilde{C}$, its stabilizer subgroup $\Gamma_{\wtilde{c}}\leq \Gamma$ is cyclic. This is because we can find a small disk $D\subset C$ in the base curve containing $\pi(\wtilde{c})$ such that over the punctured disk $D\setminus\wtilde{c}$, $\pi$ is an unramified $\Gamma$-cover and noting that the finite connected covers of the punctured disk are necessarily cyclic.   

Let $(\gamma_1, \dots ,\gamma_n)$ be generators of stabilizers in $\Gamma$ of the points $(\widetilde{p}_1,\dots, \widetilde{p}_n)$ determined by using the orientation on the complex curve $\tildeC$. The element $\gamma_i$ will be called the monodromy around the point $\wtilde{p}_i$. Suppose for 
each point $\wtilde{p}_i$ we have attached a highest weight integrable module $\mathcal{H}_{\lambda_i}(\frg,\gamma_i)$ of weight $\lambda_{i} \in P^{\ell}(\frg,\gamma_i)$. Let $$\mathcal{H}_{\vec{\lambda}}:=\mathcal{H}_{\lambda_1}(\frg, \gamma_1)\otimes \dots \otimes \mathcal{H}_{\lambda_n}(\frg, \gamma_n).$$ Then the  {\em space of twisted covacua}  can be  defined as 
$$\mathcal{V}_{\vec{\lambda},\Gamma}(\widetilde{C}, C, \widetilde{\bf p}, {\bf p},{\bf{\widetilde{z}}}):=\mathcal{H}_{\vec{\lambda}}/\big(\frg\otimes H^0(\widetilde{C},\mathcal{O}_{\widetilde{C}}(\ast \Gamma.{\widetilde{\bf p}})\big)^{\Gamma}\mathcal{H}_{\vec{\lambda}},$$ where ${\bf{\widetilde{z}}}$ denote a choice of a formal parameter along the points ${\bf{\widetilde{p}}}$. The corresponding vector bundles on $\GMod$ will be denoted by $\mathbb{V}_{\vec{\lambda},\Gamma}(\widetilde{C}, C, \widetilde{\bf p}, {\bf p})$.
Starting with the work of Pappas-Rapoport \cite{PR2,PR1}, the moduli stack of Bruhat-Tits torsor $\mathscr{G}$ associated to a pair $(G, \Gamma\subset \operatorname{Aut}(G))$ has been studied by several authors \cite{BalajiSesh,Hein}. Results of Hong-Kumar \cite{KH} (with some restrictions on level) and Zelaci \cite{Zel} (order two automorphisms of $\operatorname{SL}(r)$) connect twisted conformal blocks with global sections of line bundles on $\operatorname{Bun}\mathscr{G}$. Pappas-Rapoport \cite{PR2,PR1} ask the following: 
\begin{question}
	Is there is a Verlinde formula for these spaces of non-abelian theta functions or twisted conformal blocks?
	
\end{question}

Birke-Fuchs-Schweigert \cite{BFS} conjectured a twisted Verlinde formula for $\Gamma=\mathbb{Z}/2$ and a double cover of $\mathbb{P}^1$ ramified at two points. 
In this article, we prove a twisted Verlinde formula that computes the rank of the $\Gamma$-twisted  conformal blocks for an arbitrary finite group $\Gamma$ with the assumption that ``$\Gamma$ preserves a Borel subalgebra of $\frg$". We now state one of the main results in this paper--a general twisted Verlinde formula.

Let $(C,\mathbf{p})$ be a smooth genus $g$ curve with $n$marked points $\mathbf{p}$. The fundamental group of $C\setminus\mathbf{p}$  has a presentation of the form 
$$\pi_1(C\setminus\mathbf{p},\star)=\langle\alpha_1,\beta_1,\cdots,\alpha_g,\beta_g,\gamma_1,\cdots,\gamma_n|[\alpha_1,\beta_1]\cdots[\alpha_g,\beta_g]\gamma_1\cdots\gamma_n=1\rangle.$$ 
\noindent {\it Assumption on $\widetilde{C}$.}
Let us fix a group homomorphism 
\begin{equation}\label{eq:grouphomochi}\chi:\pi_1(C\setminus\mathbf{p},\star)\onto \Gamma^{\circ}\subset \Gamma\end{equation} 
with image $\Gamma^\circ$. Let $m_i\in \Gamma$ be the image of the loop $\gamma_i \in \pi_1(C\setminus\mathbf{p},\star)$. This determines (see Section \ref{sec:arbgenusdescription} and also \cite[\S2.3]{JKK}) an $n$-pointed admissible $\Gamma$-cover $(\tildeC\to C,\wtilde{\mathbf{p}},\mathbf{p})$ such that all the lifts $\wtilde{\mathbf{p}}$ lie in the {\em same connected component of $\tildeC$} and the monodromy around the points $\wtilde{\mathbf{p}}$ is given by $(m_1,\cdots,m_n)$. As before, let us assume that we have an action of $\Gamma$ on the simple lie algebra $\frg$ and we fix a level $\ell\in \ZBbb_{\geq 1}$. Let $\vec{\lambda}=(\lambda_1,\cdots\lambda_n)$ with $\lambda_i\in P^{\ell}(\frg,m_i)$. We state the following Verlinde formula: 
\begin{theorem}\label{conj:main1}Assume that ``$\Gamma$ preserves a Borel subalgebra of $\frg$". Then the  rank of the twisted conformal blocks bundle $\mathbb{V}_{\vec{\lambda},\Gamma}(\widetilde{C},C, \widetilde{{\bf{p}}}, {\bf p})$ at level $\ell$ is given by the following formula:
	\begin{equation}\label{eq:twistedverlindemain}\operatorname{rank}\mathbb{V}_{\vec{\lambda},\Gamma}(\widetilde{C},C, \widetilde{{\bf{p}}}, {\bf p})=\sum_{\mu \in P_{\ell}(\frg)^{\Gamma^\circ}}\frac{{S}_{\lambda_1,\mu}^{{m_1}}\cdots {S}_{\lambda_n,\mu}^{m_n}}{\big({S_{0,\mu}}\big)^{n+2g-2}},\end{equation}
	where the summation is over the set of the untwisted dominant integral level-$\ell$ weights of $\frg$ fixed by the subgroup $\Gamma^\circ\subset \Gamma$,  the matrices ${S}^{m_i}$  are the $m_i$-crossed S-matrices  at level $\ell$ corresponding to the diagram automorphism class of $m_i\in \Gamma$ and which are described (see also Chapter 13 in \cite{Kac}) explicitly in Proposition \ref{prop:charactertableviaSmatrix} (see also Corollary \ref{cor:signinverlinde}) and where $0$ is the trivial weight in $P_\ell(\frg)$.
\end{theorem}

\begin{remark}\label{rem:gammapreservesborel}
	We conjecture that the same Verlinde formula holds without the assumption that ``$\Gamma$ preserves a Borel subalgebra of $\frg$", however some 2-cocycles may appear in the general case. This is motivated by the equivariantization$-$de-equivariantization correspondence \cite{DGNO} between  $\Gamma$-crossed modular tensor categories and modular tensor categories containing $\Rep(\Gamma)$. The corresponding modular functor associated with rational vertex algebras arising from orbifold constructions has been considered by Nagatomo-Tsuchiya \cite{NagTsu} and others. Hence we expect the twisted conformal blocks to satisfy the factorization and propagation of vacua axioms without the assumption of $\Gamma$ preserving a Borel. This assumption appears here since it appears in the paper of Hong-Kumar \cite{KH}, where the proofs of the {\em factorization theorem} and {\em propagation of vacua} use this assumption. All our results will generalize if this assumption is dropped in the work \cite{KH}.
\end{remark}

\begin{remark}
	It may appear a priori from the Verlinde formula (\ref{eq:twistedverlindemain}) that the ranks of the twisted conformal blocks do not depend upon the connected component of the $\GMod(\mathbf{m})$. But the formula is dependent on the choice a group homomorphism $\chi:\pi_1(C\backslash  {\bf p},\star)\rightarrow \Gamma$ from (\ref{eq:grouphomochi}), namely it depends on the image $\Gamma^\circ$ of $\chi$ (along with the monodromy data $\mathbf{m}$). The fact that the ranks only depend on the image $\Gamma^\circ$ is a curious consequence of the triviality of the 2-cocycles (mentioned in Remark \ref{rem:gammapreservesborel}) in the case when $\Gamma$ preserves a Borel subalgebra. We also refer to the general categorical analog (Corollary \ref{cor:highergenus}) where we see that the right hand side will in general depend on the homomorphism $\chi$ and not just its image.  We refer the reader to Section \ref{sec:arbgenusdescription} and \cite[\S2.3]{JKK} for the relations between covers and holonomies $\chi$. 
\end{remark}

\begin{remark}
	It is well known that the untwisted S-matrix $S$ of type $\frg$ and level $\ell$ is a $P_{\ell}(\frg)\times P_{\ell}(\frg)$ symmetric unitary matrix. On the other hand,  for each $m\in \Gamma$, the $m$-crossed S-matrix $S^m$ is a $P^{\ell}(\frg,m)\times P_{\ell}(\frg)^m$ matrix. It is not immediately clear, but nevertheless true that $S^m$ is a square matrix. Moreover, $S^m$ is also unitary.
\end{remark}

Theorem \ref{conj:main1} is proved in Section \ref{sec:modfuntwistedconformal} (see Corollary \ref{cor:main1}). We also refer the reader to Section \ref{sec:dimensioncrosscheck1} where we give a more concrete form of the formula (\ref{eq:twistedverlindemain}) in some special cases.  We now discuss the remaining results of 
this article which are motivated by Theorem \ref{conj:main1} and also the main steps in the proof of Theorem \ref{conj:main1}.  We have divided our paper into three parts.\\

\noindent{\bf Part I.}
There has been a comprehensive study by several authors \cite{MS1,Segal,TUY:89,Witten1} trying to understand the relations between modular tensor categories, 3-dimensional topological quantum field theory and 2-dimensional modular functors. A key bridge between these three topics has been conformal blocks associated to untwisted affine Lie algebras. Moore-Seiberg \cite{MS1} has proved a Verlinde formula for any modular fusion category. In the twisted setup, $\Gamma$-crossed modular fusion categories have been introduced by V. Turaev \cite{TuraevHQFT}.\\

\noindent {\it Rigidity of $\Gamma$-graded categories.}
We first prove a rigidity result in the categorical set-up of weakly rigid categories. Weakly rigid categories are sometimes also referred to as Grothendieck-Verdier categories or $\ast$-autonomous categories (see also \cite{BoDr:13}). We refer the reader to Section \ref{sec:drinfeldsquestions} for relevant definitions and a proof (Corollary \ref{cor:rigidity of graded categories}) of the following: 
\begin{theorem}\label{thm:catrigid}
	Let $\Gamma$ be a finite group and $\mathcal{C}=\oplus_{\gamma \in \Gamma}\Ccal_{\gamma}$ be a $\Gamma$-graded weakly fusion category such that the identity component of $\mathcal{C}_1$ is rigid. Then $\mathcal{C}$ is rigid. 
\end{theorem}

\noindent {\it Categorical Verlinde formula.}
Motivated by the result of Moore-Seiberg \cite{MS1} and Verlinde \cite{Verlinde}, next we prove a Verlinde formula to compute the fusion rules for $\Gamma$-crossed modular fusion categories introduced by Turaev \cite{TuraevHQFT}. This problem of determining fusion rules of $\Gamma$-extensions of fusion categories and of $\Gamma$-crossed modular categories is also of independent interest and has also been considered in the contemporaneous work of several authors (see Bischoff-Jones \cite{BisJon}).

In \cite{Desh:17,Desh:18a}, the first coauthor introduced for each $\gamma \in \Gamma$, the notion of a categorical $\gamma$-crossed S-matrix denoted by $S^{\gamma}$ and twisted characters for any $\Gamma$-graded Frobenius $\ast$-algebra arising from a Grothendieck group $K(\Ccal)$ of a braided $\Gamma$-crossed fusion category $\Ccal$. We prove a general twisted Verlinde formula  in the setting of $\Gamma$-crossed braided fusion categories which computes fusion coefficients in terms of the crossed S-matrices. This generalizes the work of the first coauthor \cite{Desh:18}.  We refer to Sections \ref{sec:braidedcrossed}, \ref{sec:catverlinde} and Theorem \ref{thm:twisted verlinde} for more details and also to Corollary \ref{cor:highergenus} for a higher genus analog of the following 3-point genus zero version.
\begin{theorem}\label{thm:verlindeforbraided}
	Let $\Ccal=\bigoplus\limits_{\gamma\in \Gamma}\Ccal_\gamma$ be a $\Gamma$-crossed modular fusion category. Let $\gamma_1,\gamma_2\in \Gamma$ and let $A\in \Ccal_{\gamma_1}, B\in \Ccal_{\gamma_2}$ and $C\in \Ccal_{\gamma_1\gamma_2}$ be simple objects. Then the multiplicity $\nu_{A,B}^C$ of $C$ in the tensor product $A\otimes B$ is given by 
\begin{equation}
\label{eqn:verlindecateqnintro}
\nu_{A,B}^C=\sum\limits_{D\in P_1^{\langle\gamma_1,\gamma_2\rangle}}\frac{\overline{S^{\gamma_1}_{A,D}}\cdot\overline{S^{\gamma_2}_{B,D}}\cdot S^{\gamma_1\gamma_2}_{C,D}}{S_{\unit,D}}=\sum\limits_{D\in P_1^{\langle\gamma_1,\gamma_2\rangle}}\frac{{S^{\gamma_1}_{A,D}}\cdot{S^{\gamma_2}_{B,D}}\cdot\overline{S^{\gamma_1\gamma_2}_{C,D}}}{S_{\unit,D}},
\end{equation}
	where the summation is over the simple objects $D\in \Ccal_1$ fixed by both $\gamma_1,\gamma_2$ and where the crossed S-matrices are chosen in a compatible way.
\end{theorem}

\begin{remark}
	In general, a $\gamma$-crossed S-matrix depends on certain choices and is only well defined up to rescaling rows by roots of unity. In the formula given by \eqref{eqn:verlindecateqnintro}, we have assumed that the crossed S-matrices are chosen in a compatible way. In general some cocycles appear in the Verlinde formula Corollary \ref{cor:highergenus}. However, in the set up of $\Gamma$-twisted conformal blocks with $\Gamma$ preserving a Borel, the cocycles are trivial  (see also Remark \ref{rem:cocyleszero}).
\end{remark}

We also study (see Section \ref{sec:twistedfusion}) a categorical twisted fusion ring $K_{\mathbb{C}}(\Ccal_1, \gamma)$ for every $\gamma \in \Gamma$. We further show that for each $\gamma$, the crossed $S$-matrix $S^{\gamma}$ determines the character table (Theorem \ref{thm:chartableoftwistedfusion}) of the ring $K_{\mathbb{C}}(\Ccal_1,\gamma)$. The following remark is important:
\begin{remark}
	The non-negative integer $\nu_{A,B}^{C}$ is a fusion coefficient in the Grothendieck ring $K(\Ccal)$ of the $\Gamma$-crossed category and not in the categorical twisted fusion ring $K_{\mathbb{C}}(\Ccal_1,\gamma)$ where the fusion coefficients may be non-integral in general. 
\end{remark}

\noindent{\bf Part II.}
In Section \ref{sec:deftwistconf}, we give a coordinate free construction of the twisted conformal blocks and discuss the associated descent data coming form {\em propagation of vacua}. We show that like the untwisted case \cite{Fakhruddin:12}, twisted conformal blocks with at least one trivial weight are pull backs (see Proposition \ref{prop:importantprop}) along along forgetful-stabilization morphism  $\overline{\mathcal{M}}{}^{\Gamma}_{g,n+1}({\bf m},1)\rightarrow \GMod({\bf m})$. We prove the following: 
\begin{theorem}\label{thm:baby}
	Suppose that $\Gamma$ preserves a Borel subalgebra of $\frg$.	Twisted conformal blocks give a $\Gamma$-crossed modular fusion category $\Ccal(\frg,\Gamma,\ell)=\oplus_{\gamma \in \Gamma} \Ccal_{\gamma}$ such that the simple objects of each $\Ccal_\gamma$ are parameterized by the set $P^{\ell}(\frg,\gamma)$ and $$\nu_{\lambda_1,\lambda_2}^{\lambda_3^*}=\dim_{\mathbb{C}}\mathcal{V}_{\lambda_1,\lambda_2,\lambda_3,\Gamma}(\widetilde{C},\mathbb{P}^1, \widetilde{\bf p},{\bf p},\widetilde{z}),$$ where $\lambda_i$ is a simple object of $\Ccal_{\gamma_i}$,  $\gamma_1.\gamma_2.\gamma_3=1$ and $\widetilde{\bf p}$ lies in the same connected component of $\widetilde{C}$.
\end{theorem}
We refer the reader to Section \ref{sec:threepointmtc} for a more  precise version (Theorem \ref{thm:oneofmain}) of Theorem \ref{thm:baby}. The proof is given in Part III using the notion of complex analytic $\Gamma$-crossed  modular functors. 

Now given Theorem \ref{thm:baby}, by the categorical Verlinde formula in Theorem \ref{thm:verlindeforbraided}, we get a formula to compute the dimension of the three pointed twisted conformal block with the base curve $\mathbb{P}^1$. However to complete the proof of Theorem \ref{conj:main1}, we still need to determine the crossed S-matrices. \\

\noindent {\it Compute categorical crossed $S$-matrices.}
We consider the twisted fusion ring $\mathcal{R}_{\ell}(\frg,\gamma)$ coming from finite order  automorphisms of the Lie algebra $\frg$. This ring was also studied by J. Hong \cite{Hong,Hong1}. In the set up of diagram automorphism, 
Hong gives a complete description of the character table of $\mathcal{R}_{\ell}(\frg,\sigma)$ in terms 
of the characters of the representations rings of the fixed point algebra ${\frg_{\sigma}}$ (see Fuchs-Schweigert \cite{FS}).

In Section \ref{sec:arbtwistedfusion}, we observe that 
the twisted fusion ring $\mathcal{R}_{\ell}(\frg,\gamma)$ constructed from tracing automorphisms of untwisted conformal blocks only depends on the diagram automorphism class $\sigma$ of $\gamma$. 
We also observe in (see Proposition \ref{cor:corofoneofmain}), that for each $\gamma \in\Gamma$, the twisted Kac-Moody fusion ring $\Rcal_{\ell}(\frg,\gamma)$ and the categorical twisted fusion ring $K_{\mathbb{C}}(\Ccal(\frg,\ell),\gamma)$ agree.   We apply these results (along with  Proposition \ref{prop:twistedviaKMtwisted}) to compute the categorical $\gamma$-crossed S-matrices that appear in the categorical Verlinde formula for the case of twisted conformal blocks in terms of the character of the ring $\Rcal_{\ell}(\frg,\gamma)$.  This completes the overview of the steps in the proof of Theorem \ref{conj:main1}. 
We are now left to discuss the proof of Theorem \ref{thm:baby} and or its more precise version Theorem \ref{thm:oneofmain}. 

\begin{remark}
	Works of G. Faltings and C. Teleman \cite{Faltings:94,Teleman} determine the set of 
	characters of $R_{\ell}(\frg)$ as restrictions of representations of the Grothendieck ring $R(\frg)$ of representations of the lie algebra $\frg$.  This combined with the works of Yi-Zhi Huang (\cite{Huang:08b}, \cite{Huang:08a}) on rigidity determines the S-matrices in the untwisted setting. We refer the reader to the work of Dong-Lin-Ng \cite{Donlinngs} for a related discussion in the setting of Vertex algebras. 
\end{remark}

\begin{remark}
	A. Ginnory \cite{AG} showed that the fusion coefficients of  $\mathcal{R}_{\ell}(\frg,\sigma)$ could be negative. This negated a suggestion of Hong \cite{Hong} relating it to the dimensions of the twisted conformal blocks. 
	It is important to point out that the statement of Theorem \ref{conj:main1} involves both crossed and uncrossed S-matrices for various elements of $\Gamma$ whereas the Verlinde formula (see also Deshpande\cite[Thm. 2.12(iii)]{Desh:17}) for the structure constants of $\mathcal{R}_{\ell}(\frg,\sigma)$ involves only the $\sigma$-crossed S-matrix. In the categorical set-up (\cite{Desh:17}) for twisted fusion rings, the fusion coefficient may not be integers. 
\end{remark}

\noindent {\it Atiyah algebras of a flat projective connection.} The description of the Atiyah algebras will be required in Part  III to define the notions of $\Gamma$-crossed analytic modular functors. 
Following the approach of the Beilinson-Bernstein localization functors in Frenkel-Ben-Zvi\cite{FBz}, Szczesny \cite{szcz}, we extend the construction of the flat projective connections in \cite{KH,szcz} to a twisted logarithmic connection over the boundary $\Delta_{g,n,\Gamma}$ of $\GMod$. 

Let $\mathbb{L}$ be a line bundle on $\GMod$, we will denote by $\mathcal{A}_{\mathbb{L}}$(resp $\mathcal{A}_{\mathbb{L}}(-\log \Delta_{g,n,\Gamma})$) the Atiyah algebra (respectively logarithmic Atiyah algebra) of differential operators (respectively differential operators preserving $\Delta_{g,n,\Gamma}$) acting on $\mathbb{L}$.
 The following theorem determines the Atiyah algebra (Theorem \ref{thm:atiyahalgebra}) of the twisted logarithmic $\mathcal{D}$-module $\mathbb{V}_{\vec{\lambda},\Gamma}(\widetilde{C}, C, \widetilde{\bf p},{\bf p})$ on $\GMod(m_1,\dots, m_n)$, where $\bm m=(m_1,\dots, m_n)$ is the monodromy data and  $\vec{\lambda}=(\lambda_1,\dots, \lambda_n)$ as in the statement of Theorem \ref{conj:main1}.
\begin{theorem}\label{thm:atiyahalg}Let $\Lambda$ be the pullback of the Hodge line bundle to $\GMod$ under the natural forgetful back from $\GMod \rightarrow \overline{\mathcal{M}}_{g,n}$ and ${\tildeL}_i$'s denote the line bundles corresponding to $i$-th Psi-class on $\GMod$. Then the logarithmic Atiyah algebra
	$$\frac{\ell\dim \frg}{2(\ell+h^{\vee}(\frg))}\mathcal{A}_{\Lambda}(-\log {\Delta}_{g,n,\Gamma})+\sum_{i=1}^{n}N_i\Delta_{\lambda_i}\mathcal{A}_{{\tildeL}_i}(-\log {\Delta}_{g,n,\Gamma})$$ acts on the vector bundle $\mathbb{V}_{\vec{\lambda},\Gamma}(\widetilde{C},C, \widetilde{\bf{p}}, {\bf{p}})$ of twisted conformal blocks at level $\ell$. Here $\Delta_{\lambda_i}$ is the eigenvalue of the twisted Virasoro operators $L_{0,{\langle m_i\rangle }}$ at level $\ell$ of the twisted affine Kac-Moody algebra $\widehat{L}(\frg,m_i)$, $N_i$ is the order the element $m_i\in \Gamma$ and $h^{\vee}(\frg)$ is the dual Coxeter number of $\frg$.
\end{theorem}
The constants $\Delta_{\lambda_i}$ appearing in the statement of Theorem \ref{thm:atiyahalg} are known as {\em trace anomalies} and are explicitly computed in Lemma 3.6 in Wakimoto \cite{Wakimoto}. The Atiyah algebra in Theorem \ref{thm:atiyahalg} appears in our formulation (Section \ref{def:cextenmodularfunctor}) of the axioms of $\mathcal{C}$-extended $\Gamma$-crossed modular functor.

\noindent{\bf Part III.}
In the untwisted set up,  Bakalov-Kirillov \cite{BK:01} introduce the notion of a $2$-dimensional complex analytic modular functor and uses it to construct a weakly ribbon braided tensor category. We also refer the reader to the works of Andersen-Ueno \cite{Andersen-Ueno1,Andersen-Ueno2}.  In \cite{KP:08}, Kirillov-Prince define the notion of a topological $\Gamma$-crossed modular functor. Analogously, in Section \ref{sec:crossedmodular}, we define the notion of a $\Gamma$-crossed complex analytic modular functor generalizing the notion due to \cite{BK:01}.   
We refer the reader to Sections \ref{def:cextenmodularfunctor} and \ref{sec:crossedneutral} for the definition of the notion of a $\Gamma$-crossed modular functor and for a more precise version of the following: 

\begin{theorem}\label{conj:main2}Let $\Ccal$ be a $\mathbb{C}$-linear finite semisimple $\Gamma$-crossed abelian category (i.e. equipped with a $\Gamma$-grading, $\Gamma$-action and some additional structures satisfying certain compatibilities, see Defn. \ref{def:crossedabeliancategory} for details). A $\mathcal{C}$-extended $\Gamma$-crossed complex analytic modular functor defines the structure of a braided $\Gamma$-crossed weakly ribbon category on the category $\Ccal$.  Further if the neutral component $\mathcal{C}_1$ is rigid, then the corresponding $\Gamma$-crossed weakly ribbon category $\Ccal$ is also rigid. 
\end{theorem}
\begin{remark}
	The last part of Theorem \ref{conj:main2} follows by applying Theorem \ref{thm:catrigid} to the first part of Theorem \ref{conj:main2}.
\end{remark}

\noindent{\it $\Gamma$-crossed modular functor from twisted conformal blocks.}
We now address the question of constructing a $\Gamma$-crossed modular fusion category given the action of $\Gamma$ on $\frg$ and a level $\ell$. This is the final step in completing the proof of Theorem \ref{thm:baby}.

For this firstly we need the underlying finite semisimple $\Gamma$-crossed abelian category $\Ccal(\frg,\Gamma,\ell)=\bigoplus\limits_{\gamma\in \Gamma}\Ccal_\gamma$. For $\gamma\in \Gamma$, we take $\Ccal_\gamma$  to be the finite semisimple category whose simple objects are the irreducible, integrable, highest weight representations of $\widehat{L}(\frg,\gamma)$ of level $\ell$ parameterized by $P^{\ell}(\frg,\gamma)$. This is the underlying $\Gamma$-crossed abelian category (see \S\ref{ss:twistedconfblockcrossedabelian} for details) on which we want to define the structure of a $\Gamma$-crossed modular category. Motivated by results of \cite{BFM,BK:01}, we prove the following (see Theorem \ref{conj:conformalblockismodular}) in Section \ref{sec:proofofconfgammafunc}:

\begin{theorem}\label{conj:main3}Let $\Gamma$ be a finite group acting on the simple Lie algebra $\frg$ and let $\ell\in \ZBbb_{\geq 1}$. Suppose that $\Gamma$ preserves a Borel in $\frg$. Then the corresponding $\Gamma$-twisted conformal blocks define a $\mathcal{C}$-extended $\Gamma$-crossed  modular functor. 
\end{theorem}

We refer the reader to Section \ref{sec:proofofconfgammafunc} for the details. Now Theorem \ref{conj:main3} and the first part of Theorem \ref{conj:main2}  provides $\Ccal(\frg,\Gamma,\ell)$ with the structure of a $\Gamma$-crossed weakly ribbon category. However to finish the proof of Theorem \ref{thm:baby}, we  still need to answer the following question: 
\begin{question}
	Is the $\Gamma$-crossed weakly ribbon category $\Ccal(\frg,\Gamma,\ell)$ arising from $\Gamma$-twisted conformal blocks rigid? 
\end{question}
For the untwisted case, it is well known that conformal blocks form a weakly rigid braided tensor category \cite{BK:01}. Rigidity for these categories has been proved by Y. Huang \cite{Huang:08b,Huang:08a} and also by Finkelberg \cite{Fink2,Fink1}(with some restriction on Lie algebras and levels). In fact, both the proofs of Huang and Finkelberg use some variant of the Verlinde formula only as an input. Hence in principle, the same problem persists in the twisted setting.   However in the twisted case, we  circumvent this issue by applying the results of Huang \cite{Huang:08b} which implies $\Ccal_1$ in $\Ccal(\frg,\Gamma,\ell)$ is rigid and hence rigidity follows from the second part of  Theorem \ref{conj:main2}.

\paragraph*{\bf Part I: The categorical twisted Verlinde formula}\label{p:catprel}
\addcontentsline{toc}{section}{\bf Part I: The categorical twisted Verlinde formula}
In this first part of the paper we develop some results about monoidal and fusion categories that we need in the paper. However, the results obtained in this part are also of independent interest from the point of view of the theory of fusion and braided crossed fusion categories. 
The main result in Section \ref{sec:drinfeldsquestions} is a criterion for rigidity in weakly fusion categories. 
In Section \ref{sec:catverlinde} we prove the categorical twisted Verlinde formula which computes the fusion coefficients of any $\Gamma$-crossed modular category in terms of the categorical crossed S-matrices. 
\section{Rigidity in weakly fusion categories}\label{sec:drinfeldsquestions}
We begin by reviewing the notions of weak duality and rigidity in monoidal categories and prove a useful criterion for rigidity. We refer the reader to Etingof-Nikshych-Ostrik \cite{ENO:05} for more details on the theory of monoidal categories and fusion categories and to Boyarchenko-Drinfeld \cite{BoDr:13} for more on weak duality and pivotal/ribbon structures in this setting.
\subsection{Monoidal r-categories and weakly fusion categories}\label{sec:monrcatweaklyfusion} We begin by recalling the notion of a monoidal r-category and a weakly fusion category.
\begin{definition}
	A monoidal category $(\Ccal,\otimes,\unit)$ is said to be an r-category if (cf. \cite{BoDr:13}):
	\\ (i) For each $X\in \Ccal$, the functor $\Ccal\ni Y\mapsto \Hom(\unit, X\otimes Y)$ is representable by an object $X^*$, i.e. we have functorial identifications $\Hom(\unit,X\otimes Y)=\Hom(X^*,Y)$.\\
	(ii) The functor $\Ccal\ni X\mapsto X^*\in \Ccal^{\operatorname{op}}$ is an equivalence of categories, with the inverse functor being denoted by $X\mapsto { }^*X$.
\end{definition}

Note that the notion of a monoidal r-category that is defined here is dual to the one considered in \cite{BoDr:13}, namely it corresponds to the weak duality for the ``second tensor product'' constructed in {\it op. cit. \S3.1}. 
\begin{remark}
	Using both (i) and (ii) from the definition, it follows that we have functorial identifications for any pair of objects $X,Y$ in a monoidal r-category $\Ccal$:
	\begin{equation}\label{eq:defrcategory}
	\Hom(\unit,X\otimes Y)=\Hom(X^*,Y)=\Hom({ }^*Y,X).\end{equation}

\end{remark}

\begin{definition}\label{def:weaklyfusion}
	We say that a monoidal category $\Ccal$ is a weakly fusion category  over an algebraically closed field $k$ if $\Ccal$ is a finite semisimple $k$-linear abelian monoidal r-category such that the unit $\unit$ is a simple object.
\end{definition}

\begin{remark}\label{rk:fusioncoeff}
	Let $A,B,C$ be simple objects in a weakly fusion category $(\Ccal,\otimes,\unit)$. As a first step in understanding the tensor product structure, it is important to consider the multiplicity $\nu^C_{A,B}$ of any simple object $C$ in the tensor product $A\otimes B$. These multiplicities are called the fusion coefficients. Since $\Ccal$ is weakly fusion, this multiplicity is given by 
	\begin{equation}\label{eq:fusioncoeffs}\nu^C_{A,B}=\dim\Hom(C,A\otimes B)=\dim\Hom(\unit,A\otimes B\otimes C^*)=\dim\Hom(\unit, {}^*C\otimes A\otimes B).\end{equation} 
	Note that all the $\Hom$-spaces in Equation \eqref{eq:fusioncoeffs} are canonically identified by the weak duality. Hence the information of all the fusion coefficients is contained in the multiplicity spaces $\Hom(\unit, A_1\otimes\cdots\otimes A_n)$ for simple objects $A_i\in \Ccal$. The dimensions of the spaces (\ref{eq:fusioncoeffs}) are known as the fusion coefficients.
\end{remark}

Let $\Ccal$ be any monoidal r-category. By definition, for each object $X\in \Ccal$, we have $\Hom(\unit,X\otimes X^*)=\Hom(X^*,X^*)$.  In particular, we have a canonical coevaluation morphism $\coev_X:\unit\to X\otimes X^*$, which we denote pictorially (where morphisms are read from top to bottom) by
$$\coev_X=\scalebox{0.8}{\begin{tikzpicture}[baseline={([yshift=-.5ex]current bounding box.center)},]
	\node (B1) {$X$};
	\node[circle, inner sep = 0.4pt, draw, above right =.6 of B1] (Cap) {$X$};
	\node[below right =.6  of Cap] (B2) {$X^*$};
	\node[above =.4 of Cap] (unit) {$\unit$};
	\draw (B1) to[out=90,in=-180] (Cap);
	\draw (B2) to[out=90,in=0] (Cap);
	\end{tikzpicture}=\begin{tikzpicture}[baseline={([yshift=-.5ex]current bounding box.center)},scale=0.75]
	\node (B1) {$X$};
	\node[circle, inner sep = 0.4pt, draw, above right =.6 of B1] (Cap) {$X$};
	\node[below right =.6  of Cap] (B2) {$X^*$};
	\draw (B1) to[out=90,in=-180] (Cap);
	\draw (B2) to[out=90,in=0] (Cap);
	\end{tikzpicture},}$$ often dropping the unit $\unit$ from the diagram. Using the identification, $$\Hom(\unit,X\otimes Y)=\Hom(X^*,Y)=\Hom({ }^*Y,X),$$ any morphism $f:\unit \to X\otimes Y$ corresponds to a unique morphism $\wtilde{f}:X^*\to Y$ and ${}^*\wtilde{f}:{ }^*Y\to X$ such that we have the equality of morphisms
$$
\scalebox{0.8}{
	\begin{tikzpicture}[baseline={([yshift=-.5ex]current bounding box.center)}, ]
	\node[draw, text centered, text width=1cm] (f) {$f$};
	\node[below=of f.210] (B1) {$X$};
	\node[below=of f.330] (B2) {$Y$};
	\node[above = .9 of f] (unit) {$\unit$};
	\draw (B1) -- (f.210);
	\draw (B2) -- (f.330);
	\end{tikzpicture}
	=
	\begin{tikzpicture}[baseline={([yshift=-.5ex]current bounding box.center)}]
	\node (B1) {};
	\node[circle,inner sep = 0.4pt, draw, above right =.6 of B1] (Cap) {$X$};
	\node[below right =.2 and .3  of Cap] (B2) {$X^*$};
	\node[below =1.2 of B1] (E1) {$X$};
	\node at (E1-|B2) (E2) {$Y$};
	\node[draw, above = .3 of E2] (f) {$\wtilde{f}$};
	\node[above=.2 of Cap] (unit) {$\unit$};
	\draw (E1) to[out=90,in=-180] (Cap);
	\draw (B2) to[out=90,in=0] (Cap);
	\draw (B2) -- (f) -- (E2);
	\end{tikzpicture}
	=
	\begin{tikzpicture}[baseline={([yshift=-.5ex]current bounding box.center)}]
	\node (B1) {};
	\node[circle,inner sep = 0pt, draw, above left =.6 of B1] (Cap) {${}^*Y$};
	\node[below left =.2 and .3  of Cap] (B2) {${}^*Y$};
	\node[below =1.2 of B1] (E1) {$Y$};
	\node at (E1-|B2) (E2) {$X$};
	\node[draw, above = .3 of E2] (f) {${}^*\wtilde{f}$};
	\node[above=.2 of Cap] (unit) {$\unit$};
	\draw (E1) to[out=90,in=0] (Cap);
	\draw (B2) to[out=90,in=180] (Cap);
	\draw (B2) -- (f) -- (E2);
	\end{tikzpicture}.}$$
Let us further assume that $\Ccal$ is weakly fusion over an algebraically closed field $k$. Now by the semisimplicity of $\Ccal$ and weak duality, it follows that for each simple object $X\in \Ccal$, the tensor product $X\otimes X^*$ contains $\unit$ as a direct summand with multiplicity one. Hence the map $\coev_X$ has a unique splitting $e_X:X\otimes X^*\to \unit$, which we denote pictorially as $e_X=\scalebox{0.8}{\begin{tikzpicture}[baseline={([yshift=-.5ex]current bounding box.center)}]
	\node (B1) {$X$};
	\node[circle, inner sep = 0.4pt, draw, below right =.6 of B1] (Cup) {$X$};
	\node[above right =.6 of Cup] (B2) {$X^*$};
	\draw (B1) to[out=-90,in=180] (Cup);
	\draw (B2) to[out=-90,in=0] (Cup);
	\end{tikzpicture}}$
such that we have
$\scalebox{0.8}{\begin{tikzpicture}[baseline={([yshift=-.5ex]current bounding box.center)}]
	\node (B1) {$X$};
	\node[circle, inner sep = 0.4pt, draw, below right =.6 of B1] (Cup) {$X$};
	\node[above right =.6 of Cup] (B2) {$X^*$};
	\node[circle, inner sep = 0.4pt, draw, above right =.6 of B1] (Cap) {$X$};
	\draw (B1) to[out=-90,in=180] (Cup);
	\draw (B2) to[out=-90,in=0] (Cup);
	\draw (B1) to[out=90,in=-180] (Cap);
	\draw (B2) to[out=90,in=0] (Cap);
	\end{tikzpicture}=1.}$
\subsection{Rigidity in monoidal r-categories}
Let us now consider the notion of rigid duals. Let us begin by recalling the definition (see \cite{ENO:05} for more):
\begin{definition}
	An object $X$ in a monoidal category $\Ccal$ is said to have a left (rigid) dual $X^*$ if and only if the functor $X^*\otimes (\cdot)$ is left adjoint to the functor $X\otimes (\cdot)$. Equivalently, there should exist two morphisms $\coev_X:\unit \to X\otimes X^*$ and $\ev_X:X^*\otimes X\to \unit$ denoted pictorially as $\coev_X=\scalebox{0.8}{\begin{tikzpicture}[baseline={([yshift=-.5ex]current bounding box.center)}]
		\node (B1) {$X$};
		\node[circle, inner sep = 0.4pt, draw, above right =.6 of B1] (Cap) {$X$};
		\node[below right =.6  of Cap] (B2) {$X^*$};
		\draw (B1) to[out=90,in=-180] (Cap);
		\draw (B2) to[out=90,in=0] (Cap);
		\end{tikzpicture}}$ and $\ev_X=\scalebox{0.8}{\begin{tikzpicture}[baseline={([yshift=-.5ex]current bounding box.center)}]
		\node (B1) {$X^*$};
		\node[inner sep = 0.4pt, draw, below right =.6 of B1] (Cup) {$X$};
		\node[above right =.6 of Cup] (B2) {$X$};
		\draw (B1) to[out=-90,in=180] (Cup);
		\draw (B2) to[out=-90,in=0] (Cup);
		\end{tikzpicture}}$ such that we have
	\begin{equation}\label{equation:rigidity diagram}\scalebox{0.8}{
		\begin{tikzpicture}[baseline={([yshift=-.5ex]current bounding box.center)},scale=0.75]
		\node (C) {$X$};
		\node[above=.6of C] (B1) {$X$};
		\node[circle, inner sep = 0.4pt, draw, above right =.3 and .2 of B1] (Cap) {$X$};
		\node[below right =.3 and .2 of Cap] (B2) {$X^*$};
		\draw (B1) to[out=90,in=-180] (Cap);
		\draw (B2) to[out=90,in=0] (Cap);
		\node[inner sep = 0.4pt, draw, below right =.3 and .2 of B2] (Cup) {$X$};
		\node[above right =.3 and .2 of Cup] (B3) {$X$};
		\draw (B2) to[out=-90,in=180] (Cup);
		\draw (B3) to[out=-90,in=0] (Cup);
		\node[above=.6of B3] (A) {$X$};
		\draw (B1)--(C);
		\draw (A)--(B3);
		\end{tikzpicture}=
		\begin{tikzpicture}[baseline={([yshift=-.5ex]current bounding box.center)},scale=0.75]
		\node (A) {$X$};
		\node[below=1.7of A] (B){$X$};
		\draw (A) -- (B);
		\end{tikzpicture}
		\mbox{ and }
		\begin{tikzpicture}[baseline={([yshift=-.5ex]current bounding box.center)}]
		\node (C) {};
		\node[above=.6of C] (B1) {$X^*$};
		\node[above=.6of B1] (A0) {$X^*$};
		\node[inner sep = 0.4pt, draw, below right =.3 and .2 of B1] (Cup) {$X$};
		\node[above right =.3 and .2 of Cup] (B2) {$X$};
		\draw (B1) to[out=-90,in=-180] (Cup);
		\draw (B2) to[out=-90,in=0] (Cup);
		\node[circle,inner sep = 0.4pt, draw, above right =.3 and .2 of B2] (Cap) {$X$};
		\node[below right =.3 and .2 of Cap] (B3) {$X^*$};
		\draw (B2) to[out=90,in=180] (Cap);
		\draw (B3) to[out=90,in=0] (Cap);
		\node[below=.6of B3] (C0) {$X^*$};
		\draw (B3)--(C0);
		\draw (A0)--(B1);
		\end{tikzpicture}=
		\begin{tikzpicture}[baseline={([yshift=-.5ex]current bounding box.center)}]
		\node (A) {$X^*$};
		\node[below=1.76of A] (B){$X^*$};
		\draw (A) -- (B);
		\end{tikzpicture}.}
	\end{equation}
	Similarly, an object $X$ in a monoidal category $\Ccal$ is said to have a right (rigid) dual ${}^*X$ if and only if the functor ${}^*X\otimes (\cdot)$ is right adjoint to the functor $X\otimes (\cdot)$. A monoidal category is said to be \emph{rigid} if each of its objects has both a left dual and a right dual in this sense. A \emph{fusion category} is a weakly fusion category which is also rigid.
\end{definition}
We will now see that in a monoidal r-category, the two equations in (\ref{equation:rigidity diagram})  are in fact equivalent. In other words, to check rigidity of an object in an r-category, it is sufficient to only check one of the  conditions in (\ref{equation:rigidity diagram}).
\begin{lemma}\label{lemma:rigidity in r-categories}
	Let $\Ccal$ be a monoidal r-category. Let $X\in \Ccal$ and let $\coev_X:\unit\to X\otimes X^*$ be the canonical coevaluation map. Let $\ev_X:X^*\otimes X\to \unit$ be any morphism. Then the two equations in (\ref{equation:rigidity diagram}) are equivalent to each other.
\end{lemma}
\begin{proof}
	Let us assume that the first equality holds and deduce the second equality. We need to verify the equality of two morphisms in $\Hom(X^*,X^*)$. Since $\Ccal$ is an r-category, this is equivalent to verifying the equality of the corresponding morphisms in $\Hom(\unit,X\otimes X^*)$. The desired equality follows, since using the first equality from (\ref{equation:rigidity diagram}) we obtain:
	$$\scalebox{0.8}{\begin{tikzpicture}[baseline={([yshift=-.5ex]current bounding box.center)}]
		\node (C) {};
		\node[above=.6of C] (B1) {};
		\node[above=.6of B1] (A0) {$X^*$};
		\node[inner sep = 0.4pt, draw, below right =.3 and .2 of B1] (Cup) {$X$};
		\node[above right =.3 and .2 of Cup] (B2) {};
		\node[circle,inner sep = 0.4pt, draw, above right =.3 and .2 of B2] (Cap) {$X$};
		\node[below right =.3 and .2 of Cap] (B3) {};
		\draw (Cup) to[out=0, in=180] (Cap);
		\node[below=.6of B3] (C0) {$X^*$};
		\draw (C0) to[out=90,in=0] (Cap);
		\draw (A0) to[out=-90,in=-180] (Cup);
		\node[circle,inner sep = 0.4pt, draw, above left = .2 of A0] (Cap') {$X$};
		\node[left = 2.6 of C0] (C1) {$X$};
		\draw (C1) to[out=90,in=180] (Cap');
		\draw (Cap') to[out=0,in=90] (A0);
		\end{tikzpicture}=
		\begin{tikzpicture}[baseline={([yshift=-.5ex]current bounding box.center)}]
		\node (C) {};
		\node[above=.6of C] (B1) {};
		\node[above=.6of B1] (A0) {};
		\node[inner sep = 0.4pt, draw, below right =0.4 and .2 of B1] (Cup) {$X$};
		\node[above right =.3 and .2 of Cup] (B2) {};
		\node[above right =.3 and .2 of B2] (X) {};
		\node[circle,inner sep = 0.4pt, draw, above=.7 of X] (Cap) {$X$};
		\node[below right =.3 and .2 of X] (B3) {};
		\draw (Cup) to[out=0, in=180] (Cap);
		\node[below=.6of B3] (C0) {$X^*$};
		\draw (C0) to[out=90,in=0] (Cap);
		\node[circle,inner sep = 0.4pt, draw, below left = .35 of A0] (Cap') {$X$};
		\node[left = 2.6 of C0] (C1) {$X$};
		\draw (C1) to[out=90,in=180] (Cap');
		\draw (Cap') to[out=0,in=180] (Cup);
		\end{tikzpicture}
		=
		\begin{tikzpicture}[baseline={([yshift=-.5ex]current bounding box.center)}]
		\node[circle, draw,inner sep=.4pt] (Cap) {$X$};
		\node[below right = 1.6 and .4 of Cap] (B2) {$X^*$};
		\node[below left = 1.6 and .4 of Cap] (B1) {$X$};
		\node[above=.4 of Cap] (X) {};
		\draw (B1) to[out=90,in=180] (Cap);
		\draw (Cap)to[out=0,in=90] (B2);
		\end{tikzpicture}.}
	$$
	The equivalence in the other direction follows from a similar argument.
\end{proof}

Now let us get back to our setting of a weakly fusion category $\Ccal$ over an algebraically closed field $k$. For each simple object $X$ we have the coevaluation map $\coev_X:\unit\to X\otimes X^*$, but we do not know whether the desired evaluation map necessarily exists. But instead, we  have the previously defined map $e_{^*X}:{ }^*X\otimes X\to \unit$  which splits the corresponding coevaluation map $\coev_{{}^*X}:\unit\to { }^*X\otimes X$. Hence for each simple object $X\in \Ccal$ we can construct the following two morphisms in $\Ccal$:
\begin{equation}\label{equation:weak snakes}
\begin{tikzpicture}[baseline={([yshift=-.5ex]current bounding box.center)}]
\node (C) {};
\node[above=.6of C] (B1) {${}^*X$};
\node[above=.6of B1] (A0) {${}^*X$};
\node[circle, inner sep = 0pt, draw, below right =.3 and .2 of B1] (Cup) {${}^*X$};
\node[above right =.3 and .2 of Cup] (B2) {$X$};
\draw (B1) to[out=-90,in=-180] (Cup);
\draw (B2) to[out=-90,in=0] (Cup);
\node[circle,inner sep = 0.4pt, draw, above right =.3 and .2 of B2] (Cap) {$X$};
\node[below right =.3 and .2 of Cap] (B3) {$X^{*}$};
\draw (B2) to[out=90,in=180] (Cap);
\draw (B3) to[out=90,in=0] (Cap);
\node[below=.6of B3] (C0) {$X^{*}$};
\draw (B3)--(C0);
\draw (A0)--(B1);
\end{tikzpicture}
\mbox{ and }
\begin{tikzpicture}[baseline={([yshift=-.5ex]current bounding box.center)}]
\node (C) {$X$};
\node[above=.6of C] (B1) {$X$};
\node[circle, inner sep = 0.4pt, draw, above right =.3 and .2 of B1] (Cap) {$X$};
\node[below right =.3 and .2 of Cap] (B2) {$X^*$};
\draw (B1) to[out=90,in=-180] (Cap);
\draw (B2) to[out=90,in=0] (Cap);
\node[circle, inner sep = 0pt, draw, below right =.3 and .2 of B2] (Cup) {$X^*$};
\node[above right =.3 and .2 of Cup] (B3) {$X^{**}$};
\draw (B2) to[out=-90,in=180] (Cup);
\draw (B3) to[out=-90,in=0] (Cup);
\node[above=.6of B3] (A) {$X^{**}$};
\draw (B1)--(C);
\draw (A)--(B3);
\end{tikzpicture}.
\end{equation}
\begin{lemma}\label{lemma:rigidity in weakly fusion categories}
	Let $X$ be a simple object in a weakly fusion category $\Ccal$. Then the following statements are equivalent:\\
	(i) The object $X$ has a rigid left dual.\\
	(ii) The first morphism in (\ref{equation:weak snakes}) is invertible (or equivalently, nonzero).\\
	(iii) The second morphism in (\ref{equation:weak snakes}) is invertible (or equivalently, nonzero).\\
\end{lemma}
\begin{proof}
	Let us first prove (ii)$\Rightarrow$(i). Suppose that the first morphism in (\ref{equation:weak snakes}) is an  isomorphism with inverse $\delta:X^*\to { }^*X$. Define $\ev_X$ to be the composition $$\ev_X:X^*\otimes X\xrightarrow{\delta\otimes \id_X} { }^*X\otimes X\xrightarrow{e_{{}^*X}}\unit.$$ Hence by construction, the second equation in (\ref{equation:rigidity diagram}) is satisfied. By Lemma \ref{lemma:rigidity in r-categories}, the first condition must also hold and we conclude that $X^*$ must in fact be the rigid left dual of $X$.
	
	To prove (i)$\Rightarrow$(ii), suppose that $X^*$ is a rigid left dual of $X$. Hence we have an evaluation map $\ev_X:X^*\otimes X\to \unit$ satisfying (\ref{equation:rigidity diagram}). In particular by semisimplicity, $\unit$ must be a direct summand of $X^*\otimes X$. Hence we must have ${}^*X\cong X^*$. Choosing such an isomorphism we obtain a non-zero morphism ${}^*X\otimes X\to X^*\otimes X\xrightarrow{\ev_X} \unit$ which differs from $e_{{}^*X}:{}^*X\otimes X\to \unit$ by an element of $k^\times$. Statement (ii) now follows.
	
	\noindent The proof of (i)$\Leftrightarrow$(iii) is similar. 
\end{proof}

\noindent We obtain the following well-known result as a corollary:
\begin{corollary}\label{cor:doubledual}
	For any object $X$ in a fusion category $\Ccal$, its left dual $X^*$ and right dual ${}^*X$ are isomorphic, or equivalently $X\cong X^{**}$.
\end{corollary}

\begin{remark}
	It is important to note that the isomorphisms in Corollary \ref{cor:doubledual} are not canonical. More precisely, the two monoidal functors $\id_\Ccal$ and $(\cdot)^{**}:\Ccal\to \Ccal$ are non-canonically isomorphic as functors between abelian categories. Moreover such an isomorphism may not be compatible with the monoidal structure of these two functors.  The choice of an isomorphism which is also compatible with the monoidal structures is known as a \emph{pivotal} structure. We refer to \cite[\S2.4.2, \S2.4.3]{DGNO} for more about this. We will also encounter this notion is Section \ref{sec:catcrosssmat} below. Following \cite[\S5]{BoDr:13} we will also define pivotal structures in monoidal r-categories in \ref{sec:weaklyribbon}.
\end{remark}
\subsection{A criterion for rigidity} We prove the following useful criterion that guarantees the existence of rigid duals in a weakly fusion category.
\begin{proposition}\label{prop:rigidity in weakly fusion categories}
	Let $\Ccal$ be a weakly fusion category. Let $M\in \Ccal$ be a simple object such that $M\otimes { }^*M$ has a rigid left dual. Then $M$ has a rigid left dual.
\end{proposition}
\begin{proof}
	Consider the morphism $e_{{}^*M}\otimes \id_{{}^*M}:{}^*M\otimes M \otimes {}^*M\to {}^*M.$ It is non-zero, since the morphism $\coev_{{}^*M}\otimes \id_{{}^*M}:{}^*M\to {}^*M\otimes M \otimes {}^*M$ is its right inverse. Let us rewrite this non-zero morphism diagrammatically as below:
	\begin{equation}\label{eq:diagramredbox}\scalebox{0.8}{\begin{tikzpicture}[baseline={([yshift=-.5ex]current bounding box.center)}]
		\node (B1) {${ }^*M$};
		\node[circle, inner sep = 0pt, draw, below right =0.8 and .25 of B1] (Cup) {${ }^*M$};
		\node[right =.4 of Cup |- B1] (B2) {$M$};
		\node[right = .3 of B2] (B3) {${ }^*M$};
		\node[below left = 3 and .1 of B2] (E) {${ }^*M$};
		\draw (B1) to[out=-90,in=180] (Cup);
		\draw (B2) to[out=-90,in=0] (Cup);
		\draw (B3) to[out=-90,in=90] (E);
		\end{tikzpicture}=
		\begin{tikzpicture}[baseline={([yshift=-.5ex]current bounding box.center)}]
		\node (A1) {${ }^*M$};
		\node[right=1of A1] (A2) {$M\otimes { }^*M$};
		\node[below=.3of A2, draw, minimum width = 1cm] (Id) {$\operatorname{Id}$};
		\draw[double] (A2) -- (Id);
		\node[below=.3of Id.210] (B1) {$M$};
		\node[below=.3of Id.330] (B2) {${ }^*M$};
		\draw (Id.210)--(B1);
		\draw (Id.330)--(B2);
		\node[below right =3 and .8 of A1] (C) {${ }^*M$};
		\draw (B2) to[out=-90,in=90] (C);
		\node[circle, inner sep =0pt, draw, below left = .25 and .35 of B1] (Cup) {${ }^*M$};
		\draw (A1) to[out=-90, in=180] (Cup);
		\draw (Cup) to[out=0, in=-90] (B1);
		\end{tikzpicture}
		=
		\begin{tikzpicture}[baseline={([yshift=-.5ex]current bounding box.center)}]
		\node (A1) {${ }^*M$};
		\node[right=1of A1] (A2) {};
		\node[right=1of A2] (A3) {$M\otimes { }^*M$};
		\node[below=.8 of A2, draw, minimum width = 1cm] (Id) {$\operatorname{Id}$};
		\node[below right =3 and .8 of A1] (C) {${ }^*M$};
		\draw (Id.330) to[out=-90,in=90] (C);
		\node[circle, inner sep =0pt, draw, below left = 1 and .5 of Id.210] (Cup) {${ }^*M$};
		\draw (A1) to[out=-90, in=180] (Cup);
		\draw (Cup) to[out=0, in=-90] (Id.210);
		\node[above right =.4 and .05 of Id] (CapA) {};
		\node[below right = 1.8 and .6 of CapA] (CupA) {};
		\draw[double] (Id) to[out=90,in=180] (CapA) to[out=0,in=180] (CupA) to[out=0,in=-90] (A3);
		\node[below right =0.1 and -1.4 of Id] (string) {$M$};
		\node at (.5,-.25) (R1) {};
		\node at (2.8,-2) (R3) {};
		\draw[red, dotted, very thick] (R1) rectangle (R3);
		\end{tikzpicture}.}\end{equation}
	In the second equality in Equation \eqref{eq:diagramredbox}, we have used the existence of the rigid left dual of the object $M\otimes {}^*M\in \Ccal$. Let $f:\unit \to M\otimes \left({}^*M\otimes (M\otimes {}^*M)^*\right)$ be the morphism represented by  the red dotted box in (\ref{eq:diagramredbox}). Let $\wtilde{f}: M^*\to {}^*M\otimes (M\otimes {}^*M)^*$ be the morphism corresponding to $f$ using weak duality. Hence the non-zero morphism $e_{{}^*M}\otimes \id_{{}^*M}$ can be expressed as:
	\begin{equation}\label{eqn:blueandredbox}
	\scalebox{0.8}{\begin{tikzpicture}[baseline={([yshift=-.5ex]current bounding box.center)}]
		\node (A1) {${ }^*M$};
		\node[right=2.5 of A1] (A2) {$M\otimes { }^*M$};
		\node[below right =4 and .8 of A1] (C) {${ }^*M$};
		\node[circle, inner sep =0pt, draw, below right = 2 and .2 of A1] (Cup) {${ }^*M$};
		\node[circle, inner sep=0.4pt, draw, below right =.2 and 1.2 of A1] (Cap) {$M$};
		\node[below right=.4 and 0 of Cap, minimum width = 1cm, draw] (f) {$\wtilde{f}$};
		\draw (A1) to[out=-90, in=180] (Cup);
		\draw (Cup) to[out=0, in=180] (Cap);
		\node[below right = 1.8 and 1 of Cap] (CupA) {};
		\draw (Cap) to[out=0,in=90] (f);
		\draw[double] (f.300) to[out=-90,in=180] (CupA) to[out=0,in=-90] (A2);
		\draw (f.230) to[out=-90,in=90] (C);
		\node at (.55,-.25) (R1) {};
		\node at (3.2,-2.1) (R3) {};
		\draw[red, dotted, very thick] (R1) rectangle (R3);
		\node[below right=-0.2 and .4 of Cap] (string) {$M^{*}$};
		\node[below left=0.5 and .2 of Cap] (str) {$M$};
		\end{tikzpicture}
		=\begin{tikzpicture}[baseline={([yshift=-.5ex]current bounding box.center)}]
		\node (A1) {${ }^*M$};
		\node[right=2.5 of A1] (A2) {$M\otimes { }^*M$};
		\node[below right =4 and .8 of A1] (C) {${ }^*M$};
		\node[circle, inner sep =0pt, draw, below right = 1.35 and .2 of A1] (Cup) {${ }^*M$};
		\node[circle, inner sep=.4pt, draw, below right =.55 and 1.3 of A1] (Cap) {$M$};
		\node[below right=1.4 and -.1 of Cap, minimum width = 1cm, draw] (f) {$\wtilde{f}$};
		\draw (A1) to[out=-90, in=180] (Cup);
		\draw (Cup) to[out=0, in=180] (Cap);
		\node[below right = 2.8 and .9 of Cap] (CupA) {};
		\draw (Cap) to[out=0,in=90] (f);
		\draw[double] (f.300) to[out=-90,in=180] (CupA) to[out=0,in=-90] (A2);
		\draw (f.230) to[out=-90,in=90] (C);
		\node at (-0.3,-.6) (R1) {};
		\node at (4.2,-2.1) (R3) {};
		\draw[blue, dotted, very thick] (R1) rectangle (R3);
		\node[below right=0.3 and .4 of Cap] (string) {$M^{*}$};
		\end{tikzpicture}}.
		\end{equation}
	By looking at the blue dotted box in Equation \eqref{eqn:blueandredbox}, we conclude that the morphism $\scalebox{1.0}{\begin{tikzpicture}[baseline={([yshift=-.5ex]current bounding box.center)}]
		\node (C) {};
		\node[above=.6of C] (B1) {};
		\node[above=.6of B1] (A0) {${}^*M$};
		\node[circle, inner sep = -0.1pt, draw, below right =.3 and .2 of B1] (Cup) {\tiny${}^*M$};
		\node[above right =.3 and .2 of Cup] (B2) {};
		\node[circle,inner sep = 0.4pt, draw, above right =.3 and .2 of B2] (Cap) {$M$};
		\node[below right =.3 and .2 of Cap] (B3) {};
		\node[below=.6of B3] (C0) {$M^{*}$};
		\draw (Cup) to[out=0,in=180] (Cap);
		\draw (C0) to[out=90,in=0] (Cap);
		\draw (A0)to[out=-90,in=-180] (Cup);
		\end{tikzpicture}}
	$ must be non-zero and hence an isomorphism. From Lemma \ref{lemma:rigidity in weakly fusion categories} we conclude that $M$ has a rigid left dual.
\end{proof}
Finally we deduce the following corollary of Proposition \ref{prop:rigidity in weakly fusion categories}:
\begin{corollary}\label{cor:rigidity of graded categories}
	Let $\Gamma$ be a finite group and let $\Ccal=\bigoplus\limits_{\gamma\in\Gamma}\Ccal_\gamma$ be a $\Gamma$-graded weakly fusion category such that the identity component $\Ccal_1$ is rigid. Then $\Ccal$ is rigid and hence a fusion category.
\end{corollary}
\begin{proof}
	This follows directly from Proposition \ref{prop:rigidity in weakly fusion categories}. Let $M$ be a simple object in $\Ccal$, lying in say $\Ccal_\gamma$. Then ${}^*M\in \Ccal_{\gamma^{-1}}$ and $M\otimes {}^*M$ lies in $\Ccal_1$ which we have assumed is a fusion category. Hence $M\otimes {}^*M$ has a rigid left dual. From the proposition it follows that $M$ has a rigid left dual. Hence $\Ccal$ is a fusion category.
\end{proof}

\section{On Grothendieck rings and fusion rules}\label{sec:grothringfrobalg}
In this section we recall some notions related to Grothendieck rings of fusion categories, sometimes also called fusion rings and describe the structures on it. We refer to \cite[\S1]{Lus87}, \cite{Desh:17,Desh:18a}, \cite{Arote:18,Fossum:71} for more details on the various notions introduced in this section. We will also relate this to the notion of fusion rules and fusion rings as described in \cite{BeauHirz}. 
\subsection{Properties of Grothendieck group of $(\Ccal, \otimes, \unit)$}
Let $(\Ccal,\otimes,\unit)$ be a fusion category. Then its Grothendieck group, which we denote by $K(\Ccal)$ has the structure of a based ring in the sense of \cite[\S1]{Lus87}. It is often necessary and useful to extend scalars to some rings or fields $R$ which are equipped with an involution $\overline{(\cdot)}:R\to R$ and define $K_R(\Ccal):=K(\Ccal)\otimes_{\ZBbb}R$ equipped with the following structures:
\begin{enumerate}
	\item There is the non-degenerate linear functional $\nu:K_R(\Ccal)\to R$ which records the coefficient of the class of the unit $[\unit]$ in any element of $K_R(\Ccal)$. Moreover using rigidity it is easy to see (using Corollary \ref{cor:doubledual}) that $\nu$ is symmetric, i.e. $\nu([X\otimes Y])=\nu([Y\otimes X])$. This allows us to identify $K_R(\Ccal)^*\cong K_R(\Ccal)$ as $K_R(\Ccal)$-bimodules and provides $K_R(\Ccal)$ with the structure of a symmetric Frobenius $R$-algebra (cf. \cite{Fossum:71,Arote:18}).
	\item The rigid duality on $\Ccal$ and Corollary \ref{cor:doubledual} can be used to define an $R$-semilinear anti-involution \newline $(\cdot)^*:K_R(\Ccal)\to K_R(\Ccal)$ such that we have $\nu(a^*)=\overline{\nu(a)}$ for any $a\in K_R(\Ccal)$.
	This allows us to define a Hermitian pairing  on $K_R(\Ccal)$:
	\begin{equation}
	\label{eqn:hermitianform}
	\langle a,b\rangle:=\linfun(ab^*).
		\end{equation}
	\item The classes of the simple objects of $\Ccal$ form an orthonormal $R$-basis of $K_R(\Ccal)$ with respect to the Hermitian form $\langle \ ,\ \rangle$ defined by Equation \eqref{eqn:hermitianform}.
\end{enumerate}

\begin{definition}\label{def:frobstaralgarote}(See  also \cite{Arote:18}.) A Frobenius $\ast$-$R$-algebra is a symmetric Frobenius $R$-algebra $(A,\nu:A\to R)$ equipped with an $R$-semilinear anti-involution $(\cdot)^*:A\to A$ satisfying  $\nu(a^*)=\overline{\nu(a)}$ and which admits an orthonormal basis with respect to the Hermitian inner product $\langle a,b\rangle=\nu(ab^*)$.
\end{definition}	
Note that by the definition of the $R$-Hermitian inner product, we indeed have that
\begin{equation}\label{eq:frobalginnerprod}
\langle a ,b\rangle = \langle ab^*, 1\rangle=\nu(ab^*)=\overline{\nu((ab^*)^*)}=\overline{\nu(ba^*)}=\overline{\langle b,a\rangle}
\end{equation}
for any two elements $a,b$ of the Frobenius $\ast$-$R$-algebra. We can think of this as a decategorified version of the natural identification (\ref{eq:defrcategory}).  

The $R$ that we will mostly be interested in is $\Qab$, the maximal abelian extension of $\QBbb$. By the Kronecker-Weber theorem it is obtained by adjoining all roots of unity to $\QBbb$. Hence it has a distinguished involution called `complex conjugation' and denoted $\bar{(\cdot)}$ which maps each root of unity to its inverse. Sometimes we will also consider $R=\CBbb$ equipped with complex conjugation. Sometimes it is also useful to consider $R=\ZBbb[\omega]$ with complex conjugation, where $\omega$ is some root of unity. Frobenius $*$-algebras over any of these $R$ (equipped with the `complex conjugation' involution) are necessarily semi-simple $R$-algebras (see \cite{Arote:18}).

We record the above discussion along with an important result of \cite{ENO:05} in the following:
\begin{proposition}\label{prop:recalldef}
	Let $\Ccal$ be a fusion category.\\
	(i) The fusion ring $\KQab(\Ccal)$ over $\Qab$ has the structure of a Frobenius $\ast$-algebra with the classes of simple objects of $\Ccal$ forming an orthonormal basis of $\KQab(\Ccal)$. \\
	(ii) The Frobenius $\ast$-algebra $\KQab(\Ccal)$ is a split semisimple $\Qab$-algebra, i.e. all its irreducible complex representations are already defined over $\Qab$ (see \cite[Cor. 8.53]{ENO:05}).
\end{proposition}

We now relate Proposition \ref{prop:recalldef} and the discussion preceding it with the notion of (non-degenerate) fusion rules as defined in \cite[\S5]{BeauHirz}. Let $\NBbb$ denote the set $\ZBbb_{\geq0}$ of non-negative integers. Let $I$ be a finite set equipped with an involution $\ast$ and let $\NBbb^I$
denote the free commutative monoidal generated by $I$. Although in \cite{BeauHirz} additive notation is used for this commutative monoid, here we find it more convenient to use multiplicative notation and we denote a typical element of $\NBbb^I$ by $\prod\limits_{\alpha\in I}\alpha^{n_\alpha}$. 

Now also consider the free abelian group $K$ with basis given by $I$ and extend the involution $\ast$ to an involution of $K$. The goal is to now define the structure of a commutative Frobenius $\ast$-$\ZBbb$-algebra on $K$ (equipped with the $\ast$ involution) so that $I$ becomes a $\ast$-invariant orthonormal basis containing the unit of the desired ring structure. As is proved in \cite[\S5]{BeauHirz} this structure is equivalent to the axioms of non-degenerate fusion rules. 

Essentially we want to define the product of the basis elements $I$ of $K$, or in other words we would like to expand the products $\prod\limits_{\alpha\in I}\alpha^{n_\alpha}\in \NBbb^I$ as $\ZBbb$-linear combinations of $I$. The fusion rule, which we denote by $\nu:\NBbb^I\to \ZBbb$ records the coefficient of $1\in I$ in such a product. By (\ref{eq:frobalginnerprod}) this is enough to decompose any product into a linear combination of the basis elements. We refer to {\it loc. cit.} for details.

\section{Braided crossed categories and crossed S-matrices}\label{sec:braidedcrossed}
Let $\Gamma$ be a finite group. Since the notion of braided $\Gamma$-crossed monoidal categories plays a central role in this paper, let us begin by recalling the definition. We refer to the paper Drinfeld-Gelaki-Nikshych-Ostrik \cite{DGNO} and the book of  V. Turaev \cite{TuraevHQFT} for a more detailed discussion. From now on we will assume that all our categories are $\CBbb$-linear.
\begin{definition}\label{d:braidedgammacrossed} A braided $\Gamma$-crossed monoidal category is a monoidal category $\Ccal$ equipped with the following structures:
	\begin{enumerate}
		\item a $\Gamma$-grading $\Ccal=
		\bigoplus\limits_{\gamma\in \Gamma}\Ccal_\gamma$ such that $\Ccal_{\gamma_1}\otimes \Ccal_{\gamma_2}\subset \Ccal_{\gamma_1\gamma_2}$;
		\item a monoidal $\Gamma$-action on $\Ccal$ such that the action of an element $g\in \Gamma$ maps the component $\Ccal_\gamma$ to $\Ccal_{g\gamma g^{-1}}$; 
		\item isomorphisms (which will be called crossed braiding isomorphisms)
		$$\beta_{M,N}:M\otimes N \xrightarrow{\cong} \gamma(N) \otimes M \mbox{ for } \gamma\in\Gamma, M\in \Ccal_\gamma, N\in \Ccal$$
		which are functorial in $M,N$, are compatible with the $\Gamma$ action, i.e. $g(\beta_{M,N})=\beta_{g(M),g(N)}$ for all $g\in \Gamma$, with $\beta_{\unit,N}=\beta_{N,\unit}=\id_N$ and satisfy certain compatibilities as described in \cite[Def. 4.41]{DGNO}.
		%
		%
		%
	\end{enumerate}
\end{definition}
	\begin{remark}
		It follows from the axioms in Definition \ref{d:braidedgammacrossed} that $\Ccal_1$ is a braided monoidal category equipped with a braided action of $\Gamma$. The braided $\Gamma$-crossed monoidal category $\Ccal$ is said to be a braided $\Gamma$-crossed extension of the braided monoidal category $\Ccal_1$.
	\end{remark}

	\begin{remark}
		In this paper we will only consider braided $\Gamma$-crossed monoidal categories which are also weakly fusion. By Corollary \ref{cor:rigidity of graded categories}, rigidity of such a braided $\Gamma$-crossed weakly fusion category $\Ccal$ is equivalent to the rigidity of the identity component $\Ccal_1$.
	\end{remark}

	\begin{definition}\label{d:equivariantization} (cf. \cite[\S4]{DGNO}.)
		Let $\Ccal$ be a category equipped with an action of a finite group $\Gamma$. Then its $\Gamma$-equivariantization denoted by $\Ccal^\Gamma$ is the category whose objects are pairs $(C,\psi_{\cdot})$ where $C$ is an object of $\Ccal$ and $\psi_{\cdot}$ is a family of isomorphisms $\psi_\gamma:\gamma(C)\xoto{\cong} C$ for each $\gamma\in \Gamma$ such that for each $\gamma_1,\gamma_2\in \Gamma$, $\psi_{\gamma_1\gamma_2}$ equals the composition  $$\gamma_1\gamma_2(C)\xoto{=}\gamma_1(\gamma_2(C))\xoto{\gamma_1(\psi_{\gamma_2})}\gamma_1(C)\xoto{\psi_{\gamma_1}}C.$$ The morphisms in $\Ccal^\Gamma$ are defined to be the morphisms in $\Ccal$ that commute with the isomorphisms $\psi_\gamma$. If $\Ccal$ is monoidal with a monoidal action of $\Gamma$, then $\Ccal^\Gamma$ is also a monoidal category.
	\end{definition}
	
	\noindent We recall the following well-known result:
	\begin{proposition}\label{prop:equivisbraided}
		(cf. {\cite[\S4.4.4]{DGNO}}.)
		Let $\Ccal$ be a braided $\Gamma$-crossed monoidal category. Then $\Ccal^\Gamma$ has the structure of a braided monoidal category.
	\end{proposition}

	\subsection{Braided $\Gamma$-crossed fusion categories and their Grothendieck rings} From now on, let us assume that $\Ccal$ is a braided $\Gamma$-crossed fusion category, i.e. we also assume rigidity and finite semi-simpleness of $\Ccal$ and that the unit $\unit$ is simple. For each $\gamma\in \Gamma$ we let $P_\gamma$ denote the set of (isomorphism classes of) simple objects in the component $\Ccal_\gamma$. 
	
	As in Section \ref{sec:grothringfrobalg}, let us denote by $\KQab(\Ccal)$ the $\Qab$-algebra obtained from the Grothendieck ring of $\Ccal$ by extension of scalars to $\Qab$. Then 
\begin{equation}\label{eq:gradingongrothendieckalgebra}\KQab(\Ccal)=\bigoplus\limits_{\gamma\in \Gamma}\KQab(\Ccal_\gamma)\end{equation} 
	is a $\Gamma$-crossed Frobenius $\Qab$-$*$-algebra as defined in \cite{JKK,AroDes}. In other words the braided $\Gamma$-crossed structure on $\Ccal$ provides similar structures on $\KQab(\Ccal)$. Hence we have the grading as in Equation \eqref{eq:gradingongrothendieckalgebra}, an action of $\Gamma$ on the algebra $\KQab(\Ccal)$ coming from Definition \ref{d:braidedgammacrossed}(2) and finally the crossed braiding isomorphisms give us the similar crossed commutation relations in $\KQab(\Ccal)$ at the level of Grothendieck rings.
	
	As before, the non-degenerate symmetric linear functional $\linfun:\KQab(\Ccal)\to\Qab$ which records the coefficient of the class of the unit $[\unit]$ in any element of $\KQab(\Ccal)$ gives us an identification 
	\begin{equation}\label{eq:identificationwithdualCgamma}
	\KQab(\Ccal_\gamma)^*\cong \KQab(\Ccal_{\gamma^{-1}})\mbox{ as }\KQab(\Ccal_1)\mbox{-bimodules.}\end{equation}
	It also follows from the crossed braiding isomorphisms that the identity component $\KQab(\Ccal_1)$ is central in $\KQab(\Ccal)$. 
	
	For $\gamma\in \Gamma$, let $P_\gamma$ denote the finite set of (representatives of isomorphism classes of) simple objects of the semisimple category $\Ccal_\gamma$. We will also think of $P_\gamma$ as a $\Qab$-basis of $\KQab(\Ccal_\gamma)$. It is an orthonormal basis with respect to the natural Hermitian form.

	Our goal in this section is to describe the fusion coefficients of $\KQab(\Ccal)$. For $\gamma_1,\gamma_2\in \Gamma, A\in \Ccal_{\gamma_1}, B\in \Ccal_{\gamma_2}$ and $C\in \Ccal_{\gamma_1\gamma_2}$, the fusion coefficient $\nu_{A,B}^C\in \mathbb{Z}_{\geq 0}$ is defined to be the multiplicity of $C$ in the product $A\otimes B$, so that we have the equality $$[A]\cdot[B]=\sum\limits_{C\in P_{\gamma_1\gamma_2}}\nu_{A,B}^C\cdot[C]$$ in the Grothendieck ring. Equivalently (see also Remark \ref{rk:fusioncoeff}), the fusion coefficient can be thought of as the multiplicity of the unit $\unit$ in $A\otimes B\otimes C^*$, where $C^*\in \Ccal_{\gamma_2^{-1}\gamma_1^{-1}}$ is the rigid dual of $C$, i.e. $$\nu_{A,B}^C=\linfun([A]\cdot[B]\cdot[C^*])=\dim\Hom(\unit,A\otimes B\otimes C^*).$$

	\subsection{Twisted characters}\label{sec:twistedchars}
	We will always assume that the grading on $\Ccal$ is faithful, i.e. $\Ccal_\gamma$ is non-zero for all $\gamma\in \Gamma$. We will also assume that the identity component $\Ccal_1$ is non-degenerate as a braided fusion category. We refer to \cite[\S2.8]{DGNO} for the notion of non-degeneracy of braided fusion categories. Consider an element $\gamma\in \Gamma$. Under our assumptions, the component $\Ccal_\gamma$ becomes an invertible $\Ccal_1$-bimodule category corresponding (under the equivalence $\underline{\operatorname{Pic}}(\Ccal_1)\cong \underline{\operatorname{EqBr}}(\Ccal_1)$ given by \cite[Thm. 5.2]{ENO:10}) to the braided autoequivalence $\gamma:\Ccal_1\xrightarrow{\cong} \Ccal_1$ induced by the $\Gamma$-action on $\Ccal$. Hence we are in the setting studied in \cite{Desh:17,Desh:18a}. Note that we also have the braided $\langle\gamma\rangle$-crossed category $\Ccal_{\langle\gamma\rangle}:=\bigoplus\limits_{g\in\langle\gamma\rangle}\Ccal_g\subseteq \Ccal$, where $\langle\gamma\rangle\leq \Gamma$ denotes the cyclic subgroup generated by $\Gamma$. 
	
	Let us recall the notion of twisted characters that can be defined in the setting of $\Gamma$-graded algebras. We refer to \cite{Arote:18} for more on the notion of twisted characters for general graded Frobenius $\ast$-algebras, which is in turn based on the Clifford theory of group-graded rings \cite{Dade}. According to this, the grading (\ref{eq:gradingongrothendieckalgebra}) determines a {\emph{partial action}} of $\Gamma$ on the set $\Irr(\KQab(\Ccal_1))$ of simple modules of the identity component $\KQab(\Ccal_1)$. This partial action is defined as follows: for each $\gamma\in\Gamma$ and a simple $\KQab(\Ccal_1)$-module $M$, we define the $\KQab(\Ccal_1)$-module ${}^{(\gamma)} M:=\KQab(\Ccal_\gamma)\otimes_{\KQab(\Ccal_1)}M$. The module ${}^{(\gamma)} M$ is either simple or $0$, see \cite[\S3.1]{Desh:18a} or \cite[\S4]{Arote:18} for details. 
	
	In our setting, in addition to the grading (\ref{eq:gradingongrothendieckalgebra}), $\KQab(\Ccal)$ is a $\Gamma$-crossed Frobenius $\Qab$-$*$-algebra. In particular we have an action of $\Gamma$ on the commutative Frobenius $\Qab$-$*$-algebra $\KQab(\Ccal_1)$. We also know from \cite{ENO:05} that $\KQab(\Ccal_1)$ is semisimple and split. Hence $\Irr(\KQab(\Ccal_1))$ is the set of all 1-dimensional representations $\rho:\KQab(\Ccal)\to \Qab$. The $\Gamma$-action on $\KQab(\Ccal_1)$ induces an honest action of $\Gamma$ on $\Irr(\KQab(\Ccal_1))$, which is different from the partial action coming from the $\Gamma$-grading (\ref{eq:gradingongrothendieckalgebra}) using Clifford theory (from \cite{Dade}) as described in the previous paragraph. They are related as follows:
	\begin{lemma} Let us denote the $\Gamma$-action on $\Irr(\KQab(\Ccal_1))$ coming from the $\Gamma$-action on $\KQab(\Ccal_1)$ as $\gamma :M\mapsto {}^\gamma M$ for $\gamma\in \Gamma,M\in \Irr(\KQab(\Ccal_1))$. Then the partial action coming from Clifford theory can be described as follows:
	\[{}^{(\gamma)}M=\begin{dcases*}
			M & if ${}^\gamma M\cong M$,\\
			0 & else.
			\end{dcases*}\]
	In particular for each $\gamma\in \Gamma$, the $\gamma$-fixed points for the action as well as the partial action coincide, and hence the fixed point set $\Irr(\KQab(\Ccal_1))^\gamma$ is unambiguously defined.
	\end{lemma}
	\begin{proof}
	This follows from Lemmas 3.3 and 3.5 from \cite{Desh:18a}.
	\end{proof}
 	
\noindent The following result from \cite{Desh:18a} is used to define twisted characters:
	\begin{lemma}\label{lem:lemfortwistedchars} (See \cite[Thm. 2.11(i)]{Desh:18a}.)
	Let $\gamma\in \Gamma$ be an element of order $m$ and let $\rho\in \Irr(\KQab(\Ccal_1))^\gamma$. Then $\rho$ can be extended to a 1-dimensional character $\wtilde{\rho}:\KQab(\Ccal_{\langle\gamma\rangle})\to\Qab$ of the $\langle\gamma\rangle$-crossed Frobenius $\ast$-algebra $\KQab(\Ccal_{\langle\gamma\rangle})$ in exactly $m$-distinct ways which differ on $\KQab(\Ccal_\gamma)$ up to scaling by $m$-th roots of unity. 
	\end{lemma}

	\begin{definition}
		Let $\gamma\in\Gamma$ and $\rho\in \Irr(\KQab(\Ccal_1))^\gamma$. Let $\wtilde{\rho}:\KQab(\Ccal_{\langle\gamma\rangle})\to\Qab$ be an extension. The $\gamma$-twisted character 
		${\rho}^\gamma:\KQab(\Ccal_\gamma)\to\Qab$
		is defined to be the restriction of $\wtilde{\rho}$ to $\KQab(\Ccal_\gamma)\subset \KQab(\Ccal_{\langle\gamma\rangle})$. Using the identification $\KQab(\Ccal_\gamma)^*\cong\KQab(\Ccal_{\gamma^{-1}})$ (\ref{eq:identificationwithdualCgamma}), let ${\alpha}^\gamma_\rho\in \KQab(\Ccal_{\gamma^{-1}})$ be the element corresponding to $\rho^\gamma$, namely 
		$$\alpha^\gamma_\rho:=\sum\limits_{A\in P_\gamma}\wtilde{\rho}([A])\cdot [A^*]\in \KQab(\Ccal_{\gamma^{-1}}).$$ 
		We will call either of $\rho^\gamma$ or $\alpha^\gamma_\rho$ as the $\gamma$-twisted character associated with $\rho\in \Irr(\KQab(\Ccal_1))^\gamma$. 
	\end{definition}
	\begin{remark}
		Note that $\rho^\gamma$ and $\alpha^\gamma_\rho$ depend on the choice of the extended character $\wtilde{\rho}$. Hence they are well-defined only up to scaling by $m$-th roots of unity, where $m$ is the order of $\gamma$ in $\Gamma$. We also refer to \cite{Arote:18} for the definition of twisted characters and their orthogonality relations in the setting of graded Frobenius $\ast$-algebras.
	\end{remark}
	For $\gamma=1$, we have the well-defined element $\alpha_\rho:=\alpha^1_\rho\in \KQab(\Ccal_1)$ for each $\rho\in \Irr(\KQab(\Ccal_1))$. The following results about these characters and $\gamma$-twisted characters are proved in \cite{Ost:15,Desh:18a}:
	
	\begin{theorem}\label{thm:twisted characters}
		\begin{enumerate}[(i)]
			\item Let $\rho,\rho'\in\Irr(\KQab(\Ccal_1)$. Then $\rho'(\alpha_\rho)=0$ if $\rho'\neq\rho$ and $f_\rho:=\rho(\alpha_\rho)\in\Qab$ is a totally positive cyclotomic integer known as the formal codegree of $\rho$. In other words, $e_\rho:=\frac{\alpha_\rho}{f_\rho}\in \KQab(\Ccal_1)$ is the minimal idempotent corresponding to $\rho\in\Irr(\KQab(\Ccal_1))$.
			\item  For $\rho,\rho'\in\Irr(\KQab(\Ccal_1))$, we have orthogonality of characters, i.e. $\langle\alpha_\rho,\alpha_{\rho'}\rangle=\delta_{\rho\rho'}\cdot f_\rho.$
			\item Let $\rho\in \Irr(\KQab(\Ccal_1))^{\gamma}$ and $\alpha^\gamma_{\rho}\in \KQab(\Ccal_{\gamma^{-1}})$ the corresponding $\gamma$-twisted character. Then for any $\rho'\in\Irr(\KQab(\Ccal_1))$, we have the following:
			$$\alpha_{\rho'}\cdot \alpha^\gamma_\rho = \alpha^\gamma_{\rho}\cdot \alpha_{\rho'}=
			\begin{dcases*}
			0 & if $\rho'\neq \rho$,\\
			f_\rho\cdot \alpha^\gamma_\rho & if $\rho'=\rho$.
			\end{dcases*}$$
			We have a direct sum decomposition
			$$\KQab(\Ccal_{\gamma^{-1}})=\bigoplus\limits_{\rho\in\Irr(\KQab(\Ccal_1))^{\gamma}} \Qab\cdot \alpha^\gamma_\rho \mbox{ as a $\KQab(\Ccal_1)$-module.}$$
			\item For $\rho,\rho'\in\Irr(\KQab(\Ccal_1))^{\gamma}$, we have orthogonality of twisted characters, $$\langle\alpha^\gamma_\rho,\alpha^\gamma_{\rho'}\rangle=\delta_{\rho\rho'}\cdot f_\rho.$$ In other words the $\gamma$-twisted characters $\{\alpha^\gamma_\rho|\rho\in \KQab(\Ccal_1)^\gamma\}$ form an orthogonal basis of $\KQab(\Ccal_{\gamma^{-1}})$ made up of eigenvectors for the action of $\KQab(\Ccal_1)$.
			\item  For $\rho\in \Irr(\KQab(\Ccal_1))^{\gamma}$, let $\wtilde{\rho}:\KQab(\Ccal_{\langle\gamma\rangle})\to\Qab$ be an extension. Let $n=|\langle\gamma\rangle|$. Then the formal codegree $f_{\wtilde{\rho}}$ equals $n\cdot f_\rho$. Define $e^\gamma_\rho:=\frac{\alpha^\gamma_\rho}{f_\rho}$. Then $(e^\gamma_\rho)^n=e_\rho$.
			\item For each $\gamma\in\Gamma$, we have $|P_1^\gamma|=|P_\gamma|=|\Irr(\KQab(\Ccal_1))^\gamma|$.
		\end{enumerate}
	\end{theorem}
	This result describes the structure of $\KQab(\Ccal_{\gamma^{-1}})$ as a $\KQab(\Ccal_1)$-module. Now suppose we have $\gamma_1,\gamma_2\in \Gamma$. Then for $\rho_i\in\Irr(\KQab(\Ccal_1))^{\gamma_i}$, we now describe the product $\alpha^{\gamma_2}_{\rho_2}\cdot \alpha^{\gamma_1}_{\rho_1}\in \KQab(\Ccal_{\gamma_2^{-1}\gamma_1^{-1}})$ of the twisted characters. We first consider the case where the $\rho_1,\rho_2$ are distinct. In this case, the product of the primitive idempotents $e_{\rho_2}\cdot e_{\rho_1}$ equals zero. Hence the product $$\alpha^{\gamma_2}_{\rho_2}\cdot \alpha^{\gamma_1}_{\rho_1}=\alpha^{\gamma_2}_{\rho_2}\cdot e_{\rho_2}\cdot e_{\rho_1}\cdot \alpha^{\gamma_1}_{\rho_1}$$ also equals zero in the case $\rho_1\neq \rho_2$.

	For $\rho\in\Irr(\KQab(\Ccal_1))$, let $\Gamma_\rho\leq \Gamma$ be the stabilizer of $\rho$ under the action of $\Gamma$ and $e_\rho\in\KQab(\Ccal_1)$ the associated primitive idempotent. For each $\gamma\in\Gamma_\rho$, we have the $\gamma$-twisted character $\alpha^\gamma_\rho$ and the element $e^\gamma_\rho$ (determined up to scaling by roots of unity). It follows from Theorem \ref{thm:twisted characters} that $$e_\rho\KQab(\Ccal)=\bigoplus\limits_{\gamma\in \Gamma_\rho}e_\rho\KQab(\Ccal_{\gamma^{-1}})=\bigoplus\limits_{\gamma\in \Gamma_\rho}\Qab\cdot e^{\gamma}_\rho.$$ 
	
	From this $\Gamma_\rho$-graded algebra, we obtain a central extension $$0\to{\Qab}^\times \to \wtilde{\Gamma}_\rho\to\Gamma_\rho\to 0.$$ Our choice of the elements $e^\gamma_\rho$ determines a 2-cocycle $\varphi_\rho:\Gamma_\rho\times \Gamma_\rho\to{\Qab}^\times$. Hence we obtain
	\begin{corollary}\label{cor:2cocyclerelations}
		\begin{enumerate}[(i)]
			\item If $\gamma_1,\gamma_2\in\Gamma$, $\rho_1\in\Irr(\KQab(\Ccal_1))^{\gamma_1}$, $\rho_2\in\Irr(\KQab(\Ccal_1))^{\gamma_2}$ and $\rho_1\neq \rho_2$, then the product of the corresponding twisted characters $\alpha^{\gamma_2}_{\rho_2}\cdot \alpha^{\gamma_1}_{\rho_1}$ equals $0$.
			\item Let $\rho\in\Irr(\KQab(\Ccal_1))$ and let $\gamma_1,\gamma_2\in\Gamma_\rho$. Then $\alpha^{\gamma_2}_{\rho}\cdot \alpha^{\gamma_1}_{\rho}=\varphi_\rho(\gamma_1,\gamma_2)\cdot f_\rho\cdot\alpha^{\gamma_1\gamma_2}_\rho$ and hence $e^{\gamma_2}_{\rho}\cdot e^{\gamma_1}_{\rho}=\varphi_\rho(\gamma_1,\gamma_2)\cdot e^{\gamma_1\gamma_2}_\rho$, where $\varphi_\rho$ is a 2-cocycle corresponding to the central extension $$0\to{\Qab}^\times \to \wtilde{\Gamma}_\rho\to\Gamma_\rho\to 0.$$ Equivalently, for $r_i\in \KQab(\Ccal_{\gamma_i})$ we have the relation $$\rho^{\gamma_1}(r_1)\rho^{\gamma_2}(r_2)=\varphi_\rho(\gamma_1,\gamma_2)\rho^{\gamma_1\gamma_2}(r_1r_2).$$
		\end{enumerate}
	\end{corollary}
	\begin{remark}\label{rk:ncocycle}
		Let $\rho\in\Irr(\KQab(\Ccal_1))$. For any $\gamma_1,\cdots,\gamma_n$ in the stabilizer $\Gamma_\rho$ of $\rho$, let $\varphi_\rho(\gamma_1,\cdots,\gamma_n)$ be such that $\rho^{\gamma_1}(r_1)\cdots\rho^{\gamma_n}(r_n)=\varphi_\rho(\gamma_1,\cdots,\gamma_n)\rho^{\gamma_1\cdots\gamma_n}(r_1\cdots r_n)$. These scalars will appear in our twisted categorical Verlinde formula. Furthermore, for each pair $(\rho,\gamma)$ such that $\gamma\in\Gamma_\rho$ we have chosen the elements $e^\gamma_\rho\in \KQab(\Ccal_{\gamma^{-1}})$ to be such that $(e^\gamma_\rho)^m=1$ where $m$ is the order of $\gamma$. It will be notationally convenient to further assume that we always have $e^\gamma_\rho\cdot e^{\gamma^{-1}}_\rho=1$. It is clear that such a choice can always be made. This ensures the equality $\varphi_\rho(\gamma,\gamma^{-1})=1$.
	\end{remark}
	
	\subsection{Categorical crossed S-matrices}\label{sec:catcrosssmat}
	Let us now suppose that the braided $\Gamma$-crossed fusion category $\Ccal$ is also equipped with a ribbon or spherical structure. We refer to \cite[\S2.4.3]{DGNO},\cite[\S4.1]{Desh:17} for details, here we just give a brief sketch. In any fusion category $\Ccal$, the autoequivalence $\Ccal\ni C\mapsto C^{**}$ is monoidal. 
	
\begin{definition}	
A \emph{pivotal structure} is a tensor isomorphism between the identity functor on $\Ccal$ and the double duality functor $C \mapsto C^{**}$.
\end{definition} 
A pivotal structure allows us to define traces of endomorphisms of objects of $\Ccal$, and in particular we can define the categorical dimensions of objects of $\Ccal$ as the trace of the identity endomorphism of the object. Moreover, taking categorical dimensions defines a character $\dim:\KQab(\Ccal)\to \Qab$ and we have $\dim(C^*)=\overline{\dim(C)}\in \Qab$ for any $C\in \Ccal$. 
	
\begin{definition}	A pivotal structure is said to be \emph{spherical} if the categorical dimension of any object and its dual are equal, or equivalently, if the categorical dimension is a totally real cyclotomic number. 
\end{definition}
We recall the definition of a modular fusion category and $\Gamma$-crossed modular fusion category.
	\begin{definition}\label{def:crossedmodcat}
		A modular fusion category is a non-degenerate braided fusion category equipped with a spherical structure. A $\Gamma$-crossed modular fusion category (often also called $\Gamma$-crossed modular category) is a braided $\Gamma$-crossed fusion category $\Ccal=\bigoplus\limits_{\gamma\in \Gamma}\Ccal_\gamma$ equipped with a spherical structure such that $\Ccal_1$ is a non-degenerate braided fusion category and each component $\Ccal_\gamma$ is non-zero. The multiplicative central charge of a $\Gamma$-crossed modular fusion category is defined to be the multiplicative central charge of the modular category $\Ccal_1\subset \Ccal$. 
	\end{definition}
	\begin{remark}
		For the definition of multiplicative central charge of a modular fusion category we refer to \cite[Def. 5.7.9 and Thm. 5.7.11]{BK:01} and \cite[\S6.2]{DGNO} where this is defined using the notion of Gauss sums of modular categories. We will not make any use of these notions in this paper. When we study the notion of modular functors in Part III we will see that the multiplicative central charge can be thought of as the ``anomaly'' of the corresponding modular functor.
	\end{remark}
	
	Let $\Ccal$ be a $\Gamma$-crossed modular fusion category. Then $\Ccal_1$ is a modular fusion category equipped with a modular action of $\Gamma$ and each $\Ccal_\gamma$ is equipped with a $\Ccal_1$-module trace. Because of the spherical structure, trace of endomorphisms in $\Ccal$, and hence categorical dimensions of objects in $\Ccal$ are now well-defined. In this setting we can define a $P_\gamma\times P_1^\gamma$-matrix $\wtilde{S}^\gamma$ known as the unnormalized $\gamma$-crossed S-matrix as follows. 
	\begin{definition}
		For each $\gamma$-stable simple object $C\in P_1^\gamma$, let us choose an isomorphism $\psi_C:\gamma(C)\xrightarrow{\cong} C$ as in \cite[\S2.1]{Desh:17}, i.e. such that the induced composition $$C=\gamma^m(C)\xrightarrow{\gamma^{m-1}(\psi_C)} \gamma^{m-1}(C)\to\cdots \to \gamma(C)\rar{\psi_C} C$$ is the identity, where $m\in \ZBbb$ is the order of $\gamma$. Note that this condition on $\psi_C$ determines it up to scaling by $m$-th roots of unity. For simple objects $M\in P_\gamma, C\in P_1^\gamma$, we set 
		$$\wtilde{S}^\gamma_{M,C}=\tr(C\otimes M\xrightarrow{\beta_{C,M}}M\otimes C\xrightarrow{\beta_{M,C}}\gamma(C)\otimes M\xrightarrow{\psi_C\otimes\id_M} C\otimes M).$$
	\end{definition}
	\begin{remark}\label{rk:unitconvention}
		We follow the convention from \cite{Desh:17} that in case the simple object $C\in P_1^\gamma$ is the unit $\unit$, we always choose $\psi_{\unit}:\gamma(\unit)=\unit\to \unit$ to be the identity. With this convention we see that 
	\begin{equation}\label{eqn:cattwistedSmatrix}
	\wtilde{S}^\gamma_{M,\unit}=\tr(\id_M)=\dim (M),
	\end{equation}
		the categorical dimension of the object $M\in \Ccal_\gamma$.
	\end{remark}
	
	\begin{remark}\label{rk:diffsofnotations}
		The $\gamma$-crossed S-matrix as defined in Equation \eqref{eqn:cattwistedSmatrix} is the transpose of the crossed S-matrix defined in \cite{Desh:17}. Also in {\it op. cit.} the first author only deals with unnormalized crossed S-matrices as defined in Equation \eqref{eqn:cattwistedSmatrix}, whereas in the current paper we will mostly work with the normalized unitary crossed S-matrix as defined in Definition \ref{d:normalizedcrossedsmatrix} below. This explains the differences of notation in the categorical Verlinde formula of the two papers.
	\end{remark}
	In the presence of the spherical structure, we have the unnormalized S-matrix $\wtilde{S}$ of the modular fusion category $\Ccal_1$. Using this S-matrix, we can identify $P_1\cong \Irr(\KQab(\Ccal_1))$ as follows: $$P_1\ni C\mapsto (\rho_C:[D]\mapsto \frac{\wtilde{S}_{D,C}}{\dim C}),$$ i.e. the S-matrix is essentially the character table of $\KQab(\Ccal_1)$. Similarly, the $\gamma$-crossed S-matrix is essentially the table of $\gamma$-twisted characters. More precisely we recall the following results from \cite{Desh:17,Desh:18a}:
	\begin{theorem}\label{thm:crossedsmatrixchartable}
		(i) For $M\in P_\gamma, C\in P_1^\gamma$ the numbers $\frac{\wtilde{S}^\gamma_{M,C}}{\dim M}=\frac{\wtilde{S}^\gamma_{M,C}}{\wtilde{S}^\gamma_{M,\unit}}$ and $\frac{\wtilde{S}^\gamma_{M,C}}{\dim C}=\frac{\wtilde{S}^\gamma_{M,C}}{\wtilde{S}_{\unit,C}}$ are cyclotomic integers. For $C\in P_1^\gamma$ the linear functional $${\rho}^\gamma_C:\KQab(\Ccal_{\gamma})\to \Qab, \ [M]\mapsto \frac{\wtilde{S}^\gamma_{M,C}}{\dim C}$$ is the $\gamma$-twisted character associated with $\rho_C\in \Irr(\KQab(\Ccal_1))^\gamma$.\\
		(ii) The categorical dimension $\dim \Ccal_1$ is a totally positive cyclotomic integer. We have $\wtilde{S}^\gamma\cdot \overline{\wtilde{S}^\gamma}^T=\overline{\wtilde{S}^\gamma}^T\cdot \wtilde{S}^\gamma=\dim \Ccal_1\cdot I.$
	\end{theorem}
	For $C\in P_1$, the formal codegree $f_C=f_{\rho_C}$ of the corresponding character $\rho_C$ equals the totally positive cyclotomic integer $\frac{\dim \Ccal_1}{\dim^2 C}$.
	\begin{definition}\label{d:normalizedcrossedsmatrix}
		We choose a positive square root $\sqrt{\dim \Ccal_1}$ and define the normalized $\gamma$-crossed S-matrix to be $S^\gamma:=\frac{1}{\sqrt{\dim\Ccal_1}}\cdot \wtilde{S}^\gamma$. The matrix $S^\gamma$ is a $P_\gamma\times P_1^\gamma$ unitary matrix with entries in $\Qab$. 
	\end{definition}
	
	\begin{remark}\label{rk:0column}
		By Remark \ref{rk:unitconvention} we see that the $\unit$-th column of the normalized crossed S-matrix is given by 
		$S^\gamma_{M,\unit}=\frac{\dim M}{\sqrt{\dim \Ccal_1}}.$
		Now the categorical dimensions are all real since they are defined using a spherical structure on $\Ccal$. Hence all the entries in the $\unit$-th column are real numbers. 
	\end{remark}
	
	\section{Categorical twisted fusion rings and their character tables}\label{sec:twistedfusion}
	We now recall from \cite[\S2.2,\S2.3]{Desh:17} the definition of categorical twisted fusion rings (also known as twisted Grothendieck rings) associated with a fusion category equipped with an autoequivalence. Let $\Ccal$ be a finite semisimple $\CBbb$-linear abelian category equipped with a $\mathbb{C}$-linear autoequivalence $\gamma:\Ccal\to \Ccal$. Let $P$ denote the finite set of (isomorphism classes of) simple objects of $\Ccal$. Consider the (non-semisimple) abelian category $\Ccal^\gamma$ obtained by $\gamma$-equivariantization whose objects are pairs $(C,\psi)$ where $C$ is an object $\Ccal$ and $\psi:\gamma(C)\xoto\cong C$. 
	\begin{definition}\label{def:categoricaltwistedfusion}
		In the setting above, the twisted complexified Grothendieck group $\KC(\Ccal,\gamma)$ is defined as follows:  Consider the complexified Grothendieck group $\KC(\Ccal^\gamma)$ and then quotient out by the relation $[C,c\psi]=c[C,\psi]$ for each class $[C,\psi]\in \KC(\Ccal^\gamma)$ and $c\in \CBbb^\times$. We continue to use the notation $[C,\psi]$ to denote the corresponding class in the quotient $\KC(\Ccal,\gamma)$.
	\end{definition}
	
	Note that $\KC(\Ccal,\gamma)$ is just a  $\mathbb{C}$-vector space that has been defined as a quotient of the (typically infinite dimensional) $\mathbb{C}$-vector space $\KC(\Ccal^\gamma)$ by a subspace spanned by certain relations. We see now that it is in fact finite dimensional:
	
	\begin{proposition}\label{p:dimoftwistedspace}
		(i) We have $\dim\KC(\Ccal,\gamma)=|P^\gamma|$, the number of simple objects stabilized by $\gamma$. For each $C\in P^\gamma$, choose an isomorphism $\psi_C:\gamma(C)\to C$. Then the set $\{[C,\psi_C]|C\in P^\gamma\}$ is a basis of $\KC(\Ccal,\gamma)$.\newline
		(ii) If $\Ccal$ also has a monoidal structure and if $\gamma$ is a monoidal autoequivalence, then $\KC(\Ccal,\gamma)$ has the structure of a $\CBbb$-algebra, which we call the (complexified) twisted fusion ring.
	\end{proposition}
	\begin{proof}
		Note that we can consider $\Ccal$ as a module category over the monoidal category $\Vec$ of finite dimensional $\CBbb$-vector spaces. Now the statement (i) follows from the proof of \cite[Prop. 2.4(i)]{Desh:17}. We essentially need to prove that for any $\gamma$-orbit $O\subset P$ of cardinality at least 2, $\KC(\Ccal_O,\gamma)=0$, where $\Ccal_O\subset \Ccal$ is the full subcategory generated by the simple objects in $O$. As in {\it loc. cit.} this can be done by looking at the isomorphism classes of simple objects in $\Ccal_O^\gamma$ and noting that $1-\zeta\in \CBbb^\times$ must act by 0 on $\KC(\Ccal_O,\gamma)$ for each $|O|$-th root of unity $\zeta\in \CBbb^\times$. Now to prove (ii), it is straightforward to check that the subspace of $\KC(\Ccal^\gamma)$ that we are quotienting by is an ideal.
	\end{proof}
	
	Suppose we have an object $(A,\psi)\in \Ccal^\gamma$ for a finite semisimple $\CBbb$-linear abelian category $\Ccal$. For each simple object $C$, define the multiplicity vector space $N_A^C:=\Hom(C,A)$. We obtain the isomorphism of vector spaces
	$$N_A^C=\Hom(C,A)\cong \Hom(\gamma(C),\gamma(A))\xoto{\psi\circ(\cdot)}\Hom(\gamma(C),A)=N_A^{\gamma(C)}.$$
	Now for each simple object in $\Ccal$ stabilized by $\gamma$ let us choose an isomorphism $\psi_C:\gamma(C)\xoto\cong C.$ These choices give us the linear isomorphisms $N_C^{\gamma(C)}\xoto{}N_A^C$ and hence the linear automorphisms $\phi_{A,\psi}^{C,\psi_C}:N_A^C\xoto{\cong}N_A^C.$ Then using Proposition \ref{p:dimoftwistedspace}, we have the following equality in the twisted Grothendieck vector space $\KC(\Ccal,\gamma)$
	\begin{equation}\label{eq:equalityintwistedgrothendieck}
	[A,\psi]=\sum\limits_{C\in P^\gamma}\tr(\phi_{A,\psi}^{C,\psi_C})[C,\psi_C].
	\end{equation}

	\subsection{The twisted fusion ring in terms of traces on multiplicity spaces} Let us now assume that $\Ccal$ is a weakly fusion category equipped with a monoidal action of a finite group $\Gamma$. Let $\gamma\in \Gamma$ be an element of order $m$. In particular we have an equivalence $\gamma^m\cong \id_{\Ccal}$ of monoidal functors. Then the $\langle\gamma\rangle\cong \ZBbb/m\ZBbb$-equivariantization $\Ccal^{\langle\gamma\rangle}$ is also a weakly fusion category. We can consider the twisted fusion ring $\KC(\Ccal,\gamma)$ as a quotient of $\KC(\Ccal^{\langle\gamma\rangle})$ modulo the relations $[C,\zeta\psi]=\zeta[C,\psi]$ for any $m$-th root of unity $\zeta$ (see \cite[\S2.3]{Desh:17} for details). As in {\it loc. cit.} the weak duality on $\Ccal^{\langle\gamma\rangle}$ induces a $\CBbb$-semilinear involution $(\cdot)^*$ on $\KC(\Ccal,\gamma)$ as a vector space.
	
	Consider objects $(C_1,\psi_1),\cdots (C_n,\psi_n)$ in $\Ccal^{\langle\gamma\rangle}$. Then we have the following composition of natural isomorphisms 
	\begin{align*}\phi_{C_1,\psi_1,\cdots,C_n,\psi_n}:&\Hom(\unit,C_1\otimes\cdots \otimes C_n)\xoto{\cong}\Hom(\gamma(\unit),\gamma(C_1\otimes\cdots \otimes C_n))\\
	&\quad \xoto{\cong}\Hom(\unit,\gamma(C_1)\otimes\cdots \otimes \gamma(C_n))\xoto\cong\Hom(\unit,C_1\otimes\cdots \otimes C_n).\end{align*}
	
	Now let $A,B,C$ be simple objects of $\Ccal$. Since we have assumed that $\Ccal$ is weakly fusion, the multiplicity space $N^C_{A\otimes B}=\Hom(C,A\otimes B)$ is canonically identified with $\Hom(\unit,{ }^*C\otimes A\otimes B)$.

	\begin{proposition}\label{p:tracesonhomspaces}
		For each simple object $C\in \Ccal$ stabilized by $\gamma$, let us fix a $\langle\gamma\rangle$-equivariant structure $\psi_C:\gamma(C)\to C$ such that $(C,\psi_C)\in \Ccal^{\langle\gamma\rangle}$. Then the set $\{[C,\psi_C]|C\in P^\gamma\}$ is a basis of the $\gamma$-twisted fusion ring $\KC(\Ccal,\gamma)$ and for each $A,B\in P^\gamma$ we have the following equality in the twisted fusion ring $\KC(\Ccal,\gamma)$
		\begin{equation}\label{eq:equalityintwistedfusionring}
		[A,\psi_A]\cdot[B,\psi_B]=\sum\limits_{C\in P^\gamma}\tr(\phi_{{}^*C,{}^*\psi_C,A,\psi_A,B,\psi_B})[C,\psi_C].
		\end{equation}
		In other words, the fusion coefficients for the twisted fusion ring $\KC(\Ccal,\gamma)$ with respect to the basis $\{[C,\psi_C]|C\in P^\gamma\}$ are obtained by taking the traces of $\gamma$ on the multiplicity spaces  $N^C_{A\otimes B}=\Hom(C,A\otimes B)=\Hom(\unit,{ }^*C\otimes A\otimes B)$.
	\end{proposition}
	\begin{proof}
		The result follows by using (\ref{eq:equalityintwistedgrothendieck}) for the tensor product $(A,\psi_A)\otimes (B,\psi_B)\in \Ccal^{\langle\gamma\rangle}$.
	\end{proof}
	\begin{remark}\label{rk:twistedfusionringsagree}
		This means that the twisted fusion ring $\KC(\Ccal,\gamma)$ can also be defined by taking the fusion rules (see Section \ref{section:hongfusion} below) obtained by tracing out automorphisms of multiplicity spaces. This approach was used in \cite{Hong1,Hong} to define twisted fusion rings at level $\ell$ associated to a diagram automorphism of a simple Lie algebra $\frg$.
	\end{remark}
	\subsection{Twisted fusion ring of fusion categories}
	Let $\Ccal$ be a fusion category with a $\Gamma$-action and let $\gamma\in \Gamma$ be an element of order $m$. Then the equivariantization $\Ccal^{\langle\gamma\rangle}$ is also a fusion category and hence by Section \ref{sec:grothringfrobalg} the extension by scalars of the Grothendieck ring, $K_{\ZBbb[\omega]}(\Ccal^{\langle\gamma\rangle})$ is a Frobenius $\ast$-$\ZBbb[\omega]$-algebra where we take $\omega$ to be a primitive $m$-th root of unity. As in \cite[\S2.3]{Desh:17} we can define the $\gamma$-twisted fusion ring $K_{\ZBbb[\omega]}(\Ccal,\gamma)$ to be the quotient of the above Frobenius $\ast$-algebra obtained by identifying $[C,\omega\psi]=\omega[C,\psi]$ for each $(C,\psi)\in \Ccal^{\langle\gamma\rangle}$. As always, for each $\gamma$-stable simple object $C\in P^\gamma$, let us fix an equivariantization $(C,\psi_C)\in \Ccal^{\langle\gamma\rangle}$. Then by \cite{Desh:17} we have
	\begin{proposition}\label{p:twistedfusionfrobalg}
		The $\gamma$-twisted fusion ring $K_{\ZBbb[\omega]}(\Ccal,\gamma)$ is a Frobenius $\ast$-$\ZBbb[\omega]$-algebra with an orthonormal basis given by $\{[C,\psi_C]|C\in P^\gamma\}$ and hence the same is true for the extensions of scalars $\KQab(\Ccal,\gamma)$ and $\KC(\Ccal,\gamma)$.
	\end{proposition}

	\subsection{Character tables of twisted fusion rings of modular categories}
	Let us now assume that $\Ccal=\bigoplus\limits_{\gamma\in \Gamma}\Ccal_\gamma$ is a $\Gamma$-crossed modular category. In particular, we have a modular action of $\Gamma$ on the identity component $\Ccal_1$ which is a modular category. Hence for each $\gamma\in \Gamma$ we have the twisted fusion ring $\KC(\Ccal_1,\gamma)$ which in this case is a commutative Frobenius $\ast$-algebra that has been studied in \cite{Desh:17}. We recall that for each $\gamma\in \Gamma$, $P_\gamma$ denotes the set of simple objects of $\Ccal_\gamma$. For each simple object $C$ of $\Ccal_1$ stabilized by $\gamma$, we choose an isomorphism $\psi_C:\gamma(C)\to C$ such that $(C,\psi_C)\in \Ccal_1^{\langle\gamma\rangle}$. 
	
	The set $\{[C,\psi_C]|C\in P_1^\gamma\}$ is a basis of $\KC(\Ccal_1,\gamma)$. It was proved in \cite[Thm. 2.12(i)]{Desh:17} that the categorical $\gamma$-crossed S-matrix $S^\gamma$ (it is a $P_\gamma\times P_1^\gamma$ matrix, see Definition \ref{d:normalizedcrossedsmatrix}) is essentially the character table of $\KC(\Ccal_1,\gamma)$:
	\begin{theorem}\label{thm:chartableoftwistedfusion}
		Let $M$ be a simple object of $\Ccal_\gamma$. For each $C\in P_1^\gamma$, define $\chi_M([C,\psi_C])=\frac{S^\gamma_{M,C}}{S^{\gamma}_{M,\unit}}$ and extend linearly. Then $\chi_M:\KC(\Ccal_1,\gamma)\to \CBbb$ is a character and the set $\{\chi_M|M\in P_\gamma\}$ is the set of all irreducible characters of $\KC(\Ccal_1,\gamma)$. 
	\end{theorem}

	\begin{remark}
		In fact the $\gamma$-twisted fusion ring can be defined over $\ZBbb[\omega]$ and hence over $\Qab$, where $\omega$ is a primitive $m$-th root of unity (see \cite[\S2.2]{Desh:17} for details). Hence $\KQab(\Ccal_1,\gamma)$ is also well-defined and Theorem \ref{thm:chartableoftwistedfusion} shows that it is split over $\Qab$.
	\end{remark}

	\begin{remark}
		Let us compare Theorems \ref{thm:crossedsmatrixchartable} and  \ref{thm:chartableoftwistedfusion}. On the one hand Theorem \ref{thm:crossedsmatrixchartable} says that the $\gamma$-crossed S-matrix gives the $\gamma$-twisted character table of the untwisted fusion ring $\KQab(\Ccal_1)$. On the other hand Theorem \ref{thm:chartableoftwistedfusion} says that the $\gamma$-crossed S-matrix also gives the usual character table of the $\gamma$-twisted fusion ring $\KQab(\Ccal_1,\gamma)$. Both these properties are going to be important for us.
	\end{remark}

	\section{A categorical  twisted Verlinde formula}\label{sec:catverlinde}
	We will now state and prove a Verlinde formula for braided $\Gamma$-crossed categories. For $\gamma_1,\gamma_2\in \Gamma$, $A\in P_{\gamma_1}, B\in P_{\gamma_2}, C\in P_{\gamma_1\gamma_2}$ we will compute the fusion coefficient $\nu_{A,B}^C$ in terms of the crossed S-matrices, or equivalently, in terms of the twisted character tables. By the rigid duality, it is clear that $$\nu_{A,B}^C=\nu_{B^*,A^*}^{C^*}=\dim\Hom(\unit,A\otimes B\otimes C^*)=\nu([A]\cdot[B]\cdot[C^*]).$$ Note that the fusion product in $\KQab(\Ccal)$ is given by: 
	\begin{equation}
	\label{eqn:basicmultiplication}
		[A]\cdot [B]=\sum\limits_{C\in P_{\gamma_1\gamma_2}}\nu^C_{A,B}[C].
	\end{equation}
	
	\begin{remark}
		The coefficients $\nu_{A,B}^C$ described in Equation \eqref{eqn:basicmultiplication} are multiplicities in the Grothendieck ring $\KQab({\Ccal})$ of the braided $\Gamma$-crossed category and not in the categorical $\gamma$-twisted fusion ring $K_{\Qab}(\Ccal_1, \gamma)$. In particular the $\nu^C_{A,B}$ are non-negative integers. As we will see, it is these numbers which equal the ranks of bundles conformal blocks where the base curve is $\PBbb^1$ with 3 points having ramification given by $\gamma_1,\gamma_2$ and $\gamma_2^{-1}\gamma_1^{-1}\in \Gamma$. On the other hand, by Proposition \ref{p:twistedfusionfrobalg}, the fusion coefficients in the twisted fusion ring $\KQab(\Ccal_1,\gamma)$ lie in the ring $\ZBbb[\omega]$ where $\omega$ is a primitive $|\gamma|$-th root of unity and need not be non-negative integers in general.
	\end{remark}
	More generally, let $\gamma_1,\cdots,\gamma_n\in \Gamma$ with $\gamma_1\cdots\gamma_n=1$. We will in fact prove an $n$-pointed twisted Verlinde formula to compute $\dim\Hom(\unit,A_1\otimes\cdots\otimes A_n)=\nu([A_1]\cdots[A_n])$ for $A_i\in P_{\gamma_i}$. We will use the following:
	\begin{lemma}\label{lem:unfixedchars}
		Let $\gamma_1,\cdots,\gamma_n\in \Gamma$ and let $A_i\in\Ccal_{\gamma_i}$, so that $A_1\otimes\cdots \otimes A_n\in \Ccal_{\gamma_1\cdots\gamma_n}$. Let $\rho\in \Irr(\KQab(\Ccal_1))^{\gamma_1\cdots\gamma_n}$ be such that (the $\gamma_1\cdots\gamma_n$-twisted character) $\rho^{\gamma_1\cdots\gamma_n}([A_1]\cdots[A_n])\neq 0$. Then $\rho\in \Irr(\KQab(\Ccal_1))^{\langle\gamma_1,\cdots,\gamma_n\rangle}$. 
	\end{lemma}
	\begin{proof}
		Consider any $A\in \Ccal_1$.. Then by the crossed braiding isomorphisms, for each $i$ we have $$A\otimes (A_1\otimes\cdots\otimes A_i)\cong (A_1\otimes\cdots\otimes A_i)\otimes A\cong (\gamma_1\cdots\gamma_i)(A)\otimes (A_1\otimes\cdots\otimes A_i).$$ Hence for each $i$
		$$A\otimes A_1\otimes\cdots\otimes A_n\cong (\gamma_1\cdots\gamma_i)(A)\otimes A_1\otimes\cdots\otimes A_n\mbox{ as objects of }\Ccal_{\gamma_1\cdots\gamma_n}.$$ 
		Hence $$\rho^{\gamma_1\cdots\gamma_n}([A]\cdot[A_1]\cdots[A_n])=\rho^{\gamma_1\cdots\gamma_n}((\gamma_1\cdots\gamma_n)[A]\cdot[A_1]\cdots[A_n])$$.  Thus $$\rho([A])\cdot\rho^{\gamma_1\cdots\gamma_n}([A_1]\cdots[A_n])=\rho(\gamma_1\cdots\gamma_i[A])\cdot\rho^{\gamma_1\cdots\gamma_n}([A_1]\cdots[A_n]).$$ Since we have assumed that $\rho^{\gamma_1\cdots\gamma_n}([A_1]\cdots[A_n])\neq 0$, we conclude that $\rho([A])=\rho(\gamma_1\cdots\gamma_i[A])$ for any $A\in\Ccal_1$ and $1\leq i\leq n$. Hence $\rho\in \Irr(\KQab(\Ccal_1))^{\gamma_1\cdots\gamma_i}$ for each $i$ and we conclude that $\rho$ is fixed by each $\gamma_i$ as desired.
	\end{proof}
	\begin{theorem}\label{thm:twisted verlinde} (Categorical twisted Verlinde formula in genus 0.)
		Let $\gamma_1,\cdots,\gamma_n\in \Gamma$ with the condition $\gamma_1\cdots\gamma_n=1$. Let $A_i\in \Ccal_{\gamma_i}$ for $1\leq i \leq n$. Then
		\begin{enumerate}[I]
			\item $$\dim\Hom(\unit,A_1\otimes \cdots \otimes A_n)=\sum\limits_{{{\rho\in\Irr(\KQab(\Ccal_1))^{\langle\gamma_1,\cdots,\gamma_n\rangle}}}}\frac{{\rho^{\gamma_1}([A_1])}\cdots\rho^{\gamma_n}([A_n])}{f_\rho \cdot\varphi_\rho(\gamma_1,\cdots,\gamma_n)},$$
			where $f_\rho$ is the formal codegree and $\varphi_\rho(\gamma_1,\cdots,\gamma_n)$ are the scalars defined in Remark \ref{rk:ncocycle}.
			\item Now let $A_i\in P_{\gamma_i}$ be simple objects. If $\Ccal$ is equipped with a spherical structure, then in terms of the crossed S-matrices we have
			$$\dim\Hom(\unit,A_1\otimes\cdots\otimes A_n)=\sum\limits_{D\in P_1^{\langle\gamma_1,\cdots,\gamma_n\rangle}}\frac{{S^{\gamma_1}_{A_1,D}}\cdots{S^{\gamma_n}_{A_n,D}}}{{S_{\unit,D}}^{(n-2)}\cdot \varphi_D(\gamma_1,\cdots,\gamma_n)}.$$
		\end{enumerate}
	\end{theorem}
	\begin{proof}
		To prove (i), observe that in the Frobenius $*$-algebra $\KQab(\Ccal_1)$ the unit $$1=\sum\limits_{\rho\in \Irr(\KQab(\Ccal_1))}\frac{\alpha_\rho}{f_\rho}$$ is expressed as a sum of the minimal idempotents. Using the fact $\KQab(\Ccal_1)\cong \KQab(\Ccal_1)^*$, this translates to the equality $\nu=\sum\limits_{\rho\in \Irr(\KQab(\Ccal_1))}\frac{\rho}{f_\rho}$ in $\KQab(\Ccal_1)^*$. Hence
		\begin{eqnarray*}
			\dim\Hom(\unit,A_1\otimes\cdots\otimes A_n)&=&\nu([A_1]\cdots[A_n]),\\
			&=&\sum\limits_{\rho\in \Irr(\KQab(\Ccal_1))}\frac{\rho([A_1]\cdots[A_n])}{f_\rho},\\
			&=&\sum\limits_{\rho\in \Irr(\KQab(\Ccal_1))^{\langle\gamma_1,\cdots,\gamma_n\rangle}}\frac{\rho([A_1]\cdots[A_n])}{f_\rho}\mbox{ }\mbox{ }\mbox{(by Lemma \ref{lem:unfixedchars})},\\
			&=&\sum\limits_{\rho\in \Irr(\KQab(\Ccal_1))^{\langle\gamma_1,\cdots,\gamma_n\rangle}}\frac{\rho^{\gamma_1}([A_1])\cdots\rho^{\gamma_n}([A_n])}{f_\rho\cdot \varphi_{\rho}(\gamma_1,\cdots,\gamma_n)}.\\
		\end{eqnarray*}
		The last line follows from Remark \ref{rk:ncocycle}. Now 
		to prove (ii), observe that in the spherical setting we have a $\Gamma$-equivariant bijection $P_1\cong\Irr(\KQab(\Ccal_1))$ denoted by $D\leftrightarrow \rho_D$. By Theorem \ref{thm:crossedsmatrixchartable} and Definition \ref{d:normalizedcrossedsmatrix} for each $D\in P_1^{\langle\gamma_1,\cdots,\gamma_n\rangle}$ we have $\rho_D^{\gamma_i}([A_i])=\frac{S^{\gamma_i}_{A_i,D}}{S_{\unit,D}}$ and the formal codegree $f_D=f_{\rho_D}=\frac{\dim \Ccal_1}{\dim^2 D}=\frac{1}{(S_{\unit,D})^2}$. Statement (ii) now follows from (i).
	\end{proof}
	\subsection{A higher genus twisted Verlinde formula}\label{sss:highergenus}
	In order to motivate the higher genus Verlinde formula for twisted conformal blocks, let us derive a categorical version of such a formula. Let $a,b\in \Gamma$. For a simple object $A\in P_a$, $b(A^*)$ is a simple object of $\Ccal_{ba^{-1}b^{-1}}$. We define the object 
	\begin{equation}\label{eq:defomegaab}
	\Omega_{a,b}:=\bigoplus\limits_{A\in P_a}A\otimes b(A^*)\in \Ccal_{[a,b]}, \mbox{ where }[a,b]=aba^{-1}b^{-1}.
	\end{equation}
	\begin{lemma}\label{lem:fixedbycommutator}
		Let $a,b\in \Gamma$ and let $\rho\in\Irr(\KQab(\Ccal_1))^{[a,b]}$ be a character fixed by the commutator $[a,b]$. Then 
		$$\rho^{[a,b]}([\Omega_{a,b}])=\begin{dcases*}
		\frac{f_\rho}{\varphi_\rho(a,b,a^{-1},b^{-1})} & if $\rho\in \Irr(\KQab(\Ccal_1))^{\langle a,b\rangle}$ (see Remark \ref{rk:ncocycle})\\
		0 & else.
		\end{dcases*}$$
	\end{lemma}
	\begin{proof}
		We have the following equality $$\rho^{[a,b]}([\Omega_{a,b}])=\rho^{[a,b]}\left(\sum\limits_{A\in P_a}[A]\cdot b[A^*]\right)=\sum\limits_{A\in P_a}\rho^{[a,b]}\left([A]\cdot b[A^*]\right).$$ By Lemma \ref{lem:unfixedchars}, if $\rho^{[a,b]}\left([A]\cdot b[A^*]\right)\neq 0$ for some $A\in P_a$, then we must have $\rho$ as an element in $\Irr(\KQab(\Ccal_1))^{\langle a, ba^{-1}b^{-1}\rangle}$. Otherwise each individual term in the summation and hence  $\rho^{[a,b]}([\Omega_{a,b}])$ must be zero. Hence let us assume that $\rho\in \Irr(\KQab(\Ccal_1))^{\langle a, ba^{-1}b^{-1}\rangle}$. In this case, 
		\begin{equation}\label{eq:omegaab}
		\rho^{[a,b]}([\Omega_{a,b}])=\sum\limits_{A\in P_a}\frac{\rho^{a}([A])\cdot \rho^{ba^{-1}b^{-1}}(b[A^*])}{\varphi_\rho(a,ba^{-1}b^{-1})}=\sum\limits_{A\in P_a}\frac{\rho^{a}([A])\cdot (\rho^{ba^{-1}b^{-1}}\circ b)([A^*])}{\varphi_\rho(a,ba^{-1}b^{-1})}.
		\end{equation}
		Recall that here $\rho^{ba^{-1}b^{-1}}:\KQab(\Ccal_{ba^{-1}b^{-1}})\to \Qab$ is the restriction of some choice of character 
		$$\wtilde{\rho^{ba^{-1}b^{-1}}}:\KQab(\Ccal_{\langle ba^{-1}b^{-1}\rangle})\to \Qab\mbox{ extending }\rho:\KQab(\Ccal_1)\to \Qab.$$ 
		Hence $\wtilde{\rho^{ba^{-1}b^{-1}}}\circ b:\KQab(\Ccal_{\langle a^{-1}\rangle})\to \Qab$ is a character extending $\rho\circ b:\KQab(\Ccal_1)\to \Qab.$ Hence the composition ${\rho^{ba^{-1}b^{-1}}}\circ b:\KQab(\Ccal_{a^{-1}})\to \Qab$ differs from the chosen twisted character $(\rho\circ b)^{a^{-1}}$ by some $|\langle{a}\rangle|$-th root of unity. Hence by the orthogonality of twisted characters (Theorem \ref{thm:twisted characters}(iv)) and (\ref{eq:omegaab}), we get that $\rho^{[a,b]}([\Omega_{a,b}])=0$ if $\rho\neq \rho\circ b$.
		
		In other words, we have proved that if $\rho^{[a,b]}([\Omega_{a,b}])\neq 0$, then we must have $\rho\in \Irr(\KQab(\Ccal_1))^{\langle a,b\rangle}$.
		
		Hence now suppose that $a,b\in \Gamma_\rho$. Again by the twisted orthogonality relations we have $$\sum\limits_{A\in P_a}\rho^a([A])\cdot \overline{\rho^{a}([A])}=\sum\limits_{A\in P_a}\rho^a([A])\cdot {\rho^{a^{-1}}([A^*])}=f_\rho.$$ Also using Remark \ref{rk:ncocycle}, we have $\rho^{ba^{-1}b^{-1}}\circ b=\frac{\rho^{a^{-1}}}{\varphi_{\rho}(b,a^{-1},b)}$. Combining with (\ref{eq:omegaab}) we complete the proof of the lemma.
	\end{proof}
	Finally using Theorem \ref{thm:twisted verlinde} and Lemma \ref{lem:fixedbycommutator} we obtain the following:
	\begin{corollary}\label{cor:highergenus} (Categorical twisted Verlinde formula for any genus.)
		Let $g,n$ be non-negative integers and let $a_1,\cdots,a_g, b_1,\cdots, b_g, m_1,\cdots , m_n$ be elements in $ \Gamma$ that  satisfy the relation
		$[a_1,b_1]\cdots[a_g,b_g]\cdot m_1\cdots m_n=1.$ Let $\Gamma^\circ\leq \Gamma$ be the subgroup generated by $a_j,b_j,m_i$. Let $M_i\in \Ccal_{m_i}$. Then the dimension of $\Hom(\unit,\Omega_{a_1,b_1}\otimes\cdots\otimes\Omega_{a_g,b_g}\otimes M_1\otimes\cdots\otimes M_n)$ is  $$\sum\limits_{\rho\in\Irr(\KQab(\Ccal_1))^{\Gamma^\circ}}\frac{({f_\rho)}^{g-1}\cdot\rho^{m_1}([M_1])\cdots\rho^{m_n}([M_n])}{\varphi_\rho(a_1,b_1,a_1^{-1},b_1^{-1},\cdots,m_1,\cdots,m_n)}.$$If $\Ccal$ is equipped with a spherical structure and if $M_i\in P_{m_i}$ are simple then the dimension of $\Hom(\unit,\Omega_{a_1,b_1}\otimes\cdots\otimes\Omega_{a_g,b_g}\otimes M_1\otimes\cdots\otimes M_n)$ is 
		$$\sum\limits_{D\in P_1^{\Gamma^\circ}}\frac{\left(\frac{1}{S_{\unit,D}}\right)^{n+2g-2}\cdot S^{m_1}_{M_1,D}\cdots S^{m_n}_{M_n,D}}{\varphi_D(a_1,b_1,a_1^{-1},b_1^{-1},\cdots,m_1,\cdots,m_n)}.$$
	\end{corollary}
	\begin{remark}
		The fundamental group of a smooth complex genus $g$ curve with $n$ punctures has a presentation $\langle\alpha_1,\beta_1,\cdots,\alpha_g,\beta_g,\gamma_1,\cdots,\gamma_n|[\alpha_1,\beta_1]\cdots[\alpha_g,\beta_g]\gamma_1\cdots\gamma_n=1\rangle$. In other words, the choice of the elements $a_j,b_j,m_i\in \Gamma$ in Corollary \ref{cor:highergenus} is equivalent to the choice of a group homomorphism from the fundamental group to $\Gamma$ and the subgroup $\Gamma^\circ$ is just the image of this homomorphism.
	\end{remark}

	\paragraph*{\bf Part II:  Conformal blocks for twisted affine Kac-Moody Lie algebras}
\addcontentsline{toc}{section}{\bf Part II: Conformal blocks for twisted affine Kac-Moody Lie algebras}
	We give a coordinate free description of twisted conformal blocks following \cite{Tsuchimoto} over the moduli space of $n$-pointed admissible $\Gamma$-covers $\GMod$. 
	For admissible covers of $\mathbb{P}^1\rightarrow \mathbb{P}^1$  ramified at two points, we embed twisted conformal blocks into the space of invariants of  tensor product of representations  and get natural bounds on the dimensions of these conformal blocks. 
	
	In Section \ref{sec:threepointmtc}, we state one of the main theorems (See Theorem \ref{thm:oneofmain}) of the paper of contructing a $\Gamma$-crossed modular fusion category from twisted conformal blocks. We also relate the categorical crossed S-matrices discussed in Part I with characters of the twisted fusion ring associated with twisted affine Kac-Moody algebras. 
	
	Finally we end Part II, by reconstructing twisted conformal blocks using the Beilinson-Bernstein localization and given a flat projective connection with logarithmic singularities. We also derive a formula for the log Atiyah algebra of the associated projective connection in terms of natural line bundle on $\GMod$.

	\section{Twisted Kac-Moody Lie algebras and their representations}\label{sec:repofrep} Our goal is to apply the Verlinde formula for braided crossed categories to conformal blocks for twisted affine Lie algebras. We first recall following Kac \cite{Kac} some notations and facts about twisted Kac-Moody Lie algebras and their representations. Our notations and conventions will be similar to \cite{Kac,KH}, with 
	key differences at some places.

	\subsection{Twisted affine Lie algebras}\label{sec:affineLiealgtwisted}
	Let $\gamma$ be an automorphism of a finite dimensional simple Lie algebra $\frg$ of 
	order $|\gamma|$. We
	fix a $|\gamma|$-th root of unity $\epsilon:=e^{\frac{2\pi \sqrt{-1}}{|\gamma|}}$ of unity and then we have an eigen-decomposition $\frg=\oplus_{i=0}^{|\gamma|-1}\frg_i$, where $\frg_i:=\{X\in \frg| \gamma(X)=\epsilon^{i}X\}.$
	
	In particular $\frg_{0}$ (also denoted by $\frg^{\gamma}$) is the Lie subalgebra of $\frg$ invariant
	under $\gamma$. 
	 We let $\gamma$ act on $\mathbb{C}((t))$ by the formula $\gamma.t:=\epsilon^{-1}t$.
	Then the automorphism $\gamma$ induces a natural automorphism (also denoted by) $\gamma$ of the affine Lie algebra 
	$
	\frg\otimes \mathbb{C}((t))\oplus \mathbb{C}c$, where $c$ is a central element and the Lie bracket is defined by the formula 
	$$[X\otimes f, Y\otimes g]=[X,Y]\otimes fg+ (X,Y)_{\frg}\operatorname{Res}_{t=0}gdf\cdot c,$$ 
	where 
	$(,)_{\frg}$ is the normalized Killing form on $\frg$ such that $(\theta,\theta)=2$ for any long root $\theta$ of $\frg$. 
	
	We define the {\em twisted affine Lie algebra} $\widehat{L}(\frg,\gamma):=(\frg\otimes \mathbb{C}((t))\oplus \mathbb{C}c)^{\gamma}$.
	Using the eigen-decomposition of $\frg$, we get 
	$$\widehat{L}(\frg,\gamma)=\oplus_{i=0}^{|\gamma|-1}\frg_i\otimes \mathcal{A}_i\oplus \mathbb{C}c,$$ 
	where $\mathcal{A}_i=\{f(t)\in \mathbb{C}((t))|\gamma(f(t))=\epsilon^{-i}f(t)\}$. The Lie bracket is defined by the formula:
	$$[X\otimes f, Y\otimes g]=[X,Y]\otimes fg+ \frac{1}{|\gamma|}(X,Y)_{\frg}\operatorname{Res}_{t=0}gdf\cdot c,$$ 
	where $X\otimes f$ and $Y\otimes g$ are elements of $(\frg\otimes \mathbb{C}((t)))^{\gamma}$. 
	The classification \cite[Proposition 8.1]{Kac} of finite order automorphisms  of $\frg$ tells us: 
	\begin{proposition}\label{prop:classdia}Let $\gamma$ be an automorphism of $\frg$ of order $|\gamma|$.
		Then there exists
		a Borel subalgebra 
		$\mathfrak{b}$ of $\frg$ containing a Cartan subalgebra $\mathfrak{h}$ such
		that $\gamma=\sigma\exp(\operatorname{ad}\frac{2\pi{\sqrt{-1}}}{|\gamma|}h),$ where $\sigma$ is 
		a diagram automorphism of $\frg$, such that we have the following:
		\begin{itemize}
			\item both $\gamma$ and $\sigma$ preserve $\mathfrak{h}$.
			\item $h$ is a element of the subalgebra of $\mathfrak{h}$ fixed
			by $\sigma$. 
			\item $\sigma$ preserves a set of simple roots $\Pi'=\{\alpha'_1,\dots,\alpha'_{\rank{\frg}}\}$ and $\alpha'_i(h)\in \mathbb{Z}$ for
			each $1\leq i\leq \rank{\frg}$. 
			\item $\sigma$ and $\exp(\operatorname{ad}\frac{{2\pi}{\sqrt{-1}}}{|\gamma|}h)$ 
			commute.
		\end{itemize}
	\end{proposition}
	We know that diagram automorphisms of finite dimensional simple 
	Lie algebras have 
	been classified and the order $m$ of $\sigma$ is in the set $\{1,2,3\}$. Now suppose $\gamma$ and $\sigma$ are 
	related by Proposition \ref{prop:classdia}, then $m$ divides $|\gamma|$. Consider the following 
	natural map 
	\begin{equation}
	\label{eqn:basicisomor}
	\phi_{\sigma,\gamma}:\widehat{L}(\frg,\sigma)\rightarrow 
	\widehat{L}(\frg,\gamma), \ X[t^{j}]\rightarrow X[t^{\frac{|\gamma|}{m}j+k}],
	\ c\rightarrow c,
	\end{equation} where $X$ is an $\epsilon^{\frac{|\gamma|}{m}j}$-eigenvector of $\frg$ and a $k$-eigenvector for $\exp(\operatorname{ad}\frac{{2\pi}{\sqrt{-1}}}{|\gamma|}h)$.  
	By Proposition 8.6 in \cite{Kac}, we 
	get that the map in Equation \eqref{eqn:basicisomor} is an isomorphism of Lie algebras. Thus 
	we are reduced to the case of studying twisted affine Lie algebras when $\gamma$ is a diagram automorphism of $\frg$. 

	\subsubsection{Horizontal subalgebra and weight lattice} Let $X_N$ denote a type
	of finite dimensional complex Lie algebra of $\rank{N}$ (as in the classification in \cite{Kac})  
	and the 
	associated Lie algebra is $\frg(X_N)$. Similarly, let $X_N^{(m)}$ be the type of the  affine Kac-Moody
	Lie algebra associated to a diagram automorphism $\gamma$ of $\frg(X_N)$ of order $m$. 
	The corresponding Lie algebra of type $X_{N}^{(m)}$ 
	will be denoted by $\frg(X_N^{(m)})$. 
	
	Denote the set of simple roots  $\prod=\{ \alpha_0,\alpha_1,\dots,\alpha_N\}$ and
	the simple coroots $\prod^{\vee}=\{\alpha_0^{\vee},\dots, \alpha_N^{\vee}\}$. Similarly consider  $\theta:=\sum_{i=1}^{N}a_i\alpha_i$, where $a_i$ (respectively  $a_i^{\vee}$) are the Coxeter (respectively dual Coxeter labels). Our {\em numbering} of the vertices of the Dynkin 
	diagram of $\frg(X_N)$ and $\frg(X_N^{(m)})$ is
	the same as in Kac \cite{Kac}.

	We denote the finite dimensional Lie algebra 
	obtained by deleting the $0$-th 
	vertex of the Dynkin diagram of $\frg(X_N^{(m)})$ by $\GO$. Let $\mathring{\mathfrak{h}}$ denote
	the Cartan subalgebra of $\GO$. Clearly $\mathring{\mathfrak{h}}$ is generated by
	the coroots $\alpha_{1}^{\vee},\dots, \alpha_N^{\vee}$. Let  $Q(\GO)$ (resp. ${Q}^{\vee}(\GO)$) 
	denote the root lattice (resp. coroot lattice) of $\GO$ and $\mathring{\Delta}$
	denote the positive simple roots of $\GO$. 
	
	For $1\leq i \leq N$, we define the affine
	fundamental weights $\Lambda_i$ that satisfies the equation $\Lambda_i(\alpha_j^{\vee})=\delta_{ij}.$ Then we can write $\Lambda_i$ 
	in terms of its horizontal projection as follows:
	\begin{equation}
	\label{eqn:affinefunda}
	\Lambda_i:=\overline{\Lambda}_i+a_i^{\vee}\Lambda_0, 
	\end{equation}
	where $\Lambda_0$ is the zero-th affine weight . 
	Then $\overline{\Lambda}_1,\dots, \overline{\Lambda}_N$ are the fundamental weights of the horizontal subalgebra $\GO$. 
	\begin{remark}
		The numbering of the vertices of the Dynkin diagram of $\GO$ as in Kac \cite{Kac} may not extend to a numbering scheme of 
		the Dynkin diagram of $\GA$. Hence we might need to reorder the set $\overline{\Lambda}_1,\dots, \overline{\Lambda}_N$ to match
		up with the usual conventions for $\GO$.
	\end{remark}
	\subsubsection{Level $\ell$-weights}
	The weight lattice $P$ is the $\mathbb{Z}$-lattice 
	generated by $\Lambda_0,\dots,\Lambda_N$.  Let $\alpha_0^{\vee}$ be the zero-th coroot of $\frg(X_N^{(m)})$.
	The set $P^{\ell}(\GA)$ of dominant
	integral weights of level $\ell$  is defined as follows:
	$$P^{\ell}(\GA)=\{\lambda \in P| \lambda(\alpha^{\vee}_i)\geq 0 \ \mbox{for all $0\leq i \leq N$ and $\lambda(K)=\ell$} \},$$ where $K$ is the generator of the center of $\frg(X_N^{(m)})$ and is given by the formula $K=\sum_{i=0}^Na_{i}^{\vee}\alpha_i^{\vee}$. 
	
	The set of irreducible, integrable, highest weight representations at level $\ell$ of the affine 
	Kac-Moody Lie algebra $\frg(X_N^{(m)})$ is 
	in bijection with the set $P^{\ell}(\frg(X_N^{(m)}))$. If $m=1$, then the set $P^{\ell}(\frg(X_N^{(1)})$
	will often be denoted by $P_{\ell}(\frg)$. The following 
	lemma can be checked directly:
	\begin{lemma}\label{lem:levelell}
		Let ${P}_{+}(\GO)$ denote the set of dominant integral weights of $\GO$, then 
		$$P^{\ell}(\GA)=\{\lambda \in {P}_+(\GO) | \kappa_{\mathfrak{g}(X_N^{(m)})}(\lambda,\theta)\leq \ell\},$$ where $\kappa_{\frg(X_N^{(m)})}$ is the normalized Killing form on $\frg(X_N^{(m)})$.
	\end{lemma}
	We refer the reader to Appendix \ref{section:crossedS} for an explicit description of the set $P^{\ell}(\GA)$ in the various cases.

	\subsection{Modules for twisted affine Lie algebras}\label{sec:arbtodia}For any finite order automorphism $\gamma$, let $P^{\ell}(\frg,\gamma)$ be a subset of $P_{+}(\frg^{\gamma})$ parameterizing the set of integrable, irreducible highest weight representations of $\widehat{L}(\frg,\gamma)$. By the isomorphism of $\widehat{L}(\frg,\gamma)\simeq \widehat{L}(\frg,\sigma)$ and the realization of $\widehat{L}(\frg,\sigma)$ as  a twisted affine Kac-Moody algebra \cite{Kac} $\frg(X_N^{(m)})$, we get a natural bijection 
	\begin{equation}\label{eqn:bijfromarbtodia}
	P^{\ell}(\frg,\gamma)\leftrightarrow P^{\ell}(\GA). 
	\end{equation}
	We also refer the reader to \cite[Section 2]{KH} for an alternate description and construction. 
	For $\lambda\in P^{\ell}(\frg,\gamma)$, let $V_{\lambda}$ denote the highest weight irreducible module of $\frg^{\gamma}$ 
	of highest weight $\lambda$. Similarly, we denote the highest weight irreducible
	integrable modules by $\mathcal{H}_{\lambda}(\frg,\gamma)$ (also denoted by $\mathcal{H}_{\lambda}$ if there is no confusion). 
	They are characterized by the property:
	\begin{enumerate}
		\item $\mathcal{H}_{\vec{\lambda}}(\frg,\gamma)$ are infinite dimensional. 
		\item $V_{\lambda} \subset \mathcal{H}_{\lambda}(\frg,\gamma)$.
		\item The central element $c$ acts on $\mathcal{H}_{\lambda}(\frg,\gamma)$ by 
		multiplication by $\ell$.

	\end{enumerate}
	In Section \ref{sec:atiyahalg}, we give a coordinate free construction of the modules $\mathcal{H}_{\lambda}(\frg,\gamma)$.

	\section{Sheaf of twisted conformal blocks}\label{sec:deftwistconf}In this section, we give a brief {\em coordinate 
		free construction} of the sheaf of twisted conformal blocks (\cite{Fakhruddin:12,FS,KH,TUY:89,Tsuchimoto}).	
	\subsection{Family of pointed $\Gamma$-curves}\label{sec:pointedGammacurves}
	Let $T$ be a smooth variety and consider a proper, flat family  $\pi:\widetilde{C}\rightarrow T$  of curves with at most nodal singularities. 
	Let $\Gamma$ be a finite group that acts on $\widetilde{C}$ such that the map $\pi$ is $\Gamma$ equivariant. 
	Let the quotient $\overline{\pi}:C=\widetilde{C}/\Gamma\rightarrow T$ be the induced family of curves over $T$. 
	The genus of the fibers of $\pi$ and $\overline{\pi}$ are related by the Hurwitz formula. We further choose mutually disjoint sections ${\bf{p}}=(p_1,\dots,p_n)$ (respectively $\widetilde{\bf{p}}=(\widetilde{p}_1,\dots,\widetilde{p}_n)$)  of $\overline{\pi}$ 
	(respectively $\pi$) such that
	\begin{enumerate}\label{enu:pointedcurves}
		\item We have $pr(\widetilde{\bf{p}}(t))={\bf{p}}(t)$ for all points $t\in T$,  where $pr: \widetilde{C}\rightarrow C$
		is the canonical quotient map. 
		\item The points ${\bf{p}}(t)$ are all smooth.		
		\item For any $t\in T$, the smooth branching points of the $\Gamma$-cover $\tildeC_t\to C_t$ are contained in $\cup_{i=1}^np_i(t)$. 
		\item For any point $t\in T$, the data $(\widetilde{C}_t,{C}_t,\widetilde{{\bf p}}(t),{\bf{p}}(t))$ is a $n$-pointed 
		admissible $\Gamma$-cover in the sense of Jarvis-Kimura-Kaufmann (See Definition \ref{def:admissiblecovers}) \cite{JKK}.
		\item We assume that $\widetilde{C}\setminus \Gamma\cdot\widetilde{\bf {p}}(T)$ is affine over $T$.
	\end{enumerate}
	Let ${\bf m}=(m_1,\dots,m_n)\in \Gamma^n$ be the monodromies around the sections $\widetilde{\bf{p}}$, in other words for all $1\leq i\leq n$, $m_i$ is the generator of the stabilizer $\Gamma_i\leq \Gamma$ of $\widetilde{p}_i$  determined by the orientation of the curve.
	
	\subsection{Coordinate free highest weight integrable modules}We start with a family of pointed $\Gamma$-curves satisfying the conditions in Section \ref{sec:pointedGammacurves}. Let $\mathcal{I}_{\widetilde{\bf {p}}}$ 
	(respectively $\mathcal{I}_{\widetilde{p}_i})$ be the ideal of the image of $\widetilde{\bf{p}}$ 
	(respectively $\widetilde{p}_i$) in $\widetilde{C}$. Let $\widehat{\mathcal{O}}_{\widetilde{C}, \widetilde{p}_i}$ 
	be the formal completion of $\mathcal{O}_{\widetilde{C}}$ along the image of $\widetilde{p}_i$ and $\mathcal{K}_{\widetilde{C}, \widetilde{p}_i}$
	be the sheaf of formal meromorphic functions along $\widetilde{p}_i$. Observe that if $\Gamma_{i}=\langle m_i\rangle$
	denotes the stabilizer in $\Gamma$ of $\widetilde{p}_i$, then $\Gamma_i$ acts on $\mathcal{K}_{\widetilde{C}, \widetilde{p}_i}$.  We define the {\em coordinate free twisted affine $\Ocal_T$-Lie algebra } as 
	\begin{equation}
	\wfrg_{\widetilde{p}_i}:=\big(\frg\otimes \Kor \big)^{\Gamma_i}\oplus \mathcal{O}_{{T}}.c, 
	\end{equation}Suppose 
	$\widehat{\mathfrak{p}}_{\widetilde{p}_i}:=\big(\frg\otimes \widehat{\mathcal{O}}_{\widetilde{C}, \widetilde{p}_i}\big)^{\Gamma_i}\oplus \mathcal{O}_Tc$. For $\lambda_i \in P^{\ell}(\frg,\Gamma_i):=P^{\ell}(\frg,m_i)$, let $V_{\lambda}$ denote the $\frg^{\Gamma_i}$ 
	module of highest weight $\lambda$.  
	We let $\big(\frg\otimes \widehat{\mathcal{O}}_{\widetilde{C}, \widetilde{p}_i}\big)^{\Gamma_i}$ 
	act on $V_{\lambda}$ via evaluation at $\widetilde{p}_i$ and $c$ acts on 
	$V_{\lambda}$ by multiplication by $\ell$. 
	Thus $\widehat{\mathfrak{p}}_{\widetilde{p}_i}$ act on $V_{\lambda}$. We denote by $M_{\lambda_i,\widetilde{p}_i}:=
	\operatorname{Ind}^{\wfrg_{\widetilde{p}_i}}_{\widehat{\mathfrak{p}}_{\widetilde{p}_i}}V_{\lambda}$. It follows that $M_{\lambda_i,\widetilde{p}_i}$ admits a unique
	irreducible quotient which we denote by $\mathbb{H}_{\lambda_i,\widetilde{p}_i}$. 
	\begin{remark}
		If $T$
		is a point and we choose formal coordinates around $\widetilde{p}_i$, 
		we get an isomorphism of $\wfrg_{\widetilde{p}_i}$ with $\widehat{L}(\frg,\Gamma_i)=\widehat{L}(\frg,m_i)$. 
		Under this isomorphism $\mathbb{H}_{{\lambda}_i,\widetilde{p}_i}$ gets identified with $\mathcal{H}_{\lambda}(\frg,\Gamma_i)$. 
	\end{remark}
	
	Consider the sheaf of Lie algebras $$\wfrg_{{ p_i}}:=\left(\frg\otimes \big(\bigoplus\limits_{\wtilde{p}_i'\in pr^{-1}(p_i)}\mathcal{K}_{\widetilde{C}, \widetilde{p}_i'}\big)\right)^{\Gamma} \oplus \Ocal_T.c.$$ This is canonically identified with $\wfrg_{\widetilde{p}_i}$. Similarly consider the sheaf of Lie algebras $\wfrg_{\widetilde{C},{\bf{p}}}:=(\oplus_{i=1}^n\widehat{\frg}_{p_i})/\mathcal{Z}$, where $\mathcal{Z}$ is a subsheaf of $ \oplus_{i=1}^n \mathcal{O}_T.c$ consisting of tuples $(f_1,\dots, f_n)$ such that $f_1+\dots+f_n=0$. 
	\subsubsection{Sheaf of twisted  covacua}Let $\vec{\lambda}=(\lambda_1,\dots,\lambda_n)$
	be an $n$-tuple 
	of weights such 
	that each $\lambda_i \in P^{\ell}(\frg,\Gamma_i)$. 
	The sheaf $\mathbb{H}_{\vec{\lambda}}:=\mathbb{H}_{\lambda_1,\widetilde{p}_1}\otimes \dots \otimes \mathbb{H}_{\lambda_n,\widetilde{p}_n}$ of $\mathcal{O}_T$-modules is also a representation of the sheaf of Lie algebras $\wfrg_{\widetilde{C},{\bf p}}$. By the residue formula, we have a homomorphism  of sheaves of Lie algebras $(\frg\otimes\Ocal_{\tildeC}(*\Gamma\cdot\widetilde{\bf{p}}))^\Gamma\to \wfrg_{\widetilde{C},{\bf p}}$. 
	
	We define the {\em sheaf of twisted covacua} to be
	$$\mathbb{V}_{\vec{\lambda},\Gamma}(\widetilde{C},C, \widetilde{\bf{p}}, {\bf{p}}):=\mathbb{H}_{\vec{\lambda}}/(\frg\otimes \mathcal{O}_{\widetilde{C}}(*\Gamma\cdot\widetilde{\bf{p}}))^{\Gamma}\mathbb{H}_{\vec{\lambda}},$$ where $\Gamma\cdot\widetilde{\bf{p}}$ denotes the union of the $\Gamma$ orbits of $\widetilde{p}_i$ for $1\leq i\leq N$. Similarly we define the the {sheaf of vacua} as 
	$$\mathbb{V}^{\dagger}_{\vec{\lambda},\Gamma}(\widetilde{C},C, \widetilde{\bf{p}}, {\bf{p}}):=\{\langle \Psi| \in \mathbb{H}_{\vec{\lambda}}^{*}| \langle \Psi|X[f]=0, \ \mbox{for all} \ X \otimes f \in (\frg\otimes \mathcal{O}_{\widetilde{C}}(*\Gamma\cdot\widetilde{\bf{p}}))^{\Gamma}\}.$$
	We now recall some basic properties of the sheaf $\mathbb{V}_{\vec{\lambda},\Gamma}(\widetilde{C},C,\widetilde{\bf{p}}, {\bf p})$. 
	\subsubsection{Gauge Symmetry}
	If $(\widetilde{C}, C,\widetilde{\bf p}, {\bf p})$ is an admissible curve, the equation $$\langle \Psi| X[f]=\sum_{j=1}\langle \Psi | \rho_j(X\otimes f)=0,$$ where $\rho_j$ is the $j$-th component of the map $(\frg\otimes H^0(\widetilde{C},\mathcal{O}_{\widetilde{C}}(*\Gamma \cdot \widetilde{\bf p}))^{\Gamma}$ to $\widehat{\frg}_{\widetilde{C},\widetilde{\bf p}}$ will be referred to as the {\em Gauge Symmetry} or the {\em Gauge condition}.
	\subsection{Properties of twisted vacua}\label{sec:propertiesofvacua}We record some important properties of twisted Vacua that we will use. 
	
	\begin{proposition}
		The sheaves  $\mathbb{V}_{\vec{\lambda},\Gamma}(\widetilde{C},C,\widetilde{\bf{p}},{\bf{p}})$ and $\VOC$ are coherent $\mathcal{O}_T$-modules \cite{KH}, (\cite[Lemma 2.5.2]{Sorger} in the untwisted set up) which are compatible with base change.
		Moreover, like
		in the untwisted case, Hong-Kumar \cite{KH} (under the assumption that $\Gamma$ preserves a Borel subalgebra of $\frg$) show
		that the sheaf $\VOC$ is locally free and $\VOC$ and $\COV$ are dual to each 
		other. 
	\end{proposition}
	
	We also refer the reader to \cite{szcz}, for a more general statement on the interior $\mathcal{M}_{g,n}^{\Gamma}$ of $\GMod$ in the setting of orbifold vertex algebras. 
	
	Let ${\bf{q}}=(q_1,\dots, q_m)$ be $m$ disjoint sections of 
	$\overline{\pi}: C\rightarrow T$ (also disjoint from $\bf{p}$) marking \'etale points of the $\Gamma$-cover $\tildeC\to C$ and let  $\widetilde{\bf{q}}$ be a choice of lifts to $\widetilde{C}$. This data endows the family $\widetilde{C}\rightarrow T$ as a $n+m$-pointed $\Gamma$-cover with monodromy data ${\bf{m}'}=({\bf{m}},1,\dots,1)$. 
	Then we have the following \cite{Dam17,KH,TUY:89} which is often referred to as {\em Propagation of Vacua}:
	\begin{proposition}\label{prop:propavacua}
		Let $\vec{0}=(0,\dots ,0)\in P_{\ell}(\frg)^n$ and assume that $\Gamma$ preserve a Borel subalgebra of $\frg$. Then there is a natural isomorphism between $\mathcal{O}_T$-modules $\COV\simeq \mathbb{V}_{\vec{\lambda}\sqcup\vec{0},\Gamma}(\widetilde{C},C,\widetilde{\bf{p}}\sqcup\widetilde{\bf{q}},{\bf{p}}\sqcup{\bf{q}})$. Moreover, these isomorphisms are compatible with each other for different choices of the \'etale points chosen. 
	\end{proposition}
	\begin{remark}
		We only need ``Propagation of Vacua" along the \'etale points, since by assumption all the ramifications points are already marked. In Hong-Kumar\cite{KH}, Proposition \ref{prop:propavacua} is proved when the extra points ${\bf q}$ are not necessarily \'etale which is a stronger version than the one stated in Proposition \ref{prop:propavacua}. However it must be pointed out that $0$ may not be a weight in $P^{\ell}(\frg,\gamma)$ for all $\ell$. 
	\end{remark}
	\subsubsection{Descent data}\label{sec:descentdata} Following the discussion in \cite{Fakhruddin:12}, we use Proposition \ref{prop:propavacua} to  drop the condition that $\widetilde{C}\setminus \Gamma\cdot\widetilde{\bf{p}}(T)$ is affine. Let $\widetilde{C}\rightarrow T$ be a family satisfying conditions 1-4 in Section \ref{enu:pointedcurves}. 
	We can find an \'etale cover $T'$  of $T$ and $m$-sections marking \'etale 
	points $\widetilde{\bf{q}}$ (respectively $\bf{q}$) of the induced family $\widetilde{C}'$ such that it 
	satisfies all the conditions in Section \ref{enu:pointedcurves}. Thus we can associate a sheaf
	$\mathbb{V}_{\vec{\lambda}\sqcup\vec{0},\Gamma}(\widetilde{C},C,\widetilde{\bf{p}}\sqcup\widetilde{\bf{q}},{\bf{p}}\sqcup{\bf{q}})$ on $T'$. 
	We can define $\COV$ to be the natural descent of $\mathbb{V}_{\vec{\lambda}\sqcup\vec{0},\Gamma}(\widetilde{C},C,\widetilde{\bf{p}}\sqcup\widetilde{\bf{q}},{\bf{p}}\sqcup{\bf{q}})$ given by Proposition \ref{prop:propavacua}. The same discussion (see Proposition 2.1) in \cite{Fakhruddin:12}, tells us that the $\COV$ is independent of the choice of the \'etale cover of $T'$ of $T$. Thus, we get a well defined sheaf of covacua $\COVx(\widetilde{C},C, \widetilde{\bf p}, {\bf p})$ on the moduli stack $\overline{\mathcal{M}}{}^{\Gamma}_{g,n}({\bf{m}})$ of $n$-pointed admissible covers defined in \cite{JKK} (see Appendix \ref{ap:modstackofadmcovers} for the definition of $\GMod({\bf m})$). 
	\subsubsection{Factorization}
	
	Let $\widetilde{C}\rightarrow C\rightarrow T$ be a family of stable $n+2$ pointed stable covers with group $\Gamma$. Let $\widetilde{\bf{p}}'=(\widetilde{\bf{p}},\tildeq_1,\tildeq_2)$ be  the $n+2$ sections such that the monodromies around the sections $\tildeq_1$ and $\tildeq_2$ are inverse to each other, say $\gamma$ and $\gamma^{-1}$ respectively. Then identifying the family $\widetilde{C}$ along the sections $\tildeq_1$ and $\tildeq_2$, we get a new family  of stable $n$-pointed cover of curves $\widetilde{D}\rightarrow T$ with group $\Gamma$ along with  $n$-sections $\widetilde{\bf{p}}$. Assume that ``$\Gamma$ preserves a Borel subalgebra of $\frg$". Then by the factorization theorem in \cite{KH}, we have the following isomorphisms of locally free sheaves on $T$.
	\begin{proposition}\label{prop:factorizationofhongkumar}
		$$\mathbb{V}_{\vec{\lambda},\Gamma}(\widetilde{D},D, \widetilde{\bf{p}},{\bf{p}})\simeq \bigoplus_{\mu \in P^{\ell}(\frg,\gamma)}\mathbb{V}_{\vec{\lambda}\sqcup\{{\mu},{\mu^{*}}\},\Gamma}(\widetilde{C},C,\widetilde{\bf{p}}',{\bf{p}'}).$$
		
	\end{proposition}
	\subsubsection{Global Properties} We continue to assume that the group $\Gamma$ preserves a Borel subalgebra of $\frg$. The bundles of twisted covacua are compatible with natural morphisms on $\overline{\mathcal{M}}{}^{\Gamma}_{g,n}({\bf{m}})$ which we state below. 
	
	\begin{proposition}\label{prop:importantprop}Let $\vec{\lambda}$ be an $n$-tuple
		of weights corresponding to $\mathbf{m}$ and let $\COVx$ denote the sheaf of twisted covacua on $\overline{\mathcal{M}}{}^{\Gamma}_{g,n}({\bf{m}})$, where ${\bf{m}}=(m_1,\dots,m_n)\in \Gamma^n$. Then the following holds: 
		\begin{enumerate}
			\item Assume that for some $i$, we have  $m_i=1$, $\lambda_i=0$ and that $(g,n-1)$ is a stable pair, i.e. $2g-2+(n-1)>0$.
			Consider the natural forgetful stabilization map $f_i: \overline{\mathcal{M}}{}^{\Gamma}_{g,n}({\bf{m}})\rightarrow \overline{\mathcal{M}}{}^{\Gamma}_{g,n-1}({\bf{m}'})$
			obtained by
			forgetting the point $\widetilde{p}_i$ (and contracting any unstable component), then
			there is a natural isomorphism
			$$\COVx\simeq f_i^*\mathbb{V}_{\vec{\lambda}',\Gamma},$$ where $\vec{\lambda}'$ (respectively ${\bf{m}'}$) 
			is obtained by 
			deleting $\lambda_i=0$ 
			(respectively $m_i=1$) from $\vec{\lambda}$ (respectively ${\bf{m}}$).
			\item Let $\xi_{1,2,\gamma}: \overline{\mathcal{M}}{}^{\Gamma}_{g_1,n_1+1}({\bf{m}}_1,\gamma) \times \overline{\mathcal{M}}{}^{\Gamma}_{g_2,n_2+1}({\bf{m}}_2,\gamma^{-1})\rightarrow  \overline{\mathcal{M}}{}^{\Gamma}_{g_1+g_2,n_1+n_2}({\bf{m_1}},{\bf{m}_2})$ be 
			the morphism obtained by gluing two curves of genus $g_1$ and $g_2$ along the last marked point with monodromy $\gamma$ and $\gamma^{-1}$, then there is a natural isomorphism
			$$\xi_{1,2,\gamma}^*\COVx\simeq \bigoplus_{\mu\in P^{\ell}(\frg,\gamma)}
			\mathbb{V}_{\vec{\lambda}',\Gamma}
			\boxtimes \mathbb{V}_{\vec{\lambda}'',\Gamma}
			,$$
			where $\vec{\lambda}'=(\lambda_1,\dots,\lambda_{n_1},\mu)$ and $\vec{\lambda}''=(\lambda_{n_1+1},\dots,\lambda_{n_2},\mu^{*})$.
			
			\item Let $\xi_{\gamma}: \overline{\mathcal{M}}{}^{\Gamma}_{g-1,n+2}({\bf{m}},\gamma,\gamma^{-1})
			\rightarrow\overline{\mathcal{M}}{}^{\Gamma}_{g,n}({\bf{m}})$ be the morphism obtained by gluing a curve along two points with opposite monodromies. Then there is a canonical isomorphism
			$$\xi_{\gamma}^*\COVx\simeq \bigoplus_{\mu\in P^{\ell}(\frg,\gamma)}	\mathbb{V}_{\vec{\lambda}\sqcup \{\mu,\mu^*\},\Gamma}.$$ 
		\end{enumerate}
		Moreover, these isomorphisms induced by $\xi_{1,2,\gamma}$, $f_i$ and $\xi_{\gamma}$ are compatible with each other.
	\end{proposition}
	\begin{proof}The second and third part of the proposition follows directly from Proposition \ref{prop:factorizationofhongkumar}. We now discuss the proof of the first part which is similar to the discussion in Section 2.2 of \cite{Fakhruddin:12}. 
		
		Let $\widetilde{C}\rightarrow T$ be a family of stable $n+1$-pointed covers with group $\Gamma$ and $\widetilde{\bf{p}}$ are the sections. Assume that the points in the fibers of $\widetilde{C}\rightarrow T$ marked by the  $n+1$-th section in $\widetilde{\bf{p}}$ are always \'etale.  Let $\vec{\lambda}=(\lambda_1,\dots,\lambda_{n+1})$ be such that $\lambda_{n+1}=0$ in $P_{\ell}(\frg)$. Then forgetting the $n+1$-th section, gives a family of $n$-pointed $\Gamma$ covers which may not be stable. To make the new family stable, we have to contract unstable components to a point. Hence two situations can occur. 
		
		First we can contract a rational curve that has two marked points  and $p_{n}$ and $p_{n+1}$ and meets the other component $E$ at a nodal point $q$. In the case after stabilization we obtain the curve $E$ is smooth at $q$ and we declare $q$ to be the $n$-th marked point $p_n$. 
		Secondly, we can contract a rational component which has one marked point $p_{n+1}$ and meets the other components $E_1$ and $E_2$ at two distinct nodal points $q_1$ and $q_2$. In this case, after stabilization we get a curve obtained by joining $E_1$ and $E_2$ by identifying $q_1$  with $q_2$. We refer the reader to Section 2 in \cite{JKK} for more details on these stabilization morphisms. Now in both these cases, part (1) of Proposition \ref{prop:importantprop} follows from Proposition \ref{prop:factorizationofhongkumar} and Lemma \ref{lem:linebundlecase}.
	\end{proof}
	
	\begin{lemma}\label{lem:linebundlecase}Let ${\bf m}=(\gamma,\gamma^{-1},1)\in \Gamma^{3}$ 
		and $\vec{\lambda}=(\lambda,\lambda^*,0)$, where $\lambda \in P^{\ell}(\frg,\gamma)$. Assume that the marked points $\widetilde{\bf{p}}$ are in the same connected component of $\widetilde{C}$. Then the fibers of the vector bundle $\mathbb{V}_{\vec{\lambda},\Gamma}(\widetilde{C},\mathbb{P}^1,\widetilde{{\bf p}},{\bf p})$ restricted to such points are one dimensional. 
	\end{lemma}
	\begin{proof}From the assumption it suffices to assume  that $\widetilde{C}=\mathbb{P}^1$ and $\Gamma=\langle \gamma \rangle$. Let ${\mathbb{P}^1}\rightarrow \mathbb{P}^1$ be a $\Gamma$-cover with $n$ marked points. Assume that $n-2$ of the marked points are \'etale. In this set up, by Proposition \ref{prop:embeddingintopone} , the fibers of $\mathbb{V}^{\dagger}_{\vec{\lambda},\Gamma}(\mathbb{P}^1, \mathbb{P}^1, \widetilde{{\bf p}}, {\bf p})$ of the twisted conformal bundle embeds in  $\operatorname{Hom}_{\frg^{\Gamma}}(V_{\lambda_1}\otimes\dots \otimes  V_{\lambda_n}, \mathbb{C})$ where $V_{\lambda}$ is the irreducible $\frg^{m_i}$-module  of highest weight $\lambda$ and  $m_i$ is the monodromy around the point $\widetilde{p}_i$. 
		
		Since $\operatorname{Hom}_{\frg^{\gamma}}(V_{\lambda}\otimes V_{\lambda^*})$ is one dimensional,  this applied to the situation of the lemma give us that the rank of $\mathbb{V}_{\vec{\lambda},\Gamma}(\widetilde{C},\mathbb{P}^1,\widetilde{{\bf p}},{\bf p})$ is at most one dimensional. 
		Now the result follows from the fact the one dimensional space of $\frg^{\gamma}$-invariants of $V_{\lambda}\otimes V_{\lambda^*}$ extends (Equation \eqref{eqn:bilform}) the bilinear form $(\ |\ )_{\lambda}$ on $\mathcal{H}_{\lambda}(\frg,\gamma)\otimes \mathcal{H}_{\lambda^*}(\frg,\gamma^{-1})$ as an element (Equation \eqref{eqn:bilo}) of the twisted conformal block
	\end{proof}
	Let $z$ be a global coordinate on $\mathbb{C}$ and assume that $\mathbb{P}^1=\mathbb{C}\cup \{\infty\}$. Observe that $w=1/z$ is a local coordinate at the point $\infty$ and $\xi(\widetilde{q})=z-\widetilde{q}$ is a local coordinate at any other point $\widetilde{q}\in \mathbb{C}\{0\}$. The proof of the following proposition is analogous to Proposition 6.1 in \cite{U}. We include a proof for completeness 
	\begin{proposition}\label{prop:embeddingintopone}Let $\Gamma=\langle \gamma\rangle$ be a cyclic group of order $N$ and $\mathbb{P}^1\rightarrow \mathbb{P}^1$ be an admissible cover with Galois group $\Gamma$ and two ramification points $0$ and $\infty$ of order $N$ and monodromy $\gamma$ and $\gamma^{-1}$ respectively. Moreover assume that $\widetilde{q}_1,\dots, \widetilde{q}_n$ are distinct \'etale points on $\mathbb{C}^{\times}$ with non-intersecting $\Gamma$-orbits, then the twisted conformal blocks at any level injects into the space of $\frg^{\gamma}$ invariants:
		$$\mathcal{V}_{\vec\lambda,\mu_1,\mu_2,\Gamma}(\mathbb{P}^1,\mathbb{P}^1, \widetilde{\bf p}, {\bf p})\hookrightarrow \operatorname{Hom}_{\frg^{\gamma}}(V_{\vec{\lambda}}\otimes V_{\mu_1}\otimes V_{\mu_2},\mathbb{C}),$$
		where $\widetilde{\bf p}=(\widetilde{q}_1,\dots, \widetilde{q}_n,0,\infty)$ and $V_{\vec{\lambda}}=V_{\lambda_1}\otimes \dots \otimes V_{\lambda_n}$. 
	\end{proposition}
	\begin{proof}Let $\langle \Psi|$ be an element of the twisted conformal blocks i.e. in-particular it is a function on the tensor product of integrable representations 
		$$\mathcal{H}_{\vec{\lambda},\mu_1,\mu_2}:=\mathcal{H}_{\lambda_1}(\frg)\otimes \mathcal{H}_{\lambda_n}(\frg)\otimes \mathcal{H}_{\mu_1}(\frg,\gamma)\otimes \mathcal{H}_{\mu_2}(\frg,\gamma^{-1}).$$
		The tensor product $V_{\vec{\lambda}}\otimes V_{\mu_1}\otimes V_{\mu_2}$ naturally embeds into $\mathcal{H}_{\vec{\lambda},\mu_1,\mu_2}$,  hence by restriction of $\langle \Psi|$ we get a map
		$$\iota: \mathcal{V}_{\vec\lambda,\mu_1,\mu_2,\Gamma}(\mathbb{P}^1,\mathbb{P}^1, \widetilde{\bf p}, {\bf p})\rightarrow \operatorname{Hom}(V_{\vec{\lambda}}\otimes V_{\mu_1}\otimes V_{\mu_2},\mathbb{C}),$$
		Now, let $X\in \frg^{\gamma}$ and $1$ be the identity function on $\mathbb{P}^1$. By the ``Gauge Condition", we get for any vector $\vec{v} \in V_{\vec{\lambda}}\otimes V_{\mu_1}\otimes V_{\mu_2}$:
		$$\langle \Psi\cdot (X\otimes 1)|\vec{v}\rangle =\sum_{j=1}^{n+2}\langle \Psi|\rho_j(X\otimes 1) \vec{v}\rangle=0,$$ where $\rho_j$ denotes the action of $X\otimes 1$ on the $j$-th component of $\vec{v}$. 
		Thus, we conclude that the image of $\iota$ in indeed in the space $\operatorname{Hom}_{\frg^{\gamma}}(V_{\vec{\lambda}}\otimes V_{\mu_1}\otimes V_{\mu_2},\mathbb{C})$.
		
		We now show that the map $\iota$ is indeed injective. 	Assume $\iota(\langle \Psi|)=0$. By induction on the natural filtration $F_{p} \mathcal{H}_{\vec{\lambda},\mu_1,\mu_2}$, we prove that $\langle \Psi|$ restricted to $F_{p}\mathcal{H}_{\vec{\lambda},\mu_1,\mu_2}$ is zero. For $p=0$, this is true by the assumption. Assume that $\langle \Psi|_{|F_p\mathcal{H}_{\vec{\lambda},\mu_1,\mu_2}}=0$.
		
		Now any element of $F_{p+1}\mathcal{H}_{\lambda,\mu_1,\mu_2}$ is of the form $\rho_j(X(-M))|\Phi\rangle $, for some $1\leq j\leq n+2$ and  a positive integer $M$, where  $|\Phi\rangle \in F_{p}\mathcal{H}_{\lambda,\mu_1,\mu_2}$. Moreover  either $X(-M)$ is an element of $\widehat{L}(\frg)$ for $1\leq j \leq n$ or $X(-M)\in \widehat{L}(\frg,\gamma)$ and $\widehat{L}(\frg,\gamma^{-1})$ for $j=n+1$, $j= n+2$ respectively. We split up the rest into two cases
		
		First we consider the case $j$ is $n+1$. The case $j$ is $n+2$ is similar. Further we can assume that $X(-M)$ is in the space $\frg_i\otimes \mathcal{A}_i$ in $\widehat{L}(\frg,\gamma)$, where $-i=M\mod |\gamma|$. 
		
		The  function $f(z)=z^{-M}$ has a pole of order $M$ only at zero and is holomorphic else where. Moreover $X\otimes f$ is an element of $\Gamma$-invariant $\frg$ valued functions on $\mathbb{P}^1\backslash \Gamma \cdot\widetilde{\bf p}$. Thus again by ``Gauge Symmetry", we get 
		\begin{eqnarray*}
			\langle \Psi| \rho_{n+1}(X(-M))\Phi\rangle &=& \langle \Psi| \rho_{n+1}(X\otimes f)\Phi\rangle, \\
			&=&-\sum_{j\neq n+1}\langle \Psi| \rho_j(X\otimes f)\Phi \rangle =0.
		\end{eqnarray*}
		Since $f$ is holomorphic outside of zero, it follows that $\rho_j(X\otimes f)|\Phi \rangle \in  F_{p}\mathcal{H}_{\lambda,\mu_1,\mu_2}$. 
		
		Next consider the case $1\leq j\leq n$. Without loss of generality assume that $j=1$. Since the action of $\gamma$ fixes $0$ and $\infty$, it follows that the action of $\gamma$ on $z$ by via multiplication of $\epsilon^{-1}$, where $\epsilon$ is a chosen $|\gamma|$-th root of unity, for every $0\leq k< |\gamma|$, we can produce a  meromorphic function $f_k$ on $\mathbb{P}^1$ with the following property:
		\begin{enumerate}
			\item $f_k$	 has poles only along the $\Gamma$-orbit of $\widetilde{q}_1$ of order  exactly $M$ and holomorphic elsewhere.
			\item $\gamma. f_k(z)=\epsilon^{-k}f_k(z)$, where $z$ is the global coordinate on $\mathbb{C}$. 
		\end{enumerate}
		For example the function $f_k(z)=\frac{z^{k}}{(z^{|\gamma|}-\widetilde{q}_1^{|\gamma|})^M}$ satisfies the conditions (1) and (2) since $k<|\gamma|$.  
		Now by projecting into the eigenspaces $\frg_k$ of the $\gamma$ action on $\frg$, we can assume without loss of generality that $X$ is of the form $X_k \in \frg_k$. Now by construction $X_k\otimes f_k$ is an element of $\Gamma$ invariant $\frg$-valued functions on $\mathbb{P}^1\backslash \Gamma\cdot \widetilde{q}_1$.  We conclude via ``Gauge Symmetry" as follows:
		\begin{eqnarray*}
			\langle \Psi| \rho_{1}(X_k(-M)\Phi\rangle &=& \langle \Psi| \rho_{1}(X_k\otimes f_k)\Phi\rangle, \\
			&=&-\sum_{j\neq 1}\langle \Psi| \rho_j(X_k\otimes f_k)\Phi \rangle
			=0  \mbox{ (again by induction hypothesis)}.	
		\end{eqnarray*}
	\end{proof}
	\begin{remark}In the untwisted case, Fakhruddin \cite{Fakhruddin:12} showed  that the bundles of sheaf of covacua is globally generated by generalizing Proposition 6.1 in \cite{U} to the nodal case. Similar results also holds for twisted conformal blocks of arithmetic genus zero. We will study this further in a future paper where we compute Chern classes of $\Gamma$-twisted conformal blocks.
	\end{remark}
	\subsubsection{Equivariance with respect to permutations and conjugation}\label{subsubsec:equivariance}
	As before let $\mathbf{m}=(m_1,\cdots,m_n)\in \Gamma^n$ and consider the moduli stack $\barM^\Gamma_{g,n}(\mathbf{m})$. Let $\sigma\in S_n$ be a permutation. Then permutation of the marked points induces an isomorphism $\xi_\sigma:\barM^\Gamma_{g,n}(\mathbf{m})\rar{\cong} \barM^\Gamma_{g,n}(\mathbf{m}_\sigma)$, where $\mathbf{m}_\sigma:=(m_{\sigma(1)},\cdots,m_{\sigma(n)})$. 
	
	If $\vec{\lambda}=(\lambda_1,\cdots,\lambda_n)$ is an $n$-tuple of weights with $\lambda_i\in P^{\ell}(\frg,m_i)$, we set $\vec{\lambda}_\sigma$ to be the permuted weights. Then we have a natural isomorphism between the sheaves of covacua 
	\begin{equation}\label{eqn:pullbackandvacua}
	\COVx\simeq \xi_\sigma^*\mathbb{V}_{\vec{\lambda}_\sigma,\Gamma}.
	\end{equation}
	
	Now let $\pmb\gamma=(\gamma_1,\cdots,\gamma_n)\in \Gamma^n$. Let ${}^{\pmb{\gamma}}\mathbf{m}=({}^{\gamma_1}m_1,\cdots,{}^{\gamma_n}m_n)$ be the conjugated $n$-tuple, where ${}^{\gamma_{i}}m_i:=\gamma_im_i\gamma_i^{-1}$. Then acting on the marked points $\wtilde{\mathbf{p}}$ in $\tildeC$ by $\pmb\gamma$ induces an isomorphism $\xi_{\pmb\gamma}:\barM^\Gamma_{g,n}(\mathbf{m})\rar{\cong}\barM^\Gamma_{g,n}({}^{\pmb\gamma}\mathbf{m})$. 
	
	If $\vec{\lambda}=(\lambda_1,\cdots,\lambda_n)$ is an $n$-tuple of weights as before, then we obtain the weights $\pmb\gamma\cdot \vec{\lambda}$ with $\gamma_i\cdot \lambda_i\in P^{\ell}(\frg,{}^{\gamma_i}m_i)$ (see also Section \ref{subsubsec:gammaaction}). Then we have a natural isomorphism between the sheaves of covacua
	\begin{equation}\label{eqn:conjugation}
	\COVx\simeq \xi_{\pmb\gamma}^*\mathbb{V}_{\pmb\gamma\cdot\vec{\lambda},\Gamma}.
	\end{equation}
	Combining Equations \eqref{eqn:pullbackandvacua} and \eqref{eqn:conjugation}, we see that the sheaves of covacua are equivariant for the action of the wreath product $S_n\ltimes \Gamma^n$ on the moduli stacks of pointed admissible $\Gamma$-covers.

	\section{Twisted fusion rings associated to automorphisms of $\frg$}\label{sec:fusionforaut}
	Let $\mathfrak{g}$ be a 
	finite dimensional Lie algebra which is simply laced. 
	Let $D(\frg)$ be the Dynkin 
	diagram of $\frg$. A diagram automorphism 
	$\sigma$ of $\frg$ is a graph automorphism 
	of the Dynkin diagram $D(\frg)$. Let $N$ be
	the order of $\sigma$. It is well known that $N\in \{1,2,3\}$. We
	can extend $\sigma$ to an automorphism of the Lie algebra $\frg$. 
	To this data, one can (see \cite{FSS})
	attach a new Lie algebra 
	$\frg_{\sigma}$ known as the orbit Lie algebra. 
	We enumerate the vertices $I$ of $D(\frg)$ as
	in \cite{Kac} by integers. 
	Further  denote by $\frg^{\sigma}$ 
	the Lie algebra of $\frg$ fixed by $\sigma$.  The Lie algebras 
	$\frg_{\sigma}$ 
	and 
	$\frg^{\sigma}$ are Langlands dual to each other. 

	We fix an enumeration the orbit representatives such
	that it is the smallest in the orbit. More precisely 
	$$\check{I}=\{i \in I | i \leq \sigma^{a}(i), \ \mbox{and}\  0\leq a \leq N-1\}.$$Let $N_i$ 
	denote the
	order of the orbit at $i\in I$. 
	Let $P(\frg)^{\sigma}$ denote the set of integral weights which are invariant 
	under $\sigma$ and consider the fundamental weights $\omega_1,\dots, \omega_{\operatorname{rank}{\frg}}$ of $\frg$.  The following is an easy observation:
	\begin{lemma}\label{lem:bijecofweights}
		There is a natural bijection $\iota:P(\frg)^{\sigma}\rightarrow P(\frg_{\sigma})$ which 
		has 
		the following properties:
		\begin{enumerate}
			\item For $i\in \check{I}$, we get $\sum_{a=0}^{N_i}\iota(\omega_{\sigma^{a}i})=\omega_{i}$.
			\item $\iota(\overline{\rho})=\overline{\rho}_{\sigma}$, where $\overline{\rho}$ and $\overline{\rho}_{\sigma}$ are 
			sums of the fundamental weights of $\frg$ and $\frg_{\sigma}$ respectively. 
		\end{enumerate}
	\end{lemma}

	\subsection{Fusion rules associated to automorphisms}\label{section:hongfusion}
	%
	%
	
	Let $\frg$ be any simple  Lie algebra. For any $n$-tuple $\vec{\lambda}=(\lambda_1,\dots,\lambda_n)$, consider the dual conformal block 
	$\confv$ at level $\ell$. 
	Now given any diagram automorphism $\sigma: \mathfrak{g} \rightarrow \mathfrak{g}$, we get an 
	automorphism of $\sigma^*: P_{\ell}(\frg)\rightarrow P_{\ell}(\frg)$. 
	This induces a map 
	of the dual conformal blocks 
	$\sigma: \mathcal{V}_{\vec{\lambda}}(\mathbb{P}^1,\vec{z})\rightarrow \mathcal{V}_{\sigma^*({\vec{\lambda}})}(\mathbb{P}^1,\vec{z}),$
	where $\sigma^*({\vec{\lambda}})=(\sigma^*\lambda_1,\dots, \sigma^*\lambda_n)$.  
	
	Let $P_{\ell}(\mathfrak{g})^{\sigma}$ 
	denote the set of level $\ell$ weights fixed by $\sigma$. Thus, 
	if $\vec{\lambda} \in (P_{\ell}(\mathfrak{g})^{\sigma})^n$, then we get 
	an element $\sigma$ in $\operatorname{End}(\mathcal{V}_{\vec{\lambda}}(\mathbb{P}^1,\vec{z}))$. 
	In \cite{Hong}, Jiuzu Hong defines a fusion rule associated to $\sigma$. 
	\subsubsection{Twisted fusion rules for diagram automorphisms}\label{{section:hongfusion}}Now we restrict to 
	the case when $\mathfrak{g}$ is simply laced and $\sigma$ is diagram automorphism 
	of the Lie algebra $\mathfrak{g}$.  The following proposition can be found in  \cite{Hong}:
	\begin{proposition} \label{pro:twistedfusionrule}
		The map $N_{\sigma}: \mathbb{N}^{P_{\ell}^{\sigma}(\mathfrak{g})}\rightarrow \mathbb{C}$ 
		given by 
		$N_{\sigma}(\sum_i \lambda_i ):=\operatorname{Tr}(\sigma| \mathcal{V}_{\vec{\lambda}}( \mathbb{P}^1, \vec{z}))$ satisfies 
		the hypothesis of  fusion rules (see \cite{BeauHirz}).
	\end{proposition}
	\begin{definition}
		The  ring given by Proposition \ref{pro:twistedfusionrule} will be denoted by $\mathcal{R}_{\ell}(\frg,\sigma)$ and will be called the twisted (Kac-Moody) fusion ring.  
	\end{definition}
	
	The set of characters of the fusion ring $\mathcal{R}_{{\ell}}(\frg,\sigma)$ has been described in \cite{Hong}. We recall the details below. 
	Let $G_{\sigma}$ be the simply connected group associated to the
	Lie algebra $\mathfrak{g}_{\sigma}$ and $T_{\sigma}$ be a maximal
	torus in $G_{\sigma}$. Let $Q(\frg)$ be the root lattice of $\frg$ and let $Q(\frg)^{\sigma}$ be elements in the
	root lattice of $\mathfrak{g}$ which are fixed by the automorphism
	$\sigma$. Then under the transformation $\iota:P(\frg)^{\sigma}\rightarrow P(\frg_{\sigma})$, we get 
	\begin{equation}\label{lemma:iotamagic}
	\iota (Q(\mathfrak{g})^{\sigma})= \left\{
	\begin{array}{ll}
	Q(\mathfrak{g}_{\sigma}) & \mathfrak{g}\neq A_{2n}, \\
	P(\mathfrak{g}_{\sigma})=\frac{1}{2}Q(\frg_{\sigma}) & \mathfrak{g}=A_{2n}. \\
	\end{array} 
	\right.
	\end{equation}
	
	Consider the following subset $T_{\sigma,\ell}$ of $T_{\sigma}$ given by 
	$$T_{\sigma,\ell}:=\{t \in T_{\sigma}|e^{\alpha}(t)=1 \  \mbox{for} \ \alpha \in (\ell+h^{\vee})\iota(Q^{\sigma}(\mathfrak{g}))\}.$$ An element $t$ of $T_{\sigma,\ell}$ is called {\em regular} if the Weyl group $W_{\sigma}$ of $\frg_{\sigma}$ acts freely on $t$ and let $T^{\operatorname{reg}}_{\sigma,\ell}$ is set of regular elements in $T_{\sigma,\ell}$.  
	Now the main result in \cite{Hong} is the following:
	\begin{proposition}\label{prop:mainhong1}For $t \in T^{\reg}_{\sigma,\ell}/W_{\sigma}$,  
		the set $\operatorname{Tr}_{*}(t)$ gives the characters of 
		the fusion ring $\mathcal{R}_{{\ell}}(\frg,\sigma)$.  
	\end{proposition}
	Thus to study the character table of the fusion ring $\mathcal{R}_{{\ell}}(\frg,\sigma)$, we need to give an  explicit description of the set $T_{\sigma,\ell}^{\reg}/W_{\sigma}$ in terms of weights of the level $\ell$-representations of the twisted affine Lie algebra $\GA$ associated to $\sigma$. 
	\begin{remark}
		In the case of fusion ring $\mathcal{F}_N$ associated to untwisted conformal blocks, there is a natural bijection \cite{BeauHirz} between $P_{\ell}(\frg)$ and $T^{\operatorname{reg}}_{\ell}/W$, where $W$ is the Weyl group of $\frg$. However, in the case of twisted fusion ring $\mathcal{R}_{{\ell}}(\frg,\sigma)$, there
		is no such natural bijection. 
	\end{remark}
	
	The following proposition should be considered as a generalization of Proposition 6.3 in \cite{BeauHirz} to the case of all affine Kac-Moody Lie algebras. 
	\begin{proposition}\label{prop:vitalbijectionthathongsshouldhavewrittenproperly}
		There is a natural bijection between $T_{\sigma,\ell}^{\operatorname{reg}}/W_{\sigma}$ with $P^{\ell}(\GA)$.
	\end{proposition}
	\begin{proof}We divide our proof in the two major cases. First, we consider the case when the twisted affine Kac-Moody Lie algebra is of type $A_{2n-1}^{(2)}, D_{n+1}^{(2)}, E_{6}^{(2)}$ or $D_{4}^{(3)}$. In \cite{Hong}, Hong introduces the following subset set of coweights $\check{P}(\frg_{\sigma})$:
		$$\check{P}_{\sigma,\ell}:=\{\check{\mu}\in \check{P}(\frg_{\sigma})|\theta_{\sigma}(\check{\mu})\leq \ell\}.$$ Now Lemma 5.20 in \cite{Hong}, shows that there is a natural bijection between $\check{P}_{\sigma,\ell}$ and $T_{\sigma,\ell}^{\operatorname{reg}}/W_{\sigma}$. Thus we will be done if we can show that $P^{\ell}(\frg,\sigma)$ is in a natural bijection with $\check{P}_{\sigma,\ell}$. This set can be precisely described as 
		\begin{enumerate}\label{enu:hongpain1}
			\item If $\frg=A_{2n-1}$, then $\check{P}_{\sigma,\ell}=\{\sum_{i=1}^ne_i\check{\omega}_i| e_1+2(e_2+\dots+e_n)\leq \ell\}$.
			\item If $\frg=D_{n+1}$, then $\check{P}_{\sigma,\ell}:=\{\sum_{i=1}^ne_i\check{\omega}_i | 2(e_1+\dots+e_{n-1})+e_n\leq \ell\}$.
			\item If $\frg=D_{4}$, then $\check{P}_{\sigma,\ell}:=\{e_1\check{\omega}_1+e_2\check{\omega}_2|3e_1+2e_2\leq \ell\}$.
			\item If $\frg=E_{6}$, then $\check{P}_{\sigma,\ell}:=\{\sum_{i=1}^4e_i\check{\omega}_i| 2e_1+4b_2+3b_3+2e_4\leq \ell\}$.
		\end{enumerate}Here $\check{\omega}_i$'s are the fundamental coweights of the Lie algebra $\frg_{\sigma}$ and $e_i$'s are non-negative integers. 	By the explicit description of the set $P^{\ell}(\frg(X_N^{(m)}))$ by equations \eqref{eqn:levellweightforA2n-12}, \eqref{eqn:levellweightforDn+12},\eqref{eqn:levellweigtforD43} and \eqref{eqn:levellweigtforE62}, it follows that there is a natural bijection between $P^{\ell}(\frg,\sigma)$ and $\check{P}_{\sigma,\ell}$. 
		
		Now we consider the remaining case when $A=A_{2n}^{(2)}$. In this case, the horizontal subalgebra $\GO$ is same as $\frg_{\sigma}$. We define the map $\beta:P^{\ell}(\frg(A))\rightarrow T_{\sigma,\ell}$ given by the formula $\beta(\lambda)=\exp\big(\frac{2\pi i}{\ell+h^{\vee}}\nu_{\frg(A)}(\lambda+\overline{\rho})\big)$, where $\nu_{\frg(A)}$ is the isomorphism between $\mathfrak{H}^{*}(A) \rightarrow \mathfrak{H}(A)$ induced by the normalized Cartan-Killing form $\kappa_{\frg(A)}$ on $\frg(A)$ and $\overline{\rho}$ is the sum of the fundamental weights of the horizontal subalgebra $\GO$. Now the rest of the proof follows as in Proposition 9.3 in \cite{BeauHirz}.
		
	\end{proof}

	\subsubsection{Fusion rings associated to arbitrary automorphisms}\label{sec:arbtwistedfusion} In the previous section, we discussed how one can associate fusion rings to diagram automorphism. 
	Let $\gamma:\frg \rightarrow \frg$ be an arbitrary finite order automorphism and $\sigma$ be the diagram automorphism of $\frg$ related to $\gamma$ by Proposition \ref{prop:classdia}. Then $\gamma $ 
	and $\sigma$ differ by an inner 
	automorphism of $\frg$ and hence we get  a bijection between $P(\frg)^{\gamma}$ and $P(\frg)^{\sigma}$. Now since $\gamma$ and $\sigma$ both preserve a normalized Cartan Killing 
	form it follows that there exists a bijection between $P_{\ell}(\frg)^{\gamma}$ and $P_{\ell}(\frg)^{\sigma}$. 
	
	If $\vec{\lambda}=(\lambda_1,\dots,\lambda_n)\in (P_{\ell}(\frg)^{\sigma})^n$, then as in Section \ref{section:hongfusion}, we get an element in $\operatorname{End}(\mathcal{V}_{\vec{\lambda}}(\mathbb{P}^1,\vec{z}))$. Since the untwisted conformal block $\mathcal{V}_{\vec{\lambda}}(\mathbb{P}^1,\vec{z})$ is constructed as coinvariants of representation of the Lie algebra $\frg\otimes \mathcal{O}_{\mathbb{P}^1}(*\vec{z})$, it follows \cite[Section 5.3]{FS} that inner automorphism of $\frg$ acts trivially on the untwisted conformal block $\mathcal{V}_{\vec{\lambda}}(\mathbb{P}^1,\vec{z})$. The discussion can be summarized as follows:
	\begin{proposition}
		Let $\vec{\lambda}\in (P_{\ell}(\frg)^{\sigma})^n$, we get $\operatorname{Tr}(\gamma|\mathcal{V}_{\vec{\lambda}}(\mathbb{P}^1,\vec{z}))=\operatorname{Tr}(\sigma|\mathcal{V}_{\vec{\lambda}}(\mathbb{P}^1,\vec{z}))$.
	\end{proposition}
	Hence the twisted Kac-Moody fusion ring $\Rcal_{\ell}(\frg,\gamma)$ associated to an arbitrary automorphism $\gamma$ just depends on the class of the diagram automorphism $\sigma$. 
	
	\section{The $\Gamma$-crossed modular fusion category associated with twisted conformal blocks} \label{sec:threepointmtc}We now state one of the main theorems in this paper. We prove this theorem in the later sections (Section \ref{sec:proofofoneofmain}) using the formalism of $\Gamma$-crossed modular functors that will be developed in Part III of the paper. We use this theorem to connect the fusion ring studied by J. Hong \cite{Hong} as a categorical twisted fusion ring as in Section \ref{sec:twistedfusion}. We have the following set up:
	
	Let $m_1,m_2,m_3$ be elements of $\Gamma$ such that $m_1\cdot m_2\cdot m_3=1$. Consider the $3$-pointed marked curve $(\mathbb{P}^1, \mu_3)$, where the three marked points are the 3-rd roots of unity $\mu_{3}\subset \mathbb{C}\subset \mathbb{P}^1$, the associated tangent vector is again $\omega$ considered as a tangent vector at $\omega$. Consider $0$ as the basepoint on $\PBbb^1\setminus \mu_3$. Consider paths $C_j$ in $\PBbb^1\setminus \mu_3$ based at 0 and obtained by going in a straight line from $p_0$ to $p_j$ and encircling the point $e^{\frac{2\pi\sqrt{-1}j}{3}}$ counterclockwise. Then we obtain a presentation of the fundamental group  $\pi_1(\PBbb^1\setminus \mu_3,0)=\langle C_1,C_2,C_3 | C_1\cdot C_2\cdot C_3=1\rangle.$ Hence if $m_1m_2 m_3=1$, then $C_j\mapsto m_j$ defines a group homomorphism $\phi:\pi_1(\PBbb^1\setminus \mu_3,0)\to \Gamma$. Then by \cite[\S2.3]{JKK} this determines an $3$-marked admissible $\Gamma$-cover $(\tildeC\to \PBbb^1,\wtilde{\mathbf{p}},\mu_n,\wtilde{\mathbf{v}})$ in the distinguished component $\xi_{0,3}(m_1,m_2,m_3)$ of $\overline{\mathcal{M}}{}^{\Gamma}_{0,3}(m_1,m_2,m_3)$ (see \cite[\S2.3]{JKK})). 
	\begin{theorem}\label{thm:oneofmain}
		Assume that $\Gamma$ preserves a Borel subgroup of $\frg$.	The $\Gamma$-twisted conformal blocks for a simple Lie algebra $\frg$ at level $\ell$ define a $\Gamma$-crossed modular fusion category
		$\mathcal{C}(\frg,\Gamma,\ell)=\bigoplus_{\gamma \in\Gamma}\mathcal{C}_{\gamma},$ with fusion product denoted by $\dotimes$ such that 
		\begin{enumerate}
			\item The identity component $\Ccal_1$ is the modular fusion category defined by untwisted conformal blocks for $\frg$ at level $\ell$.
			\item The simple objects of $\mathcal{C}_{\gamma}$ are parameterized by the set $P^{\ell}(\frg,\gamma)$.
			\item The monoidal structure $\dotimes$ is defined by the following $\operatorname{Hom}$-spaces:\\
			For simple objects $\lambda_i \in \mathcal{C}_{m_i}$
			\begin{equation}\label{eq:threepointhomspaces}
			\Hom(\unit,\lambda_1\dotimes\lambda_2\dotimes \lambda_3)=
			\begin{cases}
			\Vcal_{\vec{\lambda},\Gamma}(\tildeC\to \PBbb^1,\wtilde{\mathbf{p}},\mu_3,\wtilde{\mathbf{v}}), & \text{if $m_1\cdot m_2\cdot m_3=1$}\\
			0 & \text{otherwise}.
			\end{cases}
			\end{equation}
		\end{enumerate}
	\end{theorem}
	\subsection{The relationship with categorical twisted fusion rings}
	Let us go back to the untwisted situation, i.e. when the group $\Gamma$ is trivial. In this setting Huang \cite{Huang:08a}, \cite{Huang:08b} has proved Theorem \ref{thm:oneofmain}, namely that untwisted conformal blocks (see Theorem \ref{thm:oldconf}) define a modular fusion category $\Ccal(\frg,\ell)$. If $\gamma$ is a finite order automorphism of $\frg$, then we have the induced autoequivalence of the modular fusion category $\Ccal(\frg,\ell)$ which comes from the action of $\gamma$ on the spaces of conformal blocks described in Section \ref{sec:arbtwistedfusion}. Note that the classical untwisted version of Equation (\ref{eq:threepointhomspaces}) says that
	$\Hom(\unit,\lambda_1\dotimes\lambda_2\dotimes \lambda_3)=\Vcal_{{\lambda_1,\lambda_2,\lambda_3}}(\PBbb^1,\vec{z})$ and hence tracing out automorphisms on the $\Hom$-spaces that appear on the left is same as tracing out automorphisms on conformal blocks that appear on the right. Recall that the categorical twisted fusion ring $\KC(\Ccal(\frg,\ell),\gamma)$ (see Definition \ref{def:categoricaltwistedfusion}) is described using the former operation (by Proposition \ref{p:tracesonhomspaces}), whereas the twisted Kac-Moody fusion ring $\Rcal_{\ell}(\frg,\gamma)$ has been defined using the latter operation. Hence we obtain:
	\begin{proposition}\label{cor:corofoneofmain}For any finite order automorphism $\gamma$ of $\frg$, the categorical $\gamma$-twisted fusion ring $\KC(\Ccal(\frg,\ell),\gamma)$ is isomorphic to $R_{\ell}(\frg,\gamma)$. 
	\end{proposition}
	
	Note that if $\gamma$ is an automorphism of order $m$ we get an action of a finite cyclic group $\langle\gamma\rangle$ on $\frg$ which preserves a Borel subalgebra. Hence by Theorem \ref{thm:oneofmain} we get a $\langle\gamma\rangle$-crossed modular fusion category $$\Ccal(\frg,\langle\gamma\rangle,\ell)=\Ccal_1\oplus\Ccal_{\gamma}\oplus\cdots\oplus\Ccal_{\gamma^{m-1}}.$$ By \cite{ENO:10} we know that we have a correspondence between modular autoequivalences of $\Ccal_1$ and invertible $\Ccal_1$-module categories. By the same work we also know that the invertible $\Ccal_1$-module category $\Ccal_\gamma$ is precisely the one corresponding to the modular autoequivalence $\gamma$ of $\Ccal_1=\Ccal(\frg,\ell)$.
	
	\subsection{Twisted conformal blocks and categorical crossed S-matrices}\label{sec:crossedmatrix} 
	In this section, we describe how to  relate  the normalized character table  for the fusion ring $\mathcal{R}_{\ell}(\frg,\sigma)\cong \Rcal_{\ell}(\frg,\gamma)$ associated to diagram automorphisms $\sigma$ (resp. automorphism $\gamma$ in the diagram automorphism class of $\sigma$) at level $\ell$ as discussed in Section \ref{section:hongfusion} to compute the categorical crossed S-matrix using Theorem \ref{thm:oneofmain}.

	Let $\mathcal{F}$ be a fusion ring coming from fusion rules with basis $I$ and let $S$ be the set of characters. Consider a matrix $\Sigma$ whose rows are parameterized by $S$ and columns are parameterized by $I$. For $(\chi,\lambda)\in S\times I$, set: $\Sigma_{\chi,\lambda}:=\chi(\lambda).$
	
	Consider the element $\omega=\sum_{\lambda \in I}\lambda\lambda^*$ 
	and the diagonal
	matrix $D_{\omega}$ whose entries are $\sum_{I}|\chi(\lambda)|^{2}$ for $\chi \in S$.  Now we get $\Sigma \cdot \overline{\Sigma}^{t}=D_{\omega}.$ 
	\begin{definition}\label{def:normalizedcharactertablefortwistedfusionofconf}We define the normalized character table $\Sigma'$  of the ring $\mathcal{R}_{\ell}(\frg,\sigma)$  to be the matrix $D_{\omega}^{-\frac{1}{2}}\Sigma$. 
	\end{definition}
	
	A formula for $D_{\omega}$ will be discussed later as a Weyl denominator. Recall by Proposition \ref{prop:mainhong1}, the rows of the matrix $\Sigma'$ are parameterized by the set $T^{reg}_{\sigma,\ell}/W_{\sigma}$, where 
	$$T_{\sigma,\ell}:=\{t \in T_{\sigma}|e^{\alpha}(t)=1 \  \mbox{for} \ \alpha \in (\ell+h^{\vee})\iota(Q^{\sigma}(\mathfrak{g}))\},$$ where $T_{\sigma}$ is a maximal torus of the orbit Lie algebra $\mathfrak{g}_{\sigma}$ and $\iota(Q^{\sigma})$ be as in equation \ref{lemma:iotamagic}. 
	
	Let $\beta :P^\ell(\frg,\gamma)\cong P^{\ell}(\mathfrak{g}(A))\rightarrow T^{reg}_{\sigma,\ell}/W_{\sigma}$ be the natural bijection given by Proposition \ref{prop:vitalbijectionthathongsshouldhavewrittenproperly} and Equation \ref{eqn:bijfromarbtodia} in Section \ref{sec:arbtodia}, where $A$ is the Cartan matrix type associated to the Lie algebra $\frg$ and the diagram automorphism $\sigma$. The columns of the matrix $\Sigma'$ are parameterized by the set $P_{\ell}(\frg)^{\sigma}$. 
	\begin{remark}
		Note that by proposition \ref{prop:mainhong1}, the column of $\Sigma'$ corresponding to the unit in $I$ is real positive.
	\end{remark}
	
	The following Proposition relates the categorical crossed S-matrices with the character table. 
	\begin{proposition}\label{prop:twistedviaKMtwisted}For any $\gamma \in \Gamma$, the categorical $\gamma$-crossed $S$-matrix $S^{\gamma}$ of the $\Gamma$-crossed modular fusion category given by Theorem \ref{thm:oneofmain} and the normalized character table of the ring $\Rcal_{\ell}(\frg,\sigma)$ (where $\sigma$ is the diagram automorphism associated to $\gamma$) are related by the formula 
		\begin{equation}\label{eq:categoricalsmatrixequalskacmoody}
		S^\gamma_{\lambda,\mu}=\pm\Sigma'_{\beta(\lambda),\mu},
		\end{equation} where the sign $\pm$ only depends on $\gamma$  as in Lemma \ref{lem:sign}. Here we have used the natural bijection (see Section \ref{sec:arbtwistedfusion}) between $P_{\ell}(\frg)^{\gamma}\cong P_{\ell}(\frg)^{\sigma}$ to identify the columns of both matrices.
	\end{proposition}
	\begin{proof}By Proposition  \ref{cor:corofoneofmain}, for $\gamma\in \Gamma$ the categorical crossed fusion ring of $\mathcal{C}(\frg,\Gamma,\ell)$ is isomorphic to $R_{\ell}(\frg,\gamma)$. Now by construction  the categorical $\gamma$-crossed S-matrix $S^\gamma$ at level $\ell$ is the normalized character table of the categorical twisted fusion ring (Theorem \ref{thm:chartableoftwistedfusion}) constructed for each $\gamma \in \Gamma$ for the category $\mathcal{C}(\frg,\Gamma,\ell)$. 
		
		On the other hand the ring $\mathcal{R}_{\ell}(\frg,\gamma)$ and $\mathcal{R}_{\ell}(\frg,\sigma)$ are isomorphic and hence $\Sigma'$ is also the normalized character table of $\Rcal_{\ell}(\frg,\gamma)$. By Remark \ref{rk:0column}, the 0-th column of the categorical crossed S-matrix is real since the entries are the categorical dimensions of the simple objects (divided by the positive real constant $\sqrt{\dim \Ccal_1}$) of $\Ccal_\gamma$. Moreover, by Huang's \cite[Proof of Theorem 4.5]{Huang:08b} computations of the pivotal trace along with Equation 13.8.10 and Remark 13.8 in \cite{Kac}, we know that in the untwisted part $\Ccal_1$ all categorical dimensions are positive, i.e. the spherical structure on the untwisted modular fusion category $\Ccal_1$ is positive. 
		
		Now the spherical structure on $\Ccal_\gamma$ induces a compatible and normalized $\Ccal_1$-module trace on $\Ccal_\gamma$ (see \cite[\S1.3]{Desh:17}). As observed in \cite[\S1.3]{Desh:17} there are exactly two such possible module traces and these two traces differ by a sign. Since $\Ccal_1$ is positive spherical, $\Ccal_\gamma$ has the canonical positive Frobenius-Perron trace. And hence the $\Ccal_1$-module trace on $\Ccal_\gamma$ coming from the spherical structure must be either this one, or its negative. Hence in $\Ccal_\gamma$ it follows that all dimensions must be either all positive or all negative. 
		
		Now both $S^{\gamma}$ and $\Sigma'$ are normalized character table of the same fusion ring with the same basis and set of characters, the matrix $\Sigma'$ has the column corresponding to $0\in P_\ell(\frg)^\gamma$ positive real, while in the matrix $S^\gamma$ the 0-th column is either all positive or all negative. Hence $S^{\gamma}$ and $\Sigma'$ must agree up to a sign which only depends on $\gamma$.
	\end{proof}
	\begin{lemma}\label{lem:sign}
		The sign in Proposition \ref{prop:twistedviaKMtwisted} defines a group homomorphism $\sign:\Gamma\to \{\pm 1\}$.
	\end{lemma}
	\begin{proof}
		Since the categorical dimensions of all objects in the untwisted part $\Ccal_1$ are positive, the dimensions of objects in $\Ccal_\gamma$ have the same sign which we denote $\sign{(\gamma)}$. Now given an object $A_1\in \Ccal_{\gamma_1}, A_2\in \Ccal_{\gamma_2}$ we have $\dim(A_1)\cdot\dim(A_2)=\dim(A_1\dotimes A_2)$. Hence we conclude that $\sign:\Gamma\to {\pm 1}$ is a group homomorphism.
	\end{proof}
	
	\begin{corollary}\label{cor:signinverlinde}
		In the Verlinde formula Theorem \ref{conj:main1} we could take the matrices $\Sigma'$ (from Definition \ref{def:normalizedcharactertablefortwistedfusionofconf}) in place of the categorical crossed S-matrices. In particular, for our purposes, we can ignore the sign in Proposition \ref{prop:twistedviaKMtwisted}.
	\end{corollary}
	\begin{proof}
		Note that the elements $m_1,\cdots ,m_n\in \Gamma$ that appear in the Verlinde formula satisfy a relationship of the form
		$[a_1,b_1]\cdots[a_g,b_g]m_1\cdots m_n=1.$
		
		Applying the $\sign$ homomorphism to the relationship $[a_1,b_1]\cdots[a_g,b_g]m_1\cdots m_n=1$ we obtain that $$\sign(m_1)\cdots \sign(m_n)=1.$$ In other words if we change all the crossed S-matrices $S^{m_i}$ by a sign, all the signs will cancel and we will get the same answer.
	\end{proof}

	\section{Explicit description of $\Sigma'$}Recall that $\Sigma'$ is the normalized character table of the fusion ring $\mathcal{R}_{\ell}(\frg,\sigma)$ whose rows are parameterized (Proposition \ref{prop:mainhong1}) by $T^{reg}_{\sigma,\ell}/W_{\sigma}$ and columns by $P_{\ell}(\frg)^{\sigma}$, where $\sigma$ is a diagram automorphism of $\frg$ and $\frg(A)$ is the twisted affine Kac-Moody Lie algebra associated to $(\frg,\sigma)$.

	The Weyl character formula and the proof of Lemma 9.7 in \cite{BeauHirz} and \cite{Hong}, we get
	\begin{equation}
	\label{eqn:J}
	D_{\omega}(t)={|T_{\sigma,\ell}|}\bigg(\prod_{\alpha \in \Delta_{\sigma,+}}(e^{\frac{\alpha}{2}}(t)-e^{-\frac{\alpha}{2}}(t))\bigg)^{-2}, 
	\end{equation} where $D_{\omega}$ as in Section \ref{sec:crossedmatrix}.
	
	By the definition of $\Sigma'$ and  the Weyl character formula, we get the following: 
	\begin{equation}
	\label{eqn:normalizedcharacter}
	\Sigma'_{t, \mu}={|T_{\sigma,\ell}|}^{-\frac{1}{2}}\frac{\prod_{\alpha \in \Delta_{\sigma,+}}|e^{\frac{\alpha}{2}}(t)-e^{-\frac{\alpha}{2}}(t)|}{\prod_{\alpha \in \Delta_{\sigma,+}}(e^{\frac{\alpha}{2}}(t)-e^{-\frac{\alpha}{2}}(t))}\bigg(\sum_{w\in W_{\sigma}}\epsilon(w)e^{(w(\iota(\mu)+\overline{\rho}_{\sigma}))}(t)\bigg),
	\end{equation}
	where $\overline{\rho}_{\sigma}$ is the sum of the fundamental weights of $\frg_{\sigma}$ and $W_{\sigma}$ is the Weyl group of $\frg_{\sigma}$. 
	Recall that by Proposition \ref{prop:vitalbijectionthathongsshouldhavewrittenproperly} there is a natural bijection $\beta :P^{\ell}(\mathfrak{g}(A))\rightarrow T^{reg}_{\sigma,\ell}/W_{\sigma}$, where $\frg(A)$ is a twisted affine Kac-Moody Lie algebra. 
	\begin{proposition}
		\label{prop:charactertableviaSmatrix}Let $\gamma$ be a finite order automorphism of $\frg$ and $\sigma$ be the associated diagram automorphism. Let  $\Delta_{\sigma}$ denote the set of roots of the Lie algebra $\frg_{\sigma}$. Then the relations between the matrix $\Sigma'$ and the twisted Kac-Moody $S$-matrices are given by the following:
		$$\sign(\gamma)S^{\gamma}_{\lambda,\mu}=\Sigma'_{\beta(\lambda),\mu}=i^{|\mathring{\Delta}_{+}|}\bigg(\frac{\prod_{\alpha \in \Delta_{\sigma,+}}|e^{\frac{\alpha}{2}}(\beta(\lambda))-e^{-\frac{\alpha}{2}}(\beta(\lambda))|}{\prod_{\alpha \in \Delta_{\sigma,+}}(e^{\frac{\alpha}{2}}(\beta(\lambda))-e^{-\frac{\alpha}{2}}(\beta(\lambda)))}\bigg)\overline{{\mathscr{S}}_{\lambda,\iota(\mu)}^{(\ell)}},$$ where $\iota$ is the natural bijection between $P_\ell(\frg)^{\sigma}  $ and $P_\ell(\frg_{\sigma})$ and ${\mathscr{S}}^{(\ell)}_{\lambda,\iota(\mu)}$ denote the twisted Kac-Moody $S$-matrices given in Equations  \ref{{eqn:SmatrixforA2n-12}}, \ref{{eqn:SmatrixforDn+12}},\ref{{eqn:SmatrixforD43}},\ref{{eqn:SmatrixforE62}} and \ref{{eqn:SmatrixforA2n2}}, and $\sign(\gamma)$ is the sign from Lemma \ref{lem:sign}. 
	\end{proposition}
	\begin{proof}The first equality is Proposition \ref{prop:twistedviaKMtwisted} and Lemma \ref{lem:sign}. We split up the rest of the proof of the proposition to case by case for the untwisted affine Lie algebras.
		\subsection{The case $A=A_{2n}^{(2)}$}Since $\mathring{\frg}=\mathfrak{g}_{\sigma}$, we get $W_{\sigma}=\mathring{W}$ and $\overline{\rho}_{\sigma}=\overline{\rho}$. Let $J_{\iota(\mu)}(\beta(\lambda))=\sum_{w\in W_{\sigma}}\epsilon(w)e^{(w(\iota(\mu)+\overline{\rho}_{\sigma}))}(\beta(\lambda))$. In this case, the bijection $\beta$ is given by the following $\beta(\lambda)=\exp(\frac{2\pi i} {\ell+h^{\vee}}(\nu_{\frg(A)}(\lambda+\overline{\rho})))$, where $h^{\vee}$ is the dual Coxeter number. Hence we get 
		\begin{equation}
		\label{eqnarray:mainfuldeduction1}
		J_{\iota(\mu)}(\beta(\lambda))=\sum_{w\in \mathring{W}}\epsilon(w)\exp\bigg(\kappa_{\mathfrak{g}(A)}(w(\iota(\mu)+\overline{\rho}),\lambda+\overline{\rho})\bigg)
		=i^{|\mathring{\Delta}_{+}|}(\ell+2n)^{\frac{1}{2}}{\overline{{\mathscr{S}}^{(\ell)}_{\lambda,\iota({\mu})}}(A_{2n}^{(2)})}.
		\end{equation}Hence we get the following 
		\begin{equation}
		\label{equation:characterfora2n2n}
		\Sigma'_{\beta(\lambda),\mu}(\ell)=i^{|\mathring{\Delta}_{+}|}\bigg(\frac{\prod_{\alpha \in \Delta_{\sigma,+}}|e^{\frac{\alpha}{2}}(\beta(\lambda))-e^{-\frac{\alpha}{2}}(\beta(\lambda))|}{\prod_{\alpha \in \Delta_{\sigma,+}}(e^{\frac{\alpha}{2}}(\beta(\lambda))-e^{-\frac{\alpha}{2}}(\beta(\lambda)))}\bigg)\overline{{\mathscr{S}}^{(\ell)}_{\lambda,\iota({\mu})}}(A_{2n}^{(2)}).
		\end{equation}
		This completes the proof in the case $A=A_{2n}^{(2)}$. 
	\end{proof}

	\subsection {The case $A$ is either $A_{2n-1}^{(2)}$, $D_{n+1}^{(2)}$, $D_{4}^{(3)}$, $E_6^{(2)}$}In these cases the Lie algebras $\mathring{\frg}$ and the $\mathfrak{g}_{\sigma}$ are Langlands dual. In particular, $\mathring{W}=W_{\sigma}$ and $\mathring{Q}^{\vee}$ is equal to $Q(\mathfrak{g}_{\sigma})$ and vice versa. As before, let $\beta: P^{\ell}(\frg(A))\rightarrow T^{reg}_{\sigma,\ell}$ be the natural bijection. Then define a function $J_{A}$ from the set $T^{reg}_{\sigma,\ell}\times P^{\ell}(\mathfrak{g})^{\sigma}$ given by the formula 
	$$ J_{A}(\beta(\lambda), \mu)=\sum_{w\in \mathring{W}}\epsilon(w)\exp\bigg((w(\iota(\mu)+\overline{\rho}))(\beta(\lambda))\bigg).$$
	Let $\check{P}_{\ell}(A)=\{\check{\lambda}\in \check{P}(\mathring{\frg})|\theta(\check{\lambda})\leq \ell\}$, where $\theta$ is the highest root of $\mathring{\frg}$. There is a natural bijection $\eta: P^{\ell}(\frg)^{\sigma}\rightarrow \check{P}_{\ell}(A)$. Similarly define another function $J'_A$ on $P^{\ell}(\frg(A))\times \check{P}_{\ell}(A)$ by the formula:
	$$ J'_{A}(\tilde{\lambda}, \tilde{\mu})=\sum_{w\in {W}_{\sigma}}\epsilon(w)\exp\frac{2\pi i }{\ell+h^{\vee}}\bigg((w(\lambda+\overline{\rho}))(\tilde{\mu}+\check{\overline{\rho}})\bigg),$$ where $\overline{\rho}$ and $\check{\rho}$ are the sum of the fundamental weights (respectively coweights) of the Lie algebra $\mathring{\frg}$. By a direct calculation, we can check that 
	$J_A(\beta(\lambda),\mu)=J'_A(\lambda,\eta(\mu)).$
	
	To a Cartan matrix $A$ of types $A_{2n-1}^{(2)}$, $D_{n+1}^{(2)}$, $E_6^{(2)}$, V. Kac \cite{Kac} associates another Cartan matrix $A'$ of type $D_{n+1}^{(2)},A_{2n-1}^{(2)}$, $E_6^{(2)}$ respectively. Since $\frg_{\sigma}$ and $\mathring{\frg}$ are Langlands dual, we observe by the calculations in Appendix \ref{section:crossedS} that there is a natural  bijection $\eta': \check{P}_{\ell}(A)\rightarrow P^{\ell}(\frg(A'))$. In particular:
	\begin{equation}
	\label{eqn:complicatedcalculationthattookwednesday}J_{A}(\beta(\lambda),\mu)=i^{|\mathring{\Delta}_{+}|}|T_{\sigma,\ell}|^{\frac{1}{2}}\overline{{\mathscr{S}}^{(\ell)}_{\lambda,\eta'\circ\eta({\mu})}}(A).
	\end{equation}Substituting Equation \eqref{eqn:complicatedcalculationthattookwednesday} in Equation \eqref{eqn:normalizedcharacter},  we get the required result by observing that $\iota=\eta'\circ\eta$.

	\section{Atiyah algebra of the twisted WZW connection on $\GModhat$}\label{sec:atiyahmitcoordinates}
	In this section, we extend the twisted WZW connection on $\mathbb{V}_{\vec{\lambda},\Gamma}$ on the compact moduli $\GMod$ to a projective connection with logarithmic singularities. We also compute the Atiyah algebra \cite{BS,Tsuchimoto} of the connection. The corresponding result in the untwisted case is due to Tsuchimoto \cite{Tsuchimoto}.
	
	Let $\pi: \widetilde{C}\rightarrow C $ be a $\Gamma$-cover of
	nodal curves and let ${\bf{p}}=(p_1,\dots,p_n)$ be a sequence of $n$-distinct points of $C$ and $\widetilde{\bf{p}}$ be a lift of $\bf{p}$ to
	$\widetilde{C}$.  Let $\widetilde{p}_i$ be a lift of $p_i$ with stabilizer $\Gamma_i=\langle m_i\rangle$ of order $N_i$. We have the following definition:
	\begin{definition}\label{def:specialformal}
		A formal coordinate of $\widetilde{p}_i$ is called { special} if $m_i. \widetilde{z}_i=\exp(2\pi{\sqrt{-1}/N_i})\widetilde{z}_i$. 
	\end{definition}
	
	Let ${\widetilde{\bf{z}}}=(\widetilde{z}_1,\dots,\widetilde{z}_n)$ be an $n$-tuple of formal special coordinates of $\widetilde{\bf{p}}$ and ${\bf{z}}=(z_1,\dots,z_n)$ denote an $n$-tuple of formal coordinates of ${\bf{p}}$. The local coordinates $\wtilde{\mathbf{z}}$ and ${\bf{z}}$ are related by
	the formula $z_i=\widetilde{z}_i^{N_i},$ where $N_i$ is the order of the stabilizer of $\Gamma_i$ at the points $\widetilde{p}_i$.
	
	The functor  $\GMod$ along with a choice of formal special coordinates will be denoted by $\GModhat$. Similarly, we denote 
	the functor $\GMod$ along with a choice of one-jets along the marked points by $\GModtilde$. We refer the reader to Appendix \ref{sec:localcoord} for the definitions and the associated morphisms.

	Let ${\bf m}=(m_1,\dots,m_n)$ be an $n$-tuple of elements of $\Gamma^n$. The coherent sheaf of twisted conformal blocks on $\GModhat({\bf m})$ along with choice of formal coordinates will 
	be denoted by $\mathcal{V}^{\dagger}_{\vec{\lambda},\Gamma}(\widetilde{C},C, \widetilde{\bf{p}}, {\bf{ p}}, \widetilde{{\bf{z}}})$. We refer the reader to \cite{KH,szcz,WangShen} for more details. In this section, we give another construction of $\mathcal{V}^{\dagger}_{\vec{\lambda},\Gamma}\widetilde{C},C, \widetilde{\bf{p}}, {\bf{ p}}, \widetilde{{\bf{z}}})$ via the Beilinson-Bernstein localization principal following \cite{szcz}. 
	By taking invariant push-forward, we also show that the coherent sheaf $\mathcal{V}_{\vec{\lambda},\Gamma}^{\dagger}(\widetilde{C},C,\widetilde{\bf{p}},\mathbf{p}, \widetilde{{\bf z}})$ descends to a coherent sheaf on $\GModtilde({\bf m})$ which 
	we denote by $\mathcal{V}^{\dagger}_{\vec{\lambda},\Gamma}(\widetilde{C},C,\widetilde{\bf{p}}, {\bf{p}},\widetilde{\bf{v}})$. 

		\subsection{Localization of twisted $\mathcal{D}$-modules} In this section, we briefly recall the definition of Harish-Chandra
		pairs and the Beilinson-Bernstein localization functors \cite{BD3} to produce twisted $\mathcal{D}$-modules. We mostly follow the book of Frenkel and Ben-Zvi \cite{FBz}, however we discuss the adaptation to the logarithmic settings. We refer the reader to \cite[Sections 17.1-17.2]{FBz} or \cite[Definition 6.1]{szcz} for  the definition of a {\em Harish-Chandra pair} $(\mathfrak{s},K)$. 

		Let $Z$ be a smooth variety and let $\Theta_Z$ be the tangent sheaf of $Z$. Let $\Delta_Z$ be a normal crossing divisor in $Z$ and let $\Theta_{Z,\Delta_Z}$ be the sheaf of tangent vector fields preserving the divisor $\Delta_Z$. 
		\begin{definition}	\label{def:harishchandrapari}
			A {\em log $(\mathfrak{s},K)$-action} on  $Z$ is  an action of $K$ on $Z$ that preserves $\Delta_{Z}$  along with a  Lie algebra map $\alpha : \mathfrak{s}\otimes \mathcal{O}_Z\rightarrow \Theta_{Z,\Delta_Z}$, that satisfies the following:
			\begin{enumerate}
				\item The differential of the map $K\times Z\rightarrow Z$ is the restriction of $\alpha$ to $\mathfrak{k}$, where $\mathfrak{k}$ is the Lie algebra of $K$. 
				\item The action intertwines the representation of $K$ on $\mathfrak{s}$ with the action of $K$ on $\Theta_{Z,\Delta_Z}$. 
			\end{enumerate} Motivated by the setting in \cite{FBz}, we say that the $(\mathfrak{s},K)$-action on $(Z,\Delta_Z)$ is {\em log-transitive} if
			$\alpha$ is surjective.  
		\end{definition}
%
		%
		%
		%
		%

		\subsubsection{Assumption of the central extension}Let $\widehat{\mathfrak{s}}$ be a central extension of $\mathfrak{s}$.
		We assume the following: 
		\begin{enumerate}\label{enu:hc1}
			\item\label{assum1} The exact sequence 
			$0\rightarrow \mathcal{O}_Z\rightarrow \widehat{\mathfrak{s}}\otimes \mathcal{O}_Z\rightarrow \mathfrak{s}\otimes \mathcal{O}_Z \rightarrow 0$,
			splits over $\operatorname{Ker}\alpha$. 
			\item\label{assum2} The central extension defining $\widehat{\mathfrak{s}}$ splits over $\mathfrak{k}$ making the pair $(\widehat{\mathfrak{s}},K)$ a new Harish-Chandra pair. 
		\end{enumerate}
		These assumptions on the central extension $\widehat{\mathfrak{s}}$ imply that the quotient of $\widehat{\mathfrak{s}}\otimes \mathcal{O}_Z$ by $\operatorname{Ker}\alpha$ is a logarithmic Atiyah algebra $\mathcal{A}(-\log \Delta_{Z})$ on $Z$ i.e an exact sequence of sheaves of Lie algebras
		\begin{equation}\label{eqn:fundexact}
		0\rightarrow \mathcal{O}_Z \rightarrow \mathcal{A}(-\log \Delta_{Z}) \rightarrow^{\alpha} \Theta_{Z,\Delta_Z}\rightarrow 0.\end{equation}
		The map $\alpha$ is called the anchor map and the exact sequence \eqref{eqn:fundexact} is known as {the fundamental exact sequence} of the Atiyah algebra.
		\subsubsection{Twisted log-localization functor}\label{sec:twistedloc}
		
		Let $\widehat{\mathfrak{l}}$ be any Lie algebra containing $\widehat{\mathfrak{s}}$ and carrying a compatible adjoint $K$ action and suppose $\widetilde{\mathfrak{l}}$ is a subsheaf of $\widehat{\mathfrak{l}}\otimes \mathcal{O}_Z$ that satisfies the following:
		\begin{enumerate}
			\item it is preserved by $\widehat{\mathfrak{s}}\otimes \mathcal{O}_Z$;
			\item it is preserved by the action of $K$;
			\item it contains $\operatorname{Ker}\alpha$, where $\alpha: \mathfrak{s}\otimes \Ocal_Z\rightarrow \Theta_{Z,\Delta_Z}$. 
		\end{enumerate}
		The Beilinson-Bernstein formalism  first takes a $\widehat{\mathfrak{l}}$-module $V$ which further has the structure of a $(\widehat{\mathfrak{s}},K)$-module and considers the following sheaf of $\mathcal{O}_Z$-modules $V\otimes \mathcal{O}_Z/\widetilde{\mathfrak{l}}.(V\otimes \mathcal{O}_Z).$ 
		Then the assumptions guarantee, this is also a $\mathcal{A}(-\log \Delta_{Z})$-module. 
		
		We recall following Section 17.2.5 in \cite{FBz} how to descend the twisted $\mathcal{A}(-\log \Delta_{Z})$ module to a free group quotient $T$ of $Z$ by $K$. 
		Let $Z$ be a principal $K$-bundle over a variety $T$ and let $\Delta_{T}$ be a normal crossing divisor whose inverse image in $Z$ is $\Delta_{Z}$.  We say a pair $(T, \Delta_T)$ has a log-$(\mathfrak{s},K)$ structure if the Harish-Chandra pair $(\mathfrak{s},K)$ acts log-transitively on $Z$ extending the action of $K$ on the fibers of the map $\pi:Z\rightarrow T$. The anchor map $\alpha$ gives a surjective map 
		$$ \overline{\alpha}:\overline{\mathfrak{s}}\rightarrow \Theta_{T,\Delta_T}, \ \ \mbox{where $\overline{\mathfrak{s}}=(\mathfrak{s}/\mathfrak{k}\otimes \mathcal{O}_Z)^{K}$}.$$ We denote the corresponding log-Atiyah algebra by $\mathcal{A}(-\log \Delta_T)$. 
		Further let $\widehat{\mathfrak{s}}$, $\widehat{\mathfrak{l}}$, $\widetilde{\mathfrak{l}}$ and $V$ be as in the beginning of Section \ref{sec:twistedloc} satisfying conditions (1)-(3). 
		Now since $Z$ is a $K$-torsor over $T$, it follows that $(V\otimes \mathcal{O}_Z)/\widetilde{\mathfrak{l}}\cdot(V\otimes \mathcal{O}_Z)$ descends as a $\mathcal{O}_T$-module on $T$ which we denote by $\widehat{BB}(V)$. Moreover, $\widehat{BB}(V)$ is also a $\mathcal{A}(-\log\Delta_T)$-module. 
		\subsubsection{Fibers of $\widehat{BB}(V)$} 
		For a point $t\in T$, consider the set $Z_t=\pi^{-1}(t)$. The group $K$ acts freely transitively on $Z_t$ and we consider the vector space $\mathcal{V}_t:=Z_{t}\times_{K} V$. Similarly for every $t\in T$, set $\widehat{\mathfrak{l}}_t:=Z_{t}\times_{K}\widehat{\mathfrak{l}}$. Since $\widetilde{\mathfrak{l}}$ is preserved by the $K$-action, we also a get a Lie subalgebra $\widetilde{\mathfrak{l}}_t$ of $\widehat{\mathfrak{l}}_t$. The fibers of $\widehat{BB}(V)$ at a point $t\in T$ are the coinvariants 
		$\mathcal{V}_t/\widetilde{\mathfrak{l}}_t\cdot \mathcal{V}_t$.
		
		We will apply the Beilinson-Bernstein localization formalism in the situation $Z=\GModhat({\bf m })\rightarrow \GModtilde({\bf{m}})=T.$ as torsors for a natural group scheme. 
		We now describe these torsors. 
		\subsection{Coordinate Torsors}In this section, we recall (following \cite{FBz}, \cite{FrenkelSzcz}) natural coordinate torsors arising out of 
		ramified cover of disks.  First, 
		we discuss the absolute situation without covers. 
		\subsubsection{Formal coordinate on disks with ramification} Let $R$ be a ring, then following Section 4.1 in Szczesny \cite{szcz}, we let $\mathcal{O}_R$
		(respectively $\mathcal{K}_R$) denote the rings $R[[z]]$ (respectively $R((z))$). 
		Consider the automorphism of the ramified
		covering ${R}[[z^{\frac{1}{N}}]]\rightarrow {R}[[z^{\frac{1}{N}}]]$ given
		by sending $z^{\frac{1}{N}} \rightarrow \varepsilon z^{\frac{1}{N}}$, where $\varepsilon$ is
		an $N$-th root of unity.  Let as before $\operatorname{Aut}R[[z^{\frac{1}{N}}]]$ denote the 
		automorphisms of $R[[z^{\frac{1}{N}}]]$ preserving the ideal $(z^{\frac{1}{N}})$. Let $\operatorname{Aut}_N\mathcal{O}_R$ denote 
		the subgroup of $\operatorname{Aut}(R[[z^{\frac{1}{N}}]])$ preserving the subalgebra $R[[z]]$. Then there is an exact sequence 
		\begin{equation}
		\label{eqn:exactseqence}
		0\rightarrow \mathbb{Z}/N\mathbb{Z}\rightarrow \operatorname{Aut}_N \mathcal{O}_R \rightarrow^{\mu} \operatorname{Aut}\mathcal{O}_R \rightarrow 0.
		\end{equation}
		As in Szczesny \cite{szcz},  we consider the group scheme (ind group-scheme) $\operatorname{Aut}_N \mathcal{O}$ (respectively $\operatorname{Aut}_N\mathcal{K}$) representing $R\rightarrow \operatorname{Aut}_N(\mathcal{O}_R)$ (respectively $R\rightarrow \operatorname{Aut}_N(\mathcal{K}_R)$)
		and its Lie algebra $\operatorname{Der}^{(0)}_N \mathcal{O}$ (respectively $\operatorname{Der}_N\mathcal{K}$). We have the following explicit descriptions:
		\begin{equation}\label{eqn:ecplicitderivations}
		\operatorname{Der}_N^{(0)}\mathbb{C}[[z^{\frac{1}{N}}]]=z^{\frac{1}{N}}\mathbb{C}[[z]]\partial_{z^{\frac{1}{N}}}, \ 
		\operatorname{Der}_N\mathbb{C}((z^{\frac{1}{N}}))=z^{\frac{1}{N}}\mathbb{C}((z))\partial_{z^{\frac{1}{N}}}.
		\end{equation}

		The isogeny $\mu$ 
		whose derivative is $z^{\frac{1}{N}+k}\partial_{z^{\frac{1}{N}}}\rightarrow Nz^{k+1}\partial_z$
		gives an isomorphism of the following Lie algebras:
		\begin{equation}
		\label{eqn:liealgebraiso}
		d\mu: \operatorname{Der}_N^{(0)}\mathcal{O} \rightarrow \operatorname{Der}^{(0)}\mathcal{O}, \ d\mu: \operatorname{Der}_N\mathcal{K} \rightarrow \operatorname{Der} \mathcal{K}.
		\end{equation}

		Similarly we define $\operatorname{Aut}_{N,+}\mathcal{O}$ to be the subgroup of $\operatorname{Aut}_{N}\mathcal{O}$ such that for any ring $R$, $\operatorname{Aut}_{N,+}(\mathcal{O}_R)$ consists of $R$-power series whose coefficient of $z$ is one. Let  $\operatorname{Der}_{N,+}\mathcal{O}$ be the Lie algebra of $\operatorname{Aut}_{N,+}\mathcal{O}$. 
		Moreover, there is a natural isomorphism between $\operatorname{Der}_{N,+}\mathcal{O}$ and $\operatorname{Der}_{+}\mathcal{O}$.
		\subsubsection{Global Situation }Let $R$ be a $\mathbb{C}$-algebra and consider a family $(\widetilde{C} \rightarrow C,\wtilde{\mathbf{p}},\mathbf{p})$ of $n$-pointed admissible $\Gamma$-covers 
		over $\operatorname{Spec}R$. 
		The monodromies around the marked sections being $\mathbf{m}=(m_1,\cdots,m_n)$. The stabilizer $\Gamma_{\widetilde{p}_i}$ of $\widetilde{p}_i$  is the cyclic subgroup of $\Gamma$ of order say $N_i$ generated by $m_i$.

		We define $$\operatorname{Aut}(\widetilde{C}/C, \mathcal{O}_R):=\operatorname{Aut}_{N_1}(\mathcal{O}_R)\times \dots \times \operatorname{Aut}_{N_n}(\mathcal{O}_R)$$ and analogously  $\operatorname{Aut}_+(\widetilde{C}/C, \mathcal{O}_R)$ and $\operatorname{Aut}_(\widetilde{C}/C, \mathcal{K}_R)$. The corresponding group schemes are denoted by  $\operatorname{Aut}(\widetilde{C}/C,\mathcal{O})$, $\operatorname{Aut}_{+}(\widetilde{C}/C,\mathcal{O})$ 
		and $\operatorname{Aut}(\widetilde{C}/C, \mathcal{K})$ respectively. We further denote their Lie algebras by $\operatorname{Der}(\widetilde{C}/C,\mathcal{O})$, $\operatorname{Der}_{+}(\widetilde{C}/C,\mathcal{O})$ and $\operatorname{Der}(\widetilde{C}/C,\mathcal{K})$.  The various  versions of the moduli stacks of pointed admissible $\Gamma$-covers are related by the following proposition:
		\begin{proposition}\label{prop:torsorrealization}
			The moduli stacks $\GMod({\bf m})$, $\GModhat({\bf m})$ and $\GModtilde({\bf m})$ are related as follows:\
			\begin{enumerate}
				\item $\GModhat({\bf m})$ is a $\operatorname{Aut}(\widetilde{C}/C,\mathcal{O})$-torsor over $\GMod({\bf m})$. 
				\item $\GModhat({\bf m})$ is a $\operatorname{Aut}_{+}(\widetilde{C}/C,\mathcal{O})$-torsor over $\GModtilde({\bf m})$.
				\item $\GModtilde({\bf m})$ is a $\Gm^{n}$-torsor over $\GMod({\bf m})$. 
			\end{enumerate}
		\end{proposition}

		Let us define $\operatorname{Der}(\widetilde{C}/C, \mathcal{K})$ 
		to be the Lie algebra of 
		the ind-scheme $\operatorname{Aut}(\widetilde{C}/C, \mathcal{K})$. If $\Gamma$ is trivial, then the Lie algebra is just denoted by $\operatorname{Der}(C. \mathcal{K})$. 
		The isogeny $\mu$ in \eqref{eqn:liealgebraiso} gives 
		the following isomorphism between $\operatorname{Der}(\widetilde{C}/C,\mathcal{K})\simeq \operatorname{Der}(C,\mathcal{K}).$
		
		\subsection{Equivariant Virasoro uniformization}
		Let us recall the following result on  $\Gamma$-equivariant Virasoro uniformization due to M. Szczesny \cite[Theorem 7.1]{szcz}. 
		\begin{theorem}\label{thm:virasorouniformization}
			The stack $\GModhat({\bf m})$ carries a log-transitive action of   $\operatorname{Der}(\widetilde{C}/C,\mathcal{K})$ extending the action of $\operatorname{Aut}(\widetilde{C}/C,\mathcal{O})$ 
			along the fibers of the natural map $\widehat{\pi}: \GModhat({\bf m})\rightarrow \GMod({\bf m})$ and preserving the boundary divisor $\widehat{\Delta}_{g,n,\Gamma}=\GModhat(\mathbf{m})\backslash {\hatM}^\Gamma_{g,n}({\bf m})$.
		\end{theorem}
		\paragraph{Comments on the proof:}
		Strictly speaking Szczesny's result \cite{szcz} is on the interior $\hatM^\Gamma_{g,n}(\mathbf{m})$. Also his strategy is to reduce to the untwisted case.  However the arguments can be easily extended to the boundary following the untwisted case as in Section 4.2 in \cite{U}. Namely the action of $\operatorname{Aut}(\widetilde{C}/C,\mathcal{O})$ extends to the boundary since we are only modifying the smooth points. It particular it also preserves the singularities. 
		
		Now log-transitivity follows from extending the argument of transitivity in Frenkel and Ben-Zvi \cite{FBz} in the untwisted case along with the result  (see \cite{ACV},\cite{ACG:11}) on deformations of nodal admissible $\Gamma$ covers $\widetilde{C}\rightarrow C$ preserving the singularities. These are exactly controlled by deformations of the pointed curve obtained by normalization of $C$ and the new marked points coming from the nodes of $C$.  Using these facts, we can argue as in \cite{FBz}. 
		If $\Gamma$ is trivial, the result of the extension to the boundary  can be also found in \cite{Tsuchimoto}.

		The following is an obvious variant of Theorem \ref{thm:virasorouniformization} where we replace $\GMod({\bf m})$ by $\GModtilde({\bf m})$. 
		\begin{corollary}\label{cor:virasorojetspaces}The same results in Theorem \ref{thm:virasorouniformization} hold if 
			we consider the natural action of the group $\operatorname{Aut}_{+}(\widetilde{C}/C,\mathcal{O})$ and along the fibers of the 
			natural map $\GModhat({\bf m}) \rightarrow \GModtilde({\bf m})$.
		\end{corollary}
		
		\subsubsection{Virasoro Uniformization and the Hodge bundle}
		
		Let $\widetilde{C}\rightarrow Z$ be a versal family in $\GModhat({\bf m})$ parameterized by a smooth scheme $Z$ (or one may work in the setting of Looijenga \cite{Looijenga}). 
		We let $\widetilde{\bf{z}}$ (respectively ${\bf{z}}$) be the corresponding choice of special local formal coordinates along the sections $\widetilde{\bf p}$ (respectively ${\bf{p}}$). 
		
		Let $\Delta_Z$ be the normal crossing divisor of nodal curves and as before $\Theta_{Z,\Delta_Z}$ be the
		subsheaf of tangent vector fields on $Z$ preserving the divisor $\Delta_Z$. Finally, let $\Sigma$ be the critical locus of $\widetilde{C}\rightarrow Z$ and the image of $\Sigma$ is $\Delta_{Z}$. The codimension of $\Sigma$ in $\widetilde{C}$ is at least two.

		By Theorem \ref{thm:virasorouniformization}, the 
		Lie algebra $\operatorname{Der}(\widetilde{C}/C,\mathcal{K})$ acts log-transitively (see Definition \ref{def:harishchandrapari} and the discussion afterward) on $Z$. In particular, we have a surjective 
		Lie algebra map: 
		\begin{equation}
		\label{eqn:surjectivityofanchormap}
		\alpha :\operatorname{Der}(\widetilde{C}/C,\mathcal{K})\otimes \mathcal{O}_Z\twoheadrightarrow \Theta_{Z,\Delta_Z}.
		\end{equation}

		Following the notation in \cite{szcz}, let $vec^{\Gamma}(\widetilde{C}\backslash \Gamma.\widetilde{\bf{p}})$ be 
		the kernel of the map $\alpha$. This is the locally free sheaf given by the $\Gamma$-invariant part of the push-forward of the locally free sheaf $\mathcal{H}om(\omega_{\widetilde{C}/Z},\mathcal{O}_{\widetilde{C}})$ to $Z$, where $\omega_{\widetilde{C}/Z}$ is the relative dualizing sheaf. We refer the reader to Section 4.2 in \cite{U} for the case when $\Gamma$ is trivial. 
		
		Alternatively $vec^{\Gamma}(\widetilde{C}\backslash \widetilde{\bf p})$ can be realized as follows:  Restrict the versal family to $Z\backslash \Delta_Z$. By \cite{szcz}, we get that 
		the fiber of $vec^{\Gamma}(\widetilde{C}\backslash \Gamma.\widetilde{\bf {p}})$ 
		at a point $(\pi: \widetilde{C}_0\rightarrow C_0, \widetilde{\bf{p}},{\bf p}, \widetilde{{\bf z}})$ 
		is given by the Lie algebra of $\Gamma$-invariant vector fields
		on $\widetilde{C}_0\backslash \cup_{i=1}^{n}\Gamma \cdot \widetilde{p}_i$. 
		
		Since 
		by assumption, $\widetilde{C}_0$ is smooth, the cover $\widetilde{C}_0\rightarrow C_0$ is \'etale 
		restricted to $\widetilde{C}_0\backslash \Gamma.\widetilde{\bf{p}}$. Hence, it follows that a $\Gamma$-invariant vector field on  $\widetilde{C}_0\backslash \Gamma.\widetilde{\bf{p}}$ 
		descends to a 
		vector field on $C_0\backslash {\bf p}$. Thus we get that $vec^{\Gamma}(\widetilde{C}\backslash \Gamma.\widetilde{\bf p})$ is canonically isomorphic to the push forward of the locally free sheaf $vec(C\backslash {\bf p}):=\operatorname{\mathcal{H}om}(\omega_{C/Z},\mathcal{O}_C)$ to $Z$.
		\subsubsection{Uniformization and the associated Atiyah algebra} Recall the natural forgetful maps 
		$\widehat{\pi}: \GModhat \rightarrow \GMod \ \mbox{and}  \ \widetilde{\pi}: \GModtilde \rightarrow \GMod.$
		The action of 
		$\operatorname{Aut}(\widetilde{C}/C,\mathcal{K})$ on $\GModhat({\bf m})$ \cite[Theorem 17.3.2]{FBz}, \cite{szcz}
		preserves the divisor $\widehat{\Delta}_{g,n,\Gamma}$. The twisted Sugawara construction \cite{Wakimoto} gives an 
		action of $\operatorname{Der}(\widetilde{C}/C,\mathcal{K})$ on $\mathcal{H}_{\vec{\lambda}}$.
		
		Let $\widehat{\operatorname{Der}}(\widetilde{C}/C,\mathcal{K})$ 
		be the central 
		extension of $\operatorname{Der}(\widetilde{C}/C,\mathcal{K})$ 
		obtained as Beaer sum of the Virasoro cocycles  for the individual factors. 
		We get a short exact sequence 
		\begin{equation}
		\label{eqn:centralexten}
		0\rightarrow \mathcal{O}_{\GModhat({\bf m})}. c 
		\rightarrow \widehat{\operatorname{Der}}(\widetilde{C}/C,\mathcal{K})\rightarrow 
		\operatorname{Der}(\widetilde{C}/C,\mathcal{K})\rightarrow 0.
		\end{equation}
		
		Moreover, the short exact sequence splits when restricted to $vec^{\Gamma}(\widetilde{C}\backslash \Gamma\cdot\widetilde{\bf{p}})$ and the Lie algebra $\operatorname{Der}_{+}(\widetilde{C}/C,\mathcal{O})$. Taking 
		quotients and using surjectivity of the short exact sequence in Equation \eqref{eqn:surjectivityofanchormap}, we  get a logarithmic Atiyah algebra 
		with the following 
		fundamental exact sequence. 
		\begin{equation}\label{eqn:fundaty}
		0\rightarrow \mathcal{O}_{\GModhat({\bf m})}\cdot c\rightarrow 
		\widehat{\operatorname{Der}}(\widetilde{C}/C,\mathcal{K}) /vec{^{\Gamma}}(\widetilde{C}\backslash \Gamma\cdot (\widetilde{\bf{p}}))\rightarrow \Theta_{\GModhat({\bf m}),\Delta_{g,n,\Gamma}}\rightarrow 0.
		\end{equation}

		We denote the log-Atiyah algebra 
		$$\hat{\mathcal{A}}(\widetilde{C}/C)(-\log \hat{\pi}^{-1}(\Delta_{g,n,\Gamma})):=
		\widehat{\operatorname{Der}}(\widetilde{C}/C,\mathcal{K}) /vec{^{\Gamma}}(\widetilde{C}\backslash \Gamma \cdot (\widetilde{\bf{p}})).$$
		Since $\GModhat({\bf m})$ is an $\operatorname{Aut}_{+}(\widetilde{C}/C,\mathcal{O})$-torsor
		over $\GModtilde({\bf m})$, the sheaf $\hat{\mathcal{A}}(\widetilde{C}/C)(-\log \widehat{\pi}^{-1}(\Delta_{g,n,\Gamma}))$ 
		descends via invariant push-forward to a log-Atiyah algebra on the stack $\GModtilde({\bf m})$ which we denote by $\tilde{\mathcal{A}}(\widetilde{C}/C)(-\log \tilde{\pi}^{-1}(\Delta_{g,n,\Gamma}))$. 
		
		\subsubsection{Virasoro Uniformization of the Hurwitz-Hodge bundle}In this 
		section, we give a more explicit description of the Atiyah algebra $\tilde{\mathcal{A}}(\widetilde{C}/C)(-\log\tilde{\pi}^{-1}(\Delta_{g,n,\Gamma}))$.  
		If $\Gamma$ is trivial, the corresponding results can be found in  \cite{ACH,BS,Kon,Tsuchimoto,TUY:89}. 
		
		Consider 
		the natural map $q:\GMod({\bf m})\rightarrow \overline{\mathcal{M}}_{g,n}$
		and let $\Lambda$ be the pullback of the Hodge line bundle on $\overline{\mathcal{M}}_{g,n}$  to $\GMod$ via $q$. 
		Then we have  the following proposition:
		\begin{proposition}Let $\mathcal{A}_{\Lambda}(-\log \Delta_{g,n,\Gamma})$ denote the log-Atiyah algebra
			associated 
			to the line bundle $\Lambda$ on $\GMod({\bf m})$, then there is a natural identification  $\hat{\mathcal{A}}(\widetilde{C}/C)(-\log \hat{\pi}^{-1}(\Delta_{g,n,\Gamma}))$ with $\frac{1}{2}\mathcal{A}_{\hat{\pi}^{*}\Lambda}(-\log \Delta_{g,n,\Gamma})$. 
			\label{prop:virasoroandatiyahalgebra}

		\end{proposition}
		\begin{proof}
			First, we consider the case when $\Gamma$ is trivial. 
			In this case results in \cite{ACH,BS,Kon,Tsuchimoto,TUY:89} show that there is a natural isomorphism of $\widehat{\operatorname{Der}}(C,\mathcal{K})/vec(C\backslash {\bf p})\cong\frac{1}{2}\mathcal{A}_{\pi^{*}\Lambda}(-\log \Delta_{g,n}),$
			where $\Delta_{g,n}$ is the boundary divisor of $\overline{\mathcal{M}}_{g,n}$ 
			and $\hat{\underline{\pi}}: \widehat{\overline{\mathcal{M}}}_{g,n}\rightarrow \overline{\mathcal{M}}_{g,n}$ is the 
			forgetful map. Now consider the following commutative diagram: 
			$$\xymatrix{
				\GModhat({\bf m})\ar[r]^{\hat{\pi}}\ar[d]^{\pi}&\GMod({\bf m}) \ar[d]^{q}\\
				\widehat{\overline{\mathcal{M}}}_{g,n}\ar[r]^{\hat{\underline{\pi}}}&\overline{\mathcal{M}}_{g,n}.}
			$$
			Since the map $\pi$ is flat \cite{JKK}, we get an exact sequence  $$0\rightarrow \pi^*(vec(C\backslash {\bf p})) \rightarrow \pi^{*}\operatorname{Der}(C,\mathcal{K})\rightarrow \pi^*\Theta_{\widehat{\overline{\mathcal{M}}}_{g,n},\Delta_{g,n}}\rightarrow 0.$$
			We  now have the following commutative diagram of sheaves of Lie algebras.
			$$\xymatrix{
				{\operatorname{Der}}(\widetilde{C}/C,\mathcal{K}) \ar[r]^{\alpha}\ar[d] &\Theta_{\GModhat({\bf{m}}),\Delta_{g,n,\Gamma}}\ar[d]\ar[r] &0\\
				\pi^*({\operatorname{Der}}(C,\mathcal{K})) \ar[r] & \pi^*\Theta_{\widehat{\overline{\mathcal{M}}}_{g,n},\Delta_{g,n}}\ar[r] &0.
			}
			$$ Since by the discussion following Equation \eqref{eqn:surjectivityofanchormap}, we get the kernel of the horizontal maps are isomorphic. Hence it suffices to show that  the vertical map on the left is an isomorphism between $\pi^{*}( {\operatorname{Der}}(C,\mathcal{K}))$ is
			${\operatorname{Der}}(\widetilde{C}/C,\mathcal{K})$. 
			This follows from the isomorphism \eqref{eqn:liealgebraiso} of $\operatorname{Der}_{N}\mathcal{K}$
			with $\operatorname{Der}\mathcal{K}$.
		\end{proof}
		\subsection{Twisted log-D module}\label{sec:twistedconnectiontwistedKZ}In this section, we recall the 
		construction of the twisted WZW-connection following \cite{szcz} using the 
		formalism of the {\em twisted localization functor } recalled in Section \ref{sec:twistedloc}. 
		\begin{remark}\label{rem:diffwithsz}
			We want to address a small but important difference from \cite{szcz}. The 
			twisted vertex algebra arising from Kac-Moody algebras are not 
			conformal 
			and hence the coherent sheaf $\mathcal{V}^{\dagger}_{\vec{\lambda},\Gamma}(\widetilde{C},C,\widetilde{\bf{p}},\bf{p},\widetilde{{\bf{z}}})$
			do not descend to a coherent logarithmic $\mathcal{D}$ module  on $\GMod({\bf m})$ under the natural $\mathbb{G}_m^{n}$-action. 
		\end{remark}
		Let ${\bf m}=(m_1,\dots,m_n)\in \Gamma^n$ be a monodromy vector. 
		For each $1 \leq i \leq n$,
		let $N_i$ be the order the $m_i$ and $\mathcal{H}_{\lambda_i}$ is an irreducible 
		highest weight integrable module for $m_i$-twisted Kac-Moody 
		Lie algebra $L(\widehat{\frg}, m_i)$ of 
		highest weight $\lambda_i$ at level $\ell$. As in Sections \ref{sec:deftwistconf}, \ref{sec:descentdata} and \cite[Remark 3.4.6]{BFM}, without loss of generality assume that we have a family of $(\widetilde{C}\rightarrow Z, {\widetilde{\bf{p}}},{\bf p},\widetilde{\bf{z}})$ in $\GModhat({\bf m})$ such that $\widetilde{C}\backslash \widetilde{\bf{p}}(Z)$ is affine and let $\Delta_{Z}$ be the divisor in $Z$ corresponding to singular curves. Following Looijenga \cite[Section 5]{Looijenga}, we put the assumption that vector fields on $Z$ tangent to $\Delta_Z$ are locally liftable to $\widetilde{C}$ (We could have also worked with the set-up of versal family as before).
		
		Similarly, we denote the corresponding family $(\widetilde{C}\rightarrow T, {\widetilde{\bf{p}}},{\bf p},\widetilde{\bf{v}})$ in $\GModtilde({\bf m})$. The smooth scheme $Z$ is a $\operatorname{Aut}_{+}(\widetilde{C}/C,\mathcal{O})$-torsor over $T$.
		For
		a fixed $n$-tuple of monodromies ${\bf{m}}$ and integrable highest weights $\vec{\lambda}=(\lambda_1,\dots,\lambda_n)$, we 
		consider the vector space $\mathcal{H}_{\vec{\lambda}}=\mathcal{H}_{\lambda_1}\otimes \dots\otimes \mathcal{H}_{\lambda_n}$ 
		and the quasi-coherent sheaf on $Z$.
		$$\underline{\mathcal{H}}_{\vec{\lambda}}:=\mathcal{H}_{\vec{\lambda}}\otimes \mathcal{O}_{Z}.$$ 
		
		The Lie algebra $\operatorname{Der}_N(\mathcal{O}_R)$ acts on $\mathcal{H}_{\lambda}$ by 
		the twisted Sugawara action under the 
		homomorphism 
		\begin{equation}\label{eqn:viractionimp}\operatorname{Der}_{N}(\mathcal{O}_R)\rightarrow \operatorname{Vir}_{\Gamma_{\widetilde{p}}}, \ \  z^{k+\frac{1}{N}}\partial_{z^{\frac{1}{N}}}\rightarrow -NL_{k,{\Gamma_{\widetilde{p}}}},
		\end{equation} where $\operatorname{Vir}_{\Gamma_{\widetilde{p}}}$ is the Virasoro algebra with generated by $L_{k,\Gamma_{\widetilde{p}}}$ which acts on the $L(\widehat{\frg}, \Gamma_{\widetilde{p}})$-module $\mathcal{H}_{\lambda}$ via the twisted Sugawara action \cite[Section 3]{Wakimoto}. The Sugawara construction  gives an action 
		of $\widehat{\operatorname{Der}}(\widetilde{C}/C,\mathcal{K})$ 
		on $\underline{\mathcal{H}}_{\vec{\lambda}}$ and one naturally 
		gets an action of the log-Atiyah algebra $\hat{\mathcal{A}}(\widetilde{C}/C)(-\log \Delta_{Z})$ on the sheaf of conformal 
		blocks $\mathcal{V}^{\dagger}_{\vec{\lambda},\Gamma}(\widetilde{C},C,\widetilde{\bf p}, {\bf p}, \widetilde{\bf z})$ on $Z$.

		We apply the formalism (\cite[Section 7.2]{szcz}) of {\em twisted log-localization} as described 
		with Section \ref{sec:twistedloc} with the following:
		\begin{itemize} \item  $\widehat{\mathfrak{s}}=\widehat{\operatorname{Der}}(\widetilde{C}/C,\mathcal{K})$,
			\item $K=\operatorname{Aut}_{+}\mathcal{O}$, 
			\item $V=\mathcal{H}_{\lambda_1}\otimes \dots\otimes \mathcal{H}_{\lambda_n}$and \item $\widetilde{\mathfrak{l}}$ be the $\operatorname{Aut}(\widetilde{C}/C,\mathcal{O})$-equivariant sheaf of Lie algebras whose fibers 
			at a point $(\widetilde{C}, C, {\widetilde{\bf{p}}}, {\bf p},{\widetilde{\bf v}})$ in
			$T$ is given by $\big(\frg\otimes H^0(\widetilde{C},\mathcal{O}_{\widetilde{C}}(\ast \Gamma.\widetilde{\bf{p}})\big)^{\Gamma}$.
		\end{itemize}
		Hence the twisted log-localization gives the sheaf of conformal
		blocks $\mathcal{V}^{\dagger}_{\vec{\lambda},\Gamma}(\widetilde{C},{C}, \widetilde{\bf{p}}, {\bf{p}},\widetilde{{\bf{v}}})$ on $T$ along with 
		an action of the log-Atiyah algebra $\tilde{\mathcal{A}}(\widetilde{C}/C)(-\log \Delta_T)$. A direct corollary of the Proposition \ref{prop:virasoroandatiyahalgebra} is the 
		following:
		\begin{corollary}
			\label{cor:Atiyahalg}
			The log-Atiyah algebra $\frac{\ell\dim\frg}{2(\ell+h^{\vee})}\mathcal{A}_{\tilde{\pi}^{*}\Lambda}(-\log\widetilde{\Delta}_{g,n,\Gamma})$ 
			acts on the dual twisted conformal block bundle or the sheaf of twisted covacua $\mathcal{V}_{\vec{\lambda},\Gamma}(\widetilde{C},C,\widetilde{\bf p},{\bf p}, \widetilde{\bf v})$.
		\end{corollary}

		\subsection{Atiyah algebra for the twisted WZW connection on $\GMod({\bf m})$}\label{sec:atiyahalg}Let ${\bf m} \in \Gamma^n$ be a monodromy vector and for an $n$-tuple of integral weights $\vec{\lambda}=(\lambda_1,\dots,\lambda_n)$ at level $\ell$, we denote the vector bundle of conformal blocks on $\GMod({\bf m})$ by $\mathbb{V}^{\dagger}_{\vec{\lambda},\Gamma}(\widetilde{C},C,\widetilde{{\bf p}}, {\bf p})$. Let $\Gamma$ be a finite group acting on $\frg$. For each $\gamma\in\Gamma$ let $L(\widehat{\frg},\gamma)$ be
		the $\gamma$-twisted affine Kac-Moody Lie algebra. 
		
		As before, let $\widetilde{C}\rightarrow C$ be a ramified Galois cover of nodal curves and let $\widetilde{p}$ be a smooth ramification point of $\widetilde{C}$ 
		whose image in $C$ is denoted by $p$. Let  $\Gamma_{\widetilde{p}}$ be the stabilizer at the point $\widetilde{p}$ and
		we further, assume that $N$ be the order of $\Gamma_{\widetilde{p}}$.  
		Let  $z^{\frac{1}{N}}$ is a special coordinate at the point $\widetilde{p}$, then there is an
		isomorphism of the residue fields $\mathcal{K}_{\widetilde{p}}\simeq \mathbb{C}((z^{\frac{1}{N}}))$ 
		(respectively $\mathcal{K}_{p}\simeq \mathbb{C}((z))$). 

		Let $R=\mathbb{C}[[z^{\frac{1}{N}}]]$ and let $\mathcal{A}ut_{N}(D_{R,\widetilde{p}})$ (respectively $\mathcal{A}ut_{N,+}(D_{R,\widetilde{p}})$) denote the set of special coordinates (respectively one-jets)
		of the point $\widetilde{p}$. This
		is an  $\operatorname{Aut}_N(\mathcal{O}_{R})$ (respectively $\operatorname{Aut}_{N,+}(\mathcal{O}_R)$)-torsor. 

		Since the $L_{0,\Gamma_{\widetilde{p}}}$-eigen values of the module $\mathcal{H}_{\lambda}$ are bounded from below. Hence it follows (from the construction of the Sugawara action) that the action of $\operatorname{Der}_{N,+}\mathcal{O}$ is locally nilpotent. 
		
		Thus by the discussion in Section 6.3 of \cite{FBz}, we get that  the following vector space is independent of the formal jets at the point $\widetilde{p}$:
		$$\widetilde{\mathcal{H}}_{\lambda,\widetilde{p}}:=\mathcal{A}ut_{N,+}(D_{R,\widetilde{p}})\times_{\operatorname{Aut}_{N,+}(\mathcal{O}_R)}\mathcal{H}_{\lambda}.$$  
		However, it turns out that the eigenvalues (\cite[Lemma 3.6]{Wakimoto}) of the {\em zero-th twisted Virasoro operators}
		$L_{0,\Gamma_{\widetilde{p}}}$ on $\mathcal{H}_{\lambda}$ are not in $\frac{1}{N}\mathbb{Z}$. In particular by Equation \eqref{eqn:viractionimp}, the eigenvalues of $z^{\frac{1}{N}}\partial_{z^{\frac{1}{N}}}$
		are non integers. We refer the reader to \cite{mukho4} for similar issue when $\Gamma$ is trivial.

		As in the untwisted case, let $\Delta_{\lambda}$ denote the eigenvalues of $L_{0,\Gamma_{\widetilde{p}}} $ on 
		the degree zero part of $\mathcal{H}_{\lambda}$. We consider the vector space $\mathcal{H}_{\lambda}\otimes \mathbb{C}{dz}^{\Delta_{\lambda}}$. 
		
		By Equation \eqref{eqn:viractionimp} of the Sugawara 
		action, the eigenvalues of $z^{\frac{1}{{N}}}\partial_{z^{\frac{1}{{N}}}}$ on $\mathcal{H}_{\lambda}$ are of the form $-N\Delta_{\lambda}+ \mathbb{Z}$, where $\Delta_{\lambda}$ is 
		a rational number (Lemma 3.6 in \cite{Wakimoto}). The Lie algebra $\operatorname{Der}_N(\mathcal{O}_R)$ acts on $dz^{\Delta_{\lambda}}$ via Lie derivatives and  in particular $z^{\frac{1}{N}}\partial_{z^{\frac{1}{N}}}$ acts with eigenvalue $N\Delta_\lambda$. 
		
		This again by Section 6.3 in \cite{FBz} implies that the action of $\operatorname{Der}_N(\mathcal{O}_R)$ on $\mathcal{H}_{\lambda}\otimes \mathbb{C}dz^{\Delta_{\lambda}}$ can be exponentiated to an action of  the group $\operatorname{Aut}_N(\mathcal{O}_R)$. 
		Thus the coordinate free 
		highest weight irreducible integrable $\widehat{\frg}_{\widetilde{p}}$-module $\mathbb{H}_{\lambda}$ (see Section \ref{sec:deftwistconf} for notation)of highest weight $\lambda$ can be redefined as 
		\begin{equation}\label{eqn:coordinatefreemodule}
		\mathcal{H}_{\lambda,\widetilde{p}}:=\mathcal{A}ut_N(D_{R,\widetilde{p}})\times_{\operatorname{Aut}_N(\mathcal{O}_R)}(\mathcal{H}_{\lambda}\otimes \mathbb{C}dz^{\Delta_{\lambda}}).
		\end{equation}
		We have the following result which summarizes the discussions in this section:
		
		\begin{theorem}
			\label{thm:atiyahalgebra}
			Let $\Delta_{\lambda_i}$ be the eigenvalues of the zeroth
			Virasoro operator $L_{0,\Gamma_{\widetilde{p}_i}}$ on the degree zero part of the highest weight $L(\widehat{\frg}, \Gamma_{p_i})$-module  
			with highest weight $\lambda_i$ via 
			the twisted Sugawara action. Then the Atiyah algebra 
			$$\frac{\ell\dim \frg}{2(h^{\vee}(\frg)+\ell)}\mathcal{A}_{\Lambda}(-\log {\Delta}_{g,n,\Gamma})+\sum_{i=1}^{n}N_i\Delta_{\lambda_i}\mathcal{A}_{{\tildeL}_i}(-\log {\Delta}_{g,n,\Gamma})$$ acts on $\mathbb{V}_{\vec{\lambda},\Gamma}(\widetilde{C},C, \widetilde{\bf{p}}, {\bf{p}})$, where ${\tildeL}_i $;s are 
			tautological line bundles corresponding to the $i$-th psi classes and $N_i$'s are the orders of the cyclic groups $\Gamma_{p_i}$. 
		\end{theorem}
		\begin{proof}First of all we observe that the $\GModtilde$ is a substack of the total stack of the $\mathbb{G}_m^{n}$-torsor on $\GMod$ given by the line bundles ${\tildeL}_1,\dots,{\tildeL}_n$. 
			Hence the pullback of each ${\tildeL}_i$ to $\GModtilde$ is trivial. Let $\tilde{\pi}:\GModtilde\rightarrow \GMod$ be the forgetful map and consider  the Atiyah algebra
			\begin{equation}\label{eqn:almostrightanswer}
			\mathcal{A}:=\frac{\ell\dim{\frg}}{2(\ell+h^{\vee}(\frg))}\tilde{\pi}^{*}\mathcal{A}_{\Lambda}(-\log\tilde{\pi}^{-1}\Delta_{g,n,\Gamma})+\sum_{i=1}^n N_i\Delta_{\lambda_i}\tilde{\pi}^{*}\mathcal{A}_{{\tildeL}_i}(-\log\tilde{\pi}^{-1}\Delta_{g,n,\Gamma}),
			\end{equation} 
			Then by Proposition \ref{prop:virasoroandatiyahalgebra}, Corollary \ref{cor:Atiyahalg} and the coordinate free construction (see \eqref{eqn:coordinatefreemodule}) tell us that the log-Atiyah algebra given by Equation \eqref{eqn:almostrightanswer} acts on $\mathbb{V}:=\tilde{\pi}^{*}(\mathbb{V}^{\dagger}_{\vec{\lambda},\Gamma}(\widetilde{C},C,\widetilde{\bf{p}}, {\bf p})).$ i.e a Lie algebra homomorphism $\nabla: \mathcal{A}\rightarrow \mathcal{A}_{\mathbb{V}}$ 
			which is also equivariant under the $\mathbb{G}_m^n$-torsor. Hence the result follows.
			
		\end{proof}
		
		\paragraph*{\bf Part III: Crossed modular functors} 
		\addcontentsline{toc}{section}{\bf Part III: Crossed modular functors}
		In this part of the paper we develop the formalism of a $\Gamma$-crossed modular functor that will allow us to use these bundles of twisted conformal blocks to construct $\Gamma$-crossed modular fusion categories. 
		
		We state and prove Theorem \ref{thm:modfunmodcat} which allows us to construct $\Gamma$-crossed weakly ribbon categories from a $\Gamma$-crossed modular functor in genus 0. We then derive some important consequences of this, for example the twisted Verlinde formula (Corollary \ref{cor:verlindeforcrossedmodfun}) in the set-up of $\Gamma$-crossed modular functors. 
		The rest of the paper is dedicated to the proof that twisted conformal blocks defines a complex analytic $\Gamma$-crossed modular functor (Theorem \ref{conj:conformalblockismodular}). We also complete the proof of Theorem \ref{conj:main1}. 

		\section{$\Gamma$-crossed modular functors} \label{sec:crossedmodular}
		Let $\Gamma$ be a finite group. In this section, we define the notion of a complex analytic $\Gamma$-crossed modular functor. Topological $\Gamma$-crossed modular functors have been defined and studied in \cite{KP:08,Prince}. We prove  that the complex analytic and topological notion are both equivalent to the notion of a weakly rigid $\Gamma$-crossed modular category. 
		In this section, all abelian categories that we consider are supposed to be finite semisimple $\CBbb$-linear, even though we may not mention this explicitly and all additive functors considered are supposed to be $\CBbb$-linear.
		\subsection{$\Gamma$-crossed abelian categories}\label{sec:crossedabelian}
		Let $\Ccal$ be a finite semisimple $\CBbb$-linear abelian category equipped with a linear action of a finite group $\Gamma$. For an object $M\in \Ccal$ let $\wtilde{\Gamma}_M:=\{(g,\psi)|g\in \Gamma, \psi:g(M)\xoto{\cong} M\}$. We can define a group structure on $\wtilde{\Gamma}_M$ in the evident way and we obtain a central extension
		\begin{equation}\label{eq:centralextstabilier}
		1\to \Aut(M)\to \wtilde{\Gamma}_M\to \Gamma_M\to 1,
		\end{equation}
		where $\Gamma_M\leq \Gamma$ is the stabilizer of the isomorphism class of the object $M$. 
		\begin{remark}\label{rk:invobj}
			We say that an object $M\in \Ccal$ is $\Gamma$-invariant if $\Gamma_M=\Gamma$ and if the central extension  (\ref{eq:centralextstabilier}) is split. Any $\Gamma$-invariant object in $\Ccal$ can be naturally lifted to an object of the $\Gamma$-equivariantization $\Ccal^\Gamma$, namely the object in the equivariantization corresponding to the $\Gamma$-invariant object $M$ and the trivial representation of $\Gamma$.
		\end{remark}
		
		Now for each $n\in \ZBbb_{\geq 0}$, the tensor power $\Ccal^{\boxtimes n}$ is also a finite semisimple $\CBbb$-linear abelian category equipped with an action of the wreath product $S_n\ltimes \Gamma^n$, where $S_n$ is the symmetric group on $n$-letters acting on $\Gamma^n$ by permutation of factors. Note that we have the diagonal subgroup $\Delta\Gamma\subset \Gamma^n$ or in other words, the fixed point subgroup in $\Gamma^n$ for the $S_n$-action. The direct product $S_n\times \Delta\Gamma$ is a subgroup of the wreath product $S_n\ltimes \Gamma^n$ and hence it acts linearly on the abelian category $\Ccal^{\boxtimes n}$.
		\begin{remark}\label{rk:syminvobj}
			By a symmetric $\Gamma$-invariant object in $\Ccal^{\boxtimes n}$, we simply mean an $S_n\times \Delta\Gamma$-invariant object of $\Ccal^{\boxtimes n}$.
		\end{remark}
		
		\begin{definition}
			An involutive duality on a finite semisimple $\CBbb$-linear category is a $\CBbb$-linear equivalence $(\cdot)^*:\Ccal\xoto{\cong}\Ccal^{\operatorname{op}}$ along with a natural isomorphism between the functors $\id_\Ccal\xoto{\cong}(\cdot)^{**}$.
		\end{definition}
		
		\begin{definition}\label{def:crossedabeliancategory}
			By a  finite semisimple $\CBbb$-linear $\Gamma$-crossed abelian category we mean a finite semisimple $\CBbb$-linear abelian category $\Ccal$ along with 
			\begin{itemize}
				\item a $\Gamma$-grading of abelian categories $\Ccal=\bigoplus\limits_{m\in\Gamma}{\Ccal_m}$, 
				\item an involutive duality $(\cdot)^*:\Ccal\to \Ccal^{\op}$, which takes each component $\Ccal_m$ to $\Ccal^{\op}_{m^{-1}}$, 
				\item a linear action of $\Gamma$ on $\Ccal$ such that $\gamma(\Ccal_m)\subset \Ccal_{\gamma m \gamma^{-1}}$ for all $\gamma, m \in \Gamma$, which is compatible with the involutive duality and
				\item a simple $\Gamma$-invariant object (see Remark \ref{rk:invobj}) $\unit\in \Ccal_1$ which is self-dual, $\unit\cong\unit^*$.  
			\end{itemize}
		\end{definition}

		\begin{remark}
			If $\Ccal$ is a $\Gamma$-crossed abelian category as in Definition \ref{def:crossedabeliancategory}, then we have a $\Gamma^n$-grading of $\Ccal^{\boxtimes n}$ 
			$$\Ccal^{\boxtimes n}=\bigoplus\limits_{\mathbf{m}\in \Gamma^n}\Ccal^{\boxtimes n}_{\mathbf{m}}, \mbox{ where } \Ccal^{\boxtimes n}_{\mathbf{m}}:=\Ccal_{m_1}\boxtimes\cdots \boxtimes \Ccal_{m_n} \subset \Ccal^{\boxtimes n}\mbox{ for } \mathbf{m}=(m_1,\cdots,m_n).$$
		\end{remark}

		\begin{definition}
			For each $m\in \Gamma$, let $P_m$ denote the set of simple objects of $\Ccal_m$. Define $R_m:=\bigoplus\limits_{M\in P_m}M\boxtimes M^*\in \Ccal^{\boxtimes 2}_{m,m^{-1}}$ and define $R=\bigoplus\limits_{m\in \Gamma} R_m\in \Ccal^{\boxtimes 2}$. Note that $R$ is a symmetric $\Gamma$-invariant self dual object of $\Ccal^{\boxtimes 2}$.
		\end{definition}
		
		\begin{remark}
			By a slight abuse of notation, we will often denote a $\Gamma$-crossed abelian category by the triple $(\Ccal,\unit,R)$, but it should be remembered that there are other structures present which are not reflected in this notation.
			
		\end{remark}
		
		\begin{remark}
			A grading $\Ccal=\bigoplus\limits_{m\in\Gamma}{\Ccal_m}$ is said to be faithful if each component $\Ccal_m$ is non-zero. We will always assume that our $\Gamma$-crossed abelian categories are faithfully graded.
		\end{remark}

		\subsection{Motivation for the definition of crossed modular functors}\label{sec:motivationcomplan}
		Let $(\Ccal,\unit,R)$ be a finite semisimple faithfully graded $\Gamma$-crossed abelian category. Let us describe the motivation behind the notion of a $\Ccal$-extended $\Gamma$-crossed modular functor. The main goal is to define a tensor product structure on $\Ccal$ which equips it with the structure of a braided $\Gamma$-crossed weakly fusion category  (see Definitions \ref{d:braidedgammacrossed} and \ref{def:weaklyfusion}) with unit $\unit$ such that the $\Gamma$-action, grading, and weak duality agree with the structures that are already present on $\Ccal$.

		The idea is to describe the desired tensor product of objects $M_1,\cdots,M_n\in \Ccal$ by describing instead the $\Hom$-spaces $$\Hom(M,M_1\otimes\cdots\otimes M_n)=\Hom(\unit,M_1\otimes\cdots\otimes M_n\otimes M^*)$$ and the relations that these spaces need to satisfy in order to define the desired tensor product structure.
		
		Let us consider objects $M_i\in \Ccal_{m_i}\subset \Ccal$ for $1\leq i\leq n$. Consider this $n$-tuple of objects as an object $$\mathbf{M}:=M_1\boxtimes\cdots\boxtimes M_n\in \Ccal_{m_1}\boxtimes\cdots \boxtimes \Ccal_{m_n}=\Ccal^{\boxtimes n}_{\mathbf{m}} \subset \Ccal^{\boxtimes n},\mbox{ where } \mathbf{m}=(m_1,\cdots,m_n).$$
		
		Let us also consider an $n$-marked admissible $\Gamma$-cover $(\tildeC\to \PBbb^1,\wtilde{\mathbf{p}},\mathbf{p},\wtilde{\mathbf{v}},\mathbf{v})$ with ramification data given by $\mathbf{m}\in \Gamma^n$ as in Definition \ref{def:nmarkedcovers}. 
		
		Informally, a $\Ccal$-extended $\Gamma$-crossed modular functor $\Vcal$ (in genus 0) is a gadget which takes in the data of an object $\mathbf{M}\in \Ccal_{\mathbf{m}}^{\boxtimes n}$ and an $n$-marked  $\Gamma$-cover $(\tildeC\to \PBbb^1,\wtilde{\mathbf{p}},\mathbf{p},\wtilde{\mathbf{v}},\mathbf{v})$, and spits out a finite dimensional vector space $\Vcal_{\mathbf{M}}(\tildeC\to \PBbb^1,\wtilde{\mathbf{p}},\mathbf{p},\wtilde{\mathbf{v}},\mathbf{v})$. These finite dimensional vector spaces are supposed to be identified with the $\Hom$-spaces that we need to define the desired tensor product structure on $\Ccal$. A modular functor $\Vcal$ in arbitrary genus can accept as an input any $n$-marked admissible $\Gamma$-cover $(\tildeC\to C,\wtilde{\mathbf{p}},\mathbf{p},\wtilde{\mathbf{v}},\mathbf{v})$ without any restriction on the genus of $C$.
		
		This assignment needs to satisfy certain axioms in order to get the desired properties for the tensor product. These axioms are stated precisely in Section \ref{def:cextenmodularfunctor}. Here we give the informal idea.
		
		Firstly, as the cover $(\tildeC\to C,\wtilde{\mathbf{p}},\mathbf{p},\wtilde{\mathbf{v}},\mathbf{v})$ varies in its moduli space $\tildeM^\Gamma_{g,n}(\mathbf{m})$ (see Definition \ref{def:nmarkedcovers}) the corresponding vector spaces should define a  vector bundle with a flat projective connection on this moduli space. Secondly, these moduli spaces $\tildeM^\Gamma_{g,n}(\mathbf{m})$ and their closures $\tildebarM^\Gamma_{g,n}(\mathbf{m})$ have various gluing maps, as well as forgetting of marked-points maps between them. The vector bundles that are given by the $\Gamma$-crossed modular functor $\Vcal$ need to be compatible with all such maps. This is what is usually called `factorization rules' and `propagation of vacua'. In addition we also need certain other functoriality conditions.
		
		In order to package all the axioms we need in a compact way, we use the category  $\Xcal^\Gamma$ of stable $\Gamma$-graphs that is described in Appendix \ref{ap:groupmarkedgraphs}. This category keeps track of the combinatorial structure underlying the gluing and forgetting marked points maps between the moduli stacks of $n$-marked admissible $\Gamma$-covers. 
		We refer to Appendix \ref{ap:groupmarkedgraphs} (see also \cite{GK:98}, \cite[\S3.3]{P:13}) for the definitions and examples. 
		
		Informally we think of a graph as a `usual graph' with vertices and edges, but we allow some legs (which we think of as one half of an edge) attached to vertices. We think of an edge as being made up of two half-edges joined together. Moreover each vertex of our graph is assigned a weight, that will be related to the genus of the base curve in a $\Gamma$-cover $\tildeC\to C$. It is then possible to define the genus and the notion of stability of a weighted-legged-graph. In addition, a stable $\Gamma$-graph has the following data:
		\begin{itemize}
			\item An element $\mathbf{m}(h)\in \Gamma$ attached to each half-edge including the legs. (This corresponds to the monodromy data of the marked cover.)
			\item An element $\mathbf{b}(h)$ (called edge conjugation data) attached to the half-edges which are part of an edge $\{h,h'\}$ such that ${ }^{\mathbf{b}(h)}\mathbf{m}(h)\cdot\mathbf{m}(h')=1,\mbox{ and} \ \mathbf{b}(h)\cdot\mathbf{b}(h')=1$.
		\end{itemize}

		\subsection{Factorization and propagation of vacua functors}
		Let $(\Ccal,\unit,R)$ be a faithfully graded $\Gamma$-crossed abelian category. We refer to the informal discussion in Section \ref{sec:motivationcomplan} and Appendix \ref{ap:groupmarkedgraphs} for the definition of the category $\Xcal^\Gamma$ of stable $\Gamma$-graphs. Let $(X,\mathbf{m},\mathbf{b})\in \Xcal^\Gamma$ be a stable $\Gamma$-graph (see Definition \ref{d:XGammagA}). With such a graph we associate the abelian category $\Ccal^{\boxtimes H(X)}_{\mathbf{m}}$, where $H(X)$ denotes the set of half-edges of the graph. 
		
		Now consider a morphism (see Definition \ref{d:XGammagA}) $(X,\mathbf{m}_X,\mathbf{b}_X)\xoto{(f,\pmb{\gamma})} (Y,\mathbf{m}_Y,\mathbf{b}_Y)$ in $\Xcal^\Gamma$ which only involves contracting edges. We have the associated functor 
		\begin{equation}\label{eq:crossedabeliangraphmaps}
		\Rcal_{f,\pmb\gamma}:\Ccal^{\boxtimes H(Y)}_{\mathbf{m}_Y}\rar{}\Ccal^{\boxtimes H(X)}_{\mathbf{m}_X} \mbox{ defined as} 
		\end{equation}
		$$\left(\bigboxtimes\limits_{h\in H(Y)}M_{h}\right)\mapsto \left(\bigboxtimes\limits_{f^*h\in f^*H(Y)}\pmb\gamma(h)^{-1}(M_h)\right)\boxtimes\left(\bigboxtimes\limits_{\{h_1,h_2\}\in E(X)\setminus f^*E(Y)}(1,\mathbf{b}_x(h_1))\cdot R_{\mathbf{m}_X({h_1})}\right).$$
		Here we are just inserting the suitable translate of a graded component (to be more precise, the object $(1,\mathbf{b}_x(h_1))\cdot R_{\mathbf{m}_X({h_1})}\in \Ccal_{\mathbf{m}_X(h_1)}\boxtimes \Ccal_{\mathbf{m}_X(h_2)}$) of the symmetric $\Gamma$-invariant object $R\in \Ccal^{\boxtimes 2}$ at the edges of $X$ which have been contracted by $f:X\to Y$. The reader can have a look at Example \ref{ex:contractingloop} below where a simple prototypical example is considered.
		
		Now suppose that the morphism $(f,\pmb\gamma):(X,\mathbf{m}_X,\mathbf{b}_X)\xoto{} (Y,\mathbf{m}_Y,\mathbf{b}_Y)$ only deletes some 1-marked legs. We define the functor 
		\begin{equation}
		\Rcal_{f,\pmb\gamma}: \Ccal^{\boxtimes H(Y)}_{\mathbf{m}_Y}\rar{} \Ccal^{\boxtimes H(X)}_{\mathbf{m}_X}
		\end{equation}
		$$\left(\bigboxtimes\limits_{h\in H(Y)}M_h\right)\mapsto \left(\bigboxtimes\limits_{f^*h\in f^*H(Y)}\pmb\gamma(h)^{-1}(M_h)\right)\boxtimes \left(\bigboxtimes\limits_{H(X)\setminus f^*H(Y)}\unit\right),$$
		i.e. we insert the $\Gamma$-invariant object $\unit$ at all the 1-marked legs which have been deleted.
		
		Hence we can define the functor $\Rcal_{f,\pmb\gamma}$  for any morphism in $\Xcal^\Gamma$ (see also discussion in Appendix \ref{ap:clutchingalongmaps}). It is clear that if we have two morphisms
		$$(X,\mathbf{m}_X,\mathbf{b}_X)\xoto{f_1,\pmb{\gamma_1}}(Y,\mathbf{m}_Y,\mathbf{b}_Y)\xoto{f_2,\pmb{\gamma_2}}(Z,\mathbf{m}_Z,\mathbf{b}_Z)$$ then we have a natural isomorphism between the two functors
		\begin{equation}
		\Rcal_{(f_2,\pmb{\gamma_2})\circ (f_1,\pmb{\gamma_1})} \cong \Rcal_{f_1,\pmb{\gamma_1}}\circ \Rcal_{f_2,\pmb{\gamma_2}}:\Ccal^{\boxtimes H(Z)}_{\mathbf{m}_Z}\rar{}\Ccal^{\boxtimes H(X)}_{\mathbf{m}_X}.
		\end{equation}
		
		\subsection{$\Ccal$-extended $\Gamma$-crossed modular functors} \label{def:cextenmodularfunctor}
		Let $\Ccal$ be a faithfully graded $\Gamma$-crossed abelian category. We will now define the notion of a $\Ccal$-extended $\Gamma$-crossed modular functor extending the notion of a modular functor as defined in \cite[Ch. 6]{BK:01}. We refer the reader to Appendix \ref{ap:groupmarkedgraphs}, \ref{ap:clutchingalongmaps}, \ref{ap:specialization} for the definitions of the stacks, the categories of twisted logarithmic $D$-modules on them and the specialization functors that appear in the definition below.
		
		\begin{definition}\label{d:gammacrossedmodfun}
			Let $\Ccal$ be a (faithfully graded) $\Gamma$-crossed abelian category with $\Gamma$-invariant object $\unit\in \Ccal$ and the involutive duality $(\cdot)^*$. Let $c\in \CBbb$. A $\Ccal$-extended modular functor of (additive) central charge $c$ consists of the following data satisfying the following conditions:
			\begin{itemize}
				\item[1.] For each stable pair $(g,A)$ (i.e. $2g-2+|A|>0$) and $\mathbf{m}\in \Gamma^A$ a conformal blocks functor
				\begin{equation}
				\label{eqn:importantdefinitionoffunctors}
			\Vcal_{g,A,\mathbf{m}}:\Ccal^{\boxtimes A}_\mathbf{m}\rar{}\mathscr{D}_c\Mod(\tildebarM^\Gamma_{g,A}(\mathbf{m})).
				\end{equation}
				Once we have functors as in Equation \eqref{eqn:importantdefinitionoffunctors}, we can canonically extend them to obtain functors
				$$\Vcal_{X,\mathbf{m},\mathbf{b}}:\Ccal^{\boxtimes H(X)}_\mathbf{m}\rar{}\mathscr{D}_c\Mod(\tildebarM^\Gamma_{X,\mathbf{m},\mathbf{b}})$$
				for each $(X,\mathbf{m},\mathbf{b})\in \Xcal^\Gamma$. We will often abuse notation and denote the functors $\Vcal_{X,\mathbf{m},\mathbf{b}}$ as just $\Vcal$.
				\item[2.](Gluing Functor) For each morphism $(f,\pmb\gamma):(X,\mathbf{m}_X,\mathbf{b}_X)\to (Y,\mathbf{m}_Y,\mathbf{b}_Y)$ in $\Xcal^\Gamma$ a natural isomorphism $G_{f,\pmb\gamma}$ as below between the two functors from $\Ccal^{\boxtimes H(Y)}_{\mathbf{m}_Y}$ to $\mathscr{D}_c\Mod(\tildebarM^\Gamma_{X,\mathbf{m}_X,\mathbf{b}_X})$:
				\begin{equation}\label{eq:gluingaxiom}
				G_{f,\pmb\gamma}:\Vcal_{X,\mathbf{m}_X,\mathbf{b}_X}\circ\Rcal_{f,\pmb\gamma}\rar{\cong}\Sp_{f,\pmb\gamma}\circ \Vcal_{Y,\mathbf{m}_Y,\mathbf{b}_Y}
				\end{equation}
				compatible with compositions of morphisms in $\Xcal^\Gamma$.
				\item[3.] A normalization $\Vcal_{0,3,\pmb{1}}(\unit\boxtimes\unit\boxtimes\unit)(\PBbb^1\times\Gamma\to \PBbb^1,\wtilde{\mathbf{p}},\mathbf{p},\wtilde{\mathbf{v}})\cong \CBbb$, where the 3 marked points $\wtilde{\mathbf{p}}$ in $\PBbb^1\times \{1\}\subset \PBbb^1\times\Gamma$ are the 3 roots of unity in $\CBbb$ and with tangent vectors being the outward pointing unit vectors.
				\item[4.] (Non-degeneracy.) Let $\gamma\in \Gamma$ and let $X\in \Ccal_{\gamma}$ be a non-zero object. Then the vector bundle $\Vcal_{0,3,(\gamma,\gamma^{-1},1)}(X\boxtimes X^*\boxtimes\unit)$ is non-zero.
			\end{itemize}
			We will often abuse notation and denote all functors $\Vcal_{X,\mathbf{m},\mathbf{b}}$ simply by $\Vcal$. Also if $(\tildeD\to D,\wtilde{\mathbf{q}},\mathbf{q},\wtilde{\mathbf{w}})\in \tildebarM^\Gamma_{X,\mathbf{m},\mathbf{b}}$ and $\mathbf{M}\in \Ccal^{\boxtimes H(X)}_{\mathbf{m}}$, the vector space $\Vcal_{\mathbf{M}}(\tildeD\to D,\wtilde{\mathbf{q}},\mathbf{q},\wtilde{\mathbf{w}})$ is defined as the fiber of the twisted $\mathcal{D}$-module $\Vcal_{X,\mathbf{m},\mathbf{b}}(\mathbf{M})$ at the point $(\tildeD\to D,\wtilde{\mathbf{q}},\mathbf{q},\wtilde{\mathbf{w}})$.
		\end{definition}
		\begin{remark}
			Since morphisms in $\Xcal^\Gamma$ involve contracting edges as well as forgetting $1$-marked legs, the condition 2 in Definition \ref{d:gammacrossedmodfun} gives the factorization isomorphisms, the propagation of vacua isomorphisms as well as all the compatibilities mentioned in \cite[\S6.7]{BK:01}. When $\Gamma$ is the trivial group, it is clear that Definition \ref{d:gammacrossedmodfun} agrees with the one in {\it loc. cit.}, except that we include non-degeneracy as part of the definition. See the example below to derive `factorization rules' from the axioms in Definition \ref{d:gammacrossedmodfun}.
		\end{remark}
		
		\begin{example}\label{ex:contractingloop} (An example of factorization.) Let $(g,n)$ be a stable pair, i.e. $2g-2+n>0$, and let $a,b,m_1,\cdots,m_n\in \Gamma$. Consider the following morphism in $\Xcal^\Gamma$ obtained by contracting the loop below whose two half-edges have monodromy data $a , ba^{-1}b^{-1}$ and edge conjugation data $(b,b^{-1})$:
			\begin{equation}
			\label{eqn:twistedgluing}
			f_g:
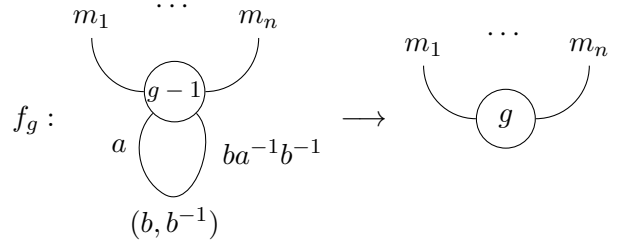
\begin{tikzpicture}[baseline={([yshift=-.5ex]current bounding box.center)}]
			\node (B1) {$m_1$};
			\node[circle, inner sep = 0.4pt, draw, below right =.6 of B1] (Cup) {$g-1$};
			\node[above right =.6 of Cup] (B2) {$m_n$};
			\node[below right =.3 of Cup] (B) {$ba^{-1}b^{-1}$};
			\node[below left =.3 of Cup] (B') {$a$};
			\node[above =.5 of Cup] (Dots) {$\cdots$};
			\node[below =.9 of Cup] (invisible) {};
			\node[below =1 of Cup] (A) {$(b,b^{-1})$};
			\draw (B1) to[out=-90,in=180] (Cup);
			\draw (B2) to[out=-90,in=0] (Cup);
			\draw (Cup) to[out=-135,in=150] (invisible) to[out=-30,in=-45] (Cup);
			\end{tikzpicture}\longrightarrow \begin{tikzpicture}[baseline={([yshift=-.5ex]current bounding box.center)}]
			\node (B1) {$m_1$};
			\node[circle, inner sep = 4.8pt, draw, below right =.6 of B1] (Cup) {$g$};
			\node[above =.5 of Cup] (Dots) {$\cdots$};
			\node[above right =.6 of Cup] (B2) {$m_n$};
			\draw (B1) to[out=-90,in=180] (Cup);
			\node[below =.7 of Cup] (invisible) {};
			\draw (B2) to[out=-90,in=0] (Cup);
			\end{tikzpicture}.
			\end{equation}
			Then we have the gluing map from Equation (\ref{eq:gluingcovers}) (see also the beginning of Appendix \ref{ap:clutchingalongmaps}):
			$$
			\xi_{f_g}: \barM{}^\Gamma_{g-1,n+2}(a,ba^{-1}b^{-1},\mathbf{m}) \rar{} \barM{}^\Gamma_{g,n}(\mathbf{m}).
			$$ The map $\xi_{f_g}$ naturally  factors through the gluing map $$\barM{}^\Gamma_{g-1,n+2}(bab^{-1},ba^{-1}b^{-1},\mathbf{m}) \rar{} \barM{}^\Gamma_{g,n}(\mathbf{m})$$ as defined in \cite[\S2.2]{JKK} via action by the element $(b,1,\mathbf{1})\in \Gamma^{n+2}$. Note that the morphism $f_g$ in Equation \eqref{eqn:twistedgluing} also factors through the corresponding graph in $\Xcal^\Gamma$.

			Recall the specialization functor from Equation (\ref{eq:specalonggraphmaps})
			$$\Sp_{f_g}:\mathscr{D}_c\Mod(\tildebarM{}^\Gamma_{g,n}(\mathbf{m})) \rar{} \mathscr{D}_c\Mod(\tildebarM{}^\Gamma_{g-1,n+2}(a,ba^{-1}b^{-1},\mathbf{m}))$$
			and the functor from Equation \ref{eq:crossedabeliangraphmaps}:
			$\Rcal_{f_g}: \Ccal^{\boxtimes n}_{\mathbf{m}}\rar{} \Ccal^{\boxtimes n+2}_{a,ba^{-1}b^{-1},\mathbf{m}},$
			defined as $$\Ccal^{\boxtimes n}_\mathbf{m}\ni\mathbf{M}\mapsto (1,b)\cdot(R_a)\boxtimes \mathbf{M}=\bigoplus\limits_{A\in P_a}A\boxtimes b(A^*)\boxtimes \mathbf{M}.$$ 
			In this case the Gluing axiom (Equation \ref{eq:gluingaxiom}) gives a natural isomorphism between two functors
			$$
			G_{f_g}:\Vcal_{g-1,n+2,(a,ba^{-1}b^{-1},\mathbf{m})}\circ\Rcal_{f_g}\rar{\cong}\Sp_{f_g}\circ \Vcal_{g,n,\mathbf{m}}$$
			i.e. for each $\mathbf{M}\in \Ccal^{\boxtimes n}_{m}$ we have an isomorphism of vector bundles with connections
			\begin{equation}\label{eq:inductionstep} 
			\bigoplus\limits_{P_a}\Vcal_{g-1,n+2,(a,ba^{-1}b^{-1},\mathbf{m})}(A\boxtimes b(A^*)\boxtimes \mathbf{M})\rar{\cong}\Sp_{f,\pmb\gamma} (\Vcal_{g,n,\mathbf{m}}(\mathbf{M})).\end{equation} We will use the isomorphism in Equation \eqref{eq:inductionstep} later to prove Theorem \ref{thm:modfunmodcatarbgenus} by induction on the genus.
		\end{example}

		\begin{example}\label{ex:braidinggluing}
			Let us consider another example of a morphism in $\Xcal^\Gamma$, namely the crossed braiding $\beta_{l_1,l_2}:(g;m_1,m_2,\cdots,m_n)\to (g;m_1m_2m_1^{-1},m_1,\cdots,m_n)$ from Example \ref{ex:crossedbraiding}. As a shorthand notation, let us set $\mathbf{m}'=(m_1m_2m_1^{-1},m_1,\cdots,m_n)$ and similarly for the other data where the first two pieces of data are swapped. This braiding isomorphism in $\Xcal^\Gamma$ gives rise to the isomorphism of the stacks
			\begin{equation}\label{eqn:bradingmodfunc}
			 \xi_{\beta_{l_1,l_2}}:\tildebarM{}^\Gamma_{g,n}(\mathbf{m})\to \tildebarM{}^\Gamma_{g,n}(\mathbf{m}'), (\tildeC\to C,\wtilde{\mathbf{p}},\wtilde{\mathbf{v}})\mapsto (\tildeC\to C;m_1(\wtilde{p}_2),\wtilde{p}_1,\wtilde{p}_3,\cdots,\wtilde{p}_n,,\wtilde{\mathbf{v}}')\end{equation}
			In this case  the specialization functor from Equation (\ref{eq:specalonggraphmaps}) is just the pullback along the isomorphism in Equation \eqref{eqn:bradingmodfunc}:
			$$\Sp_{\beta_{l_1,l_2}}=\xi^*_{\beta_{l_1,l_2}}:\mathscr{D}_c\Mod(\tildebarM{}^\Gamma_{g,n}(\mathbf{m}')) \rar{} \mathscr{D}_c\Mod(\tildebarM{}^\Gamma_{g,n}(\mathbf{m})).$$
			Also in this case the inverse of the functor from Equation \ref{eq:crossedabeliangraphmaps}:
			$\Rcal^{-1}_{\beta_{l_1,l_2}}: \Ccal^{\boxtimes n}_{\mathbf{m}}\rar{} \Ccal^{\boxtimes n}_{\mathbf{m}'},$ 
			is defined by $$\Ccal^{\boxtimes n}_\mathbf{m}\ni M_1\boxtimes\cdots\boxtimes M_n \mapsto m_1(M_2)\boxtimes M_1\boxtimes M_3\cdots\boxtimes M_n.$$ 
			Applying Gluing axiom (Equation \ref{eq:gluingaxiom}) gives a natural isomorphism between two functors
			$$
			\Vcal_{g;m_1,m_2,\cdots,m_n}\rar{\cong}\xi^*_{\beta_{l_1,l_2}}\circ \Vcal_{g;m_1m_2m_1^{-1},m_1,\cdots,m_n}\circ\Rcal_{\beta_{l_1,l_2}}^{-1}$$
			i.e. natural isomorphisms between the vector spaces
			$$\Vcal_{M_1\boxtimes\cdots\boxtimes M_n}(\tildeC\to C,\wtilde{\mathbf{p}},\wtilde{\mathbf{v}})\xoto{\cong} \Vcal_{ m_1(M_2)\boxtimes M_1\cdots\boxtimes M_n}(\tildeC\to C;m_1(\wtilde{p}_2),\wtilde{p}_1,\cdots,\wtilde{p}_n,,\wtilde{\mathbf{v}'}).$$
		\end{example}

		\noindent Next let us define a $\Ccal$-extended $\Gamma$-crossed modular functor in genus 0.
		\begin{definition}\label{d:gammacrossedmodfun0}
			Let $\Xcal^\Gamma_0\subset \Xcal^\Gamma$ denote the full subcategory formed by graphs of genus 0, i.e. those graphs which are forests and all of whose vertices have weight 0. The notion of a $\Ccal$-extended $\Gamma$-crossed modular functor in genus 0 is obtained by replacing the category $\Xcal^\Gamma$ with the full subcategory $\Xcal^\Gamma_0$ in Definition \ref{d:gammacrossedmodfun}.
		\end{definition}
		\begin{remark}
			Note that in genus 0, the Hodge line bundles on $\overline{\mathcal{M}}_{0,A}$ are trivial and hence their pull back to $\tildebarM{}^\Gamma_{0,A}$ are all trivial and hence we do not need to consider the central charge and each $\Vcal(\mathbf{M})$ is a $D$-module on $\tildebarM{}^\Gamma_{0,A}$ with log-singularities at the boundary.
		\end{remark}

		\subsection{Pivotal and ribbon structures on weakly fusion categories}\label{sec:weaklyribbon}
		Before proceeding further, let us recall the notions of pivotal and ribbon structures in the setting of monoidal r-categories and weakly fusion categories. We refer back to Section \ref{sec:monrcatweaklyfusion} where these notions were defined.  Unlike rigid categories, in monoidal r-categories, the weak duality functor $X\mapsto X^*$ is not necessarily monoidal. However, by \cite{BoDr:13} the double duality functor $(\cdot)^{**}$ is indeed monoidal and also braided in case the original category is braided. 
		\begin{definition} (cf. \cite[\S5]{BoDr:13}.)
			A pivotal structure on a monoidal r-category $\Ccal$ is a monoidal isomorphism between the monoidal functors $\id_\Ccal$ and $(\cdot)^{**}$ on $\Ccal$.
		\end{definition}
		In \cite{BoDr:13}, a ribbon r-category is defined to be a braided monoidal r-category equipped with a monoidal isomorphism between the identity functor and the double duality functor, $X\cong X^{**}$, satisfying a certain balancing property. We refer to \cite[\S7,8]{BoDr:13} for details.
		
		Now let $\Ccal$ be a braided $\Gamma$-crossed monoidal r-category. A pivotal structure on $\Ccal$ naturally defines a pivotal structure on the equivariantization $\Ccal^\Gamma$ which is a braided monoidal r-category (see Proposition \ref{prop:equivisbraided}).
		\begin{definition}\label{def:ribboning}
			A ribbon structure on a braided $\Gamma$-crossed monoidal r-category $\Ccal$ is a pivotal structure such that the induced pivotal structure on the braided monoidal r-category $\Ccal^\Gamma$ is ribbon in the sense of \cite{BoDr:13}.
		\end{definition}

		\begin{definition}\label{d:crossedweaklyribbon}
			(i) We define a $\Gamma$-crossed weakly ribbon fusion category, or in short a $\Gamma$-crossed weakly ribbon category, to be a braided $\Gamma$-crossed weakly fusion category equipped with a ribbon structure as in Definition \ref{def:ribboning}.\\
			(ii) By a weakly ribbon category we mean a braided weakly fusion category equipped with a ribbon structure.
		\end{definition}

		We will prove that a $\Ccal$-extended $\Gamma$-crossed modular functor in genus 0 defines the structure of a $\Gamma$-crossed weakly ribbon category on $\Ccal$.

		\subsection{The neutral modular functor}\label{sec:crossedneutral}
		Let $(\Ccal=\bigoplus\limits_{m\in \Gamma}\Ccal_m,\unit,\bigoplus\limits_{m\in \Gamma}R_m)$ be a $\Gamma$-crossed abelian category and suppose we are given a $\Ccal$-extended $\Gamma$-crossed modular functor $\Vcal$ of additive central charge $c\in \CBbb$. We will now define a $\Ccal_1$-extended modular functor $\Vcal_1$ in the sense of \cite[Ch. 6]{BK:01} with the same central charge $c$. 
		
		Let $X\in \Xcal$ be a stable graph and let $\mathbf{M}\in \Ccal_1^{\boxtimes H(X)}$. Then we want to define $\Vcal_1(\mathbf{M})\in \mathscr{D}_c\Mod (\tildebarM_{X})$. We have a functor $\Xcal\to \Xcal^\Gamma$ defined on objects by $X\mapsto (X,\mathbf1,\mathbf1_{H(X)\setminus L(X)})$ (mark all half-edges by $1\in\Gamma$) and on morphisms by $f\mapsto (f,\mathbf1)$. 
		
		We also have a morphism of stacks $i_X:\tildebarM_{X}\to \tildebarM{}^\Gamma_{X,\mathbf1,\mathbf1}$ defined by $(D,\mathbf{q},\mathbf{w})\mapsto (D\times \Gamma\to D,\mathbf{q}\times 1,\mathbf{q},\mathbf{w}\times 1)$. Since $\Vcal$ is a $\Ccal$-extended $\Gamma$-crossed modular functor, we have the twisted $\mathcal{D}$-module $\Vcal(\mathbf{M})$ on $\tildebarM{}^\Gamma_{X,\mathbf1,\mathbf1}$. 
		\begin{definition}\label{def:neutral}
			We define  $\Vcal_1(\mathbf{M})$ to be the pullback $i_X^*\Vcal(\mathbf{M})\in \mathscr{D}_c\Mod(\tildebarM_{X})$.
		\end{definition} 
		We can perform the same construction for the genus 0 case. Hence in this setting we obtain:
		\begin{proposition}
			\begin{enumerate}[(i)]
				\item  Given a $\Ccal$-extended $\Gamma$-crossed modular functor $\Vcal$ in genus 0, the functor $\Vcal_1$ as in Definition \ref{def:neutral} is a $\Ccal_1$-extended modular functor in genus 0. 
				In particular, this defines the structure of a weakly ribbon category on $\Ccal_1$ and the action of $\Gamma$ on $\Ccal_1$ respects this structure.
				\item  Given a $\Ccal$-extended $\Gamma$-crossed modular functor $\Vcal$ with central charge $c$, the functor $\Vcal_1$ (Definition \ref{def:neutral}) is a $\Ccal_1$-extended modular functor with central charge $c$.
			\end{enumerate}
		\end{proposition}
		\begin{proof}The fact that $\mathcal{V}_1$ is $\mathcal{C}_1$-extended modular functor with central charge $c$ follows from the construction of $\Vcal_1$ and the axioms satisfied by $\mathcal{V}$. 
			The part about $\mathcal{C}_1$-extended modular functor in genus $0$ giving a weakly ribbon category on $\mathcal{C}_1$ is due to Bakalov-Kirillov \cite{BK:01}.
		\end{proof}

		\subsection{Generalizations to the twisted setting}\label{ss:someconj}
		In this section, we state the main result  which relates the notions of $\Gamma$-crossed modular functors in genus 0 and $\Gamma$-crossed weakly ribbon categories. This result will be proved in Section \ref{sec:CATGamma} after showing that the notion of complex analytic $\Gamma$-crossed modular functors as defined in this section is equivalent to the topological version studied in \cite{KP:08}. We then derive some important consequences from the main theorem.
		
		\subsubsection{When the base curve is genus 0.}
		Suppose that $(\Ccal,\unit,R)$ is a finite semisimple $\Gamma$-crossed abelian category such that $\unit\in \Ccal_1$ is a simple object. In this case, we prove that the notion of a $\Ccal$-extended $\Gamma$-crossed modular functor in genus 0 is equivalent to giving $\Ccal$ the additional structure of a $\Gamma$-crossed weakly fusion ribbon category. In a weakly rigid category we have natural isomorphisms $\Hom(M,M')\cong \Hom(\unit, { }^*M\otimes M')$. 
		
		Now given a $\Ccal$-extended $\Gamma$-crossed modular functor we want to define a tensor product, i.e. given $\mathbf{m}\in \Gamma^n, \mathbf{M}=M_1\boxtimes\cdots \boxtimes M_n\in \Ccal^{\boxtimes n}_\mathbf{m}$, we want to describe the tensor product $$M_1\otimes \cdots\otimes M_n\in \Ccal_{m_1m_2\cdots m_n}\subset \Ccal.$$
		
		Given a $\Ccal$-extended $\Gamma$-crossed modular functor, let us describe the desired morphism space $\Hom(\unit,M_1\otimes\cdots \otimes M_n)$. If the product $m_1\cdots m_n\neq 1$, then this space is 0. Hence suppose that $m_1\cdots m_n=1.$ Consider the $n$-marked curve $(\PBbb^1,\mu_n,\mu_n)$ where the $n$ marked points are the $n$-th roots of unity $\mu_n\subset \CBbb\subset \PBbb^1$ and for each marked point $\omega\in\mu_n$, the associated tangent vector is again $\omega$ considered as a tangent vector at $\omega$. 
		
		Consider $0$ as the basepoint on $\PBbb^1\setminus \mu_n$. Consider loops $\gamma_j$ in $\PBbb^1\setminus \mu_n$ starting at $0$ and going in a straight line to $e^{\frac{2\pi\sqrt{-1}j}{n}}$ and encircling the point $e^{\frac{2\pi\sqrt{-1}j}{n}}$ counterclockwise. Then we obtain a presentation of the fundamental group  
		\begin{equation}
		\label{eqn:constructcoverviamonodromy}
		\pi_1(\PBbb^1\setminus \mu_n,0)=\langle\gamma_1,\cdots,\gamma_n|\gamma_1\cdots\gamma_n=1\rangle.\end{equation}
		
		Hence if $m_1\cdots m_n=1$, then $\gamma_j\mapsto m_j$ defines a group homomorphism 
		\begin{equation}\label{eqn:monodromycoverp1}
			\phi:\pi_1(\PBbb^1\setminus \mu_n,0)\to \Gamma.
			\end{equation} Then as described in \cite[\S2.3]{JKK}, this determines an $n$-marked admissible cover $(\tildeC\to \PBbb^1,\wtilde{\mathbf{p}},\mu_n,\wtilde{\mathbf{v}})$.
		
		Then we want to define a tensor product on $\Ccal$ for which the multiplicity spaces $\Hom(\unit,M_1\otimes\cdots \otimes M_n)$ (see also Remark \ref{rk:fusioncoeff}) are fibers of the bundles $\Vcal_{\mathbf{M}}$ at the special covers described given by Equation \eqref{eqn:monodromycoverp1}. More precisely we want: 
		\begin{equation}\label{eq:homspaces}
		\Hom(\unit,M_1\otimes\cdots \otimes M_n)=\Vcal_{\mathbf{M}}(\tildeC\to \PBbb^1,\wtilde{\mathbf{p}},\mu_n,\wtilde{\mathbf{v}}).\end{equation}

		
		We state the following result which will be proved in Section \ref{sec:proofsofthms}.
		\begin{theorem}\label{thm:modfunmodcat}
			Let $(\Ccal,\unit,R)$ be a finite semisimple faithfully graded $\Gamma$-crossed abelian category with $\unit$ being a simple object. We also continue to use all the preceding notation. Then 
			the notion of a $\Ccal$-extended $\Gamma$-crossed modular functor in genus 0 is equivalent to equipping $\Ccal$ with the structure of a $\Gamma$-crossed weakly ribbon category satisfying Equation (\ref{eq:homspaces}):
			$$\Hom(\unit,M_1\otimes\cdots \otimes M_n)=\Vcal_{\mathbf{M}}(\tildeC\to \PBbb^1,\wtilde{\mathbf{p}},\mu_n,\wtilde{\mathbf{v}}).$$
		\end{theorem}
		
		\subsubsection{When the base genus is arbitrary}\label{sec:arbgenusdescription}
		Using the genus zero result (Theorem \ref{thm:modfunmodcat}), we now derive its higher genus analog. Consider a stable pair $(g,n)$ and $\mathbf{m}\in \Gamma^n$. Then  consider the stack $\tildebarM{}^\Gamma_{g,n}(\mathbf{m})$ of $n$-marked admissible $\Gamma$-covers with monodromy data $\mathbf{m}$ with genus $g$ base curves. As before consider an object $\mathbf{M}=M_1\boxtimes\cdots \boxtimes M_n\in \Ccal^{\boxtimes n}_\mathbf{m}$. 
		
		Now given a $\Ccal$-extended $\Gamma$-crossed modular functor in arbitrary genus, we firstly get the structure of a $\Gamma$-crossed weakly ribbon category on $\Ccal$ using the previous theorem. Just as in the case of genus 0 base curves, we would like to compute the dimensions of the fibers of the vector bundle $\Vcal_{\mathbf{M}}$ (which comes as a part of the data defining a $\Gamma$-crossed modular functor) on $\tildebarM{}^\Gamma_{g,n}(\mathbf{m})$ at various points in $\tildeM^\Gamma_{g,n}(\mathbf{m})$ as the dimensions of certain $\Hom$-spaces in the category $\Ccal$.
		
		\begin{remark}\label{rk:componentranksdiffer}
			Note that the stacks $\tildebarM{}^\Gamma_{g,n}(\mathbf{m})$ are disconnected in general and as we will see below, the rank of $\Vcal_{\mathbf{M}}$ may differ on the different connected components. We will describe below the ranks of $\Vcal_{\mathbf{M}}$ at all the possible components of this stack. Note that to describe the ranks over all of $\tildebarM{}^\Gamma_{g,n}(\mathbf{m})$, it is sufficient describe the ranks at $n$-marked admissible $\Gamma$-covers in the open part $\tildeM^\Gamma_{g,n}(\mathbf{m})$ where the base curve of the $\Gamma$-cover is smooth of genus $g$.
		\end{remark}
		In view of the remark consider any $(\tildeC\to C,\wtilde{\mathbf{p}},\mathbf{p},\wtilde{\mathbf{v}},\mathbf{v})\in \tildeM^\Gamma_{g,n}(\mathbf{m})$. Let us choose a basepoint $p_0\in C_{gen}=C\setminus\mathbf{p}$ and a lift $\wtilde{p}_0\in \tildeC_{gen}$. Since $C_{gen}$ is a smooth genus $g$ curve with $n$ punctures, we can choose closed curves $\alpha_1,\beta_1,\cdots,\alpha_g,\beta_g$ in $C_{gen}$ based at $p_0$ along with smooth paths $\delta_i$ (for $1\leq i\leq n$) beginning at $p_0$ and ending at $p_i$ such that:
		\begin{itemize}
			\item For each $i$, the smooth path $\delta_i$ represents the tangent vector $v_i$ in $T_{p_i}C$ that is part of the marking data on $C$ and does not pass through any marked point $p_j$ except at the endpoint.
			\item Let $\sigma_i$ denote the loops in $C_{gen}$ based at $p_0$ and which encircle $p_i$ once counterclockwise defined using the path $\delta_i$ as constructed in \cite[\S2.3]{JKK}.
			\item These $2g+n$ loops give a presentation 
			\begin{equation}
			\pi_1(C_{gen},p_0)=\langle\alpha_1,\beta_1,\cdots,\alpha_g,\beta_g,\sigma_1,\cdots,\sigma_n|[\alpha_1,\beta_1]\cdots[\alpha_g,\beta_g]\sigma_1\cdots\sigma_n=1\rangle.
			\end{equation}
		\end{itemize}
		Note that $\tildeC_{gen}\to C_{gen}$ is an \'etale $\Gamma$-cover. We can now uniquely lift each path $\delta_i$ in $C$ to a smooth path $\wtilde\delta_i$ in $\tildeC$ beginning at $\wtilde{p}_0$ and which lies in $\tildeC_{gen}$ except for the ending point. Let $\wtilde{p}'_i$ denote the endpoint of $\wtilde\delta_i$; it will be some point of $\tildeC$ lying above $p_i$. Let $\tilde{v}'_i$ denote the tangent vector in $T_{\wtilde{p}'_i}(\tildeC)$ represented by $\wtilde\delta_i$. Then there exists a unique $\gamma_i\in \Gamma$ such that $\gamma_i$ takes the original marking data $(\wtilde{p}_i, \wtilde{v}_i)$ on $\tildeC$ to the newly constructed data $(\wtilde{p}'_i, \wtilde{v}'_i)$.
		
		In other words, the choices of the paths made determine a unique $\pmb\gamma\in \Gamma^n$ such that $$\pmb\gamma\cdot(\tildeC\to C,\wtilde{\mathbf{p}},\mathbf{p},\wtilde{\mathbf{v}},\mathbf{v})=(\tildeC\to C,\wtilde{\mathbf{p}}',\mathbf{p},\wtilde{\mathbf{v}}',\mathbf{v})\in \tildeM^\Gamma_{g,n}({}^{\pmb\gamma}\mathbf{m}).$$
		
		Now note that $\pmb\gamma$ determines a morphism in the category $\Xcal^\Gamma$ from the genus $g$ $\Gamma$-corolla with $n$ legs marked by $\mathbf{m}$ to the genus $g$ $\Gamma$-corolla with $n$ legs marked by ${}^{\pmb\gamma}\mathbf{m}$. (This is precisely the diagram drawn in Example \ref{ex:groupaction} of Appendix \ref{ap:groupmarkedgraphs}.) Hence by condition (2.) from Definition \ref{d:gammacrossedmodfun} of $\Gamma$-crossed modular functor
		\begin{equation}\label{eq:dimofspecialcovers}
		\dim \Vcal_{\mathbf{M}}(\tildeC\to C,\wtilde{\mathbf{p}},\mathbf{p},\wtilde{\mathbf{v}},\mathbf{v})=\dim \Vcal_{{}^{\pmb\gamma}\mathbf{M}}(\tildeC\to C,\wtilde{\mathbf{p}}',\mathbf{p},\wtilde{\mathbf{v}}',\mathbf{v}).
		\end{equation}
		\begin{remark}\label{rk:specialcovers}
			The equality (\ref{eq:dimofspecialcovers}) shows that to compute the rank of any line bundle $\Vcal_{\mathbf{M}}$ at any $(\tildeC\to C,\wtilde{\mathbf{p}},\mathbf{p},\wtilde{\mathbf{v}},\mathbf{v})\in \tildeM^\Gamma_{g,n}(\mathbf{m})$, it is enough to only consider the special covers of the type $(\tildeC\to C,\wtilde{\mathbf{p}}',\mathbf{p},\wtilde{\mathbf{v}}',\mathbf{v})$ that we have constructed. Hence from now on, we will restrict our attention to such special covers, and from now on, we assume that $(\tildeC\to C,\wtilde{\mathbf{p}},\wtilde{\mathbf{v}})$ is one such special covering.
		\end{remark}
		The choice of lift $\wtilde{p}_0$ determines a group homomorphism $\chi_{\wtilde{p}_0}:\pi_1(C_{gen},p_0)\to \Gamma$. Let $\chi_{\wtilde{p}_0}(\alpha_j)=a_j, \chi_{\wtilde{p}_0}(\beta_j)=b_j\in \Gamma$. Hence we must have the relationship:
		$$[a_1,b_1]\cdots[a_g,b_g]\cdot m_1\cdots m_n=1.$$
		
		Recall from Equation (\ref{eq:defomegaab}) that in any $\Gamma$-crossed weakly fusion category, given $a,b\in \Gamma$, we can define the objects
		$\Omega_{a,b}:=\bigoplus\limits_{A\in P_a}A\otimes b(A^*)\in \Ccal_{[a,b]}.$
		We can now state the following result which relates the ranks of the vector bundles defining a modular functor at special $\Gamma$-covers to certain multiplicity spaces:
		\begin{theorem}\label{thm:modfunmodcatarbgenus}
			Let us continue to use the set-up and all the notation from the discussion following Remark \ref{rk:specialcovers}. Suppose that we have a $\Ccal$-extended $\Gamma$-crossed modular functor (in arbitrary genus) of additive central charge $c$. Then the structure of $\Gamma$-crossed weakly ribbon category  on $\Ccal$ is such that we have the following relationship between the vector bundles coming from the data of the $\Gamma$-crossed modular functor $\Vcal$ and certain multiplicity spaces in $\Ccal$:
			\begin{equation}\label{eq:highergenushomspaces}
			\dim \Vcal_{\mathbf{M}}(\tildeC\to C,\wtilde{\mathbf{p}},\wtilde{\mathbf{v}})=\dim\Hom(\unit,\Omega_{a_1,b_1}\otimes\cdots\otimes\Omega_{a_g,b_g}\otimes M_1\otimes\cdots\otimes M_n).
			\end{equation}
		\end{theorem}
		\begin{proof}
			We will use induction on the genus $g$. Theorem \ref{thm:modfunmodcat} (Equation \ref{eq:homspaces}) covers the genus 0 case. Let us now assume the result up to genus $g-1$. To complete the proof we use (Equation \ref{eq:inductionstep} where we take $(a,b)$ to be our current $(a_g,b_g)$). Recall first of all that the specialization $\Sp_{f_g}$ preserves the ranks of vector bundles. Hence from Equation (\ref{eq:inductionstep}) we conclude that 
			$$\dim \Vcal_{\mathbf{M}}(\tildeC\to C,\wtilde{\mathbf{p}},\wtilde{\mathbf{v}})=\bigoplus\limits_{A\in P_{a_g}}\dim\Vcal_{A\boxtimes b_g(A^*)\boxtimes\mathbf{M}}(\tildeD\to D,(\wtilde{q}',\wtilde{q}'',\wtilde{\mathbf{p}}'),(\wtilde{w}',\wtilde{w}'',\wtilde{\mathbf{v}}'))$$
			where the special $\Gamma$-covering $(\tildeD\to D,(\wtilde{q}',\wtilde{q}'',\wtilde{\mathbf{p}}'),(\wtilde{w}',\wtilde{w}'',\wtilde{\mathbf{v}}'))$ with $D$ being smooth of genus $g-1$ is a normalization of a degeneration of $(\tildeC\to C,\wtilde{\mathbf{p}},\wtilde{\mathbf{v}})$ to the boundary of $\tildebarM{}^\Gamma_{g,n}(\mathbf{m})$ corresponding to the following morphism of stable $\Gamma$-graphs from Example \ref{ex:contractingloop}:
			$$f_g:\begin{tikzpicture}[baseline={([yshift=-.5ex]current bounding box.center)}]
			\node (B1) {$m_1$};
			\node[circle, inner sep = 0.4pt, draw, below right =.6 of B1] (Cup) {$g-1$};
			\node[above right =.6 of Cup] (B2) {$m_n$};
			\node[below right =.3 of Cup] (B) {$b_ga_g^{-1}b_g^{-1}$};
			\node[below left =.3 of Cup] (B') {$a_g$};
			\node[above =.5 of Cup] (Dots) {$\cdots$};
			\node[below =.9 of Cup] (invisible) {};
			\node[below =1 of Cup] (A) {$(b_g,b_g^{-1})$};
			\draw (B1) to[out=-90,in=180] (Cup);
			\draw (B2) to[out=-90,in=0] (Cup);
			\draw (Cup) to[out=-135,in=150] (invisible) to[out=-30,in=-45] (Cup);
			\end{tikzpicture}\longrightarrow \begin{tikzpicture}[baseline={([yshift=-.5ex]current bounding box.center)}]
			\node (B1) {$m_1$};
			\node[circle, inner sep = 4.8pt, draw, below right =.6 of B1] (Cup) {$g$};
			\node[above =.5 of Cup] (Dots) {$\cdots$};
			\node[above right =.6 of Cup] (B2) {$m_n$};
			\draw (B1) to[out=-90,in=180] (Cup);
			\node[below =.7 of Cup] (invisible) {};
			\draw (B2) to[out=-90,in=0] (Cup);
			\end{tikzpicture}.$$
			It is constructed as follows: Let us consider the loop $\alpha_g$ based at $p_0$ as being composed of a path $\delta$ from $p_0$ to a point $q\in C$ and then followed by a smooth simple closed curve $\alpha'_g$ based at $q$ and following $\delta^{-1}$ from $q$ back to $p$.  
			
			We first degenerate $(\tildeC\to C,\wtilde{\mathbf{p}},\wtilde{\mathbf{v}})$ to a nodal $(\tildeC'\to C',\wtilde{\mathbf{p}'},\wtilde{\mathbf{v}}')$ where the base $C'$ is a nodal curve obtained by degenerating the loop $\alpha'_g$ on $C$ to a node $q'$ on $C'$. We then get a path $\delta'$ on $C'$ from $p'_0$ to $q'$.  Lift this path to $\tildeC'$ from $\wtilde{p}'_0$ to the node $\wtilde{q}'\in \tildeC'$. Now normalize $C'$ at $q'$ to get a cover $(\tildeD\to D,(\wtilde{q}',\wtilde{z}'',\wtilde{\mathbf{p}}'),(\wtilde{w}',\wtilde{v}'',\wtilde{\mathbf{v}}'))$. 
			
			The monodromy data at the newly created marked points $\wtilde{q}',\wtilde{z}''$ will now be $(a_g,a_{g}^{-1})$. However this will not be a special cover, and we act on $(\wtilde{z}'',\wtilde{v}'')$ by $b_g\in \Gamma$ to obtain the desired special cover $(\tildeD\to D,(\wtilde{q}',\wtilde{q}'',\wtilde{\mathbf{p}}'),(\wtilde{w}',\wtilde{w}'',\wtilde{\mathbf{v}}'))$.

			By the induction hypothesis for the genus $g-1$ base curve $D$ we get
			\begin{align*}
			&\dim\Vcal_{A\boxtimes b_g(A^*)\boxtimes\mathbf{M}}(\tildeD\to D,(\wtilde{q}',\wtilde{q}'',\wtilde{\mathbf{p}}'),(\wtilde{w}',\wtilde{w}'',\wtilde{\mathbf{v}}'))\\
			& \quad =\dim\Hom(\unit,\Omega_{a_1,b_1}\otimes\cdots\otimes\Omega_{a_{g-1},b_{g-1}}\otimes A\otimes b_g(A^*)\otimes M_1\otimes\cdots\otimes M_n).
			\end{align*}
			for each simple object $A\in P_{a_g}$. Since by definition $\Omega_{a_g,b_g}=\bigoplus\limits_{A\in P_{a_g}}A\otimes b_g(A^*),$ we complete the proof by taking the direct sum of the previous equality over all $A\in P_{a_g}$.
		\end{proof}
		
		\noindent We now state some direct corollaries.
		\begin{corollary}\label{cor:crossedmodfunmodcat}
			Along with the assumptions of Theorem \ref{thm:modfunmodcatarbgenus}, also suppose that the corresponding neutral weakly ribbon category $\Ccal_1\subset \Ccal$ is rigid. Then in fact $\Ccal$ is a $\Gamma$-crossed modular fusion category of multiplicative central charge $e^{\sqrt{-1}\pi c}$.
		\end{corollary}
		\begin{proof}
			By Theorem \ref{thm:modfunmodcatarbgenus} the $\Gamma$-crossed modular functor $\Vcal$ defines the structure of a $\Gamma$-crossed weakly ribbon category on $\Ccal$, in particular $\Ccal$ is a $\Gamma$-graded weakly fusion category. We have assumed that the component $\Ccal_1$ is rigid. Hence by Corollary \ref{cor:rigidity of graded categories} $\Ccal$ must be rigid and hence a braided $\Gamma$-crossed fusion category. 
			
			Moreover $\Ccal_1$ corresponds to the neutral $\Ccal_1$-extended (untwisted) modular functor $\Vcal_1$ in arbitrary genus of additive central charge $c$ as described in Section \ref{sec:crossedneutral}. Since we have assumed $\Ccal_1$ is rigid, by \cite[\S6.7]{BK:01} we conclude $\Ccal_1$ is a modular fusion category of multiplicative central charge $e^{\sqrt{-1}\pi c}$. Moreover by assumption $\Ccal$ is faithfully graded. Hence we see that $\Ccal$ is a faithfully graded braided $\Gamma$-crossed fusion category equipped with a ribbon (i.e. spherical structure) and that $\Ccal_1$ is a modular category of multiplicative central charge $e^{\sqrt{-1}\pi c}$. Hence by Definition \ref{def:crossedmodcat} $\Ccal$ is a $\Gamma$-crossed modular category of multiplicative central charge $e^{\sqrt{-1}\pi c}$ as desired.
		\end{proof}
		
		Let $\Gamma^\circ\leq \Gamma$ be the image of the holonomy homomorphism $$\chi_{\wtilde{p}_0}:\pi_1(C_{gen},p_0)\onto\Gamma^\circ\subset \Gamma.$$ Recall that by Equation (\ref{eq:dimofspecialcovers}), to describe the ranks of all the bundles $\Vcal_{\mathbf{M}}$ at arbitrary $n$-marked admissible $\Gamma$-covers, it is enough to compute the ranks $\dim \Vcal_{\mathbf{M}}(\tildeC\to C,\wtilde{\mathbf{p}},\mathbf{p},\wtilde{\mathbf{v}},\mathbf{v})$ assuming that the cover is special. The following corollary does this:
		\begin{corollary}\label{cor:verlindeforcrossedmodfun}
			(Verlinde formula for \emph{rigid} $\Gamma$-crossed modular functors.) We use all the notation described in this section. Let $\Vcal$ be a $\Ccal$-extended $\Gamma$-crossed modular functor (in arbitrary genus) of additive central charge $c$. This equips $\Ccal$ with the structure of a $\Gamma$-crossed weakly ribbon category. Suppose in addition that the neutral component $\Ccal_1\subset \Ccal$ is rigid, or equivalently, that $\Ccal$ is a $\Gamma$-crossed modular fusion category. Then the ranks of the vector bundles $\Vcal_{\mathbf{M}}$ at the special $\Gamma$-covers are described by the following Verlinde formula:
			\begin{equation}\label{eq:verlindeforcrossedmodfun}
			\dim \Vcal_{\mathbf{M}}(\tildeC\to C,\wtilde{\mathbf{p}},\wtilde{\mathbf{v}})=\sum\limits_{D\in P_1^{\Gamma^\circ}}\frac{\left(\frac{1}{S_{\unit,D}}\right)^{n+2g-2}\cdot S^{m_1}_{M_1,D}\cdots S^{m_n}_{M_n,D}}{\varphi_D(a_1,b_1,a_1^{-1},b_1^{-1},\cdots,m_1,\cdots,m_n)}
			\end{equation}
			where for $m\in \Gamma$, $S^m$ are the categorical $m$-crossed S-matrices of the $\Gamma$-crossed modular fusion category $\Ccal$ and where the denominator is given by the cocycles described in Section \ref{sec:twistedchars}, Remark \ref{rk:ncocycle}.
		\end{corollary}
		\begin{proof}
			By Corollary \ref{cor:crossedmodfunmodcat} and the assumption that $\Ccal_1$ is rigid, the $\Gamma$-crossed modular functor $\Vcal$ induces the structure of a $\Gamma$-crossed modular fusion category on $\Ccal$. Hence we have available to us the categorical arbitrary genus twisted Verlinde formula proved in Corollary \ref{cor:highergenus}. Now Theorem \ref{thm:modfunmodcatarbgenus} implies that
			$$\dim \Vcal_{\mathbf{M}}(\tildeC\to C,\wtilde{\mathbf{p}},\wtilde{\mathbf{v}})=\dim\Hom(\unit,\Omega_{a_1,b_1}\otimes\cdots\otimes\Omega_{a_g,b_g}\otimes M_1\otimes\cdots\otimes M_n).$$
			By Corollary \ref{cor:highergenus} we have:
			\begin{align}
			\dim\Hom(\unit,\Omega_{a_1,b_1}\otimes\cdots\otimes\Omega_{a_g,b_g}\otimes M_1\otimes\cdots\otimes M_n)\\=\sum\limits_{D\in P_1^{\Gamma^\circ}}\frac{\left(\frac{1}{S_{\unit,D}}\right)^{n+2g-2}\cdot S^{m_1}_{M_1,D}\cdots S^{m_n}_{M_n,D}}{\varphi_D(a_1,b_1,a_1^{-1},b_1^{-1},\cdots,m_1,\cdots,m_n)}.
			\end{align}
			This completes the proof.
		\end{proof}

		\section{Complex analytic and topological $\Gamma$-crossed modular functors}\label{sec:CATGamma}
		In this section, we compare the complex analytic notion of $\Gamma$-crossed modular functors which we defined in the previous section with the topological version of $\Gamma$-crossed functors studied by Kirillov and Prince in \cite{KP:08,Prince}. Let $(\Ccal,\unit,R)$ be a finite semisimple $\Gamma$-crossed abelian category. It is proved in \cite{KP:08} that the structure of a topological $\Ccal$-extended $\Gamma$-crossed modular functor in genus zero is equivalent to the structure of a $\Gamma$-crossed weakly ribbon category on $\Ccal$. We will now prove that the notion of a $\Ccal$-extended complex analytic $\Gamma$-crossed modular functor in genus zero is equivalent to the topological notion. This will complete the proof of Theorem \ref{thm:modfunmodcat}.

		\subsection{Motivation for topological $\Gamma$-crossed modular functors}\label{sec:topmotivation} As in the previous section, let $\Ccal$ be a finite semisimple faithfully graded $\CBbb$-linear $\Gamma$-crossed abelian category. In order to define the structure of a $\Gamma$-crossed weakly ribbon category on $\Ccal$, we considered marked admissible covers  $(\tildeC\to C,\wtilde{\mathbf{p}},\wtilde{\mathbf{v}})\in \tildeM^\Gamma_{g,n}(\mathbf{m})$ in Section \ref{sec:crossedmodular}.  For the topological version, we consider the topological analogs of such covers. 
		
		Namely, given a $\Gamma$-cover of smooth projective curves $(\tildeC\to C,\wtilde{\mathbf{p}},\wtilde{\mathbf{v}})$ we may consider it as a branched covering of smooth compact oriented surfaces. Now we also have the marked points $\wtilde{\mathbf{p}}$ and the marked tangent vectors $\wtilde{\mathbf{v}}$ on the top curve/surface $\tildeC$ and their images $\mathbf{p}$, $\mathbf{v}$ on the base $C$. We then delete tiny disks around the points $\mathbf{p}$ to get an oriented surface $\Sigma$ with boundary. Let $\wtilde\Sigma\subset \tildeC$ be the preimage of $\Sigma$. Note that since $\Sigma\subset  C_{gen}$ and $\wtilde\Sigma\subset \tildeC_{gen}$, the projection $\wtilde\Sigma\to \Sigma$ is a principal left $\Gamma$-bundle. We also have the marked tangent vectors $\wtilde{\mathbf{v}}$. We start from a marked point $\wtilde{p}_i\in \tildeC$ and follow the marked tangent vectors $\wtilde{v}_i$ till we hit the corresponding boundary circle in $\wtilde\Sigma$ at some point $\tilde{q}_i$.
		
		In the topological version, instead of covers $(\tildeC\to C,\wtilde{\mathbf{p}},\wtilde{\mathbf{v}})$ we work with $(\wtilde\Sigma\to \Sigma,\wtilde{\mathbf{q}})$. Consider any object $\mathbf{M}\in \Ccal^{\boxtimes n}_{\mathbf{m}}$. Informally, a topological $\Ccal$-extended $\Gamma$-crossed modular functor $\tau$ takes as an input a pair $(\mathbf{M}, (\wtilde\Sigma\to \Sigma,\wtilde{\mathbf{q}}))$ and spits out a finite dimensional vector space $\tau_{\mathbf{M}}(\wtilde\Sigma\to \Sigma,\wtilde{\mathbf{q}})$.

		Hence let us recall the notion of $\Gamma$-covers of extended surfaces from \cite[\S3]{KP:08} and \cite{Prince}. We refer to these sources for more details and the proofs.
		\begin{definition}
			(i) An extended surface $(\Sigma,\{p_h\}_{h\in \H(\Sigma)})$ is a (possibly disconnected) smooth compact oriented surface $\Sigma$ with set of boundary components $\H(\Sigma)$ with a choice of a marked point $p_h$ on each boundary component $h\in \H(\Sigma)$. \\
			\noindent (ii) Let $\pi_0(\Sigma)$ denote the set of connected components of $\Sigma$. We have a natural map $s:\H(\Sigma)\to \pi_0(\Sigma)$ and the weight map $w:\pi_0(\Sigma)\to \ZBbb_{\geq 0}$ which assigns to a connected component of $\Sigma$ the genus of its closure. Hence given an extended surface, we can assign to it a weighted graph with vertices $\pi_0(\Sigma)$ and half-edges $\H(\Sigma)$ which is a disjoint union of corollas (a graph with one vertex, no edges and possibly having some legs, see also Remark \ref{rk:corolla}), with each weighted corolla  corresponding to a connected component of $\Sigma$.\\
			\noindent (iii) We say that an extended surface as in (i) and (ii) is \emph{stable} if each connected component is stable, i.e. if each corolla as in (ii) is a stable weighted graph.
		\end{definition}
		
		\noindent We now define $\Gamma$-coverings of extended surfaces:
		\begin{definition}\label{def:topgammacover}
			(i) A $\Gamma$-cover $(\wtilde\Sigma,\{\tilde{p}_h\}_{h\in \H(\Sigma)})$ of an extended surface $(\Sigma,\{p_h\}_{h\in \H(\Sigma)})$ is a principal left $\Gamma$-bundle $\wtilde{\Sigma}\to \Sigma$ along with a choice of lift $\tilde{p}_h\in \partial{\wtilde{\Sigma}}$ above each $p_h$.\\
			\noindent (ii) Using the choice of the marked points in $\partial\wtilde\Sigma$, for each $h\in \H(\Sigma)$, we can define a monodromy element of $\Gamma$ (see Prince \cite{Prince} for details) defining a function $\mathbf{m}:\H(\Sigma)\to \Gamma$. Hence with each $\Gamma$-cover we have an associated weighted $\Gamma$-graph which is again just a union of $\Gamma$-corollas (see also Remark \ref{rk:corolla}).\\
			\noindent (iii) We say that the $\Gamma$-cover is stable if the base surface $\Sigma$ is stable. 
		\end{definition}

		Recall that in the complex analytic situation we demanded that the $\Vcal_{\mathbf{M}}$ should form vector bundles with flat (projective) connections on $\tildeM^{\Gamma}_{g,n}(\mathbf{m})$. Note that such a connection is equivalent (via Riemann-Hilbert Correspondence) to a (projective) representation of the orbifold fundamental groupoid $\pi_1(\tildeM^\Gamma_{g,n}(\mathbf{m}))$ of the moduli stack $\tildeM^\Gamma_{g,n}(\mathbf{m})$. Hence it would have been equivalent to work with these fundamental groupoids and their representations instead. There are natural analogs of these groupoids on the topological side:
		\begin{definition}\label{def:surfcorolla}
			Let $(g,n)$ be a stable pair, i.e. $2g-2+n>0$. We define $\Surf^\Gamma_{g,n}(\mathbf{m})$ to be the groupoid whose objects are $\Gamma$-coverings of extended surfaces $(\wtilde\Sigma\to\Sigma,\mathbf{q})$ (with base being a smooth oriented surface of genus $g$ with $n$ boundary components) with monodromy data (see also Definition \ref{def:topgammacover} part (ii)) $\mathbf{m}\in \Gamma^n$. The morphisms in this groupoid are isotopy classes of orientation and marked points preserving diffeomorphisms of $\Gamma$-covers.
		\end{definition}
		
		By using the ideas from \cite[\S4.2.3]{BFM}, \cite{HainL} we now prove that the two groupoids are canonically equivalent. 
		\begin{theorem}\label{thm:fundamentalgroupoidmappingclassgp}
			For any $(g,n)$ such that $2g-2+n>0$ and for any $\mathbf{m}\in \Gamma^n$ we have a canonical equivalence of groupoids 
			$\pi_1(\tildeM^\Gamma_{g,n}(\mathbf{m}))\cong \Surf^\Gamma_{g,n}(\mathbf{m}).$
		\end{theorem}
		\begin{proof}
			We use a modification of the standard argument in the special case when $\Gamma$ is trivial using the identification of each connected component of the moduli stack $\tildeM^\Gamma_{g,n}(\mathbf{m})$ as a quotient of a certain Teichm\"uller space. Let us choose any point $\mathbf{c}:=(\tildeC\to C,\wtilde{\mathbf{p}},\wtilde{\mathbf{v}})\in \tildeM^\Gamma_{g,n}(\mathbf{m})$. 
			
			Let us strip $\tildeC\to C$ of the complex structure and consider it as a ramified covering of smooth compact oriented surfaces with markings. Consider the Teichm\"uller groupoid $\Tscr_{\mathbf{c}}$ (see also \cite[Chapter XV]{ACG:11} for details) whose objects are tuples $(\tildeD\to D,\wtilde{\mathbf{q}},\wtilde{\mathbf{w}},[\wtilde{f}])$ where $(\tildeD\to D,\wtilde{\mathbf{q}},\wtilde{\mathbf{w}})\in \tildeM^\Gamma_{g,n}(\mathbf{m})$ and $[\wtilde{f}]$ is the isotopy class of a smooth orientation and marking data preserving diffeomorphism $\tilde{f}:(\tildeD\to D,\wtilde{\mathbf{q}},\wtilde{\mathbf{w}})\to (\tildeC\to C,\wtilde{\mathbf{p}},\wtilde{\mathbf{v}})$ between ramified marked $\Gamma$-covers. 
			
			An isomorphism in $\Tscr_{\mathbf{c}}$ between two objects:
			\begin{equation}
			\label{eqn:randomuseless1}
			\wtilde\phi:(\tildeD\to D,\wtilde{\mathbf{q}},\wtilde{\mathbf{w}},[\wtilde{f}])\to (\tildeD'\to D',\wtilde{\mathbf{q}'},\wtilde{\mathbf{w}}',[\wtilde{f}'])
				\end{equation}  is a complex analytic isomorphism $\wtilde{\phi}$ between the $\Gamma$-covers as in Equation \eqref{eqn:randomuseless1} such that $[\wtilde{f}'\circ\wtilde{\phi}]=[\wtilde{f}]$.
			
			The Teichm\"uller space $\Tcal_{\mathbf{c}}$ is defined to be the space of isomorphism classes of objects of $\Tscr_{\mathbf{c}}$. We can think of the space $\Tcal_{\mathbf{c}}$ as the space of isotopy classes of complex analytic structures on the marked ramified $\Gamma$-covering of smooth oriented surfaces $(\tildeC\to C,\wtilde{\mathbf{p}},\wtilde{\mathbf{v}})$. 
			
			Note that $\tildeC\to C$ is a ramified $\Gamma$-cover of smooth oriented surfaces and hence given a complex analytic structure on $\tildeC$ we get a unique complex structure on $C$. Hence we may have defined the Teichm\"uller space $\Tcal_{\mathbf{c}}$ entirely using the top surface $\tildeC$ which may be disconnected. Nevertheless we conclude that $\Tcal_{\mathbf{c}}$ is contractible.
			
			Define the Teichm\"uller modular group $\varGamma_\mathbf{c}:=\{[\wtilde{\gamma}]|(\tildeC\to C,\wtilde{\mathbf{p}},\wtilde{\mathbf{v}},[\wtilde{\gamma}])\in \Tscr_{\mathbf{c}}\}$ which is a group under composition. Moreover we have an action of $\varGamma_{\mathbf{c}}$ on $\Tcal_{\mathbf{c}}$ which takes the isomorphism class of an object $(\tildeD\to D,\wtilde{\mathbf{q}},\wtilde{\mathbf{w}},[\wtilde{f}])$ to $(\tildeD\to D,\wtilde{\mathbf{q}},\wtilde{\mathbf{w}},[\wtilde{\gamma}\circ\wtilde{f}])$. Then the connected component of the moduli stack $\tildeM^\Gamma_{g,n}(\mathbf{m})$ containing the point $\mathbf{c}$ is precisely the orbifold quotient $\varGamma_{\mathbf{c}}\backslash\Tcal_{\mathbf{c}}$ and hence its orbifold fundamental group is precisely $\varGamma_{\mathbf{c}}$.
			
			It is also clear that $\varGamma_{\mathbf{c}}$ is canonically isomorphic to the automorphism group in $\Surf^\Gamma_{g,n}(\mathbf{m})$ of the $\Gamma$-cover of extended surfaces $(\wtilde\Sigma\to \Sigma,\wtilde{\mathbf{q}})$ associated with $(\tildeC\to C,\wtilde{\mathbf{p}},\wtilde{\mathbf{v}})$. To complete the proof, we apply the same argument for each connected component of $\tildeM^\Gamma_{g,n}(\mathbf{m})$ and use the observation that any $(\wtilde\Sigma\to \Sigma,\wtilde{\mathbf{q}})\in \Surf^\Gamma_{g,n}(\mathbf{m})$ can be closed up to a ramified marked $\Gamma$-cover of surfaces and can be equipped with a complex analytic structure.
		\end{proof}

		Just as in the complex analytic situation, we have various gluing operations for stable $\Gamma$-covers of extended surfaces along boundary components. Once again the category $\Xcal^\Gamma$ of stable $\Gamma$-graphs will help us to organize these different types of gluings. 
		
		\subsection{Categories cofibered in groupoids over $\Xcal^\Gamma$}
		In this section, we reformulate the notion of a topological $\Gamma$-crossed modular functor (as defined in \cite{KP:08}) in the spirit of \cite{BK:01} and \cite{BFM}. As described in Appendix \ref{ap:groupmarkedgraphs}, $\Xcal^\Gamma$ is a symmetric monoidal category under disjoint union of graphs denoted as $\sqcup$. Now for each object $(X,\mathbf{m},\mathbf{b})\in \Xcal^\Gamma$, we will assign two groupoids, one in the topological setting of surfaces and the other in the setting of complex algebraic curves. We will use the notion of categories cofibered in groupoids over the category $\Xcal^\Gamma$.

		So far in Section \ref{sec:topmotivation} we have only encountered stable $\Gamma$-corollas associated with stable $\Gamma$-covers. Note that we can think of any  stable $\Gamma$-graph as being obtained by gluing together stable  $\Gamma$-corollas (see also Remark \ref{rk:corolla}). If $(X,\mathbf{m},\mathbf{b})\in\Xcal^\Gamma$ then we can cut up the graph at the midpoints of all edges to obtain a collection of $\Gamma$-corollas parameterized by $V(X)$. For each vertex $v\in V(X)$, $L_v$ denotes the set of half-edges of $X$ whose source is $v$. These form the legs of the corolla corresponding to that vertex.
		
		We think of the groupoid $\Surf^\Gamma_{g,n}(\mathbf{m})$ as being associated with the corolla whose vertex has weight $g$, and which has $n$ legs marked by $\mathbf{m}\in \Gamma^n$. We will now associate a groupoid with any stable $\Gamma$-graph:
		\begin{definition}
			Let $(X,\mathbf{m},\mathbf{b})\in \Xcal^\Gamma$ and we continue to use the same notation. We define $\Surf^\Gamma(X,\mathbf{m},\mathbf{b})$ to be the product of the groupoids defined in Definition \ref{def:surfcorolla} attached to each vertex $v\in V(X)$:
			$$\Surf^\Gamma(X,\mathbf{m},\mathbf{b}):=\prod\limits_{v\in V(X)}\Surf^\Gamma_{w(v),L_v}(\mathbf{m}|_{L_v}).$$  
		\end{definition}
		
		\begin{remark}\label{rem:usefulremark}
			Equivalently $\Surf^\Gamma(X,\mathbf{m},\mathbf{b})$ is the groupoid whose objects are $\Gamma$-covers of extended surfaces $(\wtilde\Sigma\to \Sigma,\{\tilde{p}_h\}_{h\in \H(\Sigma)})$ along with identifications $V(X)\cong \pi_0(\Sigma)$ and $H(X)\cong \H(\Sigma)$ which are compatible with the weight maps, the maps `$s$' from half-edges to vertices, as well as the monodromies at the half-edges. The morphisms in this groupoid are isotopy classes of orientation preserving diffeomorphisms of $\Gamma$-covers preserving the marked points. 
		\end{remark}

		\subsubsection*{Idea behind gluing along morphisms in $\Xcal^\Gamma$}\label{sec:ideaofgluing} Using the identification $H(X)\cong \H(\Sigma)$ as in Remark \ref{rem:usefulremark}, we obtain an involution $\iota$ of the boundary components $\H(\Sigma)$. Now suppose that $\{h_1,h_2\}\in \H(\Sigma)$ form an edge, i.e. an $\iota$ orbit of size two. Then by definition we have ${}^{\mathbf{b}(h_1)}\mathbf{m}(h_1)=\mathbf{m}(h_2)^{-1}$. Now $\mathbf{m}(h_1)$ is the monodromy element defined using the marked point $\tilde{p}_{h_1}$ lying over $p_{h_1}$ and hence the monodromy determined by the point $\mathbf{b}(h_1)\cdot \tilde{p}_{h_1}$ is ${}^{\mathbf{b}(h_1)}\mathbf{m}(h_1)$. Hence we can glue the $\Gamma$-cover by identifying the point $\mathbf{b}(h_1)\cdot \tilde{p}_{h_1}$ with the point $\tilde{p}_{h_2}$ and gluing the surface $\Sigma$ along $p_{h_1}, p_{h_2}$. 
		
		Similarly, if $h\in \H(\Sigma)$ is a leg with $\mathbf{m}(h)=1$, then it means that the $\Gamma$-cover when restricted to the corresponding boundary component is trivial. Hence we can glue the trivial $\Gamma$-cover of a standard disk to our given $\Gamma$-cover in $\Surf^\Gamma(X,\mathbf{m},\mathbf{b})$.
		
		We refer to \cite[\S2.6]{Prince} for more details about gluing constructions and uniqueness of gluing. Recall that morphisms in $\Xcal^\Gamma$ are compositions of isomorphisms, edge contractions and deletion of 1-marked legs. Hence we obtain:
		\begin{proposition}\label{prop:surfacegluing}
			Let $(f,\pmb\gamma):(X,\mathbf{m}_X,\mathbf{b}_X)\to (Y,\mathbf{m}_Y,\mathbf{b}_Y)$ be a morphism in $\Xcal^\Gamma$. Then the gluing constructions described above define a gluing functor between groupoids
			\[
			\Gcal_{f,\pmb\gamma}:\Surf^\Gamma(X,\mathbf{m}_X,\mathbf{b}_X)\to \Surf^\Gamma(Y,\mathbf{m}_Y,\mathbf{b}_Y).
			\]
			The assignment $(X,\mathbf{m}_X,\mathbf{b}_X)\mapsto \Surf^\Gamma(X,\mathbf{m}_X,\mathbf{b}_X)$ defines a symmetric monoidal pseudo-functor from the category $\Xcal^\Gamma$ to the (2,1)-category of groupoids. We let $\Surf^\Gamma\to \Xcal^\Gamma$ be the corresponding symmetric monoidal category cofibered in groupoids over $\Xcal^\Gamma$. 
		\end{proposition}
		
		On the complex analytic side, we define the tower of groupoids as follows: For any stable $\Gamma$-graph $(X,\mathbf{m},\mathbf{b})\in \Xcal^\Gamma$, we have the smooth Deligne-Mumford stack $$\tildeM^\Gamma_{X,\mathbf{m},\mathbf{b}}:=\prod\limits_{v\in V(X)}\tildeM^\Gamma_{w(v),L_v}(\mathbf{m}|_{L_v})$$ as in Appendix \ref{ap:specialization}. Following \cite[\S4]{BFM}, \cite[\S6.1]{BK:01}, to $(X,\mathbf{m},\mathbf{b})$ we attach the Poincare fundamental groupoid $\pi_1(\tildeM^\Gamma_{X,\mathbf{m},\mathbf{b}})$ of the moduli stack $\tildeM^\Gamma_{X,\mathbf{m},\mathbf{b}}$. The objects are pointed admissible $\Gamma$-covers of (possibly disconnected) smooth projective curves with marked tangent vectors corresponding to $(X,\mathbf{m},\mathbf{b})$. The morphisms are 1-parameter $C^\infty$-families (considered up to homotopy) of such objects. By following the arguments from \cite{BK:01}, \cite[\S4.3.1]{BFM}, and Appendix \ref{ss:specalongclutch} we can define gluing functors for the following groupoids:
		\begin{proposition}
			Let $(f,\pmb\gamma):(X,\mathbf{m}_X,\mathbf{b}_X)\to (Y,\mathbf{m}_Y,\mathbf{b}_Y)$ be a morphism in $\Xcal^\Gamma$. Then we have a gluing functor between groupoids
			$
			\Gcal_{f,\pmb\gamma}:\pi_1(\tildeM^\Gamma_{X,\mathbf{m}_X,\mathbf{b}_X})\to \pi_1(\tildeM^\Gamma_{Y,\mathbf{m}_Y,\mathbf{b}_Y}).
			$
			
			The assignment $(X,\mathbf{m}_X,\mathbf{b}_X)\mapsto \pi_1(\tildeM^\Gamma_{X,\mathbf{m}_X,\mathbf{b}_X})$ defines a  symmetric monoidal pseudo-functor from the category $\Xcal^\Gamma$ to the (2,1)-category of groupoids. We let $\pi_1\tildeM^\Gamma
			\to \Xcal^\Gamma$ be the corresponding symmetric monoidal category cofibered in groupoids over $\Xcal^\Gamma$.
		\end{proposition}

		In order to define $\Gamma$-crossed modular functors with central charge, we will need to consider certain central extensions of some towers of groupoids associated to $\tildeM^\Gamma_{X,\mathbf{m},\mathbf{b}}$. Firstly, for a real symplectic vector space $V\neq 0$ we define $\Tcal_V$ to be the Poincare fundamental groupoid of the space of all Lagrangian subspaces of $V$. It will be convenient to set $\Tcal_0$ to be the groupoid with only one object $0$, with $\Hom_{\Tcal_0}(0,0)=\ZBbb$. For a smooth oriented surface $\Sigma$ with boundary, let $cl(\Sigma)$ be the closed oriented surface obtained by gluing disks to all boundary components $\pi_0\partial(\Sigma)$. Then we have the intersection form on $H_1(cl(\Sigma),\RBbb)$ making it a real symplectic vector space. 
		
		We set $\Tcal_{\Sigma}$ to be the groupoid $\Tcal_{H_1(cl(\Sigma),\RBbb)}.$ We refer to \cite[\S5.7]{BK:01}, \cite[\S4.1]{BFM} for more properties and details.
		\begin{definition}
			For $({X,\mathbf{m}_X,\mathbf{b}_X})\in \Xcal^\Gamma$, let $\wtilde{\Surf}^\Gamma({X,\mathbf{m}_X,\mathbf{b}_X})$ be the groupoid whose
			\begin{enumerate}
				\item  {\em objects} are tuples $$(\wtilde\Sigma\to \Sigma,\{\tilde{p}_h\}_{h\in \H(\Sigma)}, V(X)\cong \pi_0(\Sigma),H(X)\cong \pi_0\partial(\Sigma),\Ycal)$$ with $\Ycal\in \Tcal_{\Sigma}$ and the remaining data being that of an object of  ${\Surf}^\Gamma({X,\mathbf{m}_X,\mathbf{b}_X})$. 
				
				\item A {\em morphism} in this groupoid between two objects 
				\begin{center}
					\xymatrix{
						&(\wtilde\Sigma\to \Sigma,\{\tilde{p}_h\}_{h\in \H(\Sigma)}, V(X)\cong \pi_0(\Sigma),H(X)\cong \pi_0\partial(\Sigma),\Ycal)\ar[d]^{(\phi,\gamma)}\\
						& \  (\wtilde\Sigma'\to \Sigma',\{\tilde{p}'_h\}_{h\in \H(\Sigma)}, V(X)\cong \pi_0(\Sigma'),H(X)\cong \pi_0\partial(\Sigma'),\Ycal')
					}
				\end{center}
				is a pair $(\phi,\gamma)$, where $\phi$ is a morphism in the groupoid ${\Surf}^\Gamma({X,\mathbf{m}_X,\mathbf{b}_X})$ and $\gamma:\phi_*(\Ycal)\to \Ycal'$ is a morphism in the groupoid $\Tcal_{\Sigma'}$. 
			\end{enumerate}
			The forgetful functor of groupoids $\wtilde{\Surf}^\Gamma({X,\mathbf{m}_X,\mathbf{b}_X})\to {\Surf}^\Gamma({X,\mathbf{m}_X,\mathbf{b}_X})$  is a central extension by $\ZBbb$ (cf. \cite[\S4.2]{BFM}).
		\end{definition}
		
		\begin{proposition}
			Let $(f,\pmb\gamma):(X,\mathbf{m}_X,\mathbf{b}_X)\to (Y,\mathbf{m}_Y,\mathbf{b}_Y)$ be a morphism in $\Xcal^\Gamma$. Then we have a gluing functor between groupoids
			$
			\wtilde{\Gcal}_{f,\pmb\gamma}:\wtilde{\Surf}^\Gamma(X,\mathbf{m}_X,\mathbf{b}_X)\to \wtilde{\Surf}^\Gamma(Y,\mathbf{m}_Y,\mathbf{b}_Y).
			$
			The assignment $(X,\mathbf{m}_X,\mathbf{b}_X)\mapsto \wtilde{\Surf}^\Gamma(X,\mathbf{m}_X,\mathbf{b}_X)$ defines a symmetric monoidal pseudo-functor from the category $\Xcal^\Gamma$ to the (2,1)-category of groupoids. We let $\wtilde{\Surf}^\Gamma\to \Xcal^\Gamma$ be the corresponding symmetric monoidal category cofibered in groupoids over $\Xcal^\Gamma$. 
		\end{proposition}
		\begin{proof}
			By definition, objects of $\wtilde{\Surf}^\Gamma(X,\mathbf{m}_X,\mathbf{b}_X)$ consist of an object $$(\wtilde\Sigma\to \Sigma,\{\tilde{p}_h\}_{h\in \H(\Sigma)}, V(X)\cong \pi_0(\Sigma),H(X)\cong \pi_0\partial(\Sigma))\in {\Surf}^\Gamma(X,\mathbf{m}_X,\mathbf{b}_X)$$ along with $\Ycal\in \Tcal_{\Sigma}.$ By Proposition \ref{prop:surfacegluing} we can glue the first part of the data along the morphism $(f,\pmb\gamma)$ in $\Xcal^\Gamma$. The other part of the data i.e. $\Ycal\in \Tcal_{\Sigma}$ only depends on the base surface $\Sigma$ (which corresponds to the untwisted case) and hence can be glued using \cite[\S5.7]{BK:01}, \cite[\S4.3]{BFM}. 
		\end{proof}
		
		On the complex analytic side, recall that we have the Hodge line bundles, which we denoted by $\Lambda$, on the moduli stacks $\tildeM^\Gamma_{X,\mathbf{m},\mathbf{b}}$. By definition, the fiber of $\Lambda$ over a point $(\tildeC\to C,\wtilde{\mathbf{p}},\mathbf{p},\mathbf{v})\in \tildeM^\Gamma_{X,\mathbf{m},\mathbf{b}}$ is $\det(H^1(C,\Ocal_C))^*$. Observe that the bundle $\Lambda$ is the pull back from the moduli space parameterizing the base curve $C$.

		The smooth projective complex curve $C$ has the structure of a smooth oriented compact closed surface, and hence we have its associated groupoid $\Tcal_C$ of Lagrangians in the real symplectic vector space $H_1(C,\RBbb)\cong H^1(C,\RBbb)^*$ which can be identified with $H^1(C,\Ocal_C)^*$ as a real vector space. By \cite[\S4.1]{BFM}, \cite[Thm. 6.7.7]{BK:01} for any smooth complex curve we have a canonical equivalence of groupoids
		\begin{equation}\label{eq:equivofgroupoids}
		\Tcal_{C}\cong \pi_1({\det{}}^{\otimes 2}(H^1(C,\Ocal_C))^*\setminus \{0\}).
		\end{equation}
		Let us denote the total space minus the zero section of the line bundle $\Lambda^{\otimes 2}$ on $\tildeM^\Gamma_{X,\mathbf{m},\mathbf{b}}$  by $\Lambda^{\otimes 2}_{\times}\tildeM^\Gamma_{X,\mathbf{m},\mathbf{b}}\to \tildeM^\Gamma_{X,\mathbf{m},\mathbf{b}}.$ It is a $\CBbb^\times$-bundle over the base, hence the natural functor of groupoids  $\pi_1(\Lambda^{\otimes 2}_\times\tildeM^\Gamma_{X,\mathbf{m},\mathbf{b}})\to \pi_1(\tildeM^\Gamma_{X,\mathbf{m},\mathbf{b}})$ is a central extension by $\ZBbb$. Using this we can define a central extension of the complex analytic tower of groupoids by $\ZBbb$. Using the arguments of \cite[\S4.3]{BFM}, \cite[\S6.7]{BK:01} and Appendix \ref{ss:specalongclutch} we get:
		\begin{proposition}
			Let $(f,\pmb\gamma):(X,\mathbf{m}_X,\mathbf{b}_X)\to (Y,\mathbf{m}_Y,\mathbf{b}_Y)$ be a morphism in $\Xcal^\Gamma$. Then we have a gluing functor between groupoids
			$$
			\wtilde\Gcal_{f,\pmb\gamma}:\pi_1(\Lambda^{\otimes2}_{\times}\tildeM^\Gamma_{X,\mathbf{m}_X,\mathbf{b}_X})\to \pi_1(\Lambda^{\otimes2}_{\times}\tildeM^\Gamma_{Y,\mathbf{m}_Y,\mathbf{b}_Y}).
			$$
			The assignment $(X,\mathbf{m}_X,\mathbf{b}_X)\mapsto \pi_1(\Lambda^{\otimes2}_\times\tildeM^\Gamma_{X,\mathbf{m}_X,\mathbf{b}_X})$ defines a  symmetric monoidal pseudo-functor from the category $\Xcal^\Gamma$ to the (2,1)-category of groupoids. We let $\pi_1\Lambda^{\otimes2}_{\times}\tildeM^\Gamma
			\to \Xcal^\Gamma$ be the corresponding symmetric monoidal category cofibered in groupoids over $\Xcal^\Gamma$.
		\end{proposition}
		
		Finally we have:
		\begin{theorem}\label{thm:topanalytic}
			For each $(X,\mathbf{m}_X,\mathbf{b}_X)\in\Xcal^\Gamma$, we have canonical equivalences of groupoids
			$$\Surf^\Gamma(X,\mathbf{m}_X,\mathbf{b}_X)\cong \pi_1(\tildeM^\Gamma_{X,\mathbf{m}_X,\mathbf{b}_X}) \ \mbox{ and } \ \wtilde\Surf^\Gamma(X,\mathbf{m}_X,\mathbf{b}_X)\cong \pi_1(\Lambda^{\otimes2}_{\times}\tildeM^\Gamma_{X,\mathbf{m}_X,\mathbf{b}_X})
			$$ which induce symmetric monoidal equivalences of categories cofibered in groupoids over $\Xcal^\Gamma$
			$\Surf^\Gamma\cong \pi_1\tildeM^\Gamma$  and $\wtilde\Surf^\Gamma\cong \pi_1\Lambda^{\otimes2}_{\times}\tildeM^\Gamma.$
		\end{theorem}
		\begin{proof}
			The equivalence $\Surf^\Gamma(X,\mathbf{m}_X,\mathbf{b}_X)\cong \pi_1(\tildeM^\Gamma_{X,\mathbf{m}_X,\mathbf{b}_X})$ has been established for $\Gamma$-corollas in Theorem \ref{thm:fundamentalgroupoidmappingclassgp}. In general both groupoids are defined as products of groupoids coming from $\Gamma$-corollas corresponding to the vertices of $(X,\mathbf{m}_X,\mathbf{b}_X)$. Moreover these canonical equivalences are compatible with gluings along the morphisms of $\Xcal^\Gamma$. In order to prove that $$\wtilde\Surf^\Gamma(X,\mathbf{m}_X,\mathbf{b}_X)\cong \pi_1(\Lambda^{\otimes2}_{\times}\tildeM^\Gamma_{X,\mathbf{m}_X,\mathbf{b}_X})$$ it suffices to note that the additional data on both the sides comes from the base surface/curve which is the untwisted case. Hence using the same result from \cite[\S6.7]{BK:01}, \cite[\S4.2.3]{BFM} we complete the proof. 
		\end{proof}
		
		\subsection{Representations of the groupoids and their central extensions}
		Let $\Acal$ be a groupoid. We denote the category of finite dimensional representations of $\Acal$ by $\Rep\Acal$. By definition a finite dimensional representation of $\Acal$ is a functor from the group $\Acal$ to the category $\Vec$ of finite dimensional complex vector spaces. Now suppose that $\wtilde{\Acal}\to \Acal$ is a central extension of groupoids by $\ZBbb$, i.e.  for each object $\wtilde{A}\in \wtilde\Acal$ lying above $A\in \Acal$, $\Aut(\wtilde{A})$ is a central extension of $\Aut(A)$ by $\ZBbb$, in particular the object $\wtilde{A}$ should have a canonical automorphism $\gamma_{\wtilde{A}}$ which generates the central copy of $\ZBbb$ inside $\Aut(\wtilde{A})$. For $a\in \CBbb^\times$, a representation $\rho:\wtilde{\Acal}\to \Vec$ is said to be of multiplicative central charge $a$ if for every $\wtilde{A}\in \wtilde\Acal$, the automorphism $\gamma_{\wtilde{A}}$ acts on $\rho(\wtilde{A})$ by the scalar $a$. We define $\Rep_a\wtilde\Acal$ to be the full subcategory of $\Rep\wtilde\Acal$ of representations with multiplicative central charge $a$.
		
		Now for $(X,\mathbf{m},\mathbf{b})\in \Xcal^\Gamma$ we have seen that we have the two central extensions of groupoids by $\ZBbb$, namely $\wtilde\Surf^\Gamma(X,\mathbf{m},\mathbf{b})\to \Surf^\Gamma(X,\mathbf{m},\mathbf{b})$ and the complex analytic version $\pi_1(\Lambda^{\otimes 2}_\times\tildeM^\Gamma_{X,\mathbf{m},\mathbf{b}})\to \pi_1(\tildeM^\Gamma_{X,\mathbf{m},\mathbf{b}})$. Moreover, by Theorem \ref{thm:topanalytic}, these two central extensions of groupoids are canonically equivalent. This fact combined with Deligne's \cite{Deligne1} Riemann-Hilbert correspondence (see Remark \ref{rk:delignerhc}) we obtain the equivalence of the following categories for every $c\in \CBbb$:
		\begin{equation}\label{eq:topanalyticequivalence}
		\Rep_{e^{\sqrt{-1}\pi c}}\wtilde\Surf^\Gamma(X,\mathbf{m},\mathbf{b})\cong \Rep_{e^{\sqrt{-1}\pi c}}(\pi_1(\Lambda^{\otimes 2}_\times\tildeM^\Gamma_{X,\mathbf{m},\mathbf{b}}))\cong \mathscr{D}_c\Mod(\tildebarM{}^\Gamma_{X,\mathbf{m},\mathbf{b}}).
		\end{equation}
		Moreover, by Theorem \ref{thm:topanalytic} and the definition of the specialization functors of twisted $D$-modules along clutchings (in Appendix \ref{ss:specalongclutch}), we see that the equivalences in Equation \eqref{eq:topanalyticequivalence} are compatible with the gluing/clutching functors corresponding to morphisms in $\Xcal^\Gamma$.
		\subsection{Crossed modular functors}
		We will now formulate the definition of $\Gamma$-crossed modular functors in terms of the towers of groupoids defined in the previous subsection. Let $(\Ccal,\unit,R)$ be a $\Gamma$-crossed abelian category with the involutive duality denoted as $(\cdot)^*$.
		\begin{definition}\label{def:topmodfun}
			A $\Ccal$-extended topological $\Gamma$-crossed modular functor of multiplicative central charge $a\in \CBbb^\times$ consists of the following data:
			\begin{itemize}
				\item[1.] For each stable pair $(g,A)$ and $\mathbf{m}\in\Gamma^A$, (which we think of as corolla of weight $g$ and $A$-legs marked by $\mathbf{m}$) a topological conformal blocks functor
				\begin{equation}
				\label{eqn:criticlaone}
				\tau_{g,A,\mathbf{m}}:\Ccal^{\boxtimes A}_\mathbf{m}\rar{}\Rep_a\wtilde\Surf^\Gamma(g,A,\mathbf{m}).	\end{equation}
				Once we have functors as in Equation \eqref{eqn:criticlaone}, we can canonically extend them to obtain functors
				$$\tau_{X,\mathbf{m},\mathbf{b}}:\Ccal^{\boxtimes H(X)}_\mathbf{m}\rar{}\Rep_a\wtilde\Surf^\Gamma(X,\mathbf{m},\mathbf{b})$$
				for each $(X,\mathbf{m},\mathbf{b})\in \Xcal^\Gamma$. 
				\item[2.] For each morphism $(f,\pmb\gamma):(X,\mathbf{m}_X,\mathbf{b}_X)\to (Y,\mathbf{m}_Y,\mathbf{b}_Y)$ in $\Xcal^\Gamma$ a natural isomorphism $G_{f,\pmb\gamma}$ between the two functors from $\Ccal^{\boxtimes H(Y)}_{\mathbf{m}_Y}$ to $\Rep_a\wtilde\Surf^\Gamma(X,\mathbf{m}_X,\mathbf{b}_X)$:
				$$G_{f,\pmb\gamma}:\tau_{X,\mathbf{m}_X,\mathbf{b}_X}\circ\Rcal_{f,\pmb\gamma}\rar{}{\wtilde{\Gcal}}^*_{f,\pmb\gamma}\circ \tau_{Y,\mathbf{m}_Y,\mathbf{b}_Y}$$
				compatible with compositions of morphisms in $\Xcal^\Gamma$.
				\item[3.] A normalization $$\tau_{0,3,\pmb{1}}(\unit\boxtimes\unit\boxtimes\unit)(S^2_3\times \Gamma\to S^2_3
				, (p_1,1),(p_2,1),(p_3,1))\cong \CBbb,$$ where $(S^2_3,p_1,p_2,p_3)$ is some standard Riemann sphere with 3 disks removed and three boundary points $p_1,p_2,p_3$, say as in \cite[Def. 3.4]{Prince} and where we have equipped it with the trivial $\Gamma$-cover. Note that $\Tcal_{S^2}=\Tcal_0$.
				\item[4.] (Non-degeneracy.) Given any $\gamma\in \Gamma$ and non-zero $X\in \Ccal_{\gamma}$, $\tau_{0,3,(\gamma,\gamma^{-1},1)}(X\boxtimes X^*\boxtimes \unit)\neq 0.$
			\end{itemize}
		\end{definition}
		
		\begin{definition}\label{def:gen0topmodfun}
			Recall that we have the full subcategory $\Xcal^\Gamma_{0}\subset \Xcal^\Gamma$ be the full subcategory formed by forests with all vertices having weight 0. We define a $\Ccal$-extended topological $\Gamma$-crossed modular functor in genus 0 by replacing $\Xcal^\Gamma$ with just the genus 0 part $\Xcal^\Gamma_{0}$ and by replacing the tower of groupoids $\wtilde{\Surf}^\Gamma$ with $\Surf^\Gamma$ in Definition \ref{def:topmodfun}.
		\end{definition}
		\begin{remark}
			In \cite{KP:08}, a slightly different definition of a $\Gamma$-crossed modular functor in genus zero is given, namely, all $\Gamma$-covers of surfaces are considered instead of considering only stable ones. Instead, we are considering the operation of forgetting 1-marked legs, i.e. attaching the trivial $\Gamma$-cover of the disk to $\Gamma$-covers of surfaces, which plays a similar role. Hence we see that the notion of $\Ccal$-extended topological $\Gamma$-crossed modular functor in genus 0 as defined by Definition \ref{def:gen0topmodfun} agrees with the notion defined in \cite{KP:08}. Note that we have included non-degeneracy condition already in the definition of a $\Gamma$-crossed modular functor.
		\end{remark}
		\begin{proposition}\label{prop:top=analytic}
			Let $\Ccal$ be a $\Gamma$-crossed abelian category. Then the notion of a $\Ccal$-extended complex analytic $\Gamma$-crossed modular functor of additive central charge $c$ is equivalent to the notion of a $\Ccal$-extended topological $\Gamma$-crossed modular functor of multiplicative central charge $e^{\sqrt{-1}\pi c}$. Similarly, the notion of a $\Ccal$-extended complex analytic $\Gamma$-crossed modular functor in genus 0 is equivalent to the notion of a $\Ccal$-extended topological $\Gamma$-crossed modular functor in genus 0.
		\end{proposition}
		\begin{proof}
			By (\ref{eq:topanalyticequivalence}), for each $(X,\mathbf{m},\mathbf{b})\in \Xcal^\Gamma$ we have an equivalence 
			$$\Rep_{e^{\sqrt{-1}\pi c}}\wtilde\Surf^\Gamma(X,\mathbf{m},\mathbf{b})\cong \mathscr{D}_c\Mod(\tildebarM{}^\Gamma_{X,\mathbf{m},\mathbf{b}})$$ which is compatible with the morphisms in the category $\Xcal^\Gamma$. In other words we can replace $\Rep_{e^{\sqrt{-1}\pi c}}\wtilde\Surf^\Gamma(X,\mathbf{m},\mathbf{b})$ with $\mathscr{D}_c\Mod(\tildebarM{}^\Gamma_{X,\mathbf{m},\mathbf{b}})$ in the definition of topological $\Gamma$-crossed modular functor. Then we exactly get the notion of the complex analytic $\Gamma$-crossed modular functor we defined in the previous section. Hence the two notions are equivalent. The same applies to the genus 0 case. 
		\end{proof}

		\subsection{Proof of Theorem \ref{thm:modfunmodcat}}\label{sec:proofsofthms}
		Using the previous result we now complete the proof of Theorem \ref{thm:modfunmodcat}.  By the main theorem  in Section 5 of \cite{KP:08}, the notion of a $\Ccal$-extended topological $\Gamma$-crossed modular functor in genus 0 is equivalent to equipping $\Ccal$ with the structure of a $\Gamma$-crossed weakly ribbon category. 
		
		On the other hand by Proposition \ref{prop:top=analytic} the notion of a $\Ccal$-extended complex analytic $\Gamma$-crossed modular functor in genus 0 as defined in Section \ref{sec:crossedmodular} is equivalent to the notion of a $\Ccal$-extended topological $\Gamma$-crossed modular functor in genus 0. Equation (\ref{eq:homspaces}) follows from the property (2) that holds for the equivalence established in \cite[\S5]{KP:08}. This completes the proof of Theorem \ref{thm:modfunmodcat}.

		\section{$\Gamma$-crossed modular functors from twisted conformal blocks}\label{sec:modfuntwistedconformal}
		Let us now return to the setting of a finite group $\Gamma$ acting on a simple Lie algebra $\frg$. In this section, we discuss  how the twisted conformal blocks  associated to Galois covers of curves give a $\mathcal{C}$-extended $\Gamma$-crossed modular functor. We follow the notations of Sections \ref{sec:crossedabelian} and Section \ref{sec:deftwistconf}. In this section, we have the assumptions that ``$\Gamma$ preserves a Borel subalgebra of $\frg$" (See Remark \ref{rem:gammapreservesborel}). 
		
		\subsection{$\Gamma$-crossed abelian category}\label{ss:twistedconfblockcrossedabelian}
		Recall that for each $\gamma \in \Gamma$, we have the set $P^{\ell}(\frg,\gamma)$ that parameterizes representations of the twisted affine Kac-Moody Lie algebra $\widehat{L}(\frg,\gamma)$. We set $\mathcal{C}_{\gamma}$ to be the $\mathbb{C}$-linear semisimple abelian category whose simple objects are 
		parameterized by  $P^{\ell}(\frg,\gamma)$. We 
		define $\mathcal{C}:=\oplus_{\gamma\in \Gamma}\mathcal{C}_{\gamma}$ which is $\Gamma$-graded. 
		\subsubsection{$\Gamma$-action on $\mathcal{C}$}\label{subsubsec:gammaaction}Let $\sigma$ be an automorphism on $\frg$ of order $|\sigma|$. Recall from Section \ref{sec:affineLiealgtwisted}, the eigenspaces $\frg_{\sigma, j}$ of $\frg$ with 
		eigenvalue $\exp{\frac{2\pi{\sqrt{-1}}j}{|\sigma|}}$, where $0\leq j<|\sigma|$. Now if $\Gamma$ is a group acting on $\frg$ and let 
		$\gamma$ and $\sigma$ be any two elements of $\Gamma$. Consider the map $\frg_{\sigma,j}\rightarrow 
		\frg_{\gamma\sigma\gamma^{-1},j}$ obtained by sending $X$ to $\gamma X$ induces an 
		isomorphism of the affine Lie algebras 
		$$\psi_{\gamma}: \widehat{L}(\frg,\sigma)\rightarrow \widehat{L}(\frg,\gamma\sigma\gamma^{-1}).$$
		
		Since the group $\Gamma$ acts on $\frg$, it follows that $\psi_{\gamma_1\gamma_2}=\psi_{\gamma_1}\circ \psi_{\gamma_2}$, where $\gamma_1$ and $\gamma_2$ are arbitrary elements of $\Gamma$.
		This in turns induces an action on the set of the simple objects  $\oplus_{\gamma \in \Gamma}P^{\ell}(\frg, \gamma)$ of $\mathcal{C}$. Observe that $\psi_{\sigma}$ induces the identity morphism on $\frg^{\sigma}$ and hence it acts by identity on $P^{\ell}(\frg,\sigma)$.
		Extending the action $\mathbb{C}$-linearly, we get an action on $\mathcal{C}$. 
		\subsubsection{Invariant Object}Now observe that the simple objects of $\mathcal{C}_{1}$ 
		are just elements in $P_{\ell}(\frg)$. This has a special object corresponding to the vacuum representation of the untwisted affine Kac-Moody Lie algebra which we  declare to be  the {\em $\Gamma$-invariant object $\unit$}. 
		\subsubsection{Symmetric Tensor}Let $\frg^{\gamma}$ be the Lie subalgebra of $\frg$ fixed by $\gamma$. 	For every weight $\mu \in P_{+}(\frg^{\gamma})$, let $\mu^{*}$ be the dominant 
		integral weight of $\frg^{\gamma}$ which is the highest weight of the dual representation $V_{\mu}$. Now it is easy to see (Lemma 5.3 in \cite{KH}) that $\mu \in P^{\ell}(\frg,\gamma)$ if and only if	$\mu^* \in P^{\ell}(\frg,\gamma^{-1})$. We define the following:
		\begin{equation}\label{eqn:symmetricobject}
		R_\gamma:=\sum_{\mu \in P^{\ell}(\frg,\gamma)}\mu \boxtimes \mu^* \in \Ccal_{\gamma}\boxtimes \Ccal_{\gamma^{-1}},\  \mbox{and} \ R:=\oplus_{\gamma\in \Gamma}R_{\gamma}.
		\end{equation}
		By the definition of the action of $\Gamma$ on $\mathcal{C}$, the object $R$ is 
		clearly $\Gamma$-invariant. 
		Thus we have checked that the $\Gamma$-graded abelian category satisfies all the conditions of Definition \ref{def:crossedabeliancategory}.  
		\subsection{Twisted conformal blocks}For each stable pair $(g,A)$ and ${\bf m}\in \Gamma^{A}$ (see Definition \ref{def:cextenmodularfunctor} for notation) and $\vec{\lambda}\in \mathcal{C}_{{\bf m}}^{\boxtimes A}$, we assign the vector bundle of twisted covacua $\mathcal{V}_{\vec{\lambda},\Gamma}(\widetilde{C},C, \widetilde{\bf{p}}, \bf{p},\widetilde{\bf{v}})$ on $\widetilde{\overline{\mathcal{M}}}{}^{\Gamma}_{g,A}({\bf m})$. 
		By Theorem \ref{thm:atiyahalgebra}, we know that that the log Atiyah algebra $\mathcal{A}_{\Lambda}(-\log\widetilde{\Delta}^\Gamma_{g,A}(\mathbf{m}))$ acts on $\mathcal{V}_{\vec{\lambda},\Gamma}(\widetilde{C},C, \widetilde{\bf{p}}, \bf{p},\widetilde{\bf{v}})$ with central charge $\frac{\ell\dim\frg}{2(\ell+h^{\vee}(\frg))}$. Here $\Lambda$ denote the pullback of the Hodge bundle of $\widetilde{\overline{\mathcal{M}}}_{g,A}$ 
		to $\widetilde{\overline{\mathcal{M}}}{}^{\Gamma}_{g,A}({\bf m})$ and $\widetilde{\Delta}^\Gamma_{g,A}(\mathbf{m})$ denotes the boundary divisor of  $\widetilde{\overline{\mathcal{M}}}{}^{\Gamma}_{g,A}({\bf m})$. Thus this assignment defines 
		a functor from $\mathcal{C}_{\bf{m}}^{\boxtimes{A}}$ to $\mathscr{D}_{c}\operatorname{Mod}\big(\widetilde{\overline{\mathcal{M}}}{}^{\Gamma}_{g,A}({\bf{m}})\big)$, where $c=\frac{\ell\dim\frg}{2(\ell+h^{\vee}(\frg))}$. It 
		is clear that this assignment, satisfies Condition (3) of Definition \ref{def:cextenmodularfunctor}. Motivated by the results of \cite{BK:01}, we have the following theorem.
		\begin{theorem}\label{conj:twistedaffineconfblocks}
			\label{conj:conformalblockismodular}The assignment $\vec{\lambda} \rightarrow \mathcal{V}_{\vec{\lambda},\Gamma}(\widetilde{C},C, \widetilde{\bf{p}}, \bf{p},\widetilde{\bf{v}})$  realizes sheaf of twisted covacua as a  $\mathcal{C}$-extended $\Gamma$-crossed modular functor with central charge $\frac{\ell\dim\frg}{2(\ell+h^{\vee}(\frg))}$, i.e. it satisfies the axioms of Definition \ref{def:cextenmodularfunctor}.
		\end{theorem}We give a proof of Theorem \ref{conj:twistedaffineconfblocks} in Section \ref{sec:proofofconfgammafunc}.
		\subsection{Some consequences of the Theorem \ref{conj:twistedaffineconfblocks}}
		Let us first state the following well known result about untwisted conformal blocks and modular functors (\cite{BK:01,Huang:08a,Huang:08b}). We continue to use the same notation as before, in particular, we consider the category $\Ccal_1$ of level $\ell$ representations of the untwisted affine lie algebra.
		\begin{theorem}\label{thm:oldconf}
			The conformal blocks for the untwisted affine Lie algebra define a $\Ccal_1$-extended modular functor of central charge $\frac{\ell\dim\frg}{2(\ell+h^{\vee}(\frg))}$. The associated weak ribbon category structure endowed on $\Ccal_1$ is, in fact rigid, and hence $\Ccal_1$ is a modular fusion category with additive central charge $\frac{\ell\dim\frg}{2(\ell+h^{\vee}(\frg))}$.
		\end{theorem}
		\begin{remark}
			We remark that the rigidity statement in Theorem \ref{thm:oldconf} follows from the works \cite{Huang:08a,Huang:08b} of Huang. 
		\end{remark}
		\subsubsection{Proof of Theorem \ref{thm:oneofmain}}\label{sec:proofofoneofmain}
		We can now deduce Theorem \ref{thm:oneofmain} as the consequences of Theorems \ref{thm:modfunmodcat} and \ref{conj:twistedaffineconfblocks}. Firstly, using Theorem \ref{conj:twistedaffineconfblocks}, we see that twisted conformal blocks give a $\mathcal{C}$-extended $\Gamma$-crossed modular functor. Here $\mathcal{C}=\mathcal{C}(\frg,\Gamma,\ell)$ is a $\Gamma$-graded abelian category such that the simple objects of each $\gamma$-component $\Ccal_{\gamma}$ consists of $P^{\ell}(\frg,\gamma)$. 
		
		By Theorem \ref{thm:modfunmodcat}, the $\Gamma$ graded category $\Ccal(\frg,\Gamma,\ell)$ has a structure of a $\Gamma$-crossed weakly ribbon category satisfying Equation 
		\begin{equation}
		\operatorname{Hom}(\unit, \lambda_1\dotimes \lambda_2 \dotimes \lambda_3)=\begin{cases}
		\Vcal_{\vec{\lambda},\Gamma}(\tildeC\to \PBbb^1,\wtilde{\mathbf{p}},\mu_3,\wtilde{\mathbf{v}}), & \text{if $m_1\cdot m_2\cdot m_3=1$}\\
		0 & \text{otherwise}.
		\end{cases}
		\end{equation}

		Moreover, by Theorem \ref{thm:oldconf}, $\Ccal_1$ is rigid. Hence by Corollary \ref{cor:rigidity of graded categories} the whole of $\Ccal(\frg,\Gamma,\ell)$ must be rigid. This completes the proof of Theorem \ref{thm:oneofmain}. 
		
		\subsection{Proof of Theorem \ref{conj:main1}}
		We know that $\Ccal(\frg,\Gamma,\ell)$ is a $\Gamma$-crossed modular fusion category. In particular, we have the associated categorical crossed S-matrices. In the next section we will give a proof of Theorem  \ref{conj:twistedaffineconfblocks}. An explicit description of the associated categorical crossed S-matrices has already been given by Proposition \ref{prop:charactertableviaSmatrix}. 
		\begin{remark}\label{rem:cocyleszero}
			Let $\gamma$ be an automorphism of $\frg$. Then $\gamma$ induces an automorphism of the abstract Cartan algebra $\frh$ as well as of the {\em abstract root system} preserving positive roots. Using Lemma 2.1 in \cite{AHR:18}, we have ${ }^\gamma V_{\lambda}{\cong} V_{\gamma\cdot\lambda}$, where $V_\lambda$ is a finite dimensional irreducible $\frg$-representation of highest weight $\lambda$, and moreover if we choose a Borel $\mathfrak{b}$ fixed by $\gamma$, then we get a canonical isomorphism. If a group $\Gamma$ acts on $\frg$, then we have an induced $\Gamma$-action on the {\em abstract root system} preserving the positive roots and let $\Gamma_\lambda$ be the stabilizer of $\lambda$. If the action of $\Gamma$ fixes a Borel $\mathfrak{b}$, then for each $\gamma\in \Gamma_\lambda$, we have a canonical isomorphism $T_\gamma:{ }^\gamma V_\lambda\xoto{\cong} V_\lambda$.
			Using this we can ensure that the cocycles $\varphi_{\rho}$ or $\varphi_D$ in the statements of Theorem \ref{thm:twisted verlinde}, Corollary \ref{cor:verlindeforcrossedmodfun} are trivial for twisted conformal blocks (see also \cite[Rem. 2.17(2)]{AHR:18}) if $\Gamma$ preserves a Borel subalgebra. Hence we obtain:
		\end{remark}
		\begin{corollary}\label{cor:main1}
			The Verlinde formula (Theorem \ref{conj:main1}) holds for twisted conformal blocks associated with $\Gamma$-twisted affine Lie algebras with the crossed S-matrices given by Proposition \ref{prop:charactertableviaSmatrix}.
		\end{corollary}
		\begin{proof}
			By Theorem \ref{conj:twistedaffineconfblocks} the twisted conformal blocks associated with $\Gamma$-twisted affine Lie algebras define a $\Ccal$-extended $\Gamma$-crossed modular functor in arbitrary genus. Moreover by the  untwisted result Theorem \ref{thm:oldconf}, the neutral component $\Ccal_1$ is rigid. Hence the hypothesis of Corollary \ref{cor:verlindeforcrossedmodfun} hold for this $\Gamma$-crossed modular functor and hence $\Ccal$ is a $\Gamma$-crossed modular fusion category. Moreover we have seen in the Remark \ref{rem:cocyleszero} that the cocycles that appear in Corollary \ref{cor:verlindeforcrossedmodfun} in this case are all trivial. Hence Corollary \ref{cor:verlindeforcrossedmodfun} Equation (\ref{eq:verlindeforcrossedmodfun}) applied to this case gives us the Verlinde formula Theorem \ref{conj:main1}. The associated categorical crossed S-matrices have already been computed in Proposition \ref{prop:charactertableviaSmatrix}.
			%
			%
		\end{proof}

		\section{Proof of Theorem \ref{conj:twistedaffineconfblocks}}\label{sec:proofofconfgammafunc} In this section, we prove that the twisted conformal blocks, under the assumption that ``$\Gamma$ preserves a Borel subalgebra of $\frg$" satisfy the axioms of a $\mathcal{C}$-extended $\Gamma$-crossed modular functors. Our proof of Theorem \ref{conj:twistedaffineconfblocks} follows the 
		same line of proofs as in Chapter 7 of \cite{BK:01}. We construct the morphism of 
		functors $G_{f,\gamma}$ (see Definition \ref{d:gammacrossedmodfun}) in the case of twisted conformal blocks. 
		\subsection{Verdier Specialization}\label{sec:specializationdeverdier} Let $S$ be a complex manifold and $D$ be a smooth divisor in $S$ with ideal sheaf $\mathcal{I}_D$. Let $\operatorname{Conn}^{\reg}(S,D)$ be the category of integrable connections on $S\backslash D$ with logarithmic singularities along $D$. The category $\operatorname{Conn}^{\reg}(S,D)$ is a subcategory of $\operatorname{Hol}(S)$ of holonomic $\mathcal{D}_S$-modules on $S$.

		Let $j: S\backslash D \hookrightarrow S$ be the inclusion, then any object of $\operatorname{Conn}^{\reg}(S,D)$ is given by a locally free $\mathcal{O}_S$-module $\mathcal{F}$ of finite rank along with an action of $\mathcal{D}_S^0$, where $\mathcal{D}_S^0$ is the ring of differential operators on $S$ that preserve the ideal $\mathcal{I}_D$. The corresponding holonomic $\mathcal{D}_S$-module is $j_*\mathcal{O}_{S\backslash D}\otimes_{\mathcal{O}_S}\mathcal{F}$. Moreover, by Deligne's Riemann-Hilbert correspondence \cite{Deligne1}, we get an equivalence between $\operatorname{Conn}^{\reg}(S,D)$ and the category $\operatorname{Loc}(S\setminus D)$ of local systems on $S\setminus D$.

		Further, let $ND$ be the normal bundle of the smooth divisor $D$ in $S$, the Verdier specialization \cite{Ver1,Ver2} functor $\widetilde{Sp}_D: \operatorname{Hol}(S)\rightarrow \operatorname{Hol}(ND)$, restricts to a functor 
		\begin{equation}\label{eqn:special}
			\widetilde{\operatorname{Sp}}_D:\operatorname{Conn}^{\reg}(S,D)\rightarrow \operatorname{Conn}^{\reg}(ND,D).\end{equation}
		
		We have the following diagram: 
		$$
		\xymatrix{
			\operatorname{Conn}^{\reg}(S,D)\ar[r]^{\widetilde{Sp}_D} &\operatorname{Conn}^{\reg}(ND,D)
			\\
			\operatorname{Loc}(S\backslash D) \ar[r]\ar[u] &  \operatorname{Loc}(ND\backslash D),\ar[u]\\}
		$$ where the vertical arrows  are equivalence of categories given by Deligne's Riemann-Hilbert correspondence. The horizontal arrow on the bottom is given by restrictions of representations of the fundamental group obtained by applying the tubular neighborhood theorem.  The functor $\widetilde{\operatorname{Sp}}_D$ defined in Equation \eqref{eqn:special} can be described as follows. 
		
		The structure sheaf of the normal bundle $\mathcal{O}_{ND}$ is isomorphic to ${\bigoplus}_{m\geq 0}\mathcal{I}^m/\mathcal{I}^{m+1}$ and $\mathcal{D}_{ND}^{0}={\bigoplus}_{m\geq 0}\mathcal{I}^m\mathcal{D}_{S}^{0}/\mathcal{I}^{m+1}\mathcal{D}^0_S$. Now given a $\mathcal{D}_S^0$-module $\mathcal{F}$, the module 
		\begin{equation}\label{eqn:specializationformula}
		\widetilde{\operatorname{Sp}}_D(\mathcal{F}):={\bigoplus}_{m\geq 0}\mathcal{I}^m\mathcal{F}/\mathcal{I}^{m+1}\mathcal{F}.
		\end{equation}
		is naturally a $\mathcal{O}_{ND}$-module as well as a $\mathcal{D}^0_{ND}$-module. This (Equation \eqref{eqn:specializationformula}) definition of specialization functor is just the graded sum of the $V$-filtration on the $\mathcal{D}_S^0$-module $\mathcal{F}$. The notion of $V$-filtration is due to Kashiwara \cite{Kashiwara} and Malgrange \cite{Malgrange} and the fact that Equation \eqref{eqn:specializationformula} gives the Verdier specialization follows generalizations of the definition of the nearby cycles functor using the $V$-filtration. We refer the reader to  \cite[Theorem 4.7.8.5]{Ginzburg} and \cite[Section 3]{Saito}.

		Let ${\Lambda}$ be a line bundle on $S$ as before and let $D$ be a smooth divisor. Let $c$ be a rational number and consider the category $\mathscr{D}_{\Lambda,c}\operatorname{Mod}(S)$ of locally free sheaves of $\mathcal{O}_S$-modules with an action of $\mathcal{A}_{c\Lambda}(-\log D)$. Formula \eqref{eqn:specializationformula} gives a functor:
		$$\widetilde{\operatorname{Sp}}_D: \mathscr{D}_{\Lambda,c}\operatorname{Mod}(S)\rightarrow \mathscr{D}_{p^*\Lambda,c}\operatorname{Mod}(ND),$$ where $p: ND\rightarrow S$ factors through the natural projection.

		\subsection{Gluing functor along a smooth divisor}\label{sec:gluingfunckirillov}
		Let $(\widetilde{C},C, \widetilde{\bf {p}},{\bf p},\widetilde{\bf{v}} )$ be a family of $n$-pointed $\Gamma$-covers with chosen non-zero tangent vectors 
		at the 
		marked points in $\GModtilde(\bf{m})$ parameterized by a smooth variety $S$.  Consider a smooth divisor $D$ in $S$ such that $C\rightarrow S$ 
		restricted to $D$ is a family of stable $n$-pointed curves with exactly one node. We consider the variety 
		$$\widetilde{N}D=\{(d,\widetilde{v}',\widetilde{v}'') | \ d \in D, \widetilde{v}'\in T'_a\widetilde{C}_d, \widetilde{v}'' \in T''_a \widetilde{C}_d\},$$ where $\widetilde{C}_d$ is 
		the fiber of $\pi: \widetilde{C} \rightarrow S$ at a point $d \in D$, $a$ is any chosen point of $\widetilde{C}_d$ which lifts the  node in $C_d$ and $T'_a\widetilde{C}_d$ (respectively $T''_a\widetilde{C}$) is the 
		tangent space to $\widetilde{C}_d$ at $a$ along the two components. By Equation \eqref{eqn:decompositionofnormalbundle}, the normal bundle to the divisor $ND$ has fiber $T'_a \widetilde{C}_d\otimes T''_a\widetilde{C}_d$. Hence there is a natural map $ \widetilde{N}D\rightarrow ND$ that sends the tuple $(d,\widetilde{v}',\widetilde{v}'')\rightarrow (d,\widetilde{v}'\otimes \widetilde{v}'')$. Now let $FND$ be 
		the frame bundle of $ND$ that preserves the decomposition as in Equation \eqref{eqn:decompositionofnormalbundle} and let $\widetilde{N}^{\times}D$ be the variety 
		$$\widetilde{N}^{\times}D:=\{(d,\widetilde{v}',\widetilde{v}'')\in \widetilde{N}D| \ \widetilde{v}'\neq 0 \ \mbox{and} \ \widetilde{v}''\neq 0\}.$$  The natural map $\widetilde{N}D \rightarrow ND$ restricted to $ \widetilde{N}^{\times}D\rightarrow FND$ is a $\mathbb{G}_m$-torsor. From Diagram \eqref{eq:clutching for normal bundle}, we get the following:
		$$
		\xymatrix{
			\widetilde{N}^{\times}D\ar[r]& FND \ar[d]^{\mathbb{G}_m-torsor} & \\
			&D\ar[r]&S
		}
		$$
		Given the family $\widetilde{C}\rightarrow S$ restricted to $D$, we associate a natural family $\widetilde{C}_N\rightarrow \widetilde{N}^{\times}D$ of $n+1$-pointed 
		$\Gamma$ covers in $\GModtilde(\bf{m},\gamma, \gamma^{-1})$ parameterized by $\widetilde{N}^{\times}D$. The fiber over a point $(d,\widetilde{v}',\widetilde{v}'')\in \widetilde{N}^{\times}D$, 
		is the $(n+2)$-pointed curve $\widetilde{C}_{N,d}$ obtained by normalizing the curve $\widetilde{C}_d$, two more marked points $\widetilde{q}_1$ and 
		$\widetilde{q}_2$ (with image $q_1$ and $q_2$ in $C_N=\widetilde{C}_N/\Gamma$) which is in the preimage of the nodal point $a$ and $\widetilde{v}'$ 
		(respectively $\widetilde{v}''$) are non-zero tangent vectors at $q_1$ (respectively $q_2)$ to the curve $\widetilde{C}_{N,d}$.
		To the data $(\widetilde{C}_N,C_N, {\widetilde{\bf{p}}},\widetilde{q}_1,\widetilde{q}_2, {\bf{p}}, q_1,q_2, \widetilde{{\bf{v}}},\widetilde{v}',\widetilde{v}'')$, 
		we associate the vector bundle of twisted covacua:
		\begin{equation}
		\label{eqn:vectorbundleonnormalization}
		\mathcal{V}_{\vec{\lambda},\Gamma}(\widetilde{N}^{\times}D):=\bigoplus_{\mu \in P^{\ell}(\frg,\gamma)}\mathcal{V}_{\vec{\lambda}\cup\{\mu,\mu^*\},\Gamma}(\widetilde{C}_N,C_N, {\widetilde{\bf{p}}},\widetilde{q}_1,\widetilde{q}_2, {\bf{p}}, q_1,q_2, \widetilde{{\bf{v}}},\widetilde{v}',\widetilde{v}'').
		\end{equation} By the constructions in Section \ref{sec:twistedconnectiontwistedKZ}, the vector bundle in Equation \eqref{eqn:vectorbundleonnormalization} is also a twisted $\mathcal{D}$-module on $\widetilde{N}^{\times}D$. 
		Now to define the {\em gluing functor} (see Definition \ref{d:gammacrossedmodfun} ), it is enough to construct an isomorphism of twisted $\mathcal{D}_{\widetilde{N}^{\times}D}$-modules 
		$G_{\gamma}:\operatorname{Sp}_{D}(\mathcal{V}_{\vec{\lambda},\Gamma}(\widetilde{C},C,\widetilde{{\bf p}},{\bf{p}},\widetilde{\bf{v}})\simeq \mathcal{V}_{\vec{\lambda},\Gamma}(\widetilde{N}^{\times}D), $
		where $\operatorname{Sp}_{D}$ is the functor defined in Section \ref{ss:specalongclutch} and Equation \eqref{eqn:specializationdeverdier}.
		
		By the same argument as in Chapter 7.8 in \cite{BK:01}, we get that $\mathcal{V}_{\vec{\lambda},\Gamma}(\widetilde{N}^{\times }D)$ is $\mathbb{C}^{\times}$-equivariant and monodromic and hence descends to a vector bundle (also denoted by $\mathcal{V}_{\vec{\lambda},\Gamma}(N^{\times }D))$ on $ND$ with a projective action of the Lie algebra $\mathcal{D}^{0}_{ND}$-differential operators on $ND$ that preserve the zero section. Also by construction and discussion in Section \ref{sec:specializationdeverdier}, we get   $\operatorname{Sp}_D$ is the pullback from $N^{\times}D$, of the functor $\widetilde{\operatorname{Sp}}_D$. Hence it is enough to show that there exists isomorphisms. 
		$\widetilde{G}_{\gamma}: \widetilde{\operatorname{Sp}}_{D}(\mathcal{V}_{\vec{\lambda},\Gamma}(\widetilde{C},C,\widetilde{{\bf p}},{\bf{p}},\widetilde{\bf{v}})\simeq \mathcal{V}_{\vec{\lambda},\Gamma}({N}^{\times}D).$
		
		\subsection{Families over a formal base}As before let $D$ be a smooth divisor in a smooth scheme $S$ and let $\tau$ be a local equation of $D$ in $S$. 
		The $m$-th infinitesimal neighborhood $D^{(m)}$ has structure sheaf $\mathcal{O}_S/\tau^{m+1}\mathcal{O}_S$.  
		
		Given a family of nodal $\Gamma$ covers as in Section \ref{sec:gluingfunckirillov}, we choose formal parameters at the marked points and get a family  
		$(\widetilde{C},C, \widetilde{\bf p},{\bf{p}},\widetilde{\bf z})$ of $\Gamma$-covers  in $\GModhat({\bf m})$ parameterized by a smooth scheme $S$. 
		We consider the family $(\widetilde{C}_N, C_N, \widetilde{\bf{p}}', {\bf p}', \widetilde{{\bf z}}')$ 
		parameterized by $D$ obtained by equivariant  normalization of $\widetilde{C}\rightarrow S$ restricted to $D$. Let $\eta_1$ and $\eta_2$ be two chosen coordinates on the two 
		components of the family $\widetilde{C}\rightarrow S$ in a neighborhood of a double point $a(d)$ so that $\eta_1\eta_2=\tau$ and moreover, the double point is given by the equation 
		$\eta_1=0$ and $\eta_2=0$.  Since the group $\Gamma$ acts transitively on $\widetilde{C}$ and by assumption $C$ has only one node, the other double points of $\widetilde{C}$ are in 
		the orbit $\Gamma. a$. 
		
		On the normalization $\widetilde{C}_N$, both $\eta_1$ and $\eta_2$ are formal coordinates of the inverse image $\widetilde{q}_1(d)$ and $\widetilde{q}_2(d)$ of $a(d)$. By the $\Gamma$-action 
		we get formal coordinates in the full inverse image of the $\Gamma$ orbit of $a(d)$. Hence we get a family of $\Gamma$-covers in $\GModhat({\bf m}, \gamma,\gamma^{-1})$. In this section, 
		following the proof of Theorem  7.8.5 in \cite{BK:01}, we extend the family $(\widetilde{C},C, \widetilde{\bf p},{\bf{p}},\widetilde{\bf z})$ to a family  $\widetilde{C}_{N,D^{(m)}}\rightarrow D^{(m)}$ 
		of $(n+2)$-pointed $\Gamma$-covers in $\GModhat({\bf{m}, \gamma,\gamma^{-1}})$ parameterized by the $m$-th infinitesimal neighborhood $D^{(m)}$ of $D$.
		\subsubsection{Family of $\Gamma$-covers over $m$-th infinitesimal neighborhood} We need to first define a sheaf of algebras $\mathcal{O}_{\widetilde{C}_{N,D}}^{(m)}$ over $\widetilde{C}_N$ 
		with the structure of a flat $\mathcal{O}_{D^{(m)}}$-module. Let $U=\widetilde{C}_N\backslash \Gamma.\widetilde{q}_1(D)\cup \Gamma.\widetilde{q}_2(D)$. Since $\widetilde{C}_N$ is the normalization 
		of $\widetilde{C}$, we get that $U$ is also equal to $\widetilde{C}\backslash \Gamma.a(D)$. We define  ${\mathcal{O}_{\widetilde{C}_{N,D}}^{(m)}}_{|U}:={\mathcal{O}_{\widetilde{C}_D}^{(m)}}_{|U}$,
		where $\widetilde{C}_D$ is the restriction of the family $C\rightarrow S $ to $D$ and $\mathcal{O}_{\widetilde{C}_D}^{(m)}$ is the flat $\mathcal{O}_{D^{(m)}}$-module given 
		by $\mathcal{O}_{\widetilde{C}}/\tau^{m+1}\mathcal{O}_{\widetilde{C}}$. We extend the sheaf of $\mathcal{O}_{\widetilde{C}_{N,D}}^{(m)}$ across $\Gamma. \widetilde{q}_1$ and $\Gamma.\widetilde{q}_2$ 
		by defining the stalks at $\widetilde{q}_1$ and $\widetilde{q}_2$ and 
		extending it to the 
		orbits 
		of $\widetilde{q}_1$ and $\widetilde{q}_2$  via the $\Gamma$ action. 
		
		As in \cite{BK:01}, let $\mathcal{O}_0(\eta_1)$ (respectively $\eta_2$) be the ring of germs of analytic functions in $\eta_1$ 
		(respectively $\eta_2$) in  a 
		neighborhood of $\eta_1=0$ (respectively $\eta_2=0$). We define the stalks at $\widetilde{q}_1$ and $\widetilde{q}_2$ to be $\mathcal{O}_{\widetilde{C}_{N,D},\widetilde{q}_i}^{(m)}:=\mathcal{O}_0(\eta_i)\otimes \mathcal{O}_D^{(m)}$ for $1\leq i\leq 2$. For each $\widetilde{q}_i$ in the $\Gamma$-orbits of the points given by normalization, the map $\eta_i \rightarrow \eta_i$ and $\tau\rightarrow \eta_1\eta_2$ 
		gives a 
		map from $\mathcal{O}_{\widetilde{C}_N,D,\widetilde{q}_i}$ to $\mathcal{O}_{\widetilde{C}_N,D}^{(m)}(V_i)$ where $V_i\subset U$ is a punctured neighborhood  of $\widetilde{q}_i$. Hence we 
		define the sheaf $\mathcal{O}_{\widetilde{C}_{N,D}}^{(m)}$ as the gluing given by the morphism $\tau \rightarrow \eta_1\eta_2$. The action of $\Gamma$ is stable around the nodes and the local picture at a node $a$ 
		is given by the following:
		$$\operatorname{Spec}(A[x,y]/(xy=\tau))\xrightarrow{f} \operatorname{Spec}(A[x',y']/(x'y'=\tau^N))\rightarrow \operatorname{Spec}A,$$
		where $f^*x'=x$, $f^*y'=y$ and $\tau\in A$. 
		
		Moreover, if $\gamma$ is a generator of $\Gamma_a$, then $\gamma$ acts on $x$ (respectively $y$ ) by $\zeta$ (respectively $\zeta^{-1}$),
		where $\zeta$ is a chosen $N$-th root of unity and the stabilizer of the nodes is cyclic of order $N$. From this it directly follows that the gluing morphism is $\Gamma$-equivariant and hence $\mathcal{O}_{\widetilde{C}_{N,D}}^{(m)}$ defines 
		a family of curves with $\Gamma$-action. The family $\widetilde{C}_{N, D^{(m)}}$ defines a family of $(n+2)$-pointed $\Gamma$ covers where the section $\widetilde{{\bf p}}$ extend to sections of $\widetilde{C}_{N,D^{(m)}}$ trivially as $\widetilde{C}_{N,D^{(m)}}$ coincides with $\mathcal{O}^{(m)}_{\widetilde{C}_D}$ away from the locus of $\Gamma. \widetilde{q}_1 \cup \Gamma.\widetilde{q}_2$ along with formal parameters $\widetilde{\eta}$. 
		Moreover, $\widetilde{q}_1$ and $\widetilde{q}_2$ provide two additional sections with formal 
		parameters $\eta_1$ and $\eta_2$. 
		\subsection{Sheaf of covacua over a formal base}Let $\mathcal{F}$ be a $\mathcal{O}_S$ quasicoherent sheaf on $S$, then $\mathcal{F}^{(m)}:=\mathcal{F}/\tau^{m+1}\mathcal{F}$ gives 
		a sheaf over $D^{(m)}$. If $\mathcal{F}$ is coherent, then $\mathcal{F}^{(m)}$ is also coherent (Lemma 7.6.2 in \cite{BK:01}).  Let $(\widetilde{C}, C, \widetilde{{\bf p}}, {\bf p}, \widetilde{{\bf z}})$ be a family of $\Gamma$ covers, we 
		can associate the following locally free sheaf of covacua on $S$ of finite rank 
		$$\mathcal{V}_{\vec{\lambda},\Gamma}(\widetilde{C}, C, \widetilde{{\bf p}}, {\bf p}, \widetilde{{\bf z}}):=\mathcal{H}_{\vec{\lambda}}\otimes \mathcal{O}_S/
		(\frg\otimes \mathcal{O}_{\widetilde{C}}(\ast \Gamma.\widetilde{\bf{p}})^{\Gamma})\cdot\mathcal{H}_{\vec{\lambda}}\otimes \mathcal{O}_S.$$ 
		We denote by $\mathcal{V}^{(m)}_{\vec{\lambda},\Gamma}(\widetilde{C}, C, \widetilde{{\bf p}}, {\bf p}, \widetilde{{\bf z}})$, the coherent sheaf over $D^{(m)}$. By Lemma 7.6.2 in \cite{BK:01}, 
		we get the following lemma:
		\begin{lemma}\label{lemma:conformalformal}The coherent sheaf
			$\mathcal{V}^{(m)}_{\vec{\lambda},\Gamma}(\widetilde{C}, C, \widetilde{{\bf p}}, {\bf p}, \widetilde{{\bf z}})$ is isomorphic to 
			the quotient 
			of $\mathcal{H}_{\vec{\lambda}}\otimes \mathcal{O}_{D^{(m)}}$ by the sheaf $(\frg\otimes \mathcal{O}_{\widetilde{C}_{D}}^{(m)}(\ast \Gamma .\widetilde{\bf p})^{\Gamma}).\mathcal{H}_{\vec{\lambda}}\otimes \mathcal{O}_{D^{(m)}}$. \end{lemma}
		
		Thus given an $n$-tuple $\vec{\lambda}$ of level $\ell$ weights, and a family of curves over $D^{(m)}$, we can use Lemma \ref{lemma:conformalformal} to define 
		(see also Equation 7.6.3 in \cite{BK:01}) the sheaf of covacua $\mathcal{V}_{\vec{\lambda},\Gamma}(\widetilde{C}_{N,D^{(m)}}, \widetilde{{\bf p}}', {\bf p}', \widetilde{\bf z}')$ on a $m$-th 
		infinitesimal 
		neighborhood of a divisor $D$. 
		\subsubsection{Twisted $\mathcal{D}$-module structure over a formal base}Let as before $S$ be a smooth scheme and $D$ be a smooth divisor in $S$ and let $\mathcal{I}$ 
		be the ideal sheaf of $D$ in $S$. Let $\mathcal{D}_S$ be the sheaf of first order differential operators on $S$ and let $\mathcal{D}_S^0$. For each positive integer $m$ 
		consider 
		the sheaf $\mathcal{D}^0_{D^{(m)}}:=\mathcal{D}_S^0/\mathcal{I}^{m+1}\mathcal{D}^0_S$. Similarly consider the sheaf $\mathcal{D}^{(0)}_S:=\mathcal{D}^0_S/\mathcal{I}\mathcal{D}^0_S$ of $\mathcal{O}_D$-modules. By Proposition 7.8.4 (iii) in \cite{BK:01}, the sheaf $\mathcal{D}^{0}_S$ is canonically isomorphic to the sheaf $\mathcal{D}_{ND}^{(0)}:=\mathcal{D}^0_{ND}/\mathcal{I}\mathcal{D}^0_{ND}$, where $\mathcal{D}^0_{ND}$ is the sheaf of differential operators on $ND$ that preserve the ideal $\mathcal{I}$ and $\mathcal{I}$ is the ideal sheaf of the divisor $D$ in $ND$.

		Let $\mathcal{F}$ be a sheaf of $\mathcal{O}_S$-module and let $\operatorname{Gr}^m(\mathcal{F}):=\mathcal{I}^{m}\mathcal{F}/\mathcal{I}^{m+1}\mathcal{F}$, then $\operatorname{Gr}^m(\mathcal{F})$ are $\mathcal{D}^{(0)}_S$-modules. Moreover, we can define 
		\begin{equation}
		\widetilde{\operatorname{F}}^m(\widetilde{\operatorname{Sp}}_D(\mathcal{F})):= \mathcal{I}^m\widetilde{\operatorname{Sp}}_D(\mathcal{F})/\mathcal{I}^{m+1}\widetilde{\operatorname{Sp}}_D(\mathcal{F})=\operatorname{G}^m(\mathcal{F}).
		\end{equation}

	\subsection{Twisted sewing functor}Let $\mu \in P^{\ell}(\frg,\gamma)$ and $\mu^* \in P^{\ell}(\frg,\gamma^{-1})$ and let $\mathcal{H}_{\mu}(\frg,\gamma)$ and
	$\mathcal{H}_{\mu^*}(\frg,\gamma^{-1})$ be the corresponding highest weights integrable 
	modules for the twisted affine Lie algebra $\widehat{L}(\frg,\gamma)$ and $\widehat{L}(\frg,\gamma^{-1})$ respectively. Let $L_{k,{\langle \gamma\rangle }}$ be the $k$-th $\gamma$-twisted Virasoro operator as constructed in \cite{KacWakimoto:88,Wakimoto}. There is a natural grading on $\mathcal{H}_{\mu}(\frg,\gamma)$ by $\mathbb{Z}_{+}$. They are defined as follows: 
	\begin{equation}\label{eqn:eigen1}
	\mathcal{H}_{\mu}(\frg,\gamma)(d):=\{ |\Phi\rangle| L_{0,{\langle \gamma\rangle }}(|\Phi\rangle )=(\Delta_{\mu}+\frac{d}{N})|\Phi\rangle \},
	\end{equation}
	where $\Delta_{\mu}$ is eigen value of the twisted $L_{0,{\langle \gamma \rangle }}$ acting on the finite dimensional highest weight $\frg^{\gamma}$ module with highest weights $\mu$ and $N$ is the order of $\gamma$. From the explicit description of $\mathcal{H}_{\mu}(\frg,\gamma)$ as a quotient of a Verma module and the fact (\cite{Wakimoto},\cite{KacWakimoto:88}) that $[L_{0,{\langle \gamma \rangle }},X(m)]=-\frac{m}{N}X(m),$ it follows that 
	$\mathcal{H}_{\mu}(\frg,\gamma)=\bigoplus_{i=0}^{\infty}\mathcal{H}_{\mu}(\frg,\gamma)(d).$
	
	By Lemma 8.4 in \cite{KH}(see 
	Lemma 4.14 in \cite{U} for the untwisted case), there exists a unique up to constant 
	non-degenerate bilinear form 
	\begin{equation}\label{eqn:bilform}
	( \ | \ )_{\mu}: \mathcal{H}_{\mu}(\frg,\gamma )\otimes \mathcal{H}_{\mu^*}(\frg,\gamma^{-1})\rightarrow \mathbb{C}.
	\end{equation}
	such that
	for any  $|\Phi_1 \rangle $, $|\Phi_2\rangle $ in $\mathcal{H}_{\mu}(\frg,\gamma)$ and $\mathcal{H}_{\mu^*}(\frg,\gamma^{-1})$, we have the following equation:
	\begin{equation}\label{eqn:bilo}
	(X({m}).|\Phi_1 \rangle||\Phi_2\rangle  )_{\mu}+ (|\Phi_1\rangle | X({-m}).|\Phi_2\rangle ))_{\mu}=0,
	\end{equation}
	where $m$ is an integer and $X(m)=X\otimes t^{m}\in \widehat{L}(\frg,\gamma)$. 
	
	Moreover, the form $( \ | \ )_{\mu}$ has the following properties: 
	\begin{itemize}
		\item $( \ | \ )_{\mu}$ restricted to  $\mathcal{H}_{\mu}(\frg,\gamma)(d)\otimes \mathcal{H}_{\mu^*}(\frg,\gamma^{-1})(d')$ is zero unless $d=d'$. 
		\item $( \ | \ )_{\mu}$ restricted to $V_{\lambda}\otimes V_{\lambda^*}$ is the unique (up to scalars) $\frg^{\gamma}$-invariant bilinear form.
	\end{itemize}
	Here $\mathcal{H}_{\mu}(\frg,\gamma)(d)$ denote the degree $d$ part of $\mathcal{H}_{\mu}(\frg,\gamma)$. 
	
	Let $m_d$ be the dimension of $\mathcal{H}_{\mu}(\frg,\gamma)(d)$ and 
	we choose a basis $\{v_{\mu,1}(d),\dots, v_{\mu,m_d}(d)\}$ of $\mathcal{H}_{\mu}(\frg,\gamma)$ and let $\{v^{\mu,1}(d),\dots, v^{\mu,m_d}(d)\}$ be the basis
	of $\mathcal{H}_{\mu^*}(\frg,\gamma^{-1})$ dual 
	with respect to the bilinear form $( \ | \ )_{\mu}$. 
	
	Consider the following element (see Section 8 in \cite{KH}, Chapter 4 in \cite{U}):
	$$\gamma_{\mu,d}=\sum_{i=1}^{m_d}v_{\mu,i}(d)\otimes v^{\mu,i}(d)\in \mathcal{H}_{\mu}(\frg,\gamma)(d)\otimes \mathcal{H}_{\mu^*}(\frg,\gamma^{-1})(d)$$ and the {\em twisted
		sewing element} $\gamma_\mu:=\sum_{d\geq 0}\gamma_{\mu,d}\tau^d \in \mathcal{H}_{\mu}(\frg,\gamma)\otimes \mathcal{H}_{\mu^*}(\frg,\gamma^{-1})[[\tau]].$
	It
	is easy to check that both $\gamma_{\mu,d}$ and $\gamma_{\mu}$ are independent of the chosen basis of $\mathcal{H}_{\mu}(\frg,\gamma)(d)$. We further, observe that for each $\mu \in \mathcal{H})_{\mu}(\frg,\gamma)$, we have the following equality  which follows directly from the formula in Lemma 2.3 in \cite{Wakimoto}:
	\begin{equation}
	\label{eqn:equalityoftraceanomaly}
	\Delta_{\mu}=\Delta_{\mu^*}.
	\end{equation}
	Using the 
	twisted sewing element $\gamma_{\mu}$ one defines a map: 
	\begin{align*}
	\quad \iota_{\mu}:\mathcal{H}_{\vec{\lambda}}\otimes \mathcal{O}_{D}[[\tau]]\rightarrow \mathcal{H}_{\vec{\lambda}}\otimes \mathcal{H}_{\mu}(\frg,\gamma)\otimes \mathcal{H}_{\mu^*}(\frg,\gamma^{-1})[[\tau]], \quad
	\sum_{i} |\Phi \rangle \tau^{i}\rightarrow \sum_{i,d}|\Phi\rangle \otimes \gamma_{\mu}.
	\end{align*}
	We have the following proposition:
	\begin{proposition}\label{prop:importantflateness} The map
		$\bigoplus_{\mu \in P^{\ell}(\frg,\gamma)}\iota_{\mu}$ induces an isomorphism of $\mathcal{O}_{D^{(m)}}$-modules $\mathcal{V}_{\vec{\lambda},\Gamma}^{(m)}(\widetilde{C},C, \widetilde{\bf{p}}, {\bf p}, \widetilde{\bf{z}})$ and $\mathcal{V}_{\vec{\lambda},\Gamma}(\widetilde{C}_{N,D^{(m)}},\widetilde{\bf{p}}', {\bf p}', \widetilde{\bf{z}}')$ for each $m\geq 0$. Moreover, the projective action of $\mathcal{D}^0_{D^{(m)}}$ preserve each component $\iota_{\mu}$ and the projective action of $\tau\partial_{\tau}$ commutes projectively (i.e up to a scalar multiplication) with $\iota_{\mu}$. 
	\end{proposition}
	\begin{proof}
		The proof of the first part of the proposition can be found in \cite{KH} and for the untwisted case we refer the reader to \cite[Section 6]{TUY:89}, \cite{BK:01} and \cite{Looijenga}. We now discuss the second part of the proof. We can lift the vector field $\theta=\tau\partial_{\tau}$ to an $\Gamma$ invariant vector field  $\widehat{\theta}$ over $\widetilde{C}_{N,D^{(m)}}$ such that around $\widetilde{q}_1$ and $\widetilde{q}_2$, the local expansion is of the form $\alpha.\eta_1 \partial_{ \eta_1}$ and $\beta. \eta_2 \partial_{\eta_2}$ and $\alpha+\beta=1$. Here $\eta_1$ and $\eta_2$ are special formal coordinates. Thus we need to show that the projective action of $\tau\partial_{\tau}$ commutes up to a fixed scalar with the map $\iota_{\mu}$ as a map of conformal blocks. This follows directly from Lemma \ref{lem:proj}.
	\end{proof}
	\begin{lemma}\label{lem:proj}The following equality holds as operators on $\mathcal{H}_{\mu}(\frg,\gamma)\otimes \mathcal{H}_{\mu^*}(\frg,\gamma^{-1})[[\tau]]$:
		$$\tau\partial_{\tau}(f|\Phi\rangle \otimes \gamma_{\mu,d}\tau^{d})-(\tau\partial_{\tau}(f|\Phi\rangle ))\otimes \gamma_{\mu,d}\tau^{d}=-N.\Delta_{\mu}(f|\Phi\rangle \otimes \gamma_{\mu,d}\tau^d),$$ where $N$ is the order of $\gamma$.
	\end{lemma}
	\begin{proof}
		By the definition of the action of $\tau\partial_{\tau}$ and applying the Liebnitz rule, we get that the left hand side of the equation in the statement of the lemma is equivalent to 
		\begin{eqnarray*}
			&&f|\Phi\rangle\otimes (\tau\partial_{\tau}.(\gamma_{\mu,d}\tau^d))\\ 
			&&\quad =d.(f|\Phi\rangle \otimes \gamma_{\mu,d}\tau^{d})+f|\Phi\rangle \otimes (\tau\partial_{\tau}.\gamma_{\mu,d})\tau^d\\
			&&\quad = d.(f|\Phi\rangle \otimes \gamma_{\mu,d}\tau^{d})+\sum_{i=1}^{m_d}f|\Phi\rangle (\tau\partial_\tau.(v_{\mu,i}\otimes v^{\mu,i}))\tau^d\\
			&&\quad =d.(f|\Phi\rangle \otimes \gamma_{\mu,d}\tau^{d}) +\sum_{i=1}^{m_d}f|\Phi\rangle \bigg(\alpha.(\eta_1\partial_{\eta_1}.v_{\mu,i})\otimes v^{\mu,i}
			+\beta.v_{\mu,i}\otimes (\eta_2\partial_{\eta_2}.v^{\mu,i})\bigg)\tau^d\\
			&&\quad =d.(f|\Phi\rangle \otimes \gamma_{\mu,d}\tau^{d})\\
			&&\quad \quad + \sum_{i=1}^{m_d}f|\Phi\rangle\bigg(\alpha.(-N\Delta_{\mu}-d)v_{\mu,i}\otimes v^{\mu,i})
			+\beta.v_{\mu,i}(-N\Delta_{\mu}-d)v^{\mu,i}\bigg)\tau^d.
		\end{eqnarray*}
		Here we have used the fact that $\eta_i\partial_{\eta_i}$ acts (see Equation \eqref{eqn:viractionimp}) by $-NL_{0,{\langle \gamma \rangle}}$ and Equation \eqref{eqn:eigen1} to compute the action of $L_{0,\langle \gamma\rangle }$. Now since by the choice of the lift $\alpha+\beta=1$, the lemma follows.
	\end{proof}
	We have the following corollary of Proposition \ref{prop:importantflateness}:
	\begin{corollary}\label{cor:importantflatness}The maps $\bigoplus_{\mu}\iota_{\mu}$ induce an isomorphism between $ \operatorname{G}^m(\mathcal{V}_{\vec{\lambda},\Gamma}(\widetilde{C},C,\widetilde{\bf p}, {\bf p}, \widetilde{\bf z}))$ and $\operatorname{G}^m(\mathcal{V}_{\vec{\lambda},\Gamma}(N^{\times }D))$ that preserves the projective action of $\mathcal{D}_S^{(0)}$ on the left and $\mathcal{D}_{ND}^{(0)}$-action on the right for all $m\geq0$.
	\end{corollary}
	\begin{proof}The proof of the corollary follows from Proposition \ref{prop:importantflateness} along with the observation that $\operatorname{G}^m(\mathcal{F})$ is the kernel of the natural surjective map from $\mathcal{F}^{(m)}\twoheadrightarrow \mathcal{F}^{(m-1)}$.
	\end{proof}
	Since the vector bundles of twisted covacua descend as $\mathcal{D}^0$ modules on $\GModhat({\bf m})$, without loss of generality, we can assume that we have chosen formal neighborhood around the marked points. Now the twisted sewing construction and the Corollary \ref{cor:importantflatness} imply the following: 
	\begin{theorem}\label{thm:specializationfunctorexplicit}The sewing construction induces an isomorphism of locally free sheaves 
		$$\widetilde{G}_{\gamma}: \widetilde{\operatorname{Sp}}_D(\mathcal{V}_{\vec{\lambda},\Gamma}(\widetilde{C},{C},\widetilde{\bf p}, {\bf p}, \widetilde{{\bf z}}))\simeq \mathcal{V}_{\vec{\lambda},\Gamma}(N^{\times} D)$$ which preserves the projective action of $\mathcal{D}^0_{ND}$. Here  $\mathcal{V}_{\vec{\lambda},\Gamma}(N^{\times} D)$ is as in Equation \eqref{eqn:vectorbundleonnormalization}.
	\end{theorem}
	\begin{proof}The non-degenerate bilinear form $( \ | \ )_{\mu}$ given by Equation \eqref{eqn:bilo} induces a canonical element $\gamma_{\mu,d} \in \mathcal{H}_{\mu}(\frg,\gamma)\otimes \mathcal{H}_{\mu^*}(\frg,\gamma^{-1})$. Now consider the element $\sum_{d\geq 0}\gamma_{\mu,d}$ and	the following map 
		$$\widetilde{G}_{\gamma}:\mathcal{H}_{\vec{\lambda}}\rightarrow \mathcal{H}_{\vec{\lambda}}\otimes \mathcal{H}_{\mu}\widehat{\otimes}\mathcal{H}_{\mu^{*}}, \ |\varphi\rangle \rightarrow |\varphi \otimes \sum_{d\geq 0}\gamma_{\mu,d}\rangle. $$ The structure sheaf of $ND$ is $\bigoplus_{m\geq 0}\mathcal{I}^m/\mathcal{I}^{m+1}$ and hence all but finitely many elements in the sum $\sum_{d\geq 0}\gamma_{\mu,d}$ lies in $\mathcal{I}^M\mathcal{H}_{\vec{\lambda},\mu,\mu^*}$.
		
		Hence  the map
		$\widetilde{G}_{\gamma}$ induces a map between the sheaves  $\widetilde{\operatorname{Sp}}_{D}(\mathcal{V}_{\vec{\lambda},\Gamma}(\widetilde{C},{C},\widetilde{\bf p}, {\bf p}, \widetilde{{\bf z}}))$ and $\mathcal{V}_{\vec{\lambda},\Gamma}(N^{\times} D)$.  Since the twisted zero-th Virasoro operator $L_{0,\langle \gamma \rangle}$ preserve $\mathcal{H}_{\mu}(\frg,\gamma)(d)$ and acts diagonally, it follows that  $\widetilde{G}_{\gamma}$ preserves the projective connections. Thus, the rank of $\widetilde{G}_{\gamma}$ is constant (\cite[Lemma A.1]{Belkale:09}). We are now reduced to show that the map $\widetilde{G}_{\gamma}$ is an isomorphism in a formal neighborhood of $D$ in $ND$. 

		We will be done if we can show that for any $m\geq0$, the 	induced map $$\widetilde{G}^{(m)}_{\gamma}: \big(\widetilde{\operatorname{Sp}}_{D}\mathcal{F}\big)^{(m)}\rightarrow \mathcal{V}^{(m)}_{\vec{\lambda},\Gamma}({N}^{\times} D).$$ 
		Now the proof of the theorem follows from  Corollary \ref{cor:importantflatness}.
	\end{proof}

	\appendix
	\section{Moduli stacks of covers of curves}\label{ap:modstackofadmcovers}
	\subsection{Moduli stacks of admissible $\Gamma$-covers}\label{ap:modstackofadmcovers1}
	Let $\Gamma$ be a finite group. In this appendix we recall the definitions, basic properties and operations on the stacks $\GMod$ from Jarvis-Kimura-Kaufmann \cite{JKK} as well as the related stacks $\GModhat$ and $\GModtilde$ for non-negative integers $g,n$. These are moduli stacks of (balanced) admissible $\Gamma$-covers of stable $n$-pointed curves of genus $g$ with additional marking data as described in this section. 
	
	Let $\barM_{g,n}$ denote the Deligne-Mumford stack of stable $n$-pointed curves of genus $g$. Let $(C\rightarrow Z, p_1,\dots, p_n)$ be a stable $n$-pointed curve of genus $g$ over a scheme $Z$ with marked sections $p_1,\dots, p_n:Z\to C$. For a finite group $\Gamma$, the notion of a balanced admissible $\Gamma$-cover $\pi:\tildeC\rar{}C$ of such a stable pointed curve is defined in Abramovich-Corti-Vistoli \cite{ACV} which we recall below:
	
	\begin{definition}
		A finite morphism $\pi: \widetilde{C} \rightarrow C$ to a stable $n$-pointed curve $(C\rightarrow Z; p_1,\dots,p_n)$ of genus $g$ is called admissible $\Gamma$-cover if the following holds:
		\begin{itemize}
			\item $\widetilde{C}/Z$ is itself a nodal curve. However $
			\widetilde{C}$ may not be necessarily connected.
			\item There is a left action of the  finite group $\Gamma$ preserving $\pi$ such that restriction of $\pi$ to $C_{gen}$ is a principal $\Gamma$-bundle.  Here $C_{gen}$ are the points of $C$ which are neither marked points nor nodes.
			\item Points of $\widetilde{C}$ lying over marked points of $C$ the map $\widetilde{C}\rightarrow C\rightarrow Z$ locally is of the form:
			$$\operatorname{Spec}A[\widetilde{z}]\rightarrow \operatorname{Spec}A[z]\rightarrow \operatorname{Spec}A, $$ where $z=\widetilde{z}^N$ for some $N>0$.
			\item All nodes of $\widetilde{C}$ map to nodes of $C$. 
			\item All points of $\widetilde{C}$ lying over the nodes of $C$ are nodes and the structure maps of $\widetilde{C}\rightarrow C\rightarrow Z$ are locally same with that of 
			$$\operatorname{Spec}A[\widetilde{z},\widetilde{w}](\widetilde{z}\widetilde{w}-t)\rightarrow \Spec A[z,w](zw-t^r)\rightarrow \Spec A,$$ where $z=\widetilde{z}^r$, $w=\widetilde{w}^r$ for some integer $r>0$ and  $t\in A$ 
			\item The action of the stabilizer $\Gamma_{\widetilde{q}}$ at a node $\widetilde{q} \in \widetilde{C}$ is {\em balanced}: The previous conditions imply that $\Gamma_{\widetilde{q}}$ leaves  the two branches invariant and that $\Gamma_{\widetilde{q}}$ is cyclic. We now also demand that the eigenvalues of the action on the two branches of the tangent space at $\widetilde{q}$ are multiplicative inverses of each other.
		\end{itemize}
	\end{definition}
	
	
	As in \cite{ACV}, we denote the stack of such admissible covers $(\pi:\tildeC\to C;\mathbf{p})$ by $\barM_{g,n}(\Bcal \Gamma)$. Following \cite{JKK}, we now consider a certain variant of the stack of admissible $\Gamma$-covers. For any admissible $\Gamma$-cover $(\tildeC\rightarrow C, \mathbf{p})$, where $\mathbf{p}=(p_1,\cdots,p_n)$, let $\tilde{p}_i\in \pi^{-1}(p_i)$ be a choice of a point in the fiber over $p_i$ for all $1\leq i\leq n$. We denote $\widetilde{\mathbf{p}}=(\tilde{p}_1,\cdots,\tilde{p}_n)$.
	\begin{definition}\label{def:admissiblecovers}
		Let $\GMod$ denote the stack of $n$-pointed admissible $\Gamma$-covers $\pi: \tildeC\rightarrow C,\wtilde{\mathbf{p}},{\mathbf{p}}$ of $n$-pointed, genus-$g$ stable curves, where $\wtilde{\mathbf{p}}=(\tilde{p}_1,\cdots,\tilde{p}_n)$ are a choice of marked points of $\tildeC$ lying above the marked points $\mathbf{p}=(p_1,\cdots,p_n)$ in $C$. 
	\end{definition}
	The orientation of the curve $C$ and the fact that $\tildeC$ is a principal $\Gamma$-bundle over $C_{\gen}$ give rise to an $n$-tuple $\mathbf{m}=(m_1,\dots,m_n)\in \Gamma^n$ keeping track of the monodromies around the points $\tilde{p}_i$'s. For each $1\leq i\leq n$, the orientations give a small loop in $C-\{p_1,\dots, p_n\}$ around each $p_i$ and  a lift to a path in a small neighborhood around $\widetilde{p}_i$ is $\widetilde{C}-\{\widetilde{p}_1,\dots, \widetilde{p}_n\}$. Since the lift is not uniquely determined, the element $m_i$ give the difference between the starting and the ending sheets of the lifted path. 
	
	More precisely, the isotropy subgroup $\Gamma_{\tilde{p}_i}$ of the point $\tilde{p}_i$ is a cyclic subgroup, of order say $N_i$. The cyclic group $\Gamma_{\tilde{p}_i}$ acts on the tangent space $T_{\widetilde{p}_i}\tildeC$ faithfully. Then $m_i\in \Gamma$ is defined as the generator of $\Gamma_{\tilde{p}_i}$ which acts as multiplication by $\exp\frac{2\pi\sqrt{-1}}{N_i}$. 
	
	Hence we have the evaluation morphism 
	$\ev: \GMod \rightarrow \Gamma^n$ and let $\GMod(\mathbf{m}):=\ev^{-1}(\mathbf{m})$. We also have the morphism of stacks $\GMod \to \barM_{g,n}$ defined by forgetting the admissible cover. The following theorem is due to Jarvis-Kimura-Kaufmann  \cite{JKK}:
	\begin{theorem}
		The stack $\GMod$ and the open and closed substacks $\GMod(\mathbf{m})$ are smooth Deligne-Mumford stacks, flat, proper, and quasi finite over $\overline{\Mcal}_{g,n}$. Moreover, the $\GMod(\mathbf{m})$ are a finite disjoint union of connected components of $\GMod$. 
	\end{theorem}

	\subsection{Moduli stacks of pointed admissible covers with local coordinates}\label{sec:localcoord}
	We now introduce the stack of $n$-pointed admissible $\Gamma$-covers of stable $n$-pointed curves of genus $g$ with local coordinates.
	\begin{definition}
		Let $\hatbarM{}^\Gamma_{g,n}$ denote the stack of $n$-marked admissible $\Gamma$-covers 
		$(\pi:\tildeC\rightarrow C, \wtilde{\mathbf{p}},{\mathbf{p}},\wtilde{\mathbf{z}},{\mathbf{z}}),$ where $(\pi:\tildeC\rightarrow C,\wtilde{\mathbf{p}}, \mathbf{p})$ is an $n$-pointed admissible $\Gamma$-cover and $\wtilde{\mathbf{z}}=(\tilde{z}_1,\cdots,\tilde{z}_n)$ are special formal local coordinates at the points $\widetilde{\mathbf{p}}$ in $\tildeC$ such that $\mathbf{z}=(\tilde{z}_1^{N_1},\cdots,\tilde{z}_n^{N_n})$ are formal local coordinates at the points $\mathbf{p}$ in $C$, where $N_i$ is the order of the cyclic group $\Gamma_{\tilde{p}_i}$. As before, associated with such a data we have the monodromy $\mathbf{m}\in \Gamma^n$ around the points $\wtilde{\mathbf{p}}$. For notational convenience, we will often drop ${\bf z}$ from the notation of a family of $n$-pointed admissible $\Gamma$-covers with chosen coordinates. 
		
		For a finite set $A$, let $\hatbarM{}^\Gamma_{g,A}$ denote the stack of $A$-marked admissible covers, where instead of numbering the marking data, we have bijections of the marked points and formal coordinates with $A$. In this setting the associated monodromy of an $A$-marked admissible $\Gamma$-cover is a function $\mathbf{m}:A\to \Gamma$. Given a function $\mathbf{m}:A\to \Gamma$ we let $\hatbarM{}^\Gamma_{g,A}(\mathbf{m})\subseteq \hatbarM{}^\Gamma_{g,A}$ be the substack of those $A$-marked admissible $\Gamma$-covers with monodromy given by $\mathbf{m}$.
	\end{definition}

	\begin{definition}\label{def:nmarkedcovers}
		Let $A$ be a finite set and $\mathbf{m}:A\to \Gamma$. Let $\tildebarM{}^\Gamma_{g,A}(\mathbf{m})$ denote the stack of $A$-marked admissible $\Gamma$-covers with monodromy data $\mathbf{m}$ of the form
		$(\pi:\tildeC\rightarrow C, \wtilde{\mathbf{p}},{\mathbf{p}},\wtilde{\mathbf{v}},{\mathbf{v}}),$ where $(\pi:\tildeC\rightarrow C,\wtilde{\mathbf{p}}, \mathbf{p})$ is an $A$-pointed admissible $\Gamma$-cover, ${\mathbf{v}}$ is a choice of non-zero tangent vectors to $C$ at the points ${\mathbf{p}}$  and $\wtilde{\mathbf{v}}$ is a choice of tangent vectors to $\tildeC$ at the points $\wtilde{\mathbf{p}}$ compatible with $\mathbf{v}$. 
	\end{definition}We may henceforth drop  ${\mathbf{v}},\mathbf{p}$  from the notation of a family of curves in $\widetilde{\overline{\mathcal{M}}}{}^{\Gamma}_{g,n}$ since they are determined by the other data that appear in the notation.
	It is clear that we have the following commutative diagram:
	\begin{equation}
	\xymatrix{
		\hatbarM{}^\Gamma_{g,A}(\mathbf{m}) \ar[r]\ar[d] & \tildebarM{}^\Gamma_{g,A}(\mathbf{m})\ar[r]\ar[d] & \barM{}^\Gamma_{g,A}(\mathbf{m})\ar[d]\\
		\hatbarM_{g,A} \ar[r] & \tildebarM_{g,A}\ar[r] & \barM_{g,A}.
	}
	\end{equation}

	\subsection{The category of stable group-marked graphs and associated stacks.}\label{ap:groupmarkedgraphs}
	We will work with a certain category $\Xcal^\Gamma$ of stable $\Gamma$-graphs (see also \cite{GK:98}, \cite[\S3.3]{P:13}) in order to organize various gluing and forgetting marked points constructions for pointed covers. For us, a (weighted) graph is a tuple $(V,H,\iota,s,w)$, where $V, H$ are finite sets (of vertices and half-edges respectively), $\iota:H\to H$ is an involution, $s:H\to V$ is called the source map, and $w:V\to \ZBbb_{\geq 0}$ is a weight function. In other words, graphs have vertices, edges and some legs attached to the vertices. Moreover, there is a weight (or genus) attached to each vertex. The fixed points $L:=H^\iota$ are called the legs of the graph and $\iota$-orbits of size two are called the edges $E$ of the graph. The degree $\deg(v)$ of a vertex is the total number of half-edges sourcing from the vertex and the edge degree $\operatorname{edeg}(v)$ of a vertex is $\deg(v)$ minus the number of legs sourcing from the vertex. The genus of such a graph $X$ is defined to be (in particular, it does not depend on the legs):$$g(X):=|E|-|V|+\sum\limits_{v\in V}w(v)+\pi_0,$$ where $\pi_0$ is the number of connected components of the graph. We say that a (weighted) graph is stable if for every $v\in V$, we have $2w(v)-2+\deg(v)>0$. 
	\begin{definition}
		The objects in the category $\Xcal$ are stable (weighted) graphs. If $X,X'\in \Xcal$, then a morphism $f:X\to X'$ is a pair $(f^*,f_*)$, where   $f^*:H'\hookrightarrow H$ is an inclusion respecting the involutions, $H\setminus f^*H'\substack{{\stackrel{s}{\longrightarrow}}\\{\underset{s\iota}{\longrightarrow}}}V\xoto{f_*}V'$ is a coequalizer such that we have $s'=f_*\circ s\circ f^*$ and
		\begin{equation}\label{eq:weightcondition}
		2w'(v')-2+\operatorname{edeg}(v')=\sum\limits_{v\in f_*^{-1}(v')}{\left(2w(v)-2+\operatorname{edeg}(v)\right)}\mbox{ for each }v'\in V'.
		\end{equation} 
	\end{definition}
	In other words a morphism corresponds to contracting some of the edges and deleting some of the legs. Note that if there exists a morphism $f:X\to X'$ in $\Xcal$, then we must have $g(X)=g(X')$. We now define the category $\Xcal^\Gamma$ of stable $\Gamma$-graphs.
	\begin{definition}\label{d:XGammagA}
		An object of $\Xcal^\Gamma$ is a tuple $(X,\mathbf{m}:H(X)\to \Gamma,\mathbf{b}:H(X)\setminus L(X)\to \Gamma),$ where $X\in \Xcal$, such that whenever $\{h_1,h_2\}$ is an edge of $X$, we have $${ }^{\mathbf{b}(h_1)}\mathbf{m}(h_1)\cdot\mathbf{m}(h_2)=1, \ \mbox{ and} \ \mathbf{b}(h_1)\cdot\mathbf{b}(h_2)=1, $$ where ${}^{\mathbf{b}(h_1)}\mathbf{m}(h_1):=\mathbf{b}(h_1)\mathbf{m}(h_1)\mathbf{b}(h_1)^{-1}$. 
		\\A morphism $(f,\pmb\gamma):(X,\mathbf{m}_X, \mathbf{b}_X)\to (Y,\mathbf{m}_Y,\mathbf{b}_Y)$ in $\Xcal^\Gamma$ consists of a morphism $f:X\to Y$ in $\Xcal$ such that $\mathbf{m}_X(l)=1\  \forall l\in L(X)\setminus f^*L(Y)$ and $\pmb{\gamma}:H(Y)\to \Gamma$ such that 
		$$\pmb{\gamma}(h)\cdot\mathbf{m}_X(f^*h)\cdot\pmb{\gamma}(h)^{-1}=\mathbf{m}_Y(h) \mbox{ for all } h\in H(Y) \mbox{ and }$$
		$$\mathbf{b}_Y(h_1)\pmb\gamma(h_1)=\pmb\gamma(h_2)\mathbf{b}_X(f^*h_2) \mbox{ whenever } \{h_1,h_2\}\in E(Y).$$ The category $\Xcal^\Gamma$ is a monoidal category under disjoint unions of stable $\Gamma$-graphs.
	\end{definition}
	\begin{remark}\label{rk:corolla}
		A stable $\Gamma$-corolla is a stable $\Gamma$-graph with one vertex and no edges. If a corolla has weight $g$ and $n$ legs, then by definition, it is stable if and only if $2g-2+n>1$. We may sometimes denote a $\Gamma$-corolla in shorthand notation $(g,\mathbf{m})$ with $\mathbf{m}\in\Gamma^n$. By this we mean a corolla of genus $g$ with $n$ legs labeled by $\mathbf{m}$. If $(X,\mathbf{m},\mathbf{b})\in\Xcal^\Gamma$ then we can cut up the graph at the midpoints of all edges to obtain a collection of $\Gamma$-corollas parameterized by $V(X)$. For each vertex $v\in V(X)$ we let $L_v$ denote the set of half-edges of $X$ whose source is $v$. These form the legs of the corolla corresponding to the vertex. Hence  we may think of any stable $\Gamma$-graph as being built up from stable $\Gamma$-corollas.
	\end{remark}
	
	\noindent  Below we see 4 important examples of morphisms in the category $\Xcal^\Gamma$ along with some diagrams for special cases:
	\begin{example}\label{ex:groupaction} (Group action.)
		Let $(X,\mathbf{m},\mathbf{b})\in \Xcal^\Gamma$ with set of half-edges (including legs) denoted by $H$. Let $\pmb\gamma:H\to \Gamma$. Then we get a new stable $\Gamma$-graph $(X,{}^{\pmb\gamma}\mathbf{m},{}^{\pmb\gamma}\mathbf{b})$ where we have just conjugated all the group elements marking the half-edges. This gives us an isomorphism $(\id_X,\pmb\gamma):(X,\mathbf{m},\mathbf{b})\xoto{\cong}(X,{}^{\pmb\gamma}\mathbf{m},{}^{\pmb\gamma}\mathbf{b})$. This is related to the axiom of $\Gamma$-equivariance of the modular functors. 
		$$
		\begin{tikzpicture}[baseline={([yshift=-1ex]current bounding box.center)}]
		\node (B1) {$m_1$};
		\node[circle, inner sep = 4.8pt, draw, below right =.6 of B1] (Cup) {$g$};
		\node[above right =0.6 of Cup] (B2) {$m_n$};
		\node[above =.5 of Cup] (Dots) {$\cdots$};
		\draw (B1) to[out=-90,in=180] (Cup);
		\draw (B2) to[out=-90,in=0] (Cup);
		\end{tikzpicture}\xoto{(\id,\pmb\gamma)} 
		\begin{tikzpicture}[baseline={([yshift=-1ex]current bounding box.center)}]
		\node (B1) {$\gamma_1 m_1\gamma_1^{-1}$};
		\node[circle, inner sep = 4.8pt, draw, below right =.6 of B1] (Cup) {$g$};
		\node[above =.5 of Cup] (Dots) {$\cdots$};
		\node[above right =.6 of Cup] (B2) {$\gamma_n m_n\gamma_n^{-1}$};
		\draw (B1) to[out=-90,in=180] (Cup);
		\draw (B2) to[out=-90,in=0] (Cup);
		\end{tikzpicture}.
		$$
	\end{example}
	\begin{example}\label{ex:edgecontraction} (Edge contraction.)
		Let $(X,\mathbf{m},\mathbf{b})\in \Xcal^\Gamma$ and let $\{h,h'\}$ be an edge in $X$. Note that it could also be a loop. Then let $(X',\mathbf{m}',\mathbf{b}')$ be the stable $\Gamma$-graph obtained after contracting the edge $\{h,h'\}$. (The vertices of the edge get merged into one vertex and a suitable weight is attached to it.) This is the edge contraction morphism $(X,\mathbf{m},\mathbf{b})\to (X',\mathbf{m}',\mathbf{b}')\in \Xcal^\Gamma$. This morphism is related to gluing of $\Gamma$-covers and the axiom of `Factorization' for modular functors. 
		$$
		\begin{tikzpicture}[baseline={([yshift=0ex]current bounding box.center)}]
		\node (B1) {$m_1$};
		\node[circle, inner sep = 4pt, draw, below right =.4 of B1] (Cup) {$g$};
		\node[above =.3 of Cup] (Dots) {$\cdots$};
		\node[above right =.4 of Cup] (B2) {$m_n$};
		\draw (B1) to[out=-90,in=180] (Cup);
		\node[below =.7 of Cup] (invisible) {};
		\draw (B2) to[out=-90,in=0] (Cup);
		\node[circle, inner sep = 3pt, draw, below=.4 of Cup] (Cup1) {$g'$};
		\draw(Cup1)to(Cup);
		\node[below=.3 of Cup1] (Dots) {$\cdots$};
		\node[below  right= -.2  of Cup](B6){$(a^{-1}),am^{-1}a^{-1}$}; 
		\node[above left= -.2  of Cup1](B7){$(a), m$}; 
		\node[below right =.4 of Cup1] (B4) {$m'_n$};
		\node[below left =.4 of Cup1] (B5) {$m'_1$};
		\draw (B5) to[out=90,in=180] (Cup1);
		\draw (B4) to[out=90,in=0] (Cup1);
		\end{tikzpicture}\longrightarrow\begin{tikzpicture}[baseline={([yshift=-.3ex]current bounding box.center)}]
		\node (B1) {$m_1$};
		\node[circle, inner sep = 1pt, draw, below right =.8 of B1] (Cup) {$g+g'$};
		\node[above right =.6 of Cup] (B2) {$m_n$};
		\node[below right =.6 of Cup] (B4) {$m'_n$};
		\node[below left =.6 of Cup] (B5) {$m'_1$};
		\node[above =.5 of Cup] (Dots) {$\cdots$};
		\node[below=.5 of Cup] (Dots) {$\cdots$};
		\draw (B1) to[out=-90,in=160] (Cup);
		\draw (B2) to[out=-90,in=20] (Cup);
		\draw (B5) to[out=90,in=-160] (Cup);
		\draw (B4) to[out=90,in=-20] (Cup);
		\end{tikzpicture} .$$
		
	\end{example}
	\begin{example}\label{ex:legdeletion} (1-marked leg deletion.)
		Let $(X,\mathbf{m},\mathbf{b})\in \Xcal^\Gamma$ and let $l$ be a leg in $X$ such that $\mathbf{m}(l)=1$. Also suppose that the graph obtained by deleting the leg $l$ is stable. Let $(X',\mathbf{m}',\mathbf{b})$ be the stable $\Gamma$-graph obtained by deleting the leg. This is leg deletion morphism $(X,\mathbf{m},\mathbf{b})\to (X',\mathbf{m}',\mathbf{b})\in \Xcal^\Gamma$. In other words we are allowed to forget a leg marked by $1\in \Gamma$, as long as the graph does not become unstable. This morphism is related to the axiom of `Propagation of vacua'.
		
		$$\begin{tikzpicture}[baseline={([yshift=0ex]current bounding box.center)}]
		\node (B1) {$m_1$};
		\node[circle, inner sep = 4.8pt, draw, below right =.6 of B1] (Cup) {$g$};
		\node[above right =.6 of Cup] (B2) {$m_n$};
		\node[above =.5 of Cup] (Dots) {$\cdots$};
		\node[below =.5 of Cup] (leg) {$1$};
		\draw (B1) to[out=-90,in=180] (Cup);
		\draw (B2) to[out=-90,in=0] (Cup);
		\draw(leg)to(Cup);
		\end{tikzpicture}\longrightarrow \begin{tikzpicture}[baseline={([yshift=-1.2ex]current bounding box.center)}]
		\node (B1) {$m_1$};
		\node[circle, inner sep = 4.8pt, draw, below right =.6 of B1] (Cup) {$g$};
		\node[above right =.6 of Cup] (B2) {$m_n$};
		\node[above =.5 of Cup] (Dots) {$\cdots$};
		\draw (B1) to[out=-90,in=180] (Cup);
		\draw (B2) to[out=-90,in=0] (Cup);
		\end{tikzpicture}\hspace{1cm}
		\begin{tikzpicture}[baseline={([yshift=0ex]current bounding box.center)}]
		\node (B1) {$m_1$};
		\node[circle, inner sep = 4pt, draw, below right =.6 of B1] (Cup) {$g$};
		\node[above =.5 of Cup] (Dots) {$\cdots$};
		\node[above right =.6 of Cup] (B2) {$m_n$};
		\draw (B1) to[out=-90,in=180] (Cup);
		\node[below =.7 of Cup] (invisible) {};
		\draw (B2) to[out=-90,in=0] (Cup);
		\node[circle, inner sep = 3pt, draw, below=.4 of Cup] (Cup1) {$0$};
		\draw(Cup1)to(Cup);
		\node[below  right=-0.1  of Cup](B6){$(1), m^{-1}$}; 
		\node[above left=-0.1  of Cup1](B7){$(1),m$}; 
		\node[below right =.6 of Cup1] (B4) {$\gamma$};
		\node[below left =.6 of Cup1] (B5) {$1$};
		\draw (B5) to[out=90,in=180] (Cup1);
		\draw (B4) to[out=90,in=0] (Cup1);
		\end{tikzpicture}\longrightarrow\begin{tikzpicture}[baseline={([yshift=-.3ex]current bounding box.center)}]
		\node (B1) {$m_1$};
		\node[circle, inner sep = 3pt, draw, below right =.6 of B1] (Cup) {$g$};
		\node[above right =.6 of Cup] (B2) {$m_n$};
		\node[below right =.6 of Cup] (B4) {$\gamma$};
		\node[above =.5 of Cup] (Dots) {$\cdots$};
		\draw (B1) to[out=-90,in=160] (Cup);
		\draw (B2) to[out=-90,in=20] (Cup);
		\draw (B4) to[out=90,in=-20] (Cup);
		\end{tikzpicture} .$$
	\end{example}
	
	\begin{example} (Crossed braiding.)\label{ex:crossedbraiding}
		Consider a stable $\Gamma$-corolla with $n$ legs $l_1,l_2,\cdots,l_n$ marked by the group elements $m_1,m_2,\cdots,m_n\in \Gamma$. The crossed braiding isomorphism $(\sigma_{l_1,l_2},(1,m_1,1\cdots,1))$, denoted in short by $\beta_{l_1,l_2}$ in $\Xcal^\Gamma$ is a morphism from the $\Gamma$-corolla $(g;m_1,m_2,\cdots,m_n)$ to the $\Gamma$-corolla $(g;m_1m_2m_1^{-1}, m_1, m_3,\cdots,m_n)$ which flips the two legs $l_1,l_2$ and conjugates the second label $m_2$ by $m_1$. 
		\begin{equation}\label{eqn:prettypicture}
			\begin{tikzpicture}[baseline={([yshift=-.3ex]current bounding box.center)}]
		\node (B1) {$m_1$};
		\node[circle, inner sep = 2pt, draw, below right =.6 of B1] (Cup) {$g$};
		\node[above right =.6 of Cup] (B2) {$m_3$};
		\node[below right =.6 of Cup] (B4) {$m_4$};
		\node[below left =.6 of Cup] (B5) {$m_6$};
		\node[above =.4 of Cup] (Dots) {$m_2$};
		\node[below=.4 of Cup] (Dots1) {$m_5$};
		\draw[blue] (B1) to[out=-90,in=160] (Cup);
		\draw[red] (Dots) to[out=-90,in=90] (Cup);
		\draw (Dots1) to[out=90,in=-90] (Cup);
		\draw (B2) to[out=-90,in=20] (Cup);
		\draw (B5) to[out=90,in=-160] (Cup);
		\draw (B4) to[out=90,in=-20] (Cup);
		\end{tikzpicture}\xoto{\beta_{l_1,l_2}}\begin{tikzpicture}[baseline={([yshift=-.3ex]current bounding box.center)}]
		\node (B1) {$m_1m_2m_1^{-1}$};
		\node[circle, inner sep = 2pt, draw, below right =.6 of B1] (Cup) {$g$};
		\node[above right =.6 of Cup] (B2) {$m_3$};
		\node[below right =.6 of Cup] (B4) {$m_4$};
		\node[below left =.6 of Cup] (B5) {$m_6$};
		\node[above =.4 of Cup] (Dots) {$m_1$};
		\node[below=.4 of Cup] (Dots1) {$m_5$};
		\draw[red] (B1) to[out=-90,in=160] (Cup);
		\draw[blue] (Dots) to[out=-90,in=90] (Cup);
		\draw (Dots1) to[out=90,in=-90] (Cup);
		\draw (B2) to[out=-90,in=20] (Cup);
		\draw (B5) to[out=90,in=-160] (Cup);
		\draw (B4) to[out=90,in=-20] (Cup);
		\end{tikzpicture}
		\end{equation}
		
	\end{example}
	Since the category $\Xcal^\Gamma$ has all the morphisms as in Equation \eqref{eqn:prettypicture}, we do not need to formulate $\Gamma$-equivariance, permutation equivariance, Factorization and Propagation of vacua separately in Definition \ref{d:gammacrossedmodfun} of a $\Gamma$-crossed modular functor.

	For any $(X,\mathbf{m},\mathbf{b})\in \Xcal^\Gamma$ define the smooth Deligne-Mumford stacks (see Appendix \ref{ap:modstackofadmcovers1}, \ref{sec:localcoord})
	\begin{equation}\label{eqn:twistedconfusion1}
	\barM{}^\Gamma_{X,\mathbf{m},\mathbf{b}}:=\prod\limits_{v\in V(X)}\barM{}^\Gamma_{w(v),L_v}(\mathbf{m}|_{L_v}),
	\end{equation}
	\begin{equation}\label{eq:GModtildeXdefn}
	\tildebarM{}^\Gamma_{X,\mathbf{m},\mathbf{b}}:=\prod\limits_{v\in V(X)}\tildebarM{}^\Gamma_{w(v),L_v}(\mathbf{m}|_{L_v}).
	\end{equation}
	The stack $\barM{}^\Gamma_{X,\mathbf{m},\mathbf{b}}$ parameterizes $H(X)$-pointed admissible $\Gamma$-covers $(\tildeD\to D, \wtilde{\mathbf{q}},\mathbf{q})$ of curves $(D,\mathbf{q})$ in $\barM_{X}$ with connected components parameterized by $V(X)$ and with monodromies around the $H(X)$ marked points $\wtilde{\mathbf{q}}$ being given by $\mathbf{m}$.
	
	As in the untwisted case, for any $h\in H(X)$ we have the associated rank 1 point bundle $\tildeL_h$ on $\barM{}^\Gamma_{X,\mathbf{m},\mathbf{b}}$ and its dual line bundle $\tildeL^\vee_h$ whose fiber at $(\tildeD\to D,\wtilde{\mathbf{q}},{\mathbf{q}})\in \barM{}^\Gamma_{X,\mathbf{m},\mathbf{b}}$ is the tangent space $T_{\wtilde{\mathbf{q}}(h)}\tildeD$. 
	We recall the following proposition from \cite{JKK}.
	\begin{proposition}\label{prop:pullbackofpsiclasses}Consider the natural forgetful map $\pi:\overline{\mathcal{M}}{}^\Gamma_{X,{\mathbf{m}},\mathbf{b}}\rightarrow \overline{\mathcal{M}}_{X}$, then $\pi^*\mathcal{L}_h$ is naturally isomorphic to $\widetilde{\mathcal{L}}_h^{N_h}$, where $\mathcal{L}_h$ denotes the point bundle on $\overline{\mathcal{M}}_{X}$ and $N_h$ is the order of $\mathbf{m}(h) \in \Gamma$. 
	\end{proposition}
	
	\subsection{Clutching with respect to morphisms between stable graphs}\label{ap:clutchingalongmaps}
	Let $(X,\mathbf{m},\mathbf{b})\in \Xcal^\Gamma$. By definition, if $\{h_1,h_2\}\in E(X)$, then ${}^{\mathbf{b}(h_1)}\mathbf{m}(h_1)$ and $\mathbf{m}(h_2)$ are inverse to each other.  The monodromy around the marked point $\wtilde{\mathbf{q}}(h_1)\in \tildeD$ is $\mathbf{m}(h_1)$ and hence the monodromy around the point ${\mathbf{b}(h_1)}\cdot \wtilde{\mathbf{q}}(h_1)$ is ${\mathbf{m}(h_2)^{-1}}.$ In other words, using the construction of \cite{JKK} we can glue the two  points ${\mathbf{b}(h_1)}\cdot \wtilde{\mathbf{q}}(h_1)$ and $\wtilde{\mathbf{q}}(h_2)$ on the admissible cover $\tildeD\to D$ corresponding to an edge of $X$. We get the same glued pointed admissible $\Gamma$-cover if we glue with respect to the two points $\wtilde{\mathbf{q}}(h_1)$ and ${\mathbf{b}(h_2)}\cdot \wtilde{\mathbf{q}}(h_2)$.
	
	Now any morphism in $\Xcal^\Gamma$ is equivalent to one obtained by contracting some edges first (see Remark \ref{rk:legdeletion} below), and then deleting some 1-marked legs. For the moduli stacks, deleting 1-marked legs amounts to forgetting the marking data corresponding to the deleted legs. Since the deleted legs are 1-marked, the cover is unramified over the corresponding marked points, which implies that this is a well-defined operation. Hence, given any  morphism $(f,\pmb{\gamma}):(X,\mathbf{m}_X,\mathbf{b}_X)\to (Y,\mathbf{m}_Y,\mathbf{b}_Y)$ in $\Xcal^\Gamma$ we can define the clutching map
	\begin{equation}\label{eq:gluingcovers}
	\xi_{f,\pmb{\gamma}}: \barM{}^\Gamma_{X,\mathbf{m}_X,\mathbf{b}_X} \rar{} \barM{}^\Gamma_{Y,\mathbf{m}_Y,\mathbf{b}_Y}.
	\end{equation}
	
	\begin{remark}\label{rk:legdeletion}
		In general if we try to delete a 1-marked leg first from a stable $\Gamma$-graph, the graph may become unstable and hence may force us to contract some edges. For the corresponding moduli stacks this corresponds to the ``forgetting tails'' construction in \cite{JKK}.
	\end{remark}
	
	\begin{remark}
		If we have two morphisms $$(X,\mathbf{m}_X,\mathbf{b}_X)\xoto{(f_1,\pmb{\gamma_1})}(Y,\mathbf{m}_Y,\mathbf{b}_Y)\xoto{(f_2,\pmb{\gamma_2})}(Z,\mathbf{m}_Z,\mathbf{b}_Z)$$ in $\Xcal^\Gamma$ then $\xi_{(f_2,\pmb{\gamma_2})\circ (f_1,\pmb{\gamma_1})}=\xi_{f_2,\pmb{\gamma_2}}\circ \xi_{f_1,\pmb{\gamma_1}}$ and the assignment $(X,\mathbf{m}_X,\mathbf{b}_X)\mapsto \barM{}^\Gamma_{X,\mathbf{m}_X,\mathbf{b}_X}$ is functorial.
	\end{remark}
	\subsubsection{Normal bundles for edge contraction morphisms}\label{sec:normaltoedge}
	As before, let $(f,\pmb{\gamma})$ be a morphism in $\Xcal^\Gamma$ which only contracts edges.  By \cite[Ch. XIII-\S3]{ACG:11}, the normal bundle on $\barM{}^\Gamma_{X,\mathbf{m}_X,\mathbf{b}_X}$ to the map $\xi_{f,\pmb{\gamma}}$ decomposes as
	\begin{equation}\label{eqn:decompositionofnormalbundle}
	N\xi_{f,\pmb{\gamma}}=\bigoplus\limits_{\{h_1,h_2\}\in E(X)\setminus f^*E(Y)}\tildeL^\vee_{h_1}\otimes \tildeL^\vee_{h_2}.
	\end{equation}

	Our next goal is to lift the clutching construction to the stacks (see Equation \eqref{eq:GModtildeXdefn} for notation) $\tildebarM{}^\Gamma_{X,\mathbf{m}_X,\mathbf{b}_X}$ and $\tildebarM{}^\Gamma_{Y,\mathbf{m}_Y,\mathbf{b}_Y}$ while also taking into account the normal bundles.  Consider the stack ${\tildebarM{}^\Gamma_{X,\mathbf{m}_X,\mathbf{b}_X,f,\pmb{\gamma}}}$ parameterizing data of the form $(\tildeD\to D,\wtilde{\mathbf{q}},\mathbf{q},\wtilde{\mathbf{v}})$, where $(\tildeD\to D,\wtilde{\mathbf{q}},\mathbf{q})\in \barM{}^\Gamma_{X,\mathbf{m}_X,\mathbf{b}_X}$ is an $H(X)$-pointed admissible $\Gamma$-cover and $\wtilde{\mathbf{v}}$ is a choice of non-zero tangent vectors to $\tildeD$ at the points marked by $f^*H(Y)\subset H(X)$. We have the Cartesian square 
	\begin{equation}
	\xymatrix{
		{\tildebarM{}^\Gamma_{X,\mathbf{m}_X,\mathbf{b}_X,f,\pmb{\gamma}}}\ar[r]^{\wtilde{\xi_{f,\pmb{\gamma}}}}\ar[d]_{\Gm^{f^*H(Y)}-torsor} & \tildebarM{}^\Gamma_{Y,\mathbf{m}_Y,\mathbf{b}_Y}\ar[d]^{\Gm^{H(Y)}-torsor} \\
		\barM{}^\Gamma_{X,\mathbf{m}_X,\mathbf{b}_X}\ar[r]^{\xi_{f,\pmb{\gamma}}} &  \barM{}^\Gamma_{Y,\mathbf{m}_Y,\mathbf{b}_Y}.
	}
	\end{equation}
	It follows that the normal bundle $N\wtilde{\xi_{f,\pmb{\gamma}}}\onto {\tildebarM{}^\Gamma_{X,\mathbf{m}_X,\mathbf{b}_X,f,\pmb{\gamma}}}$ decomposes as
	\begin{equation}\label{eq:normal bundle tilde xi_f,gamma}
	N\wtilde{\xi_{f,\pmb{\gamma}}}=\bigoplus\limits_{\{h_1,h_2\}\in E(X)\setminus f^*E(Y)}\tildeL^\vee_{h_1}\otimes \tildeL^\vee_{h_2},
	\end{equation}
	where we use $\tildeL^\vee_h$ to also denote the line bundle on ${\tildebarM{}^\Gamma_{X,\mathbf{m}_X,\mathbf{b}_X,f,\pmb{\gamma}}}$ obtained by pullback of the corresponding line bundle on $\barM{}^\Gamma_{X,\mathbf{m}_X,\mathbf{b}_X}$.
	
	Let $FN\wtilde{\xi_{f,\pmb{\gamma}}}\to {\tildebarM{}^\Gamma_{X,\mathbf{m}_X,\mathbf{b}_X,f,\pmb{\gamma}}}$ denote the frame bundle (it will be a $\Gm^{E(X)\setminus f^*E(Y)}$-torsor) associated with the vector bundle $N\wtilde{\xi_{f,\pmb{\gamma}}}\to {\tildebarM{}^\Gamma_{X,\mathbf{m}_X,\mathbf{b}_X,f,\pmb{\gamma}}}$ preserving the decomposition into line bundles given in Equation \eqref{eq:normal bundle tilde xi_f,gamma}.  We obtain the diagram
	\begin{equation}\label{eq:clutching for normal bundle}
	\xymatrix{
		\tildebarM{}^\Gamma_{X,\mathbf{m}_X,\mathbf{b}_X}\ar[r] & FN\wtilde{\xi_{f,\pmb{\gamma}}}\ar[d]\\
		& {\tildebarM{}^\Gamma_{X,\mathbf{m}_X,\mathbf{b}_X,f,\pmb{\gamma}}}\ar[r]^-{\wtilde{\xi_{f,\pmb{\gamma}}}} & \tildebarM{}^\Gamma_{Y,\mathbf{m}_Y,\mathbf{b}_Y}
	}
	\end{equation}
	where the top horizontal map is given by 
	$$(\tildeD\to D,\wtilde{\mathbf{q}},\wtilde{\mathbf{w}})\mapsto \left((\tildeD\to D,\wtilde{\mathbf{q}},\wtilde{\mathbf{w}}|_{f^*H(Y)}),(\wtilde{\mathbf{w}}(h_1)\otimes \wtilde{\mathbf{w}}(h_2))_{\{h_1,h_2\}\in E(X)\setminus f^*E(Y)}\right).$$
	We think of $FN\wtilde{\xi_{f,\pmb{\gamma}}}$ as the open part of $N\wtilde{\xi_{f,\pmb{\gamma}}}$ obtained by deleting the hyperplane bundles of zero sections corresponding to the decomposition (\ref{eq:normal bundle tilde xi_f,gamma}) of the normal bundle into line bundles. The hyperplane bundles which are deleted are parameterized by the contracted edges $E(X)\setminus f^*E(Y)$. They give one set of boundary divisors on $N\wtilde{\xi_{f,\pmb\gamma}}$ that we will need to consider and these type of divisors will be referred to as {\em hyperplane bundle divisors}. We will need to consider another set of boundary divisors in \S\ref{ss:specalongclutch} below.
	
	\subsection{Twisted $\mathcal{D}$-modules and specialization along clutching maps}\label{ap:specialization}
	We have defined clutching maps between moduli stacks associated with morphisms in the category $\Xcal^\Gamma$. In this section, we consider twisted $\mathcal{D}$-modules on these moduli stacks and functors between them defined using specialization along the clutching maps.

	\subsubsection{Twisted $\mathcal{D}$-modules on the moduli stacks}
	Let $(Y,\mathbf{m},\mathbf{b})\in \Xcal^\Gamma$ be a stable $\Gamma$-graph and $\barM{}^\Gamma_{Y,\mathbf{m},\mathbf{b}}$, $\tildebarM{}^\Gamma_{Y,\mathbf{m},\mathbf{b}}$ the corresponding smooth Deligne-Mumford stacks defined by (\ref{eqn:twistedconfusion1}),(\ref{eq:GModtildeXdefn}) with their open parts 
	$$
	\Mcal{}^\Gamma_{Y,\mathbf{m},\mathbf{b}}:=\prod\limits_{v\in V(Y)}\Mcal{}^\Gamma_{w(v),L_v}(\mathbf{m}|_{L_v}), \quad
	\tildeM{}^\Gamma_{Y,\mathbf{m},\mathbf{b}}=\prod\limits_{v\in V(Y)}\tildeM{}^\Gamma_{w(v),L_v}(\mathbf{m}|_{L_v}).$$
	
	We have the normal crossing boundary divisors 
	$$\Delta^\Gamma_{Y,\mathbf{m},\mathbf{b}}:=\barM{}^\Gamma_{Y,\mathbf{m},\mathbf{b}}\setminus \Mcal{}^\Gamma_{Y,\mathbf{m},\mathbf{b}} \mbox { and } \ \wtilde\Delta^\Gamma_{Y,\mathbf{m},\mathbf{b}}:=\tildebarM{}^\Gamma_{Y,\mathbf{m},\mathbf{b}}\setminus \tildeM{}^\Gamma_{Y,\mathbf{m},\mathbf{b}}.$$
	
	Recall that we have defined the Hodge line bundles (which we will always denote by $\Lambda$) on the stacks $\GModtildeA$ as pull-backs of the Hodge line bundles on the stacks $\barM_{g,A}$ along the natural forgetful maps. Hence we can define Hodge line bundles (also denoted by $\Lambda$) on the product moduli stacks $\tildebarM{}^\Gamma_{Y,\mathbf{m},\mathbf{b}}$ as the pullback of the Hodge line bundle on the product $\barM_Y$. Consider the logarithmic Atiyah algebra $\Acal_\Lambda(-\log\wtilde\Delta^\Gamma_{Y,\mathbf{m},\mathbf{b}})$ on the smooth Deligne-Mumford stack $\tildebarM{}^\Gamma_{Y,\mathbf{m},\mathbf{b}}$. For any $c\in \CBbb$ consider the logarithmic Atiyah algebra $c\Acal_\Lambda(-\log\wtilde\Delta^\Gamma_{Y,\mathbf{m},\mathbf{b}})$ on $\tildebarM{}^\Gamma_{Y,\mathbf{m},\mathbf{b}}$ of additive central charge $c$.
	\begin{definition}
		Let $(Y,\mathbf{m},\mathbf{b})\in \Xcal^\Gamma$ and $c\in \CBbb$. We let $\mathscr{D}_c\Mod(\tildebarM{}^\Gamma_{Y,\mathbf{m},\mathbf{b}})$ denote the category of vector bundles on the smooth Deligne-Mumford stack $\tildebarM{}^\Gamma_{Y,\mathbf{m},\mathbf{b}}$ equipped with a $c\Acal_\Lambda(-\log\wtilde\Delta^\Gamma_{Y,\mathbf{m},\mathbf{b}})$-module structure. We will call such an object a vector bundle with twisted $\log \wtilde\Delta^\Gamma_{Y,\mathbf{m},\mathbf{b}}$ connection on $\tildebarM{}^\Gamma_{Y,\mathbf{m},\mathbf{b}}$.
	\end{definition}
	\begin{remark}\label{rk:delignerhc}
		By Deligne's Riemann-Hilbert correspondence (see \cite[Thm. 5.2.20]{HTT:08}), we have an equivalence of abelian categories $\mathscr{D}_c\Mod(\tildebarM{}^\Gamma_{Y,\mathbf{m},\mathbf{b}})\cong \mathscr{D}_c\Mod(\tildeM^\Gamma_{Y,\mathbf{m},\mathbf{b}})$, where the latter is the category of coherent $\mathcal{O}_{\tildeM^\Gamma_{Y,\mathbf{m},\mathbf{b}}}$-modules over the Atiyah algebra $c\Acal_\Lambda$ on the open part $\tildeM{}^\Gamma_{Y,\mathbf{m},\mathbf{b}}$.
	\end{remark}
	
	\subsubsection{Specialization along edge contractions}\label{ss:specalongclutch}
	Let $$(f,\pmb{\gamma}): (X,\mathbf{m}_X,\mathbf{b}_X)\to (Y,\mathbf{m}_Y,\mathbf{b}_Y)$$ be a morphism in $\Xcal^\Gamma$ which only contracts edges. We have the following commutative diagram of the associated clutchings
	\begin{equation}\label{eq:specalongclutching}
	\xymatrix{
		\tildebarM{}^\Gamma_{X,\mathbf{m}_X,\mathbf{b}_X}\ar[rr]^-{\Gm^{E(X)\setminus f^*E(Y)}\mbox{-}torsor}\ar[dd]_{\Gm^{H(X)}-torsor} & & FN\wtilde{\xi_{f,\pmb{\gamma}}}\ar[d]_-{\Gm^{E(X)\setminus f^*E(Y)}\mbox{-}torsor}\ar[r] & N\wtilde{\xi_{f,\pmb{\gamma}}}\ar[ld]\\
		& & {\tildebarM{}^\Gamma_{X,\mathbf{m}_X,\mathbf{b}_X,f,\pmb{\gamma}}}\ar[dll]^(.25){\Gm^{f^*H(Y)}-torsor}\ar[r]^-{\wtilde{\xi_{f,\pmb{\gamma}}}} & \tildebarM{}^\Gamma_{Y,\mathbf{m}_Y,\mathbf{b}_Y}\ar[d]^{\Gm^{H(Y)}-torsor}\\
		\barM{}^\Gamma_{X,\mathbf{m}_X,\mathbf{b}_X}\ar[rrr]^{\xi_{f,\pmb{\gamma}}}\ar[d] & & & \barM{}^\Gamma_{Y,\mathbf{m}_Y,\mathbf{b}_Y}\ar[d]\\
		\barM_{X}\ar[rrr]^{\xi_f} & & & \barM_{Y}.
	}
	\end{equation}
	Let ${\tildeM^\Gamma_{X,\mathbf{m}_X,\mathbf{b}_X,f,\pmb{\gamma}}}\subset {\tildebarM{}^\Gamma_{X,\mathbf{m}_X,\mathbf{b}_X,f,\pmb{\gamma}}}$ be the open part and let ${\wtilde{\Delta}^\Gamma_{X,\mathbf{m}_X,\mathbf{b}_X,f,\pmb{\gamma}}}$ be the complementary boundary divisor. Let $N{\wtilde{\Delta}^\Gamma_{X,\mathbf{m}_X,\mathbf{b}_X,f,\pmb{\gamma}}}\subset N\wtilde{\xi_{f,\pmb\gamma}}$ be the corresponding divisor obtained by pullback to the normal bundle. We consider the open part $FN\wtilde{\xi_{f,\pmb\gamma}}^\circ\subset N\wtilde{\xi_{f,\pmb\gamma}}$ obtained by restricting $FN\wtilde{\xi_{f,\pmb\gamma}}$ to the open part ${\tildeM^\Gamma_{X,\mathbf{m}_X,\mathbf{b}_X,f,\pmb{\gamma}}}$. The complement is a normal crossing divisor on $N\wtilde{\xi_{f,\pmb\gamma}}$ which is a union of the divisor $N{\wtilde{\Delta}^\Gamma_{X,\mathbf{m}_X,\mathbf{b}_X,f,\pmb{\gamma}}}$ and the hyperplane bundle divisors defined previously (last paragraph of Section \ref{sec:normaltoedge}).
	
	Let $\tildebarM{}^\Gamma_{Y,\mathbf{m}_Y,\mathbf{b}_Y,f,\pmb\gamma}\subset \tildebarM{}^\Gamma_{Y,\mathbf{m}_Y,\mathbf{b}_Y}$ be the image $\wtilde{\xi_{f,\pmb\gamma}}(\tildebarM{}^\Gamma_{X,\mathbf{m}_X,\mathbf{b}_X})$. It is the closure of a stratum in the natural stratification on $\tildebarM{}^\Gamma_{Y,\mathbf{m}_Y,\mathbf{b}_Y}$. 
	
	On the open part the map $\wtilde{\xi_{f,\pmb\gamma}}:{\tildeM^\Gamma_{X,\mathbf{m}_X,\mathbf{b}_X,f,\pmb{\gamma}}}\rar{}{\tildeM{}^\Gamma_{Y,\mathbf{m}_Y,\mathbf{b}_Y,f,\pmb{\gamma}}}$ is a stack quotient by a finite group. Moreover, we can lift this to a covering map from a tubular neighborhood of ${\tildeM{}^\Gamma_{X,\mathbf{m}_X,\mathbf{b}_X,f,\pmb{\gamma}}}$ in the normal bundle $N\wtilde{\xi_{f,\pmb\gamma}}$ to a tubular neighborhood of the stratum ${\tildeM{}^\Gamma_{Y,\mathbf{m}_Y,\mathbf{b}_Y,f,\pmb{\gamma}}}\subset {\tildebarM{}^\Gamma_{Y,\mathbf{m}_Y,\mathbf{b}_Y}}$ such that the hyperplane bundle divisors on $N\wtilde{\xi_{f,\pmb\gamma}}$ described previously map to the boundary divisor $\wtilde{\Delta}^\Gamma_{Y,\mathbf{m}_Y,\mathbf{b}_Y}$. 
	
	In other words the intersection of $FN\wtilde{\xi_{f,\pmb\gamma}}^\circ$ with the tubular neighborhood of $$\tildeM{}^\Gamma_{X,\mathbf{m}_X,\mathbf{b}_X,f,\pmb\gamma}\subset N\wtilde{\xi_{f,\pmb\gamma}}$$ maps to the intersection of the open part $\tildeM{}^\Gamma_{Y,\mathbf{m}_Y,\mathbf{b}_Y}$ with the tubular neighborhood of the stratum $\tildeM{}^\Gamma_{Y,\mathbf{m}_Y,\mathbf{b}_Y,f,\pmb\gamma}\subset \tildebarM{}^\Gamma_{Y,\mathbf{m}_Y,\mathbf{b}_Y}$. Note that the open part $\tildeM{}^\Gamma_{X,\mathbf{m}_X,\mathbf{b}_X}$ is a $\Gm^{E(X)\setminus f^*E(Y)}$-torsor over the open part $FN\wtilde{\xi_{f,\pmb\gamma}}^\circ$. We obtain homomorphisms of fundamental groups $$\pi_1(\tildeM{}^\Gamma_{X,\mathbf{m}_X,\mathbf{b}_X})\to \pi_1(FN\wtilde{\xi_{f,\pmb\gamma}}^\circ)\to \pi_1(\tildeM{}^\Gamma_{Y,\mathbf{m}_Y,\mathbf{b}_Y}).$$ 
	
	By \cite[Ch. XVII]{ACG:11}, the pullback of the Hodge line bundle on $\barM_Y$ along $\xi_f$ is the Hodge line bundle on $\barM_X$. Hence the Hodge line bundle on $\tildebarM{}^\Gamma_{Y,\mathbf{m}_Y,\mathbf{b}_Y}$ pulls back to the Hodge line bundle on $\tildebarM{}^\Gamma_{X,\mathbf{m}_X,\mathbf{b}_X}$ along the top of (\ref{eq:specalongclutching}). We also denote by $\Lambda$ the pullback of the Hodge line bundle to $N\wtilde{\xi_{f,\pmb{\gamma}}}$. 
	
	For $c\in \CBbb$ we consider the logarithmic Atiyah algebra $c\Acal_\Lambda(-\log N\wtilde{\xi_{f,\pmb\gamma}}\setminus FN\wtilde{\xi_{f,\pmb\gamma}}^\circ)$ on $N\wtilde{\xi_{f,\pmb\gamma}}$ and the corresponding category $\mathscr{D}_c\Mod(N\wtilde{\xi_{f,\pmb\gamma}})$ of twisted logarithmic $\mathcal{D}$-modules. Any object of $$\mathscr{D}_c\Mod(\tildebarM{}^\Gamma_{Y,\mathbf{m}_Y,\mathbf{b}_Y})\cong \mathscr{D}_c\Mod(\tildeM{}^\Gamma_{Y,\mathbf{m}_Y,\mathbf{b}_Y})$$ (see Remark \ref{rk:delignerhc}) after specialization to the normal bundle gives us an object of $$\mathscr{D}_c\Mod(N\wtilde{\xi_{f,\pmb{\gamma}}})\cong \mathscr{D}_c\Mod(FN\wtilde{\xi_{f,\pmb{\gamma}}}^\circ).$$ We can now pullback along the top horizontal arrow of (\ref{eq:specalongclutching}) to obtain an object of $\mathscr{D}_c\Mod(\tildebarM{}^\Gamma_{X,\mathbf{m}_X,\mathbf{b}_X})\cong \mathscr{D}_c\Mod(\tildeM{}^\Gamma_{X,\mathbf{m}_X,\mathbf{b}_X})$.
	This defines a functor which we denote as
	\begin{equation}\label{eqn:specializationdeverdier}
	\Sp_{f,\pmb{\gamma}}:\mathscr{D}_c\Mod(\tildebarM{}^\Gamma_{Y,\mathbf{m}_Y,\mathbf{b}_Y}) \rar{} \mathscr{D}_c\Mod(\tildebarM{}^\Gamma_{X,\mathbf{m}_X,\mathbf{b}_X}). 
	\end{equation}
	
	\subsubsection{Specialization for all clutchings}
	Suppose that $$(f,\pmb{\gamma}):(X,\mathbf{m}_X,\mathbf{b}_X)\to (Y,\mathbf{m}_Y,\mathbf{b}_Y)$$ is a morphism which only deletes some 1-marked legs. In this case the clutching map $$\xi_{f,\pmb\gamma}:\tildebarM{}^\Gamma_X,\mathbf{m}_X,\mathbf{b}_X\to \tildebarM{}^\Gamma_Y,\mathbf{m}_Y,\mathbf{b}_Y$$ merely forgets the marking data corresponding to the deleted 1-marked legs and the Hodge line bundle clearly pulls back to the Hodge line bundle. In this case we define the specialization to be the pullback
	\begin{equation}\label{eq:specalonggraphmaps}
	\Sp_{f,\pmb{\gamma}}:=\xi_{f,\pmb\gamma}^*:\mathscr{D}_c\Mod(\tildebarM{}^\Gamma_{Y,\mathbf{m}_Y,\mathbf{b}_Y}) \rar{} \mathscr{D}_c\Mod(\tildebarM{}^\Gamma_{X,\mathbf{m}_X,\mathbf{b}_X}). 
	\end{equation} 
	
	Since any morphism in $\Xcal^\Gamma$ is a composition of edge contractions followed by leg deletions, we can now define specialization along any morphism in $\Xcal^\Gamma$. If we have two morphisms
	$(X,\mathbf{m}_X,\mathbf{b}_X)\xoto{f_1,\pmb{\gamma_1}}(Y,\mathbf{m}_Y,\mathbf{b}_Y)\xoto{f_2,\pmb{\gamma_2}}(Z,\mathbf{m}_Z,\mathbf{b}_Z)$ then we have a natural isomorphism between the two functors
	\begin{equation}
	\Sp_{(f_2,\pmb{\gamma_2})\circ (f_1,\pmb{\gamma_1})} \cong \Sp_{f_1,\pmb{\gamma_1}}\circ \Sp_{f_2,\pmb{\gamma_2}}:\mathscr{D}_c\Mod(\tildebarM{}^\Gamma_{Z,\mathbf{m}_Z,\mathbf{b}_Z})\rar{} \mathscr{D}_c\Mod(\tildebarM{}^\Gamma_{X,\mathbf{m}_X,\mathbf{b}_X}).
	\end{equation}
	Hence the assignment $(X,\mathbf{m}_X,\mathbf{b}_X)\mapsto \mathscr{D}_c\Mod(\tildebarM{}^\Gamma_{X,\mathbf{m}_X,\mathbf{b}_X})$ is functorial.
	
	\section{Twisted Kac-Moody S-matrices}\label{section:crossedS}
	Following \cite{Kac}, we now recall the notion twisted S-matrices in the setting of twisted affine Lie algebras and connect them to the characters of the fusion ring in Section \ref{section:hongfusion}. We call these matrices twisted Kac-Moody S-matrices to differentiate between the crossed S-matrices discussed earlier. As observed in \cite{Kac}, these twisted Kac-Moody S-matrices (except for $A_{2n}^{(2)}$) are not matrices for the modular transformation of the characters with respect to the group $\operatorname{SL}_2(\mathbb{Z})$. 
	We follow the ordering and labeling of the roots and weights as in \cite{Kac}. 
	
	\begin{notation}Recall that $\GO$ denote the finite dimensional algebra whose Cartan matrix is obtained by deleting the $0$-th row and column of $X_N^{(m)}$.  Throughout this section, $\mathring{Q}$ will denote
		the root lattice $Q(\GO)$; $\Lambda_i$'s  will denote the affine fundamental weights of $\frg(X_N^{(m)})$ and $\overline{\Lambda}_i$ will denote their orthogonal horizontal projects with respect to invariant bilinear form $\frg(X_N^{(m)})$; $\omega_1,\dots, \omega_{\operatorname{rank \GO}}$ will denote the fundamental weights of $\GO$; and the weight $\omega_0$ of the trivial representation of $\GO$ will be considered as the zero-th fundamental weight. Moreover $M^*$ denotes the dual lattice of $\mathring{Q}$ with respect to the normalized Killing form $\kappa_{\mathfrak{g}}$.
	\end{notation}
	\subsection{Integrable highest weight representation of $\frg(X_N^{m})$. }
	\begin{itemize}
		
		\item  The case $A_{2n-1}^{(2)}$: In this case $\GO=C_n$. It turns out that for all $0\leq i\leq n$, the $\overline{\Lambda}_i=\omega_i$. 
		The set of level $\ell$ integrable highest weight representations of the affine Lie algebra $\mathfrak{g}(A_{2n-1}^{(2)})$ can be rewritten as follows:
		\begin{equation}\label{eqn:levellweightforA2n-12}
		P^{\ell}(\mathfrak{g}(A_{2n-1}^{(2)}) )=\{\sum_{i=1}^nb_i\omega_i\in P_+(C_{n})| b_1+2(b_2+\dots+b_n)\leq \ell \}.
		\end{equation}
		\item The case $D_{n+1}^{(2)}$:  In this case ${\GO}=B_n$. As in the previous case, it turns out that for all $0\leq i\leq n$, the $\overline{\Lambda}_i=\omega_i$. The set of level $\ell$ integrable highest weight representations of the affine Lie algebra $\mathfrak{g}(A_{2n-1}^{(2)})$  are  can be rewritten as follows:
		\begin{equation}\label{eqn:levellweightforDn+12}
		P^{\ell}(\mathfrak{g}(D_{n+1}^{(2)}) )=\{\sum_{i=1}^nb_i\omega_i\in P_+(B_n)| 2(b_1+\dots+b_{n-1})+b_n\leq \ell \}.
		\end{equation}
		
		\item The case $D_{4}^{(3)}$: In this case, $\GO=\mathfrak{g}_2$ and the orbit Lie algebra $\mathfrak{g}_{\sigma}$ is also $\mathfrak{g}_2$.  Then $\overline{\Lambda}_1=\omega_2$ and $\overline{\Lambda}_2=\omega_1$, where $\omega_1$ and $\omega_2$ are the fundamental weights of the Lie algebra $\mathfrak{g}_2$. Moreover, $\overline{\Lambda}_0=\omega_0$. With this notation, we have 
		\begin{equation}\label{eqn:levellweigtforD43}
		P^{\ell}(\mathfrak{g}(D_4^{(3)}))=\{b_1\omega_1+b_2\omega_2\in P_{+}(\mathfrak{g}_2)| 3b_1+2b_2\leq \ell\}.
		\end{equation}
		
		\item The case $E_{6}^{(2)}$: In this case $\GO=\mathfrak{f}_4$ Then $\overline{\Lambda}_1=\omega_4$ and $\overline{\Lambda}_2=\omega_3$, $\overline{\Lambda}_3=\omega_2$ and $\overline{\Lambda}_4=\omega_1$, where $\omega_1,\dots,\omega_4$ are the fundamental weights of the Lie algebra $\mathfrak{f}_4$. Moreover, $\overline{\Lambda}_0=\omega_0$. With this notation, we have 
		\begin{equation}\label{eqn:levellweigtforE62}
		P^{\ell}(\mathfrak{g}(E_6^{(2)}))=\{b_1\omega_1+\dots+ b_4\omega_4\in P_{+}(\mathfrak{f}_4)| 2b_1+4b_2+3b_3+2b_4\leq \ell\}.
		\end{equation}
	\end{itemize}
	
	\subsection{Twisted Kac-Moody S-matrix for $A_{2n-1}^{(2)}$}\label{sec:comp6}Let $P^{\ell}(\frg(A_{2n-1}^{(1)}))^{\sigma}$ denote the set of integrable level $\ell$ weights of the Lie algebra $\frg(A_{2n-1}^{(1)})$ which are fixed by the involution $\sigma:i \rightarrow {2n-i}$ of the vertices of the Dynkin diagram of $A_{2n-1}$. It can be described explicitly as. 
	$$P^{\ell}(\frg(A_{2n-1}^{(1)}))^{\sigma}:=\{\sum_{i=1}^{n-1}b_{i}(\omega_i+\omega_{2n-i})+b_n\omega_n\in P_{+}(A_{2n-1})| 2(b_1+\dots+b_{n-1})+b_n\leq \ell\}.$$
	Since the fixed point orbit Lie algebra of $A_{2n-1}$ under the action of the involution $\sigma$ is $B_n$, there is a natural bijection between $\iota:P(A_{2n-1})^{\sigma}\rightarrow P(B_n).$ Thus under the map $\iota$ restricts to a bijection:
	\begin{equation}
	\label{eqn:idf1}
	\iota:P^{\ell}(\frg(A_{2n-1}^{(1)}))^{\sigma}\simeq P^{\ell}(\frg(D_{n+1}^{(2)})) 
	\end{equation}
	and we identify the two via $\iota$. We recall the following map between the Cartan subalgebras of $B_n$ and $C_n$ following Section 13.9 in \cite{Kac}
	$$\tau_{\kappa}:\mathfrak{h}(B_n)\rightarrow \mathfrak{h}(C_n), \  \omega_i \rightarrow \frac{a_i}{a_i^{\vee}}\omega_i, \ \mbox{for $1\leq i\leq n$},$$ and where $a_{1},\dots, a_{n}$ (respectively $a_1^{\vee},\dots, a_n^{\vee}$) denote the Coxeter (respectively dual Coxeter) label of the Lie algebra $\mathfrak{g}(A_{2n-1}^{(2)})$. 
	
	The rows and columns of the twisted Kac-Moody S-matrix are parameterized by the set $P^{\ell}(\frg({A_{2n-1}^{(2)}}))$ and $ P^{\ell}(\frg(A^{(1)}_{2n-1}))^{\sigma}$ respectively. Under the identification $\iota$, the columns will be parameterized by $P^{\ell}(\frg(D_{n+1}^{(2)}))$. We recall the following formula \cite{Kac} for the twisted Kac-Moody $S$-matrix ${\mathscr{S}}^{(\ell)}(A_{2n-1}^{(2)})$ whose $(\lambda,\mu)$-th entry is given by the formula: 
	\begin{align}
	\label{{eqn:SmatrixforA2n-12}}{\mathscr{S}}_{\lambda,\mu}^{(\ell)}(A_{2n-1}^{(2)})=
	i^{|\mathring{\Delta}_{+}|}\frac{\sqrt{2}}{|M^*/(\ell+2n)\mathring{Q}|^{\frac{1}{2}}}  \sum_{w\in \mathring{W}}\epsilon(w)\exp\bigg(-\frac{2\pi i }{\ell+2n}\kappa_{\mathfrak{g}}(w( \lambda+\overline{\rho}), \tau_{\kappa}(\mu+\overline{\rho}'))\bigg), 
	\end{align}
	where $\overline{\rho}$ (respectively $\overline{\rho}'$) denotes the sum of the fundamental weights of $C_n$ (respectively $B_n$).

	\subsection{Twisted Kac-Moody S-matrix for $D_{n+1}^{(2)}$}\label{sec:comp7}Let $P^{\ell}(\frg(D_{n+1}^{(1)}))^{\sigma}$ denote the set of integrable level $\ell$ weights of the Lie algebra $\frg(D_{n+1}^{(1)})$ which are fixed by the involution which exchanges the $n$ and $n+1$-th vertices of the Dynkin diagram of $D_{n+1}$. It can be described explicitly as. 
	$$P^{\ell}(\frg(D_{n+1}^{(1)}))^{\sigma}:=\{\sum_{i=1}^{n}b_{i}\omega_i+b_n(\omega_n+\omega_{n+1})\in P_{+}(D_{n+1})| b_1+2(b_1+\dots+b_{n})\leq \ell\}.$$
	Since the fixed point orbit Lie algebra of $D_{n+1}$ under the action of the involution $\sigma$ is $C_n$, there is a natural isomorphism between $\iota:P(D_{n+1})^{\sigma}\rightarrow P(C_n)$. Moreover, the map $\iota$ restricts to a bijection of
	\begin{equation}\label{eqn:idff4}
	\iota:P^{\ell}(\frg(D_{n+1}^{(1)}))^{\sigma}\simeq P^{\ell}(\frg(A_{2n-1}^{(2)}))
	\end{equation}and we identify the two via $\iota$. We recall the following map between the Cartan subalgebras of $C_n$ and $B_n$ following Section 13.9 in \cite{Kac}
	$$\tau_{\kappa}:\mathfrak{h}(C_n)\rightarrow \mathfrak{h}(B_n), \  \omega_i \rightarrow \frac{a_i}{a_i^{\vee}}\omega_i,$$ where $1\leq i \leq n$ and $a_{1},\dots, a_{n}$ (respectively $a_1^{\vee},\dots, a_n^{\vee}$) denote the Coxeter (respectively dual Coxeter) labels of the Lie algebra $\mathfrak{g}(D_{n+1}^{(2)})$. 
	
	The rows and columns of the twisted Kac-Moody S-matrix are parameterized by the set $P^{\ell}(\frg({D_{n+1}^{(2)}}))$ and $ P^{\ell}(\frg(D^{(1)}_{n+1}))^{\sigma}$ respectively. Under the identification $\iota$, the columns will be parameterized by $P^{\ell}(\frg(A_{2n-1}^{(2)}))$. The $(\lambda,\mu)$-th entry of the twisted Kac-Moody $S$-matrix ${\mathscr{S}}_{\lambda,\mu}^{(\ell)}(D_{n+1}^{(2)})$ is
	\begin{equation}\label{{eqn:SmatrixforDn+12}}
	{\mathscr{S}}^{(\ell)}_{\lambda,\mu}(D_{n+1}^{(2)})=
	i^{|\mathring{\Delta}_{+}|}\frac{\sqrt{2}}{|M^*/(\ell+2n)\mathring{Q}|^\frac{1}{2}}  \sum_{w\in \mathring{W}}\epsilon(w)\exp\bigg(-\frac{2\pi i }{\ell+2n}\kappa_{\mathfrak{g}}(w( \lambda+\overline{\rho}), \tau_{\kappa}(\mu+\overline{\rho}'))\bigg), 
	\end{equation}
	where $\overline{\rho}$ (respectively $\overline{\rho}'$) denotes the sum of the fundamental weights of $C_n$ (respectively $B_n$).

	\subsection{Twisted Kac-Moody S-matrix for $D_{4}^{(3)}$}\label{sec:comp8}Let $P^{\ell}(\frg(D_{4}^{(1)}))^{\sigma}$ denote the set of integrable level $\ell$ weights of the Lie algebra $\frg(D_{4}^{(1)})$ which are fixed by rotation $\sigma$ that rotates the $1\rightarrow 3\rightarrow4\rightarrow 1$-th vertices of the Dynkin diagram of $D_{4}$. It can be described explicitly as. 
	$$P^{\ell}(\frg(D_{4}^{(1)}))^{\sigma}:=\{b_1(\omega_1+\omega_3+\omega_4)+b_2\omega_2\in P_{+}(\mathfrak{g}_2)| 3b_1+2b_2\leq \ell\}.$$
	Since the fixed point orbit Lie algebra of $D_{4}$ under the action of the involution $\sigma$ is $\mathfrak{g}_2$, there is a natural isomorphism between $\iota:P(D_{4})^{\sigma}\rightarrow P(\mathfrak{g}_2)$. Moreover, the map $\iota$ restricts to a bijection of
	\begin{equation}\label{idf2}
	\iota:P^{\ell}(\frg(D_{4}^{(1)}))^{\sigma}\simeq P^{\ell}(\frg(D_4^{(3)}))
	\end{equation}
	and we identify the two via $\iota$. We recall the following map between the Cartan subalgebra of $\mathfrak{g}_2$ following Section 13.9 in \cite{Kac}
	$$\tau_{\kappa}:\mathfrak{h}(\mathfrak{g}_2)\rightarrow \mathfrak{h}(\mathfrak{g}_2), \  \omega_2 \rightarrow \frac{a_2}{a_2^{\vee}}\omega_1, \ \omega_1\rightarrow \frac{a_1}{a_1^{\vee}}\omega_2,$$ where $1\leq i \leq 2$ and $a_{1},a_{2}$ (respectively $a_1^{\vee},a_2^{\vee}$) denote the Coxeter (respectively Dual Coxeter) labels of the Lie algebra $\mathfrak{g}(D_{4}^{(3)})$. 
	
	The rows and columns of the twisted Kac-Moody S-matrix are parameterized by the set $P^{\ell}(\frg({D_{4}^{(3)}}))$ and $ P^{\ell}(\frg(D^{(1)}_{4}))^{\sigma}$ respectively. Under the identification $\iota$, the columns will be parameterized by $P^{\ell}(\frg(D_{4}^{(3)}))$. We recall the following formula \cite{Kac}:
	\begin{equation}	\label{{eqn:SmatrixforD43}}
	{\mathscr{S}}^{(\ell)}_{\lambda,\mu}(D_{4}^{(3)})=
	i^{|\mathring{\Delta}_{+}|}\frac{\sqrt{3}}{|M^*/(\ell+6)\mathring{Q}|^{\frac{1}{2}}}  \sum_{w\in \mathring{W}}\epsilon(w)\exp\bigg(-\frac{2\pi i }{\ell+6}\kappa_{\mathfrak{g}}(w( \lambda+\overline{\rho}), \tau_{\kappa}(\mu+\overline{\rho}))\bigg), 
	\end{equation}
	where $\overline{\rho}$ denotes the sum of the fundamental weights of $\mathfrak{g}_2$.

	\subsection{Twisted Kac-Moody  S-matrix for $E_{6}^{(2)}$}\label{sec:comp4}Let $P^{\ell}(\frg(E_{6}^{(1)}))^{\sigma}$ denote the set of integrable level $\ell$ weights of the Lie algebra $\frg(E_{6}^{(1)})$ which are fixed by rotation $\sigma$ that interchanges $1$-st with $5$-th;$2$-th with $4$-th and fixes the $3rd$ and $6$-th vertices of the Dynkin diagram of $E_{6}$. It can be described explicitly as. 
	$$P^{\ell}(\frg(E_{6}^{(1)}))^{\sigma}:=\{b_1(\omega_1+\omega_5)+b_2(\omega_2+\omega_4)+ b_3\omega_3+ b_4\omega_6\in P_+(\mathfrak{f}_4)| 2b_1+4b_2+3b_3+2b_4\leq \ell\}.$$
	Since the orbit Lie algebra of $E_{6}$ under the action of the involution $\sigma$ is $\mathfrak{f}_4$, there is a natural bijection between $\iota:P(E_{6})^{\sigma}\rightarrow P(\mathfrak{f}_4)$. Moreover, the map $\iota$ restricts to a bijection of
	\begin{equation}
	\label{eqn:idf2}
	\iota:P^{\ell}(\frg(E_{6}^{(1)}))^{\sigma}\simeq P^{\ell}(\frg(E_6^{(2)})) 
	\end{equation}
	and we identify the two via $\iota$. We recall the following map between the Cartan subalgebra of $\mathfrak{f}_4$ following Section 13.9 in \cite{Kac}
	$$\tau_{\kappa}:\mathfrak{h}(\mathfrak{f}_4)\rightarrow \mathfrak{h}(\mathfrak{f}_4), \  \omega_4 \rightarrow \frac{a_4}{a_4^{\vee}}\omega_1, \ \omega_3\rightarrow \frac{a_3}{a_3^{\vee}}\omega_2, \ \omega_2\rightarrow \frac{a_2}{a_2^{\vee}}\omega_3, \omega_1\rightarrow \frac{a_1}{a_1^{\vee}}\omega_4,$$ where  $a_{1},a_{2},a_3, a_4$ (respectively $a_1^{\vee},a_2^{\vee},a_3^{\vee},a_4^{\vee}$) denote the Coxeter (respectively Dual Coxeter) labels of the Lie algebra $\mathfrak{g}(E_{6}^{(2)})$. 
	
	The rows and columns of the twisted Kac-Moody S-matrix are parameterized by the set $P^{\ell}(\frg({E_{6}^{(2)}}))$ and $ P^{\ell}(\frg(E^{(1)}_{6}))^{\sigma}$ respectively. Under the identification $\iota$, the columns will be parameterized by $P^{\ell}(\frg(E_{6}^{(2)}))$. We recall the following formula from \cite{Kac}:
	\begin{align}
	\label{{eqn:SmatrixforE62}}
	{\mathscr{S}}^{(\ell)}_{\lambda,\mu}(E_{6}^{(2)})=
	\sqrt{2}i^{|\mathring{\Delta}_{+}|}|M^*/(\ell+12)\mathring{Q}|^{-\frac{1}{2}}  \sum_{w\in \mathring{W}}\epsilon(w)\exp\bigg(-\frac{2\pi i }{\ell+12}\kappa_{\mathfrak{g}}(w( \lambda+\overline{\rho}), \tau_{\kappa}(\mu+\overline{\rho})\bigg), 
	\end{align}
	where $\overline{\rho}$  denotes the sum of the fundamental weights of $\mathfrak{f}_4$.

	\subsection{Twisted Kac-Moody S-matrix for $A_{2n}^{(2)}$}In this case ${\GO}=C_n$ and the orbit Lie algebra $\mathfrak{g}_{\sigma}=C_n$. Let $\Lambda_0,\dots,\Lambda_n$ denote the affine fundamental weights of $\mathfrak{g}(A_{2n}^{(2)})$ and $\overline{\Lambda}_0,\dots \overline{\Lambda}_n$, denote the projection to $C_n$ under the normalized invariant bilinear form on  denoted by $\kappa_{\mathfrak{g}}$. Now let $\omega_0,\omega_1,\dots,\omega_n$ denote the fundamental weights of $C_n$. It turns out that for all $0\leq i\leq n$, the $\overline{\Lambda}_i=\omega_i$. 
	The set of level $\ell$ integrable highest weight representations of the affine Lie algebra $\mathfrak{g}(A_{2n}^{(2)})$ can be rewritten as follows:
	\begin{equation}\label{eqn:levellweightforA2n2}
	P^{\ell}(\mathfrak{g}(A_{2n}^{(2)}) )=\{\sum_{i=1}^nb_i\omega_i\in P_{+}(C_n)| 2(b_1+b_2+\dots+b_n)\leq \ell \}.
	\end{equation}

	\begin{remark}\label{rem:altdef}
		With the convention of the numbering of the vertices of the affine Dynkin diagram in the case $A_{2n}^{(2)}$, we get the fixed point Lie algebra is of type $B_n$ whereas $\GO=C_n$. However it is sometimes convenient to choose a different ordering of the vertices of the Dynkin diagram such that the new horizontal subalgebra is same as the fixed point algebra $B_n$. This can be done if we renumber the vertices  by reflecting about the $n$-th vertex of the affine Dynkin diagram. Under the new numbering system we get an alternate description
		\begin{equation}\label{eqn:a2n}
		P^{\ell}(\frg(A_{2n}^{(2)}))=\{\sum_{i=1}^n{\widetilde{b}_i}\omega_i\in P_{+}(B_n)| \widetilde{b}_1+\dots+\widetilde{b}_{n-1}+\frac{\widetilde{b}_{n}}{2}\leq \frac{\ell}{2}, \ \ell-\widetilde{b}_n \ \mbox{is even}\},
		\end{equation}
		where $\widetilde{\omega}_1,\dots,\widetilde{\omega}_n$ are fundamental weights of the Lie algebra of type $B_n$. In particular, we observe that $P^{1}(\frg,\sigma)=\{\omega_n\}$. 
		We refer the reader to compare it with the description of $P^{\ell}(\frg,\gamma)$ in \cite{KH}.
	\end{remark}
	
	Let $P^{\ell}(\frg(A_{2n}^{(1)}))^{\sigma}$ denote the set of integrable level $\ell$ weights of the Lie algebra $\frg(A_{2n}^{(1)})$ which are fixed by the involution $\sigma:i \rightarrow {2n+1-i}$. It can be described explicitly as. 
	$$P^{\ell}(\frg(A_{2n}^{(1)}))^{\sigma}:=\{\sum_{i=1}^{n}a_{i}(\omega_i+\omega_{2n-i})\in P_+(A_{2n})| 2(b_1+\dots+b_n)\leq \ell\}.$$ Under the map $\iota: P(A_{2n})^{\sigma}\rightarrow P(C_n)$, there is a natural bijection between $\iota: P^{\ell}(\frg(A^{(1)}_{2n}))^{\sigma}\rightarrow P^{\ell}(\frg({A_{2n}^{(2)})}).$ Hence we will identify the two using the map $\iota$.
	The rows and columns of the twisted Kac-Moody  S-matrix are parameterized by the set $P^{\ell}(\frg({A_{2n}^{(2)}}))$ and $ P^{\ell}(\frg(A^{(1)}_{2n}))^{\sigma}$ respectively. We recall the following Theorem 13.8 \cite{Kac}:
	\begin{equation}
	\label{{eqn:SmatrixforA2n2}}
	{\mathscr{S}}^{(\ell)}_{\lambda,\mu}(A_{2n}^{(2)})=i^{|\mathring{\Delta}_{+}|}(\ell+2n+1)^{-\frac{1}{2}}\sum_{w\in \mathring{W}}\epsilon(w)\exp\bigg(-\frac{2\pi i }{\ell+2n+1}\kappa_{\mathfrak{g}}(w( \lambda+\overline{\rho}), \mu+\overline{\rho})\bigg), 
	\end{equation}where $\overline{\rho}$ denotes the sum of the fundamental weights of $C_n$. Moreover, the matrix ${\mathscr{S}}^{(\ell)}(A_{2n}^{(2)})$ has the following properties \cite{Kac}.
	\begin{proposition}\label{prop:propertiesofSmatrixA2n2}
		\label{prop:unitarityforSmatrixA2n2}The matrix $\big({\mathscr{S}}^{(\ell)}(A_{2n}^{(2)})\big)_{\lambda,\mu}={\mathscr{S}}^{(\ell)}_{\lambda,\mu}({A_{2n}^{(2)}})$  is symmetric and unitary. 
	\end{proposition}
	\subsection{Relation with untwisted Kac-Moody S-matrices}\label{sec:comparision}This section is motivated by the following two observations. First, the equivariantization of a $\Gamma$-crossed modular functor is a modular functor, hence crossed $S$-matrices must be submatrices of an uncrossed $S$-matrix. Secondly, in the {\tt KAC} software, the Kac-Moody $S$-matrices and the twisted Kac-Moody $S$-matrices of type $A_{2n}^{(2)}$ are available for computational purposes.   One can compute all crossed $S$-matrices in the remaining types using {\tt KAC} as follows:

	Let $A$ be a Cartan matrix of either of the following types $A: A_{2n-1}^{(2)}, D_{n+1}^{(2)}, E_6^{(2)}, D_4^{(3)}.$ 
	We consider the Cartan matrix $A^{t}$ obtained by taking transpose of the matrix $A$. In these cases,  observe that the affine Lie algebra associated to $A^t$ is of untwisted type. They are given as: $A^t: B_n^{(1)}, C_n^{(1)}, F_4^{(1)}, G_2^{(1)}.$ If $\mathring{A}$ and $\mathring{A}^t$ denote the Cartan matrix obtained by deleting the $0$-th row and column, then we get $\frg(\mathring{A})$ and $\frg(\mathring{A}^t)$ are Langlands dual. 
	Define  the map $\tau:\mathfrak{h}(A)\rightarrow {\mathfrak{h}}(A^t)$ by the following:
	\begin{equation}\label{eqn:transposemap}
	\Lambda_i \rightarrow\frac{a_i^{\vee}}{a_i}\Lambda_i^{t}, \ \delta \rightarrow \delta^t, 
	\end{equation}
	where $\Lambda_0,\dots,\Lambda_n$ (respectively $\Lambda_0^{t},\dots,\Lambda_n^{t}$) are affine fundamental weights for the affine Lie algebra associated to $A$ (respectively $A^t$) and $\delta$ (respectively $\delta^t$) is the dual of the affine fundamental weight $\Lambda_0$ (see  Section \ref{sec:repofrep}). 
	
	We now consider the map $\mathring{\tau}: P^{\ell}(\mathfrak{g}(A))\hookrightarrow P^{\ell+h^{\vee}-h}(\mathfrak{g}(A^{t}))$ given by the formula 
	$
	\mathring{\tau}(\lambda):=\tau(\lambda+\overline{\rho})-\overline{\rho}^t,
	$
	where $\overline{\rho}$(respectively $\overline{\rho}^t$) are sum of the fundamental weights of the horizontal subalgebra of $\mathfrak{g}(A)$ and $\mathfrak{g}(A^t)$ respectively. 
	
	The affine Lie algebras of type $A^{t}$ are all untwisted affine Kac-Moody Lie algebra, hence denote by $\mathscr{S}_{\lambda,\mu}^{(\ell)}(A^t)$-the $(\lambda,\mu)$-th entry of the Kac-Moody  $S$-matrix of type $A^t$ at level $\ell$. We refer the reader to Chapter 13 in \cite{Kac} for more details. The following proposition is well known  and can be checked directly. We also refer the reader to Proposition 3.2 in \cite{AG}.
	
	\begin{proposition}\label{prop:crossedvsuncrossed}
		Let $\mathfrak{g}(A)$ be a twisted affine Kac-Moody Lie algebra associated to the Cartan matrix $A$ of type $A_{2n-1}^{(2)}, D_{n+1}^{(2)}, D_4^{(3)}, E_6^{(2)}$, then the twisted Kac-Moody $S$ matrices are related to the $S$ matrices of $\mathfrak{g}(A^t)$ by the following formula
		$${\mathscr{S}}_{\lambda,\mu}^{(\ell)}(A)=|\nu_{\frg}(\mathring{Q}^{\vee})/\mathring{Q}|^{\frac{1}{2}}\mathscr{S}_{\mathring{\tau}(\lambda),{\tau}\circ\tau_{\kappa}(\mu)}^{(\ell+h^{\vee}-h)}(A^t),$$ where $h^{\vee}$ and $h$ are the dual Coxeter numbers of the Lie algebra $\mathfrak{g}(A)$ and $\mathfrak{g}(A^t)$ respectively. Moreover, $|\nu_{\frg}(\mathring{Q}^{\vee})/\mathring{Q}|$ is two when $A$ is either $A^{(2)}_{2n-1}$, $D^{(2)}_{n+1}$ and $E^{(2)}_6$ and is three when $A$ is of type $D_4^{(3)}$.
	\end{proposition}
	\subsubsection{Some computations}\label{sec:usefulcomputations}
	We list the following facts which are very useful in computing the twisted Kac-Moody S-matrices. Using proposition \ref{prop:crossedvsuncrossed}, we can use \texttt{KAC} to compute the twisted Kac-Moody $S$ matrices. 
	\begin{itemize}
		\item The case $A_{2n-1}^{(2)}$:
		\begin{itemize}
			\item $\tau\circ\tau_{\kappa}$ is identity.
			\item $\tau(\omega_i)=\omega_i^{t}$ for $1\leq i \leq n-1$ and $\tau(\omega_n)=2\omega_n^t$.
			\item $\mathring{\tau}(\omega_0)=\omega_n^t$.
		\end{itemize}
		\item The case $D_{n+1}^{(2)}$:
		\begin{itemize}
			\item $\tau\circ\tau_{\kappa}$ is identity.
			\item $\tau(\omega_i)=\omega_i^{t}$ for $1\leq i \leq n-1$ and $\tau(\omega_n)=2\omega_n^t$.
			\item $\mathring{\tau}(\omega_0)=\omega_n^t$.
		\end{itemize}
		\item The case $D_4^{(3)}$:
		\begin{itemize}
			\item $\tau(\omega_1)=3\omega_2^t$ and $\tau(\omega_2)=\omega_1^t$.
			\item $\tau\circ\tau_{\kappa}(\omega_i)=\omega_i^{t}$ for $1\leq i\le 2$.
			\item $\mathring{\tau}(\omega_0)=2\omega_2^t$.
		\end{itemize}
		\item The case $E_6^{(2)}$: 
		\begin{itemize}
			\item $\tau(\omega_1)=2\omega_4^t$, $\tau(\omega_2)=2\omega_3^t$, $\tau(\omega_3)=\omega_2^t$, $\tau(\omega_4)=\omega_1^t$. 
			\item $\tau\circ\tau_{\kappa}(\omega_i)=\omega_i$ for $1\leq i\leq 4$.
			\item $\mathring{\tau}(\omega_0)=\omega_3^{t}+\omega_4^t$.
		\end{itemize}
	\end{itemize}
	\subsection{Dimensions using the formula}\label{sec:dimensioncrosscheck1}In this subsection, we use the Verlinde formula \eqref{conj:main1} stated in the introduction to compute dimensions of some twisted conformal blocks.  We also discuss in specific cases, how to cross check our calculations. We start with examples of \'etale cases.
	\begin{example}\label{ex:etale}
		Let $\Gamma=\langle \gamma \rangle$ be a cyclic group of order $m$ and $\frg$ is either of type $A_{n\geq 2}, D_{n\geq 4}$ or $E_6$. Let $\widetilde{C}\rightarrow C$ is an \'etale cover with Galois group $\Gamma$, then the dimension of the twisted conformal block at level one
		\begin{equation}\label{eqn:arbitrarygenusverlindeexampleetale1}
		\dim_{\mathbb{C}}\mathbb{V}_{\Gamma}^{\dagger}(\widetilde{C},C)=|P_1(\frg)^{\gamma}|(S_{0,0})^{2-2g},
		\end{equation} where $S_{0,0}$ is $\frac{1}{\sqrt{n+1}}$, $\frac{1}{2}$,$\frac{1}{\sqrt{3}}$ for $\frg=A_n,D_n, E_6$ respectively. 
		
		Consider an elliptic curve $\widetilde{C}$ with rotation of order $m$. Then $C$ is also an elliptic curve and the map $\widetilde{C}\rightarrow C$ is an \'etale cyclic cover of order $m$. We can degenerate $C$ to a nodal elliptic curve and $\widetilde{C}$ to a cycle of $\mathbb{P}^1$'s with $m$ components. In this case, by factorization, we can cross check our answer by reducing it to the case of trivial \'etale cover $\widetilde{C}$ of $\mathbb{P}^1$ along with two marked points. Since we are in the \'etale case, no crossed $S$-matrices are involved and  we can apply the untwisted Verlinde formula after using the group action to bring all the marked points $\widetilde{\bf{p}}$ in $\widetilde{C}$ in the same component.
		
	\end{example}
	\begin{remark}
		Example \ref{ex:etale} is a priori neither covered by the untwisted Verlinde formula nor by the conjectural Verlinde formula in \cite{BFS}. 
	\end{remark}
	\begin{remark}
		G. Faltings \cite{Faltingsnew}  has shown that in the untwisted case that  conformal blocks for a simply laced Lie $\frg$ algebra has dimension $|Z_{G}|^{g}$, where $Z_{G}$ is the center of the simply connected group with Lie algebra $\frg$.  This can also be seen from the Verlinde formula. The formula in Equation \eqref{eqn:arbitrarygenusverlindeexampleetale1} can be thought of a generalization of the result to the twisted case. 
	\end{remark}
	
	\begin{example}\label{ex:double}Let $\frg=A_{2r-1}^{(2)}$ and consider a double cover $\widetilde{C}\rightarrow C$ ramified at $2n$ points $\widetilde{p}$ and let $g$ be the genus of $C$. The level one weights of $A^{(2)}_{2r-1}$ are $\{0,\dot{\omega}_1\}$. Then by the Verlinde formula  \ref{conj:main1}, we get 
		$\dim_{\mathbb{C}}\mathbb{V}_{\vec{0},\mathbb{Z}/2\mathbb{Z}}^{\dagger}(\widetilde{C}, C, \widetilde{\bf p}, {\bf p})=2^{g}r^{g+n-1}.$
	\end{example}
	The computations in Section \ref{sec:usefulcomputations}, tell us that the crossed $S$-matrix or equivalently the character table is given by the following $2\times 2$ matrix.
	$$\Sigma'=\sqrt{2} \cdot \begin{pmatrix} 
	\frac{1}{{2}}& \frac{1}{{2}} \\
	\frac{1}{{2}} & -\frac{1}{{2}}
	\end{pmatrix}.$$
	
	If $C=\mathbb{P}^1$ and $\widetilde{C}$ is a genus $n-1$  curve which is double cover of $\mathbb{P}^1$ ramified at $2n$ points. We now discuss how to cross check the answer using factorization theorem and invariants of representations.
	
	In this case, we can degenerate $\widetilde{C} \rightarrow {C}$ to a double cover $\widetilde{D}\rightarrow D$, where $\widetilde{D}$ is a chain of $n-1$ elliptic curves where the end components have three ramification points and all the other components have two ramfication points. Moreover the nodes of $\widetilde{D}$ are ramification points. By the factorization theorem, we can reduce to check that the following conformal blocks associated to double covers $\widetilde{C}\rightarrow \mathbb{P}^1$ ramified at four points are one dimensional 
	\begin{enumerate}
		\item $\dim_{\mathbb{C}}\mathbb{V}_{0,0,0,0,\mathbb{Z}/2\mathbb{Z}}^{\dagger}(\widetilde{C}, \mathbb{P}^1, \widetilde{\bf{p}}, {\bf {p}})=1.$
		\item $\dim_{\mathbb{C}}\mathbb{V}_{0,0,\dot{\omega}_1,\dot{\omega}_1,\mathbb{Z}/2\mathbb{Z}}^{\dagger}(\widetilde{C}, \mathbb{P}^1, \widetilde{\bf{p}}, {\bf {p}})=1.$
		\item $\dim_{\mathbb{C}}\mathbb{V}_{0,0,0,\dot{\omega}_1,\mathbb{Z}/2\mathbb{Z}}^{\dagger}(\widetilde{C}, \mathbb{P}^1, \widetilde{\bf{p}}, {\bf {p}})=1.$
	\end{enumerate}
	
	We can now degenerate the double cover  $\widetilde{C}$ of  $\mathbb{P}^1$ with four ramification points by degenerating the $\mathbb{P}^1$ to a reducible $\mathbb{P}^1$ with two components meeting at one point, which is \'etale. Moreover each component has two ramification points. Then normalizing we get two disjoint copies of $\mathbb{P}^1$ with a double cover ramified at two points. Hence by factorization, we are reduced to check the following:
	\begin{equation}\label{eqn:check1}
	\dim_{\mathbb{C}}\mathbb{V}^{\dagger}_{0,0,\omega_i,\mathbb{Z}/2\mathbb{Z}}(\mathbb{P}^1, \mathbb{P}^1, {\widetilde{\bf{p}}}, {\bf p})=\begin{cases}
	1 & \mbox{if $i$ is even }\\
	0 & \mbox{otherwise},
	\end{cases}
		\end{equation}
	where  $\omega_i$ is the $i$-th fundamental weight of $A_{2r-1}$ attached to an \'etale point and also 
		\begin{equation}\label{eqn:check2}
	\dim_{\mathbb{C}}\mathbb{V}^{\dagger}_{0,\dot{\omega}_1,\omega_i,\mathbb{Z}/2\mathbb{Z}}(\mathbb{P}^1, \mathbb{P}^1, {\widetilde{\bf{p}}}, {\bf p})=\begin{cases}
	1 & \mbox{if $i$ is odd}\\
	0 & \mbox{otherwise}.
	\end{cases}
		\end{equation}
	This can be verified directly from the fact that these conformal blocks are isomorphic to the space of invariant $\operatorname{Hom}_{\mathfrak{sp}(2r)}(\Lambda^{i}\mathbb{C}^{2r},\mathbb{C})$ (respectively $\operatorname{Hom}_{\mathfrak{sp}(2r)}(\mathbb{C}^{2r}\otimes \Lambda^{i}\mathbb{C}^{2r},\mathbb{C})$) which has dimensions one or zero depending on the parity of $i$. Further the dimension calculation for the ramified double cover $\mathbb{P}^1\rightarrow \mathbb{P}^1$ obtained in Equations \eqref{eqn:check1} and \eqref{eqn:check2} also agrees with the computations from the twisted Verlinde formula. The general case where genus of $C$ is positive can be approached similarly. 
	\begin{example}\label{ex:a2r}
		Consider the Lie algebra $A^{(2)}_{2r}$. The level one weights of $A^{(2)}_{2r}$ is just the $r$-th fundamental weight $\dot{\omega}_r$ of the Lie algebra of type $B_r$. Consider $\mathbb{P}^1$ as a ramified double cover of $\mathbb{P}^1$. Then by the Verlinde formula  \ref{conj:main1}, we get that the dimension of 
		$\mathbb{V}^{\dagger}_{\dot{\omega}_r,\dot{\omega}_r,\omega_i,\mathbb{Z}/2\mathbb{Z}}(\mathbb{P}^1, \mathbb{P}^1, {\widetilde{\bf{p}}}, {\bf p})$ is one.
			
		Here $\omega_i$ is the $i$-th fundamental weight of $A_{2r}$ attached to an \'etale point. If $i=0$, Example \ref{ex:a2r} can be crossed checked by Lemma \ref{lem:linebundlecase}. For all $i$'s,  this example can be crossed checked similarly as in Example \ref{ex:double} by embedding into the space of invariants.
	\end{example}

	\begin{example}\label{ex:noncheckable}	Consider a connected ramified Galois cover $\widetilde{C}\rightarrow \mathbb{P}^1$ of order three with three marked points $\widetilde{\bf p}$ with ramification given by $1\in \ZBbb/3\ZBbb=\{0,1,2\}$ at each marked point. We consider the triple $\vec{0}=(0,0,0)$ of weights in $P^{\ell}(D_4^{(3)})$ and consider the twisted conformal block $\mathbb{V}^{\dagger}_{\vec{0},\mathbb{Z}/3\mathbb{Z}}(\widetilde{C}, \mathbb{P}^1, {\widetilde{\bf {p}}}, {\bf{p}})$ at level $\ell$. By Theorem \ref{thm:oneofmain}, twisted conformal blocks in this case define a $\ZBbb/3\ZBbb$-crossed modular fusion category $\Ccal(D_4, \mathbb{Z}/3,\ell)=\Ccal_0\oplus\Ccal_1\oplus\Ccal_2.$
		
		Since we are looking at a $\ZBbb/3\ZBbb$-cover with ramification $1\in \ZBbb/3\ZBbb$ at each marked point, the weight $0\in P^{\ell}(D_4^{(3)})$ corresponds to a simple object say $A$ in the component $\Ccal_1$. Note that the object $A\dotimes A\dotimes A$ (where $\dotimes$ denotes the fusion product in the $\ZBbb/3\ZBbb$-crossed modular fusion category $\Ccal(D_4,\mathbb{Z}/3,\ell)$) lies in the identity component $\Ccal_0$. Hence by Theorem \ref{thm:oneofmain}(3), the rank of the bundle of twisted conformal blocks with weights $\vec0$ as equals the fusion coefficient $$\dim\Hom(\unit,A\otimes A\otimes A)=\dim\Hom(A^*, A\otimes A)=\nu_{A,A}^{A^*}$$ 
		in the Grothendieck ring $K(\Ccal(D_4,\ZBbb/3\ZBbb,\ell))$ of the $\ZBbb/3\ZBbb$-crossed modular fusion category and {\em not in the twisted fusion ring $\mathcal{R}_{\ell}(D_4,1)$}.
		Here $\unit\in \Ccal_0$ is the unit object corresponding to the weight 0 of the {\em untwisted $D_4$ at level $\ell$} i.e. an element of $P_{\ell}(D_4)$.  
		%
		%
		The fusion coefficients in the ring $\mathcal{R}_{\ell, \mathbb{Z}/3\mathbb{Z}}(D_4)$ is determined by Theorem \ref{conj:main1} and Proposition \ref{prop:charactertableviaSmatrix}:		
		\begin{itemize}
			\item Let $\ell=1$. The Verlinde formula tells us  $\dim \mathbb{V}^{\dagger}_{\vec{0},\mathbb{Z}/3\mathbb{Z}}(\widetilde{C}, \mathbb{P}^1, {\widetilde{\bf {p}}}, {\bf{p}})=2$. We use the fact that $0$ is the only fixed point of $P^1(D_{4}^{(1)})$ under the diagram automorphism and $S_{0,0}=\frac{1}{2}$. 
			\item Let $\ell=2$. By Proposition \ref{prop:crossedvsuncrossed}, Section \ref{sec:usefulcomputations} and using \texttt{KAC} software, we get the crossed $S$-matrix:
			$$S^{\gamma}=\sqrt{3}\cdot \begin{pmatrix}
			\frac{1}{\sqrt{6}} & \frac{1}{\sqrt{6}}  \\
			\frac{1}{\sqrt{6}} & -\frac{1}{\sqrt{6}} 
			\end{pmatrix}.$$
			The fixed points of $P^2(D_{4}^{(1)})$ under the diagram automorphism are $\{0,\omega_2\}$ and from the uncrossed $S$-matrix we get $S_{0,0}=\frac{1}{4\sqrt{2}}$ and $S_{0,\omega_2}=\frac{1}{2\sqrt{2}}$. Now the Verlinde formula \ref{conj:main1} implies that 
			$\dim_{\mathbb{C}}\mathbb{V}^{\dagger}_{\vec{0},\mathbb{Z}/3\mathbb{Z}}(\widetilde{C}, \mathbb{P}^1, {\widetilde{\bf {p}}}, {\bf{p}})=3.$
		\end{itemize}
		
	\end{example}
	\begin{remark}
		Since the curve $\widetilde{C}$ has genus one, $\mathbb{V}^{\dagger}_{\vec{0},\mathbb{Z}/3\mathbb{Z}}(\widetilde{C}, \mathbb{P}^1, {\widetilde{\bf {p}}}, {\bf{p}})$ do not embed in the space of invariants of tensor product representations. Hence the dimension of invariants do not give any natural upper bounds in these cases. This is in stark contrast with the situation for untwisted conformal blocks. Also since the moduli space of $\Gamma$-covers of $\mathbb{P}^1$ with three marked points have genus zero, we do not have a way to degenerate to simpler cases using factorization. 	Hence unlike the other example considered in this section, we do not have any other way to cross check the calculations in Example \ref{ex:noncheckable}.
	\end{remark}
	
	\begin{example}\label{ex:checkablenontrivial}Consider a connected ramified Galois cover $\widetilde{C}\rightarrow \mathbb{P}^1$ of order three with three marked points $\widetilde{\bf p}$. Here assume that the first two marked points are ramified and the third point is \'etale. Hence $\widetilde{C}$ is again $\mathbb{P}^1$. We consider the triple $\vec{0}=(\dot{0},\dot{0},0)$ of weights in $(P^{\ell}(D_4^{(3)}))^2\times P^{\ell}(D_4^{(1)})$ and consider the twisted conformal block $\mathbb{V}^{\dagger}_{\vec{0},\mathbb{Z}/3\mathbb{Z}}(\mathbb{P}^1, \mathbb{P}^1, {\widetilde{\bf {p}}}, {\bf{p}})$ at level $\ell$. 
		\begin{itemize}
			\item Let $\ell=1$. The Verlinde formula \ref{conj:main1}, tell us that the  dimension of the twisted conformal block $\mathbb{V}^{\dagger}_{\vec{0},\mathbb{Z}/3\mathbb{Z}}(\mathbb{P}^1, \mathbb{P}^1, {\widetilde{\bf {p}}}, {\bf{p}})$ is one.
			\item Let $\ell=2$. The calculation of the crossed $S$-matrix in Example \ref{ex:noncheckable} and the Verlinde formula \ref{conj:main1} implies that 
			$\dim_{\mathbb{C}}\mathbb{V}^{\dagger}_{\vec{0},\mathbb{Z}/3\mathbb{Z}}(\mathbb{P}^1, \mathbb{P}^1, {\widetilde{\bf {p}}}, {\bf{p}})=1.$
		\end{itemize}
		These numbers agree with the dimensions coming from Lemma \ref{lem:linebundlecase} and hence cross checks the Verlinde formula. 
	\end{example}
	\begin{remark}
		Examples \ref{ex:noncheckable} and \ref{ex:checkablenontrivial} provide  non-trivial examples of the Verlinde formula for twisted conformal blocks that is not covered by the conjectural Verlinde formula in \cite{BFS}. 
	\end{remark}

	\section*{Acknowledgements}
We sincerely thank Patrick Brosnan, Najmuddin Fakhruddin, Jochen Heinloth, Shrawan Kumar, Arvind Nair, Christian Pauly, Michael Rapoport and Catharina Stroppel for useful discussions. We acknowledge key communications with Christoph Schweigert regarding \cite{BFS} and with the \texttt{KAC} software. We also thank Vladimir Drinfeld for posing a question to the first named author that lead to Proposition \ref{prop:rigidity in weakly fusion categories}. The second named author thanks the Max-Planck Institute for Mathematics in Bonn for its hospitality.
\bibliographystyle{amsplain}
\bibliography{papers1}
\end{document}